\documentclass{preprint}

\usepackage[marginparwidth=2.4cm, marginparsep=1mm]{geometry}
\usepackage[full]{textcomp}
\usepackage[osf]{newtxtext}
\usepackage[title,titletoc]{appendix}

\usepackage{ifthen} 
\usepackage{hyperref}
\usepackage{breakurl}
\usepackage{amsmath}
\usepackage{amssymb}
\usepackage{mhequ}
\usepackage{mhenvs}
\usepackage{mhsymb}
\usepackage{mathrsfs}
\usepackage{microtype}
\usepackage{wasysym}
\usepackage{centernot}
\usepackage{booktabs}
\usepackage{tikz}
\usetikzlibrary{snakes}
\usepackage{colortbl}
\usepackage{scalerel}

\usepackage{hyperref}
\usepackage{amsmath}
\usepackage{amssymb}
\usepackage{mathrsfs}
\usepackage{mathtools}
\usepackage{bbm}
\usepackage{enumitem}
\usepackage{graphicx} 
\usepackage{marginnote}
\usepackage{tikz}
\usepackage{framed}
\usepackage{comment}
\usepackage{esint}
\usepackage{xcolor}
\usepackage{stmaryrd}
\usepackage{subfig}
\usepackage{longtable}

\allowdisplaybreaks

\renewcommand\emptyset{\varnothing}

\makeatletter
\def\mathcenterto#1#2{\mathclap{\phantom{#1}\mathclap{#2}}\phantom{#1}}
\let\old@widetilde\widetilde
\def\widetildeto#1#2{\mathcenterto{#2}{\old@widetilde{\mathcenterto{#1}{#2\,}}}}
\let\old@widehat\widehat
\def\widehatto#1#2{\mathcenterto{#2}{\old@widehat{\mathcenterto{#1}{#2\,}}}}
\def\mywidetilde#1{\widetildeto{G}{#1}}
\makeatother

\newcommand{\be}{\beta}
\newcommand{\ga}{\gamma}
\newcommand{\de}{\delta}
\newcommand{\emezo}{\mathfrak{e}}
\newcommand{\fa}{\mathfrak{a}}

\newcommand{\fm}{\mathfrak{m}}

\newcommand{\la}{\lambda}
\newcommand{\om}{\omega}

\newcommand{\sinfty}{{\scaleto{\infty}{1.5pt}}}
\newcommand{\sone}{{\scaleto{1}{3pt}}}

\def\id{\mathrm{id}}
\def\Dom{\mathrm{Dom}}

\def\lift{{\scaleto{\mathrm{lift}}{4pt}}}

\def\kone{\mathbf{k}}

\DeclareMathOperator{\uDelta}{\underline{\Delta}}

\def\Cd{\mathtt{C}}

\newcommand{\Err}{\mathtt{Err}}

\newcommand{\ee}{\mathbbm{e}}

\newcommand{\bM}{\mathbf{M}}

\def\CCV{\mathbb{V}}
\def\CCG{\mathbb{G}}

\def\var{\mathtt{var}}
\def\nil{\mathtt{nil}}

\def\L{\mathtt{L}}

\newcommand{\CGG}{\mathfrak{c}_{\ga}} 					

\newcommand{\SH}{\mathscr{H}}

\newcommand{\SD}{\mathscr{D}}

\newcommand{\SZ}{\mathscr{Z}}
\newcommand{\SL}{\mathscr{L}}
\newcommand{\SF}{\mathscr{F}}

\renewcommand{\C}{\mathcal{C}}
\newcommand{\bC}{\mathbf{C}}
\renewcommand{\CL}{\mathcal{L}}
\renewcommand{\CK}{\mathcal{K}}
\renewcommand{\CO}{\mathcal{O}}
\renewcommand{\CB}{\mathcal{B}}
\renewcommand{\CS}{\mathcal{S}}
\renewcommand{\CM}{\mathcal{M}}
\renewcommand{\CD}{\mathcal{D}}
\renewcommand{\CR}{\mathcal{R}}
\renewcommand{\CC}{\C}
\renewcommand{\CE}{\symbol{\mathcal{E}}}
\renewcommand{\CX}{\mathcal{X}}
\renewcommand{\CI}{\symbol{\mathcal{I}}}
\renewcommand{\CU}{\mathcal{U}}
\renewcommand{\CA}{\mathcal{A}}
\renewcommand{\CT}{\mathcal{T}}
\renewcommand{\CG}{\mathcal{G}}
\renewcommand{\CV}{\mathcal{V}}
\renewcommand{\CW}{\mathcal{W}}
\renewcommand{\CQ}{\mathcal{Q}}
\renewcommand{\CP}{\mathcal{P}}

\renewcommand{\E}{\mathbb{E}}
\renewcommand{\P}{\mathbb{P}}
\newcommand{\T}{\mathbb{T}}

\renewcommand{\d}{\mathrm{d}}
\newcommand{\poly}{\mathrm{poly}}
\newcommand{\ex}{\mathrm{ex}}

\newcommand{\SK}{\mathscr{K}}
\newcommand{\SR}{\mathscr{R}}

\newcommand{\ST}{\mathscr{T}}
\newcommand{\s}{\mathfrak{s}}

\renewcommand{\eps}{\varepsilon}

\def\d{\mathrm{d}}

\newcommand{\Kg}{K_\ga}							
\newcommand{\KK}{\mathfrak{K}}
\newcommand{\LN}{\T^3_N} 
\newcommand{\Le}{\T_{\eps}^3}	
\newcommand{\Lattice}{\Lambda_{\eps}}	
\newcommand{\lbg}{\lambda_{\ga}} 							
\renewcommand{\ae}{*_\eps} 							
\newcommand{\fc}{\mathfrak{c}} 
 					
\newcommand{\ff}{\mathfrak{f}}

\renewcommand{\fC}{\mathfrak{C}}

\newcommand\PPi{\boldsymbol{\Pi}}

\def\un#1{\underline{#1}}

\def\cadlag{c\`{a}dl\`{a}g }

\def\M{\mathfrak{M}}
\def\fN{\mathfrak{N}}
\def\${\vert\hspace{-1.5pt}\vert\hspace{-1.5pt}\vert}

\colorlet{symbols}{blue}
\colorlet{testcolor}{green!60!black}

\def\symbol#1{{\color{symbols}{#1}}}
\def\1{\mathbf{{1}}}
\def\X{\symbol{X}}

\newcommand{\blue}[1]{\color{blue}#1\color{black}}

\def\blueOne{\symbol{\1}}
\def\blueXi{\blue\Xi}

\makeatletter
\newcommand*\bigcdot{\mathpalette\bigcdot@{.5}}
\newcommand*\bigcdot@[2]{\mathbin{\vcenter{\hbox{\scalebox{#2}{$\m@th#1\bullet$}}}}}
\makeatother

\usetikzlibrary{shapes}
\usetikzlibrary{shapes.misc}
\usetikzlibrary{shapes.symbols}
\usetikzlibrary{decorations}
\usetikzlibrary{decorations.markings}
\usetikzlibrary{decorations.pathreplacing}
\usetikzlibrary{calc,intersections,through,backgrounds}

\colorlet{testcolor}{green!60!black}

\def\drawx{\draw[-,solid] (-3pt,-3pt) -- (3pt,3pt);\draw[-,solid] (-3pt,3pt) -- (3pt,-3pt);}
\tikzset{
	root/.style={circle,fill=testcolor,inner sep=0pt, minimum size=2mm},
	dot/.style={circle,fill=black,inner sep=0pt, minimum size=1mm},
	empty_dot/.style={circle, draw=black,fill=white,inner sep=0pt, minimum size=1mm},
	small_dot/.style={circle,fill=black,inner sep=0pt, minimum size=0.9mm},
	var/.style={circle,fill=white,draw=black,inner sep=0pt, minimum size=2mm},
	var_blue/.style={circle,fill=blue!20,draw=blue,inner sep=0pt, minimum size=2mm},
	var_very_blue/.style={circle,fill=blue,draw=blue,inner sep=0pt, minimum size=2mm},
	var_very_pink/.style={circle,fill=red!20, draw=red,inner sep=0pt, minimum size=2mm},
	var_red_square/.style={regular polygon,regular polygon sides=4, draw, fill=red!20, draw=red, inner sep=0pt, minimum size=2.5mm, shape border rotate=45},
	var_red_triangle/.style={regular polygon,regular polygon sides=3, draw, fill=red!20, draw=red, inner sep=0pt, minimum size=3mm, shape border rotate=180},
	var_red/.style={circle,fill=red!20,draw=red,inner sep=0pt, minimum size=2mm},
	circ/.style={circle,fill=white,draw=black,inner sep=0pt, minimum size=1.2mm},
	keps/.style= {semithick,shorten >=1pt,shorten <=1pt,->},
	keps-random/.style= {semithick,shorten >=1pt,shorten <=1pt,->, decorate, decoration = {zigzag, segment length=1mm, amplitude=1mm}},
	kepsdot/.style= {semithick,densely dashed,shorten >=1pt,shorten <=1pt,->},
	dotred/.style={circle,fill=black!50,inner sep=0pt, minimum size=2mm},
	generic/.style={semithick,shorten >=1pt,shorten <=1pt},
	gepsilon/.style={semithick,shorten >=1pt,shorten <=1pt,densely dashed},
	rootlab/.style={font=\scriptsize, transform canvas={yshift=-0.23cm}},
	dist/.style={ultra thick,draw=testcolor,shorten >=1pt,shorten <=1pt},
	testfcn/.style={ultra thick,testcolor,shorten >=1pt,shorten <=1pt,<-},
	testfcnx/.style={ultra thick,testcolor,shorten >=1pt,shorten <=1pt,<-,
		postaction={decorate,decoration={markings,mark=at position 0.6 with {\drawx}}}},
	kprime/.style={semithick,shorten >=1pt,shorten <=1pt,dotted,->},
	kprimex/.style={semithick,shorten >=1pt,shorten <=1pt,densely dashed,->,
		postaction={decorate,decoration={markings,mark=at position 0.4 with {\drawx}}}},
	kernel/.style={semithick,shorten >=1pt,shorten <=1pt,->},
	multx/.style={shorten >=1pt,shorten <=1pt,
		postaction={decorate,decoration={markings,mark=at position 0.5 with {\drawx}}}},
	kernelx/.style={semithick,shorten >=1pt,shorten <=1pt,->,
		postaction={decorate,decoration={markings,mark=at position 0.4 with {\drawx}}}},
	kepsilon/.style={semithick,shorten >=1pt,shorten <=1pt,densely dashed,->},
	kernel1/.style={->,semithick,shorten >=1pt,shorten <=1pt,postaction={decorate,decoration={markings,mark=at position 0.45 with {\draw[-] (0,-0.1) -- (0,0.1);}}}},
	kernel2/.style={->,semithick,shorten >=1pt,shorten <=1pt,postaction={decorate,decoration={markings,mark=at position 0.45 with {\draw[-] (0.05,-0.1) -- (0.05,0.1);\draw[-] (-0.05,-0.1) -- (-0.05,0.1);}}}},
	kernelBig/.style={semithick,shorten >=1pt,shorten <=1pt,decorate, decoration={zigzag,amplitude=1.5pt,segment length = 3pt,pre length=2pt,post length=2pt}},
	rho/.style={dotted,semithick,shorten >=1pt,shorten <=1pt},
	renorm/.style={shape=circle,fill=white,inner sep=1pt},
	labl/.style={shape=rectangle,fill=white,inner sep=1pt},
	xi/.style={circle,fill=symbols!10,draw=symbols,inner sep=0pt,minimum size=1.2mm},
	xix/.style={crosscircle,fill=symbols!10,draw=symbols,inner sep=0pt,minimum size=1.2mm},
	xib/.style={circle,fill=symbols!10,draw=symbols,inner sep=0pt,minimum size=1.6mm},
	xibx/.style={crosscircle,fill=symbols!10,draw=symbols,inner sep=0pt,minimum size=1.6mm},
	not/.style={circle,fill=symbols,draw=symbols,inner sep=0pt,minimum size=0.5mm},
	>=stealth,
	}

\makeatletter
\def\DeclareSymbol#1#2#3{\expandafter\gdef\csname MH@symb@#1\endcsname{\tikz[baseline=#2,scale=0.15,draw=symbols]{#3}}\expandafter\gdef\csname MH@symb@#1s\endcsname{\scalebox{0.7}{\tikz[baseline=#2,scale=0.15,draw=symbols]{#3}}}}
\def\<#1>{\csname MH@symb@#1\endcsname}
\makeatother

\DeclareSymbol{1}{0}{\draw[white] (-.4,0) -- (.4,0); \draw (0,0)  -- (0,1.4) node[dot] {};}
\DeclareSymbol{2}{0}{\draw (-0.5,1.3) node[dot] {} -- (0,0) -- (0.5,1.3) node[dot] {};}
\DeclareSymbol{3}{0}{\draw (0,0) -- (0,1.3) node[dot] {}; \draw (-.7,1) node[dot] {} -- (0,0) -- (.7,1) node[dot] {};}
\DeclareSymbol{30}{-3}{\draw (0,0) -- (0,-1); \draw (0,0) -- (0,1.3) node[dot] {}; \draw (-.7,1) node[dot] {} -- (0,0) -- (.7,1) node[dot] {};}
\DeclareSymbol{31}{-3}{\draw (0,0) -- (0,-1) -- (1,0) node[dot] {}; \draw (0,0) -- (0,1.3) node[dot] {}; \draw (-.7,1) node[dot] {} -- (0,0) -- (.7,1) node[dot] {};}
\DeclareSymbol{32}{-3}{\draw (0,0) -- (0,-1) -- (1,0) node[dot] {}; \draw (0,0) -- (0,-1) -- (-1,0) node[dot] {}; \draw (0,0) -- (0,1.3) node[dot] {}; \draw (-.7,1) node[dot] {} -- (0,0) -- (.7,1) node[dot] {};}
\DeclareSymbol{20}{-3}{\draw (0,0) -- (0,-1);\draw (-.7,1) node[dot] {} -- (0,0) -- (.7,1) node[dot] {};}
\DeclareSymbol{22}{-3}{\draw (0,0.3) -- (0,-1) -- (1,0) node[dot] {}; \draw (0,0.3) -- (0,-1) -- (-1,0) node[dot] {};\draw (-.7,1) node[dot] {} -- (0,0.3) -- (.7,1) node[dot] {};}

\usetikzlibrary{shapes}
\usetikzlibrary{shapes.misc}
\usetikzlibrary{shapes.symbols}
\usetikzlibrary{decorations}
\usetikzlibrary{decorations.markings}
\usetikzlibrary{decorations.pathreplacing}
\usetikzlibrary{calc,intersections,through,backgrounds}

\def\drawx{\draw[-,solid] (-3pt,-3pt) -- (3pt,3pt);\draw[-,solid] (-3pt,3pt) -- (3pt,-3pt);}
\tikzset{
	root/.style={circle,fill=testcolor,inner sep=0pt, minimum size=2mm},
	dot/.style={circle,fill=black,inner sep=0pt, minimum size=1mm},
	var/.style={circle,fill=black!10,draw=black,inner sep=0pt, minimum size=2mm},
	circ/.style={circle,fill=white,draw=black,inner sep=0pt, minimum size=1.2mm},
	dotred/.style={circle,fill=black!50,inner sep=0pt, minimum size=2mm},
	generic/.style={semithick,shorten >=1pt,shorten <=1pt},
	gepsilon/.style={semithick,shorten >=1pt,shorten <=1pt,densely dashed},
	dist/.style={ultra thick,draw=testcolor,shorten >=1pt,shorten <=1pt},
	testfcn/.style={ultra thick,testcolor,shorten >=1pt,shorten <=1pt,<-},
	testfcnx/.style={ultra thick,testcolor,shorten >=1pt,shorten <=1pt,<-,
		postaction={decorate,decoration={markings,mark=at position 0.6 with {\drawx}}}},
	kprime/.style={semithick,shorten >=1pt,shorten <=1pt,dotted,->},
	kprimex/.style={semithick,shorten >=1pt,shorten <=1pt,densely dashed,->,
		postaction={decorate,decoration={markings,mark=at position 0.4 with {\drawx}}}},
	kernel/.style={semithick,shorten >=1pt,shorten <=1pt,->},
	multx/.style={shorten >=1pt,shorten <=1pt,
		postaction={decorate,decoration={markings,mark=at position 0.5 with {\drawx}}}},
	kernelx/.style={semithick,shorten >=1pt,shorten <=1pt,->,
		postaction={decorate,decoration={markings,mark=at position 0.4 with {\drawx}}}},
	kepsilon/.style={semithick,shorten >=1pt,shorten <=1pt,densely dashed,->},
	kernel1/.style={->,semithick,shorten >=1pt,shorten <=1pt,postaction={decorate,decoration={markings,mark=at position 0.45 with {\draw[-] (0,-0.1) -- (0,0.1);}}}},
	kernel2/.style={->,semithick,shorten >=1pt,shorten <=1pt,postaction={decorate,decoration={markings,mark=at position 0.45 with {\draw[-] (0.05,-0.1) -- (0.05,0.1);\draw[-] (-0.05,-0.1) -- (-0.05,0.1);}}}},
	kernelBig/.style={semithick,shorten >=1pt,shorten <=1pt,decorate, decoration={zigzag,amplitude=1.5pt,segment length = 3pt,pre length=2pt,post length=2pt}},
	rho/.style={dotted,semithick,shorten >=1pt,shorten <=1pt},
	renorm/.style={shape=circle,fill=white,inner sep=1pt},
	labl/.style={shape=rectangle,fill=white,inner sep=1pt},
	xi/.style={circle,fill=symbols!10,draw=symbols,inner sep=0pt,minimum size=1.2mm},
	xix/.style={crosscircle,fill=symbols!10,draw=symbols,inner sep=0pt,minimum size=1.2mm},
	xib/.style={circle,fill=symbols!10,draw=symbols,inner sep=0pt,minimum size=1.6mm},
	xibx/.style={crosscircle,fill=symbols!10,draw=symbols,inner sep=0pt,minimum size=1.6mm},
	not/.style={circle,fill=symbols,draw=symbols,inner sep=0pt,minimum size=0.5mm},
	>=stealth,
	}

\makeatletter
\def\DeclareSymbol#1#2#3{\expandafter\gdef\csname MH@symb@#1\endcsname{\tikz[baseline=#2,scale=0.15,draw=symbols]{#3}}\expandafter\gdef\csname MH@symb@#1s\endcsname{\scalebox{0.7}{\tikz[baseline=#2,scale=0.15,draw=symbols]{#3}}}}
\def\<#1>{\csname MH@symb@#1\endcsname}
\makeatother

\DeclareSymbol{1}{0}{\draw[white] (-.4,0) -- (.4,0); \draw[black] (0,0)  -- (0,1.2) node[dot] {};}
\DeclareSymbol{2}{0}{\draw[black] (-0.5,1.2) node[dot] {} -- (0,0) -- (0.5,1.2) node[dot] {};}
\DeclareSymbol{3}{0}{\draw[black] (0,0) -- (0,1.2) node[dot] {}; \draw[black] (-.7,1) node[dot] {} -- (0,0) -- (.7,1) node[dot] {};}
\DeclareSymbol{30}{-3}{\draw[black] (0,0) -- (0,-1); \draw[black] (0,0) -- (0,1.2) node[dot] {}; \draw[black] (-.7,1) node[dot] {} -- (0,0) -- (.7,1) node[dot] {};}
\DeclareSymbol{31}{-3}{\draw[black] (0,0) -- (0,-1) -- (1,0) node[dot] {}; \draw[black] (0,0) -- (0,1.2) node[dot] {}; \draw[black] (-.7,1) node[dot] {} -- (0,0) -- (.7,1) node[dot] {};}
\DeclareSymbol{32}{-3}{\draw[black] (0,0) -- (0,-1) -- (1,0) node[dot] {}; \draw[black] (0,0) -- (0,-1) -- (-1,0) node[dot] {}; \draw[black] (0,0) -- (0,1.2) node[dot] {}; \draw[black] (-.7,1) node[dot] {} -- (0,0) -- (.7,1) node[dot] {};}
\DeclareSymbol{20}{-3}{\draw[black] (0,0) -- (0,-1);\draw[black] (-.7,1) node[dot] {} -- (0,0) -- (.7,1) node[dot] {};}
\DeclareSymbol{22}{-3}{\draw[black] (0,0.3) -- (0,-1) -- (1,0) node[dot] {}; \draw[black] (0,0.3) -- (0,-1) -- (-1,0) node[dot] {};\draw[black] (-.7,1) node[dot] {} -- (0,0.3) -- (.7,1) node[dot] {};}
\DeclareSymbol{4}{0}{\draw[black] (0,0) -- (-.4,1.3) node[dot] {}; \draw[black] (0,0) -- (.4,1.3) node[dot] {}; \draw[black] (-1.1,1) node[dot] {} -- (0,0) -- (1.1,1) node[dot] {};}
\DeclareSymbol{5}{0}{\draw[black] (0,0) -- (0,1.5) node[dot] {}; \draw[black] (0,0) -- (1.4, 0.9) node[dot] {}; \draw[black] (0,0) -- (-1.4, 0.9) node[dot] {}; \draw[black] (-.7,1.25) node[dot] {} -- (0,0) -- (.7,1.25) node[dot] {};}
\DeclareSymbol{10e}{-3}{\draw[thick,double,black] (0,0) -- (0,-1); \draw[black] (0,0) -- (-0.8,1) node[dot] {};}
\DeclareSymbol{20e}{-3}{\draw[thick,double,black] (0,0) -- (0,-1); \draw[black] (-.7,1) node[dot] {} -- (0,0) -- (.7,1) node[dot] {};}
\DeclareSymbol{30e}{-3}{\draw[thick,double,black] (0,0) -- (0,-1); \draw[black] (0,0) -- (0,1.2) node[dot] {}; \draw[black] (-.7,1) node[dot] {} -- (0,0) -- (.7,1) node[dot] {};}
\DeclareSymbol{40e}{-3}{\draw[thick,double,black] (0,0) -- (0,-1); \draw[black] (0,0) -- (-.4,1.3) node[dot] {}; \draw[black] (0,0) -- (.4,1.3) node[dot] {}; \draw[black] (-1.1,1) node[dot] {} -- (0,0) -- (1.1,1) node[dot] {};}
\DeclareSymbol{50e}{-3}{\draw[thick,double,black] (0,0) -- (0,-1); \draw[black] (0,0) -- (0,1.5) node[dot] {}; \draw[black] (0,0) -- (1.4, 0.9) node[dot] {}; \draw[black] (0,0) -- (-1.4, 0.9) node[dot] {}; \draw[black] (-.7,1.25) node[dot] {} -- (0,0) -- (.7,1.25) node[dot] {};}

\DeclareSymbol{oneNode}{-3}{\node[dot][blue] {};}
\DeclareSymbol{1b}{0}{\draw[white] (-.4,0) -- (.4,0); \draw[blue] (0,0)  -- (0,1.4) node[dot][blue] {};}
\DeclareSymbol{2b}{0}{\draw[blue] (-0.5,1.2) node[dot][blue] {} -- (0,0) -- (0.5,1.2) node[dot][blue] {};}
\DeclareSymbol{3b}{0}{\draw[blue] (0,0) -- (0,1.2) node[dot][blue] {}; \draw[blue] (-.7,1) node[dot][blue] {} -- (0,0) -- (.7,1) node[dot][blue] {};}
\DeclareSymbol{30b}{-3}{\draw[blue] (0,0) -- (0,-1); \draw[blue] (0,0) -- (0,1.2) node[dot][blue] {}; \draw[blue] (-.7,1) node[dot][blue] {} -- (0,0) -- (.7,1) node[dot][blue] {};}
\DeclareSymbol{30b30b3}{-3}{\draw[blue] (0,-1.5) -- (0,0) node[dot][blue] {}; \draw[blue] (0,-1.5) -- (-1.3,-0.8) node[dot][blue] {}; \draw[blue] (0,-1.5) -- (1.3,-0.8) node[dot][blue] {}; \draw[blue] (-1,0.2) -- (0,-1.5); \draw[blue] (-2,1.1) node[dot][blue] {} -- (-1,0.2) -- (-.4,1.2) node[dot][blue] {}; \draw[blue] (-1,0.2) -- (-1.2,1.4) node[dot][blue] {}; \draw[blue] (1,0.2) -- (0,-1.5); \draw[blue] (2,1.1) node[dot][blue] {} -- (1,0.2) -- (.4,1.2) node[dot][blue] {}; \draw[blue] (1,0.2) -- (1.2,1.4) node[dot][blue] {}; }
\DeclareSymbol{31b}{-3}{\draw[blue] (0,0) -- (0,-1) -- (1,0) node[dot][blue] {}; \draw[blue] (0,0) -- (0,1.2) node[dot][blue] {}; \draw[blue] (-.7,1) node[dot][blue] {} -- (0,0) -- (.7,1) node[dot][blue] {};}
\DeclareSymbol{32b}{-3}{\draw[blue] (0,0.3) -- (0,-1) -- (1,0) node[dot][blue] {}; \draw[blue] (0,0) -- (0,-1) -- (-1,0) node[dot][blue] {}; \draw[blue] (0,0) -- (0,1.2) node[dot][blue] {}; \draw[blue] (-.7,1) node[dot][blue] {} -- (0,0) -- (.7,1) node[dot][blue] {};}
\DeclareSymbol{33b}{-3}{\draw[blue] (0,0.3) -- (0,-1) -- (1,0) node[dot][blue] {}; \draw[blue] (0,-1) -- (-1,0.2) node[dot][blue] {}; \draw[blue] (0,0) -- (0,1.2) node[dot][blue] {}; \draw[blue] (0,-1) -- (-1.3,-0.7) node[dot][blue] {}; \draw[blue] (-.7,1) node[dot][blue] {} -- (0,0) -- (.7,1) node[dot][blue] {};}
\DeclareSymbol{34b}{-3}{\draw[blue] (0,0.3) -- (0,-1) -- (1,0.2) node[dot][blue] {}; \draw[blue] (0,-1) -- (-1,0.2) node[dot][blue] {}; \draw[blue] (0,0) -- (0,1.2) node[dot][blue] {}; \draw[blue] (0,-1) -- (-1.3,-0.7) node[dot][blue] {}; \draw[blue] (0,-1) -- (1.3,-0.7) node[dot][blue] {}; \draw[blue] (-.7,1) node[dot][blue] {} -- (0,0) -- (.7,1) node[dot][blue] {};}
\DeclareSymbol{20b}{-3}{\draw[blue] (0,0) -- (0,-1); \draw[blue] (-.7,1) node[dot][blue] {} -- (0,0) -- (.7,1) node[dot][blue] {};}
\DeclareSymbol{22b}{-3}{\draw[blue] (0,0.3) -- (0,-1) -- (1,0) node[dot][blue] {}; \draw[blue] (0,0.3) -- (0,-1) -- (-1,0) node[dot][blue] {}; \draw[blue] (-.7,1) node[dot][blue] {} -- (0,0.3) -- (.7,1) node[dot][blue] {};}
\DeclareSymbol{23b}{-3}{\draw[blue] (0,0.3) -- (0,-1) -- (1,0) node[dot][blue] {}; \draw[blue] (0,-1) -- (-1,0.2) node[dot][blue] {}; \draw[blue] (0,-1) -- (-1.3,-0.7) node[dot][blue] {}; \draw[blue] (-.7,1) node[dot][blue] {} -- (0,0.3) -- (.7,1) node[dot][blue] {};}
\DeclareSymbol{24b}{-3}{\draw[blue] (0,0.3) -- (0,-1) -- (1,0.2) node[dot][blue] {}; \draw[blue] (0,-1) -- (-1,0.2) node[dot][blue] {}; \draw[blue] (0,-1) -- (-1.3,-0.7) node[dot][blue] {}; \draw[blue] (0,-1) -- (1.3,-0.7) node[dot][blue] {}; \draw[blue] (-.7,1) node[dot][blue] {} -- (0,0.3) -- (.7,1) node[dot][blue] {};}
\DeclareSymbol{4b}{0}{\draw[blue] (0,0) -- (-.4,1.3) node[dot][blue] {}; \draw[blue] (0,0) -- (.4,1.3) node[dot][blue] {}; \draw[blue] (-1.1,1) node[dot][blue] {} -- (0,0) -- (1.1,1) node[dot][blue] {};}
\DeclareSymbol{5b}{0}{\draw[blue] (0,0) -- (0,1.5) node[dot][blue] {}; \draw[blue] (0,0) -- (1.4, 0.9) node[dot][blue] {}; \draw[blue] (0,0) -- (-1.4, 0.9) node[dot][blue] {}; \draw[blue] (-.7,1.25) node[dot][blue] {} -- (0,0) -- (.7,1.25) node[dot][blue] {};}
\DeclareSymbol{eb}{-6}{\draw[blue,thick,double] (0,0) -- (0,-1.5);}
\DeclareSymbol{10eb}{-3}{\draw[blue,thick,double] (0,0) -- (0,-1); \draw[blue] (0,0) -- (-0.8,1) node[dot][blue] {};}
\DeclareSymbol{20eb}{-3}{\draw[blue,thick,double] (0,0) -- (0,-1); \draw[blue] (-.7,1) node[dot][blue] {} -- (0,0) -- (.7,1) node[dot][blue] {};}
\DeclareSymbol{30eb}{-3}{\draw[blue,thick,double] (0,0) -- (0,-1); \draw[blue] (0,0) -- (0,1.2) node[dot][blue] {}; \draw (-.7,1) node[dot][blue] {} -- (0,0) -- (.7,1) node[dot][blue] {};}
\DeclareSymbol{40eb}{-3}{\draw[blue,thick,double] (0,0) -- (0,-1.2); \draw[blue] (0,0) -- (-.4,1.3) node[dot][blue] {}; \draw[blue] (0,0) -- (.4,1.3) node[dot][blue] {}; \draw (-1.1,1) node[dot][blue] {} -- (0,0) -- (1.1,1) node[dot][blue] {};}
\DeclareSymbol{50eb}{-3}{\draw[blue,thick,double] (0,0) -- (0,-1.2); \draw[blue] (0,0) -- (0,1.5) node[dot][blue] {}; \draw[blue] (0,0) -- (1.4, 0.9) node[dot][blue] {}; \draw[blue] (0,0) -- (-1.4, 0.9) node[dot][blue] {}; \draw[blue] (-.7,1.25) node[dot][blue] {} -- (0,0) -- (.7,1.25) node[dot][blue] {};}
\DeclareSymbol{50eb5}{-3}{\draw[blue,thick,double] (0,0) -- (0,-1.2); \draw[blue] (0,0) -- (0,1.5) node[dot][blue] {}; \draw[blue] (0,0) -- (1.4, 0.9) node[dot][blue] {}; \draw[blue] (0,0) -- (-1.4, 0.9) node[dot][blue] {}; \draw[blue] (-.7,1.25) node[dot][blue] {} -- (0,0) -- (.7,1.25) node[dot][blue] {}; \draw[blue] (0,-1.2) -- (0,-2.6); \draw[blue] (0,-2.6) -- (-1.4,-2.3) node[dot][blue] {}; \draw[blue] (0,-2.6) -- (-1.2, -1.3) node[dot][blue] {}; \draw[blue] (0,-2.6) -- (1.4,-2.3) node[dot][blue] {}; \draw[blue] (0,-2.6) -- (1.2, -1.3) node[dot][blue] {};}

\begin{document}

\date{\today}
\title{The dynamical Ising-Kac model in $3D$ converges to $\Phi^4_3$}
\author{P.~Grazieschi$^1$, K.~Matetski$^2$ and H.~Weber$^3$}
\institute{University of Bath, \email{p.grazieschi@bath.ac.uk} \and Michigan State University, \email{matetski@msu.edu} \and University of M\"{u}nster, \email{hendrik.weber@uni-muenster.de}}
\date{\today}
\titleindent=0.65cm

\maketitle

\begin{abstract}
We consider the Glauber dynamics of a ferromagnetic Ising-Kac model on a three-dimensional periodic lattice of size $(2 N + 1)^3$, in which the flipping rate of each spin depends on an average field in a large neighborhood of radius $\ga^{-1} <\!\!< N$. We study the random fluctuations of a suitably rescaled coarse-grained spin field as $N \to \infty$ and $\ga \to 0$; we show that near the mean-field value of the critical temperature, the process converges in distribution to the solution of the dynamical $\Phi^4_3$ model on a torus. Our result settles a conjectured from Giacomin, Lebowitz and Presutti  \cite{PresuttiGiacomin}.

The dynamical $\Phi^4_3$  model is given by a non-linear stochastic partial differential equation (SPDE) which is driven by an additive space-time white noise and which requires renormalisation of the non-linearity. A rigorous notion of solution for this SPDE and  its renormalisation is provided by the framework of regularity structures \cite{Regularity}. 
As in the two-dimensional case \cite{IsingKac}, the renormalisation corresponds to a small shift of the inverse temperature of the discrete system away from its mean-field value. 
\end{abstract}

\setcounter{tocdepth}{1}
\tableofcontents

\section{Introduction} 

We consider the Glauber dynamics of the three-dimensional Ising-Kac model on the discrete torus $\Z^3 \slash (2N + 1) \Z^3$. The spins take values $+1$ and $-1$ and flip randomly, where the flipping rate at a site $k$ depends on an average field in a large neighbourood of radius $\ga^{-1} <\!\!< N$ around $k$. We study the random fluctuations of a suitably rescaled coarse-grained spin field $X_\ga$ as $N \to \infty$ and $\ga \to 0$. We prove that there is a choice of the inverse temperature such that if the initial states converge in a suitable topology, then $X_\ga$ converges in distribution to the solution of the dynamical $\Phi^4_3$ model, which is formally given by the SPDE
\begin{equation}\label{eq:equation-intro}
(\partial_t - \Delta) X = - \frac{1}{3} X^3 + A X + \sqrt 2\, \xi, \qquad \qquad  x \in \T^3,
\end{equation}
where $\xi$ denotes a Gaussian space-time white noise.

The Ising-Kac model was introduced in the 60s to recover rigorously the van der Waals theory of phase transition \cite{Kac}.  Various scaling regimes for the Glauber dynamics were studied in the ninetees \cite{MR1275526,MR1374000,MR1373999,MR2460018} and in particular, it was conjectured, that in $1$, $2$ and $3$ dimensions and in a very specific scaling, non-linear fluctuations described by \eqref{eq:equation-intro} can be observed  \cite{PresuttiGiacomin}. For $d=4$ \eqref{eq:equation-intro} is not expected to have a non-trivial meaning \cite{MR4276286} and this is reflected in the dimension-dependent scaling relation  \eqref{constant:isingkac_variable_dimension} below which can be satisfied in dimensions $d=1,2,3$ but not for $d=4$.
The one-dimensional convergence result was  proved three decades ago in \cite{MR1317994, MR1358083}.  The two dimensional case settled much more recently \cite{IsingKac}. In this article we treat the three-dimensional case, thereby completely settling the conjecture from  \cite{PresuttiGiacomin}.

The main difference between the one-dimensional case $d=1$ and the cases $d=2$ and $d=3$ lies in the increased irregularity of solutions to \eqref{eq:equation-intro} in higher dimensions. In fact, for $d=1$ solutions are continuous functions and a solution theory is classical (see e.g.  \cite{MR3236753}).   For $d= 2,3$ solutions are Schwartz-distributions  and   \eqref{constant:isingkac_variable_dimension} has to be renormalised by adding an infinite counter-term. Formally, the equation becomes
\begin{equation*}
(\partial_t - \Delta) X = - \frac{1}{3} \bigl(X^3 - 3\, \infty \times X\bigr) + A X + \sqrt 2\, \xi.
\end{equation*} 
For $d=2$ this renormalisation procedure was implemented rigorously in the influential paper by  Da Prato-Debussche \cite{MR2016604} (see also \cite{MR3693966} for a solution theory on the full space $\R_t \times \R^2_x$).
Consequently, the convergence proof for Ising-Kac for $d=2$ consists of adapting their solution method to a discrete approximation (already found in \cite{PresuttiGiacomin}). 
A key technical step was to show that, up to well-controlled error terms, the renormalisation of products of martingales is similar to the Wick renormalisation of Gaussian processes. Moreover, the renormalisation of the non-linearity in the discrete equation corresponds to a small shift (of order $\ga^2 \log \ga$ in the notation of that work) of the inverse temperature from the critical value of the mean-field mode (in fact this shift had already been suggested in \cite{MR1467623}).

The solution theory for \eqref{eq:equation-intro} for $d=3$ is yet much more involved than the $d=2$ case and was understood only much more recently. Short-time solution theories were contained in the groundbreaking theories of regularity structures \cite{Regularity} and paracontrolled distributions \cite{MR3406823,MR3846835} and a solution theory is by now completely developed \cite{MR3846835,from-infinity,MR3951704,MR4164267}, see Section~\ref{sec:phi4Section} for a brief review. In particular, it is known that the renormalisation procedure is more complex --- beyond the leading order ``Wick" renormalisation an additional logarithmic divergence (the ``sunset diagram'') appears.  

In this article we develop an analysis for the  discrete approximation to \eqref{eq:equation-intro} provided in  \cite{PresuttiGiacomin,IsingKac} based on the theory of regularity structures. More specifically, we rely on the discretisation framework for regularity structures developed in \cite{erhard2017discretisation,HairerMatetski}, which of course has to be adapted to the situation at hand. A key part of this analysis is the construction and derivation of bounds for a suitable discrete \emph{model}. Following \cite{IsingKac} the discrete analogue of Hairer's \emph{model} is defined, based on a linearised version of the discrete equation. The elements of this model can be represented as iterated stochastic integrals with respect to a jump martingale. Our companion article \cite{Martingales} develops a systematic theory of these integrals which provides the necessary bounds. We encounter the same ``divergencies" as in the continuum, and as in the two-dimensional case, these correspond to small shifts (of oder $\gamma^3$ and of order $\gamma^6 \log \gamma^{-1}$) to the temperature. Additionally, we encounter an order $1$ shift (corresponding to a shift of order $\gamma^6$ of the temperature), in the analysis of the approximate Wick constant. This term, comes from the analysis of the predictable quadratic variation of the discrete martingales and does not have a counterpart in the continuous theory. 

\subsection{Structure of the article} 

In Section~\ref{sec:IsingKac} we define the dynamical Ising-Kac model and state in Theorem~\ref{thm:main} our main convergence result. We recall the solution theory of the dynamical $\Phi^4_3$ model \eqref{eq:equation-intro} in Section~\ref{sec:phi4Section}. In Section~\ref{sec:discreteRegStruct} we construct a regularity structure for the discrete equation describing the Ising-Kac model. Furthermore, we make the definitions of discrete models and modelled distributions on this regularity structure, which are required to solve the equation. A particular discrete renormalised model is constructed in Section~\ref{sec:lift}. Section~\ref{sec:martingales} contains some properties of the driving martingales and bounds on auxiliary processes, which allow to prove moment bounds for the discrete models in Section~\ref{sec:convergence-of-models}. In Section~\ref{sec:discrete-solution} we write and solve the discrete equation on the regularity structure. Theorem~\ref{thm:main} is proved in Section~\ref{sec:ConvFinal}. Appendix~\ref{sec:kernels} contains some properties of the discrete kernels used throughout the paper. 

\subsection{Notation}
\label{sec:notation}

We use $\N$ for the set of natural numbers $1, 2, \ldots$, and we set $\N_0 := \N \cup \{0\}$. The set of positive real numbers is denoted by $\R_+ := [0, \infty)$. We typically use the Euclidean distance $|x|$ for points $x \in \R^d$, but sometimes we need the distances $|x|_{\sone} = |x_1| + \cdots + |x_d|$ and $|x|_{\sinfty} = \max \{|x_1|, \ldots, |x_d|\}$. We denote by $B(x, r)$ the open ball in $\R^3$ containing the point $y$ such that $|y - x| < r$.

For integer $n \geq 0$, we denote by $\CC^n_0$ the set of compactly supported $\CC^n$ functions $\varphi : \R^3 \to \R$. The set $\CB^n$ contains all functions $\varphi \in \CC^n_0$, which are supported on $B(0, 1)$, and which satisfy $\| \varphi \|_{\CC^n} \leq 1$. For a function $\varphi \in \CB^n$, for $x \in \R^3$ and for $\lambda \in (0,1]$, we define its rescaled and recentered version
\begin{equation}\label{eq:rescaled-function}
\varphi_x^\lambda(y) := \frac{1}{\lambda^{3}} \varphi \Bigl(\frac{y-x}{\lambda}\Bigr).
\end{equation}

We define the three-dimensional torus $\T^3$ identified with $[-1, 1]^3$, and the space $\SD'(\T^3)$ of distributions on $\T^3$. Respectively, we denote by $\SD'(\R^d)$ the space of distributions on $\R^d$. When working with distribution-valued stochastic processes, we use the Skorokhod space $\CD \bigl(\R_+, \SD'(\T^3)\bigr)$ of \cadlag functions \cite{Billingsley}. 

For $\eta < 0$ we define the Besov space $\CC^\eta(\T^3)$ as a completion of smooth functions $f: \T^3 \to \R$, under the seminorm
\begin{equation}\label{eq:Besov}
\| f \|_{\CC^\eta} := \sup_{\varphi \in \CB^r} \sup_{x \in \R^3} \sup_{\lambda \in (0,1]} \lambda^{- \eta} |f (\varphi_x^\lambda)| < \infty,
\end{equation}
for $r$ being the smallest integer such that $r > -\eta$, where we extended $f$ periodically to $\R^3$, and where we write $f (\varphi_x^\lambda) = \langle f, \varphi_x^\lambda \rangle$ for the duality pairing. Then the Dirac delta $\delta$ is an element of the space $\CC^{-3}(\T^3)$. It is important to define these spaces as completions of smooth functions, because this makes the spaces separable and allows to use various probabilistic results.

For $\eps > 0$ we define the grid $\Lattice := \eps \Z^3$ of the mesh size $\eps$.\label{lab:Lambda} Then it is convenient to map a function $f : \Lattice \to \R$ to a distribution as 
\begin{equation}\label{eq:iota}
(\iota_\eps f)(\varphi) := \eps^{3} \sum_{x \in \Lattice} f(x) \varphi(x), 
\end{equation}
for any continuous and compactly supported function $\varphi$.

When working on the time-space domain $\R^4$, we use the \emph{parabolic scaling} $\s := (2, 1, 1, 1)$, where the first coordinate corresponds to the time variable and the other three correspond to the space variables. Then for any point $(t, x_1, x_2, x_3) \in \R^4$, we introduce the parabolic distance from the origin $\Vert (t, x) \Vert_\s := | t |^{\frac12} + |x_1| + |x_2| + |x_3|$. For a multiindex $k = (k_0, k_1, k_2, k_3) \in \N_0^4$ we define $|k|_\s := 2 k_0 + k_1 + k_2 + k_3$. \label{lab:norms-s}

We frequently use the notation $a \lesssim b$, which means that $a \leq C b$ for a constant $C \geq 0$ independent of the relevant quantities (such quantities are always clear from the context). In the case $a \lesssim b$ and $b \lesssim a$ we simply write $a \approx b$. For a vanishing sequence of values $\emezo$, the notation $a_\emezo \sim \emezo^{-1}$ means that $\lim_{\emezo \to 0} \emezo a_\emezo$ exists and is finite.

We write $\CL(V, W)$ for the space of linear bounded operators from $V$ to $W$.

\subsection*{Acknowledgements}
PG was supported by a scholarship from the EPSRC Centre for Doctoral Training in Statistical Applied Mathematics at Bath (SAMBa), under the project EP/L015684/1.

KM was partially supported by NSF grant DMS-2321493. HW was supported by the Royal Society through the University Research Fellowship UF140187, by the Leverhulme Trust through a Philip Leverhulme Prize and by the European Union (ERC, GE4SPDE, 101045082). HW acknowledges funding by the Deutsche Forschungsgemeinschaft under Germany’s Excellence Strategy EXC 2044 390685587, Mathematics M\"{u}nster: Dynamics -- Geometry -- Structure.

PG and HW thank the Isaac Newton Institute for Mathematical Sciences for hospitality during the programme \textit{Scaling limits, rough paths, quantum field theory}, which was supported by EPSRC Grant No. EP/R014604/1.

\section{The dynamical Ising-Kac model} 
\label{sec:IsingKac}

The Ising-Kac model is a mean-field model with long range potential, which was introduced to recover rigorously the van der Waals theory of phase transition \cite{Kac}. We are interested in the three-dimensional model on a periodic domain. To define the model, let us take $N \in \N$ and let $\LN := \Z^3 \slash (2N + 1) \Z^3$ be the three-dimensional discrete torus, i.e. a discrete periodic grid with $2N+1$ points per side. It will be convenient to identify $\LN$ with the set $\{-N, -N + 1, \ldots, 0, \ldots, N\}^3$ and allow points to be multiplied by real numbers in such a way that $r \cdot x = rx \, (\text{mod} \, (2N + 1))$, for any $x \in \LN$ and $r \in \R$, where the $\text{mod}$ operator is taken on each component of $x$. Each site of the grid $k \in \LN$ has an assigned spin value $\sigma(k) \in \{-1, +1\}$. The set of all spin configurations is $\Sigma_N := \{ -1, +1 \}^{\LN}$ and we write $\sigma = \bigl(\sigma(k): k \in \LN\bigr)$ for an element of $\Sigma_N$. 

Let us fix a constant $r_\star > 0$. The range of the interaction is represented by a real number $\ga \in (0, \ga_\star)$, for some $\ga_\star < r^{-1/3}_\star$, and by a smooth, compactly supported, rotation invariant function $\KK: \R^3 \to [0,1]$, supported in the ball $B(0, r_\star)$. (A high regularity of this function is required in the proof of Lemma~\ref{lem:Kg}.) We impose that $\KK(0) = 0$ and
\begin{equation}\label{eq:K-moments}
\int_{\R^3} \KK(x)\, \d x = 1, \qquad\qquad \int_{\R^3} \KK(x) |x|^2\, \d x = 6, 
\end{equation}
where $|x|$ is the Euclidean norm. Then we define the function $\KK_\ga: \LN \to [0, \infty)$ as 
\begin{equation}\label{eq:K-gamma}
\KK_\ga(k) = \varkappa_{\ga, 1} \ga^3 \KK(\ga k)
\end{equation}
for $k \in \LN$. The constant $\varkappa_{\ga, 1}$ is given by $\varkappa_{\ga, 1}^{-1}:= \sum_{k \in \LN} \ga^3 \KK(\ga k)$, and it guarantees that $\sum_{k \in \LN} \KK_\ga(k) = 1$. Our assumption $\ga < \ga_\star$ makes sure that the radius of interaction $r_\star \ga^{-1}$ does not exceed the size of the domain $N \approx \ga^{-4}$ (the precise definition of $N$ is given in \eqref{eq:scalings}). In the rest of this paper, we always consider $\ga < \ga_\star$.

The \emph{locally averaged (coarse-grained) field} $h_\ga: \Sigma_N \times \LN \to \R$ is defined as
\[ h_\ga(\sigma, k) := \sum_{j \in \LN} \KK_\ga (k - j) \sigma(j). \] 
Here and in what follows we consider the difference $k - j$ on the torus. The \emph{Hamiltonian} of the system is the function $\SH_\ga: \Sigma_N \to \R$ given by 
\begin{equation}\label{eq:Hamiltonian}
\SH_\ga(\sigma) := - \frac{1}{2} \sum_{j, k \in \LN} \KK_\ga (k - j) \sigma(j) \sigma(k) = - \frac{1}{2} \sum_{k \in \LN} \sigma(k) h_\ga (\sigma, k).
\end{equation}
In other words, two spins $\sigma(j)$ and $\sigma(k)$ interact if they are located at a distance bounded by $r_\star \ga^{-1}$, where $r_\star$ is the radius of the support of $\KK$.

For a fixed \emph{inverse temperature} $\beta > 0$, the \emph{Gibbs measure} $\lbg$ is the probability measure on $\Sigma_N$ 
\[ 
\lbg (\sigma) = \frac{1}{\SZ_\ga} \exp \big( - \beta \SH_\ga (\sigma) \big) \qquad \text{for} \quad \sigma \in \Sigma_N, 
\]
with normalization constant $\SZ_\ga := \sum_{\sigma \in \Sigma_N} \exp \big( - \beta \SH_\ga (\sigma) \big)$. Since we consider the Ising-Kac model in a finite volume, the sum is finite and $\SZ_\ga$ is always well-defined.

We are interested in the \emph{Glauber dynamics} of the Ising-Kac model, in which the spins evolve in time as a Markov process on a filtered probability space $\bigl(\Omega, \P, \SF, (\SF_t)_{t \geq 0}\bigr)$ with the infinitesimal generator 
\begin{equation}\label{eq:generator}
	\SL_\ga f (\sigma) := \sum_{j \in \LN} c_{\ga} (\sigma, j) \big( f(\sigma^j) - f(\sigma) \big), 
\end{equation}
acting on functions $f: \Sigma_N \to \R$. The configuration $\sigma^j$ is obtained from $\sigma$ by flipping the spin at the site $j$, i.e. for any $k \in \LN$
\[
	\sigma^j(k) := \left\{ \begin{aligned} & \sigma(k) &&\text{if }~ k \neq j, \\ &-\sigma(k) &&\text{if }~ k = j. \end{aligned} \right.
\]
The flipping rates $c_\ga$ are chosen such that the Gibbs measure $\lbg$ is reversible for the dynamics. For any $\sigma \in \Sigma_N$ and for any $j \in \LN$, we set 
\begin{equation}\label{eq:rates}
	c_\ga(\sigma, j) := \frac{ \lbg (\sigma^j) }{ \lbg (\sigma) + \lbg (\sigma^j) } = \frac{1}{2} \Bigl( 1 - \sigma(j) \tanh \big( \be h_\ga(\sigma, j) \big) \Bigr).
\end{equation}
One can readily check that the \emph{detailed balance condition} is satisfied (see Proposition~5.3 in \cite{Liggett} and the discussion above it)
\begin{equation*}
c_\ga(\sigma^j, j) \lbg (\sigma^j) = c_\ga(\sigma, j) \lbg (\sigma),
\end{equation*}
for each $j \in \LN$, which implies that indeed the Gibbs measure $\lbg$ is reversible. Given a time variable $t \geq 0$, we denote by $\sigma(t) = \bigl(\sigma(t, k): k \in \LN\bigr)$ the pure jump Markov process with jump rates $c_\ga$. 

We can use properties of the infinitesimal generator (see \cite[App.~1.1.5]{KipnisLandim}) to write 
\begin{equation}\label{eq:equation-for-sigma}
\sigma(t, k) = \sigma(0, k) + \int_0^t \SL_\ga \sigma(s, k)\, \d s + \fm_\ga (t, k), 
\end{equation}
where $\sigma(0) \in \Sigma_N$ is a fixed initial configuration of spins at time $0$, the generator is applied to the function $f (\sigma) = \sigma(k)$ and $t \mapsto \fm_\ga(t, k)$ is a family of \cadlag martingales with jumps of size $2$ (because each spin changes values from $+1$ to $-1$ or vice versa). Moreover, the predictable quadratic covariations of these martingales are given by the \emph{carr\'{e} du champ} operator \cite[App.~B]{Mourrat} and may be written as
\begin{equation}\label{eq:m-bracket}
\big\langle \fm_\ga(\bigcdot, k), \fm_\ga(\bigcdot, k') \big\rangle_t = 4 \delta_{k,k'} \int_0^t c_\ga ( \sigma(s), k )\, \d s,
\end{equation}
for all $k, k' \in \LN$, where $\delta_{k,k'}$ is the Kronecker delta, i.e. $\delta_{k,k'} = 1$ if $k = k'$ and $\delta_{k,k'} = 0$ otherwise. We recall that the predictable quadratic covariation in \eqref{eq:m-bracket} is the unique increasing process, vanishing at $t = 0$ and such that $t \; \mapsto \; \fm_\ga(t, k) \fm_\ga(t, k') - \big\langle \fm_\ga(\bigcdot, k), \fm_\ga(\bigcdot, k') \big\rangle_t$ is a martingale. The definitions and properties of the bracket processes for \cadlag martingales can be found in \cite{JS03}.

The dynamical version of the averaged field we denote by 
\begin{equation*}
h_\ga(t,k) := h_\ga(\sigma(t), k).
\end{equation*}

\begin{remark}
	As we stated above, we always consider $N >\!\!> \ga^{-1}$, which together with the property $\KK_\ga(0) = 0$ means that there is no self-interaction of spins. In contrast to the setting of \cite{IsingKac}, we have to avoid self-interaction by postulating $\fK(0) = 0$. The reason for this assumption can be seen in the proof of Lemma~\ref{lem:Kg}, where the function $K_\ga$ is required to be differentiable. The weaker bounds in \cite[Lem.~8.2]{IsingKac} in the two-dimensional setting allow this function to have a discontinuity at the origin.
\end{remark}

\subsection{Convergence of a rescaled model}
\label{sec:convergence-statement}

Our main interest lies in understanding the bahaviour of a rescaled version of the dynamical Ising-Kac model. For $\eps = 2 / (2 N + 1)$ we introduce the rescaled lattice \[ \Le := \big\{ \eps k : k \in \LN \big\}. \] In particular, $\Le$ is a subset of the three-dimensional torus $\T^3$. In what follows, we use the convolution on the lattice, defined for two functions $f, g : \Le \to \R$ as
\begin{equation}\label{eq:discrete-convolution}
\bigl(f \ae g\bigr) (x) := \eps^3 \sum_{y \in \Le} f(x - y) g(y).
\end{equation}
For any function $g: \Le \to \R$, we use the standard definition for the discrete Fourier transform
\begin{equation}\label{eq:Fourier}
	\widehat{g}(\om) := \eps^3 \sum_{x \in \Le} g(x) e^{-\pi i \om \cdot x} \qquad \text{for} \quad \om \in \{ -N, \ldots, N \}^3.
\end{equation}

We fix two positive real constants $\delta > 0$ and $\alpha > 0$ and define the family of rescaled martingales
\begin{equation}\label{eq:martingales}
	\M_\ga ( t, x ) := \frac{1}{\delta} \fm_\ga \Bigl( \frac{t}{\alpha}, \frac{x}{\eps} \Bigr) \qquad \text{for}~~ x \in \Le,~ t \geq 0.
\end{equation}
and also
\begin{equation}\label{eq:kernel-K}
	K_\ga(x) := \frac{1}{\eps^{3}} \KK_\ga \left( \frac{x}{\eps} \right).
\end{equation}
Then from \eqref{eq:equation-for-sigma} we can conclude that the rescaled process
\begin{equation}\label{eq:X-gamma}
X_\ga(t, x) := \frac{1}{\delta} h_\ga \Bigl( \frac{t}{\alpha}, \frac{x}{\eps} \Bigr) \qquad \text{for}~~ x \in \Le,~ t \geq 0,
\end{equation}
solves the following equation (see \cite{IsingKac} for the derivation of an analogous equation in the two-dimensional case)
\begin{align}
	X_\ga(t, x) &= X_\ga^0 (x) + \big( K_\ga \ae \M_\ga \big) (t, x) \label{eq:equation-for-X}\\
	&\qquad + \int_0^{t} \biggl( \frac{{\eps}^2}{\ga^2 \alpha} \widetilde{\Delta}_\ga X_\ga + \frac{\be-1}{\alpha} K_\ga \ae X_\ga - \frac{\be^3 \delta^2}{3 \alpha} K_\ga \ae X_\ga^3 + E_\ga \biggr) (s, x)\, \d s, \nonumber
\end{align}
where $X_\ga^0 (x) = X_\ga(0, x)$ is a rescaled initial configuration. The linear part of this equation is given by the discrete operator 
\begin{equation}\label{eq:Laplacian-gamma}
	\widetilde{\Delta}_\ga f (x) := \frac{\ga^2}{{\eps}^2} \big( K_\ga \ae f - f \big)(x),
\end{equation}
and the ``error term'' $E_\ga$ is given by
\begin{equation}\label{eq:expr_error_term} 
	E_\ga (t, x) := \frac{1}{\delta \alpha} \biggl( \tanh \bigl( \be \delta X_\ga \bigr) - \be \delta X_\ga + \frac{1}{3} \bigl(\be \delta X_\ga\bigr)^3 \biggr)(t,x).
\end{equation}
As we commented after \eqref{eq:K-gamma}, for all $\ga$ sufficiently small the function $K_\ga(x)$ is supported on $(-1, 1)^3$, and its convolutions with periodic processes in \eqref{eq:equation-for-X} make sense.

We are going to take the limit such that all the scaling parameters in \eqref{eq:X-gamma} tend to zero. In order to prevent explosion of the multiplier $(\be-1) / \alpha$ in \eqref{eq:equation-for-X}, we need to consider the inverse temperature of the form 
\begin{equation}\label{eq:beta}
\be = 1 + \alpha \big( \fC_\ga + A \big),
\end{equation}
where $A$ is a fixed constant (its value does not play any significant role and produces a linear term in the limiting equation \eqref{eq:Phi43}) and where $\fC_\ga$ is a suitably chosen renormalisation constant, which diverges as $\ga \to 0$ such that $\ga^{-1} <\!\!< \alpha^{-1}$. In other words, we consider the model near the critical mean-field value of the inverse temperature $\beta_c = 1$, and as we will see later, $\fC_\ga$ plays a role of the renormalisation constant, which is required to have a non-trivial limit of the non-linearity $X_\ga^3$ in \eqref{eq:equation-for-X}. The shift of the critical inverse temperature was observed in \cite{MR1467623}, and in the three-dimensional case it has a significantly more complicated structure than in two dimensions \cite{IsingKac} (see Theorem~\ref{thm:main}). 

From \eqref{eq:rates} and \eqref{eq:m-bracket} we conclude that the predictable quadratic covariations of the martingales \eqref{eq:martingales} are
\begin{equation}\label{eq:M-bracket}
\big\langle \M_\ga(\bigcdot, x), \M_\ga(\bigcdot, x') \big\rangle_t = \frac{2 \eps^3}{\delta^2 \alpha} \delta^{(\eps)}_{x,x'} \int_0^{t} \Bigl( 1 - \sigma \Bigl(\frac{s}{\alpha}, \frac{x}{\eps}\Bigr) \tanh \bigl( \be \delta X_\ga(s, x) \bigr) \Bigr) \d s,
\end{equation}
for any $x, x' \in \Le$, where $\delta^{(\eps)}_{x,x'} := \eps^{-3} \delta_{x,x'}$ is an approximation of the Dirac's delta.

We would like to have convergence of the operators $\widetilde{\Delta}_\ga$ to the Laplacian, and of the quadratic covariations for the martingales to those of a cylindrical Wiener process. We also want to have a non-trivial nonlinearity in the limit (given by the cubic term), which translates into the relations between the scaling parameters \[ 1 \approx \frac{{\eps}^2}{\ga^2 \alpha} \approx \frac{\delta^2}{\alpha} \approx \frac{{\eps}^3}{\delta^2 \alpha}. \]
In the rest of this article, we therefore fix them to be $\ga$-dependent as
\begin{equation}\label{eq:scalings} 
	N = \lfloor \ga^{-4} \rfloor, \qquad \eps = \frac{2}{2N + 1}, \qquad \alpha = \ga^6, \qquad \delta = \ga^3.
\end{equation}
This implies $\eps \approx \ga^4$, and such choice of $\eps$ (rather than $\eps = \ga^4$) makes the use of the discrete Fourier transform \eqref{eq:Fourier} more convenient. Moreover, we define:
\begin{equation}\label{constant:isingkac_one}
	\varkappa_{\ga, 2} := \frac{\eps^3}{\delta^2 \alpha} \approx 1,
\end{equation}
which we will use in the rest of the paper, remembering that it converges to $1$.

The scaling \eqref{eq:scalings} makes the radius of interaction for the rescaled process to be $\emezo := \eps / \gamma \approx \ga^3$. As such, the model has two scales: $\eps \approx \ga^4$ is the distance between points on the lattice, and $\emezo \approx \ga^3$ is the distance up to which the interaction between two spins is felt.

\begin{remark}
	The dynamical Ising-Kac model can be defined for any spatial dimension $d \geq 1$, where the previous conditions on the quantities $\eps$, $\delta$ and $\alpha$ become
\begin{equation}\label{constant:isingkac_variable_dimension}
		\eps \approx \ga^{\frac{4}{4-d}}, \qquad \alpha \approx \ga^{\frac{2d}{4-d}}, \qquad \delta \approx \ga^{\frac{d}{4-d}}.
		\end{equation}
	Observe that, in order to make these quantities vanish when $\ga \to 0$, we need to impose $d < 4$. This condition coincides with the \emph{local sub-criticality} condition in the solution theory of the dynamical $\Phi^4_d$ model \cite{Regularity}.
\end{remark}

Our goal is to prove convergence of the rescaled processes \eqref{eq:X-gamma} to the solution of the $\Phi^4_3$ equation (\emph{the dynamical $\Phi^4_3$ model})
\begin{equation}\label{eq:Phi43}
(\partial_t - \Delta) X = - \frac{1}{3} X^3 + A X + \sqrt 2\, \xi, \qquad X(0, \bigcdot) = X^0(\bigcdot),
\end{equation}
on $\R_+ \times \T^3$, where $\xi$ is space-time white noise, and $A$ is the same as in \eqref{eq:beta}. The notion of solution for the singular stochastic PDE \eqref{eq:Phi43} was first provided in \cite{Regularity} using the \emph{theory of regularity structures}, and later in \cite{MR3846835} using \emph{paracontrolled distributions}. 

We need to introduce the topology in which convergence of the initial states holds. Namely, for a function $f_\ga : \Le \to \R$, for $\eta < 0$ and for the smallest integer $r$ such that $r > - \eta$, we define the semi-norm
\begin{equation}\label{eq:eps-norm-1}
\| f_\ga \|^{(\emezo)}_{\CC^\eta} := \sup_{\varphi \in \CB^{r}} \sup_{x \in \Lattice} \sup_{\lambda \in [\emezo, 1]} \lambda^{-\eta} | (\iota_\eps f_\ga)(\varphi^\lambda_x) | + \sup_{\varphi \in \CB^{r}} \sup_{x \in \Lattice} \sup_{\lambda \in [\eps, \emezo)} \emezo^{-\eta} | (\iota_\eps f_\ga)(\varphi^\lambda_x) |.
\end{equation}
where we extended the function $f_\ga$ periodically to $\Lattice$, the set of test functions $\CB^{r}$ defined in Section~\ref{sec:notation} and the map $\iota_\eps$ is defined in \eqref{eq:iota}. This definition is similar to \eqref{eq:Besov}, where we  ``measure'' regularity only above the scale $\emezo$. On the smaller scale, we expect the function to be uniformly bounded by a constant multiple of $\emezo^{\eta}$. One can see that this semi-norm is finite for any function $f_\ga$, but we will be always interested in the situation when it is bounded uniformly in $\ga > 0$. If $\lambda < \emezo$, then the support of $\varphi^\lambda_x$ contains only the point $x \in \Lattice$, and we readily get 
\begin{equation}\label{eq:uniform-bound}
\sup_{x \in \Lattice} |f_\ga(x)| \leq \emezo^{\eta} \| f_\ga \|^{(\emezo)}_{\CC^\eta}.
\end{equation}
To compare this function with a distribution $f \in \CC^{\eta}(\T^3)$, we also define 
\begin{align}\label{eq:eps-norm-2}
\| f_\ga; f \|^{(\emezo)}_{\CC^\eta} &:= \sup_{\varphi \in \CB^{r}} \sup_{x \in \Lattice} \sup_{\lambda \in [\emezo, 1]} \lambda^{-\eta} | (\iota_\eps f_\ga - f)(\varphi^\lambda_x) | \\
&\qquad +\sup_{\varphi \in \CB^{r}} \sup_{x \in \Lattice} \sup_{\lambda \in [\eps, \emezo)} \emezo^{-\eta} | (\iota_\eps f_\ga)(\varphi^\lambda_x) | + \sup_{\varphi \in \CB^{r}} \sup_{x \in \R^3} \sup_{\lambda \in (0, \emezo)} \lambda^{-\eta} | f(\varphi^\lambda_x) |, \nonumber
\end{align}
where we extended $f_\ga$ and $f$ periodically to $\Lattice$ and $\R^3$ respectively. In other words, we compare the two functions on the scale above $\emezo$, and use the simple control on the smaller scale. 

The following is the main result of this article, which is proved in Section~\ref{sec:ConvFinal}. We refer to Section~\ref{sec:notation} for the definitions of the involved spaces.

\begin{theorem}\label{thm:main}
Let there exist values $-\frac{4}{7} < \eta < \bar \eta < -\frac{1}{2}$ and $\ga_\star > 0$, and a distribution $X^0 \in \CC^{\bar \eta}(\T^3)$
such that the rescaled initial state of the dynamical Ising-Kac model satisfies
\begin{equation}\label{eq:initial-convergence}
	\sup_{\ga \in (0, \ga_\star)} \| X^{0}_\ga\|^{(\emezo)}_{\CC^{\bar \eta}} < \infty, \qquad \lim_{\ga \to 0} \| X^{0}_\ga; X^0 \|^{(\emezo)}_{\CC^{\eta}} = 0.
\end{equation}
Then there is a choice of the constant $\fC_\ga$ in \eqref{eq:beta}, such that the processes $t \mapsto \iota_\eps X_\ga(t)$ converge in law as $\ga \to 0$ to $t \mapsto X(t)$ with respect to the topology of the Skorokhod space $\CD \bigl(\R_+, \SD'(\T^3)\bigr)$, where $X$ is the solution of the $\Phi^4_3$ equation \eqref{eq:Phi43} with the initial state $X^0$ and with the constant $A$ from \eqref{eq:beta}.

Furthermore, let $\widehat{K}_\ga$ be the discrete Fourier transform of the function $K_\ga$ (since $K_\ga$ is symmetric, $\widehat{K}_\ga$ is real-valued). Then for all $\ga > 0$ small enough one has the expansion
\begin{equation}\label{eq:C-expansion}
\fC_\ga = \fc_\ga^{(2)} + \fc_\ga^{(1)} + \fc_\ga^{(0)},
\end{equation}
where the constants $\fc_\ga^{(2)} \sim \emezo^{-1}$ and $\fc_\ga^{(1)} \sim \log \emezo$ are given by
\begin{align}
\fc_\ga^{(2)} &= \frac{\ga^6}{8} \sum_{0 < |\om|_{\sinfty} \leq N} \frac{|\widehat{K}_\ga(\om)|^2}{1 - \widehat{K}_\ga(\om)}, \label{eq:renorm-constants-main}\\
\fc_\ga^{(1)} &= \frac{\ga^{18}}{16} \sum_{0 < |\om_1|_{\sinfty}, |\om_2|_{\sinfty} \leq N} \frac{|\widehat{K}_\ga(\om_1)|^2 |\widehat{K}_\ga(\om_2)|^2}{(1 - \widehat{K}_\ga(\om_1)) (1 - \widehat{K}_\ga(\om_2))} \frac{\widehat{K}_\ga(\om_1 + \om_2)}{1 - \widehat{K}_\ga(\om_1) - \widehat{K}_\ga(\om_2) + \widehat{K}_\ga(\om_1 + \om_2)}, \nonumber
\end{align}
and the constant $\fc_\ga^{(0)}$ has a finite limit as $\ga \to 0$. All sums in \eqref{eq:renorm-constants-main} run over $\{-N, \ldots, N\}^3$ with the imposed restrictions, and the denominators of the terms in these sums are non-vanishing. 
\end{theorem}

\begin{remark}
One should note that the renormalisation constant $\fC_\ga$ depends non-trivially on the covariations \eqref{eq:M-bracket} of the driving martingales. It can be seen from the proof of Theorem~\ref{thm:main} (more precisely, from the renormalisation of the lift \eqref{eq:cherry} and from the definition of the renormalisation constant \eqref{eq:C-exact} in the discrete equation). The constant $\fc_\ga'$ in \eqref{eq:cherry} comes from renormalisation of the product in \eqref{eq:M-bracket}, and it is not observed in the two-dimensional case \cite{IsingKac}.
\end{remark}

\begin{remark}
The precise value of the constant $\fc_\ga^{(0)}$ may be obtained from \eqref{eq:C-exact}, which does not play a significant role and we omit it here.
\end{remark}

\begin{remark}\label{rem:regularity-explanation}
It is natural to consider the initial states of regularity strictly smaller than $-\frac{1}{2}$, because this is the spatial regularity of the solution to \eqref{eq:Phi43} (see \cite{Regularity}). We make the assumption $\eta > -\frac{4}{7}$ on the regularity of the initial state. It follows from the definition of the model that $X_\ga$ lives on the scale $\emezo \approx \ga^3$. This implies that, for any $\kappa > 0$, we expect the following a priori bound 
\begin{equation*}
    \| X_\ga(t) \|_{L^\infty(\Le)} \lesssim \emezo^{-\frac{1}{2} - \kappa}
\end{equation*}
uniformly in $\ga \in (0,\ga_\star)$. Hence, for $\kappa < \frac{1}{14}$ we can use the Taylor expansion of order $5$ for the function $\tanh$ in \eqref{eq:expr_error_term}, with the error term bounded by a positive power of $\ga$. This is the reason for our restriction $\eta = -\frac{1}{2} - \kappa > -\frac{4}{7}$. Proving Theorem~\ref{thm:main} for lower regularity of the initial state requires some technicalities. More precisely, for $\eta < -\frac{4}{7}$ we need to have a bigger regularity structure, than the one defined in Section~\ref{sec:discreteRegStruct}, we need to control blow-ups of $X_{\ga}$ at time $t = 0$, similarly to how it was done in \cite{MR3179667, Regularity}, and we may need to work in more complicated spaces (see \cite{Labbe} for continuous equations with irregular initial states).
\end{remark}

\subsection{A mild form of the equation}

In order to define the Green's function for the linear operator in \eqref{eq:equation-for-X}, it is convenient to use the discrete Fourier transform \eqref{eq:Fourier}. We start with recalling some of its basic properties. Every time when a sum runs over $\om \in \{ -N, \ldots, N \}^3$, we will simply write $|\om|_{\sinfty} \leq N$. For the function as in \eqref{eq:Fourier}, the Fourier series is
\begin{equation}\label{eq:Fourier-series}
	g(x) = \frac{1}{8}\sum_{|\om|_{\sinfty} \leq N} \widehat{g}(\om) e^{\pi i \om \cdot x}.
\end{equation}
Then, for two functions $f, g : \Le \to \R$, Parseval's theorem reads 
\begin{equation}\label{eq:Parseval}
	\eps^3 \sum_{x \in \Le} f(x) g(x) = \frac{1}{8} \sum_{|\om|_{\sinfty} \leq N} \widehat{f}(\om)\, \overline{\widehat{g}(\om)},
\end{equation}
where $\overline{\widehat{g}(\om)}$ is the complex conjugation of $\widehat{g}(\om)$. Moreover, one has the identities
\begin{equation}\label{eq:Fourier-of-product}
	\widehat{f g}(\om) = \frac{1}{8} \sum_{|\om'|_{\sinfty} \leq N} \widehat{f}(\om - \om')\, \widehat{g}(\om'), \qquad\qquad \widehat{f \ae g}(\om) = \widehat{f}(\om)\, \widehat{g}(\om),
\end{equation}
where $\ae$ is the convolution on $\Le$, defined in \eqref{eq:discrete-convolution}, and the subtraction $\om - \om'$ is performed on the torus $\{ -N, \ldots, N \}^3$. To have a lighter notation in the following formulas we will write $\SF_{\!\!\eps} f(\om)$ for the discrete Fourier transform $\widehat{f}(\om)$. One can readily see that $\SF_{\!\!\eps} f$ converges in a suitable sense as $\ga \to 0$ to the continuous Fourier transform $\SF f$ given by
\begin{equation*}
\SF f (\om) = \int_{\R^3} f(x) e^{-\pi i \om \cdot x} \qquad \text{for} \quad \om \in \R^3.
\end{equation*}

It will be convenient to include the factor $\eps^2 / (\ga^2 \alpha)$ in \eqref{eq:equation-for-X} into the definition of the linear operator. For this, we write 
\begin{equation}\label{eq:c-gamma-2}
\eps = \ga^4 \varkappa_{\ga, 3} \qquad \text{with} \quad |\varkappa_{\ga, 3} - 1| < \ga^4, 
\end{equation}
and we define a new operator 
\begin{equation}\label{eq:Laplacian-discrete}
\Delta_\ga := \varkappa_{\ga, 3}^2 \widetilde{\Delta}_\ga = \frac{{\eps}^2}{\ga^2 \alpha} \widetilde{\Delta}_\ga.
\end{equation}
One can see that $\Delta_\ga$ approximates the continuous Laplace operator $\Delta$ as $\ga \to 0$, when it is applied to a sufficiently regular function, and we can define the respective approximate heat kernel. More precisely, we define the function $P^\ga_t : \Le \to \R$ solving for $t > 0$ the ODEs
\begin{equation}\label{eq:kernels-P}
\frac{\d}{\d t} P^\ga_t = \Delta_\ga P^\ga_t,
\end{equation} 
with the initial condition $P^\ga_0(x) = \delta^{(\eps)}_{x,0}$ (the latter is defined below \eqref{eq:M-bracket}). $P^\ga$ is the Green's functions of the linear operator which appear in equation \eqref{eq:equation-for-X}. This function can alternatively be defined by its discrete Fourier transform
\begin{equation}\label{eq:tildeP}
	\SF_{\!\!\eps} P^\ga_t(\om) = \exp \Bigl( \varkappa_{\ga, 3}^2 \bigl( \widehat{K}_\ga(\om) - 1 \bigr) \frac{t}{\alpha} \Bigr), 
\end{equation}
for all $\om \in \{ -N, \ldots, N \}^3$. With a little ambiguity, we denote by $P^\ga_t$ the operator acting on functions $f : \Le \to \R$ by the convolution 
\begin{equation}\label{eq:P-gamma}
\bigl(P^\ga_t f\bigr)(x) = \eps^3 \sum_{y \in \Le} P^\ga_t(x - y) f(y).
\end{equation}
It will be also convenient to define the kernel 
\begin{equation}\label{eq:P-gamma-tilde}
\widetilde{P}^{\ga}_{t}(x) := P^\ga_{t} \ae K_\ga(x), 
\end{equation}
and the respective integral operator is defined by analogy with \eqref{eq:P-gamma}. We can then rewrite the discrete equation \eqref{eq:equation-for-X} in the mild form
\begin{align}\label{eq:IsingKacEqn}
	X_\ga(t, x) &= P^\ga_t X^0_\ga(x) + \sqrt 2\,  Y_\ga(t, x)\\
	&\qquad + \int_0^t \widetilde{P}^{\ga}_{t-s} \Bigl(-\frac{\be^3}{3} X^3_\ga + \bigl(\fC_\ga + A\bigr) X_\ga + E_\ga \Bigr)(s, x)\, \d s, \nonumber
\end{align}
where we have used the inverse temperature \eqref{eq:beta} and where
\begin{equation}\label{eq:Y-def}
	Y_\ga (t, x) := \frac{1}{\sqrt 2} \eps^3 \sum_{y \in \Le} \int_0^t \widetilde{P}^{\ga}_{t-s}(x-y) \,\d \M_\ga(s, y).
\end{equation}
Here and in the following, we always write stochastic integrals with respect to the time variable (which is $s$ in this integral). 

\subsection{A priori bounds}
\label{sec:a-priori}

In the proof of Theorem~\ref{thm:main}, we are going to show convergence of $X_\ga(t)$ in a stronger topology than $\SD'(\T^3)$. For this we need to control this process using the semi-norm \eqref{eq:eps-norm-1}. More precisely, for a fixed constant $\fa \geq 1$ and the value $\eta$ as in the statement of Theorem~\ref{thm:main} we define the stopping time 
\begin{equation}\label{eq:tau-1}
\tau^{(1)}_{\ga, \fa} := \inf \Bigl\{t \geq 0 :  \| X_\ga(t) \|^{(\emezo)}_{\CC^{\eta}} \geq \fa \Bigr\}.
\end{equation}
On the random time interval $[0, \tau^{(1)}_{\ga, \fa})$ we have the a priori bound $\| X_\ga(t) \|^{(\emezo)}_{\CC^{\eta}} \leq \fa$, while on the closed interval $[0, \tau^{(1)}_{\ga, \fa}]$ the bound is $ \| X_\ga(t) \|^{(\emezo)}_{\CC^{\eta}} \leq \fa + 2 \varkappa_{\ga, 1}$ almost surely. The two bounds are different because there may be a jump of the process at time $\tau^{(1)}_{\ga, \fa}$, and as one can see from \eqref{eq:equation-for-X} the jump size of $X_\ga(t, x)$ is bounded by the jump size of $\big( K_\ga \ae \M_\ga \big) (t, x)$, and the latter is almost surely bounded by $\frac{2 \eps^3}{\delta} \sup_{x \in \Lambda_\eps} |K_\ga(x)| \leq 2 \varkappa_{\ga, 1}$. Here, we used the properties that the jump size of the martingale $\M_\ga(t,x)$ is $\frac{2}{\delta}$ and a jump at time $t$ may almost surely happen only at one $x$. As follows from the definition of $\varkappa_{\ga, 1}$ in \eqref{eq:K-gamma}, it converges to $1$ as $\ga \to 0$. Since we always consider $\ga$ sufficiently small, we can assume that $\varkappa_{\ga, 1} \leq 2$, and hence $\| X_\ga(t) \|^{(\emezo)}_{\CC^{\eta}} \leq \fa + 4 \leq 5 \fa$ almost surely on $[0, \tau^{(1)}_{\ga, \fa}]$. Using \eqref{eq:uniform-bound} we also have a uniform bound on this process
\begin{equation}\label{eq:X-apriori}
|X_\ga(t, x)| \leq 5 \fa \emezo^{\eta}
\end{equation}
almost surely. Staying on the time interval $[0, \tau^{(1)}_{\ga, \fa}]$ is also sufficient to control the bracket process \eqref{eq:M-bracket}. More precisely, for $t < \tau^{(1)}_{\ga, \fa}$ we have \eqref{eq:X-apriori} and the random part of \eqref{eq:M-bracket} is bounded by
\begin{equation*}
\Bigl| \sigma \Bigl(\frac{t}{\alpha}, \frac{x}{\eps}\Bigr) \tanh \big( \be \delta X_\ga(t, x) \big) \Bigr| \lesssim \delta \fa \emezo^{\eta} \lesssim \fa \ga^{3 (1 + \eta)},
\end{equation*}
where we used the estimate $|\tanh(x)| \leq |x|$ for any $x \in \R$, where we estimated $\beta$ by a constant and where we used the scaling \eqref{eq:scalings}. Since $\eta > -1$, the preceding expression vanishes as $\ga \to 0$ and the bracket process \eqref{eq:M-bracket} converges to the covariance of a cylindrical Wiener process. 

To control the discrete model, constructed in Section~\ref{sec:lift}, we need to introduce another stopping time. For this we define the rescaled spin field
\begin{equation}\label{eq:S-def}
S_\ga(t,x) := \frac{1}{\delta} \sigma \Bigl( \frac{t}{\alpha}, \frac{x}{\eps} \Bigr) \qquad \text{for}~~ x \in \Le,~ t \geq 0.
\end{equation}
In Section~\ref{sec:second-symbol} we will need to control the product $S_{\ga}(t, x) X_\ga(t, x)$, appearing in the random part of the bracket process \eqref{eq:M-bracket}, in a suitable space of distributions. For this, we will show in Lemma~\ref{lem:Y-approximation} that the spin field $S_\ga$ can be replaced, up to an error, by its local average. More precisely, we take any smooth, rotation invariant function $\un{\fK} : \R^3 \to \R$, supported in the ball of radius $2$ and centered at the origin, whose continuous Fourier transform satisfies $\SF \un{\fK} (\om) = 1$, for all $\om \in \R^3$ such that $|\om|_{\sinfty} \leq 1$. Then for a fixed constant $\un{\kappa} \in (0, \frac{1}{10})$ we define
\begin{equation}\label{eq:under-K-def}
\un{\fK}_\ga(k) := \un{c}_{\ga, 1} \ga^{3(1-\un{\kappa})} \un{\fK} \bigl(\ga^{1-\un{\kappa}} k\bigr) \qquad \text{with} \quad \un{c}_{\ga, 1}^{-1} := \sum_{k \in \LN} \ga^{3(1-\un{\kappa})} \un{\fK} \bigl(\ga^{1-\un{\kappa}} k\bigr),
\end{equation}
and 
\begin{equation}\label{eq:X-under}
\un{X}_{\ga} (t, x) := \bigl(\un{K}_{\ga} \ae S_\ga\bigr) (t,x), \qquad\qquad \un{K}_{\ga}(x) := \frac{1}{\eps^{3}} \un{\fK}_\ga \Bigl( \frac{x}{\eps} \Bigr).
\end{equation}
In contrast to \eqref{eq:X-gamma}, where the local average of the rescaled spin field $S_\ga$ is computed in a ball of radius of order $\ga^{3}$, the process $\un{X}_{\ga}(t)$ is defined as a local average of spins in a ball of a smaller radius of order $\ga^{3 + \un{\kappa}}$. A precise value of $\un \kappa$ will not play any significant role, as soon as it is small enough. In particular, taking $\un{\kappa} < \frac{1}{10}$ will later allow us to use Lemma~\ref{lem:unX-X-bound}.

Then for $\eta$ as in the statement of Theorem~\ref{thm:main} and for the constant
\begin{equation}\label{eq:c-under}
\un{\fC}_\ga := 2 \varkappa_{\ga, 2} \int_{0}^\infty \eps^3 \sum_{x \in \Le} \bigl(P^\ga_t \ae \un{K}_\ga\bigr)(x) \widetilde{P}^\ga_t(x) \,\d t, 
\end{equation}
where we use \eqref{constant:isingkac_one}, we define the stopping time
\begin{equation}\label{eq:tau-2}
\tau^{(2)}_{\ga, \fa} := \inf \Bigl\{t \geq 0 : \| \un{X}_\ga (t) X_\ga(t) - \un{\fC}_\ga \|^{(\un{\emezo})}_{\CC^{-1-\un\kappa}} \geq \fa \emezo^{\un\kappa / 2 - 1} \Bigr\},
\end{equation}
where $\un{\emezo} := \emezo \ga^{\un\kappa}$. Since both of the involved processes $\un{X}_\ga$ and $X_\ga$ are expected to converge to distributions as $\ga \to 0$, the product $\un{X}_\ga X_\ga$ needs to be renormalised by subtracting the divergent constant $\un{\fC}_\ga$. From Lemma~\ref{lem:renorm-constant-underline} we have $|\un{\fC}_\ga| \lesssim \emezo^{-1}$, and we expect that $\| \un{X}_\ga (t) X_\ga(t) \|^{(\un{\emezo})}_{\CC^{-1-\un\kappa}}$ blows up with the speed $\emezo^{-1}$. The speed of blow-up in \eqref{eq:tau-2}, after renormalising the product, is slower. It can be significantly improved, but the presented speed is enough for our estimates in Section~\ref{sec:second-symbol}.

To combine the two stopping times \eqref{eq:tau-1} and \eqref{eq:tau-2}, we set 
\begin{equation}\label{eq:tau}
\tau_{\ga, \fa} := \tau^{(1)}_{\ga, \fa} \wedge \tau^{(2)}_{\ga, \fa},
\end{equation}
and we restrict the time variable to the interval $[0, \tau_{\ga, \fa}]$. For this it will be convenient to consider a stopped process $\sigma(t)$, extended beyond the random time $\tau_{\ga, \fa}$. To define such extension, we introduce a new spin system $\sigma'_{\ga, \fa}$ which starts from the configuration $\sigma'_{\ga, \fa}\bigl(\frac{\tau_{\ga, \fa}}{\alpha}\bigr) = \sigma\bigl(\frac{\tau_{\ga, \fa}}{\alpha}-\bigr)$ and which for the times $t > \frac{\tau_{\ga, \fa}}{\alpha}$ is given by the infinitesimal generator $\SL_\ga'$ given by \eqref{eq:generator} with the flip rates\footnote{The process $\sigma'$ defined by the generator $\SL_\ga'$ is called a ``voter model'' \cite{Liggett}. The scaling limit of the one-dimensional Ising-Kac model near the critical temperature was proved in \cite{MR1358083} by using a coupling of these two models.}
\begin{equation*}
c'_{\ga}(\sigma, j) = \frac{1}{2} \Bigl( 1 - \sigma(j) h_\ga(\sigma, j) \Bigr).
\end{equation*}
Then we set 
\begin{equation}\label{eq:sigma-stopped}
\sigma_{\ga, \fa}(t) := 
\begin{cases}
\sigma(t) &\text{for}~~ t < \frac{\tau_{\ga, \fa}}{\alpha}, \\
\sigma'_{\ga, \fa}(t) &\text{for}~~ t \geq \frac{\tau_{\ga, \fa}}{\alpha},
\end{cases}
\end{equation}
where $\alpha$ is from \eqref{eq:scalings}. The reason to make this particular choice for the extension is in a good control of the rescaled spin field $X'_{\ga, \fa}$, defined as in \eqref{eq:X-gamma} for the process $\sigma'_{\ga, \fa}$. More precisely, we show in Lemma~\ref{lem:X-prime-bound} that $X'_{\ga, \fa}$ solves a linear equation which allows to bound it globally in time.

 We define the martingales $\M_{\ga, \fa}$ via the process $\sigma_{\ga, \fa}$ in the same way as we defined $\M_\ga$ in \eqref{eq:martingales} via the process $\sigma$. For $t < \tau_{\ga, \fa}$ the martingale $\M_{\ga, \fa}(t)$ coincides with $\M_{\ga}(t)$, while for $t \geq \tau_{\ga, \fa}$ we denote $\M_{\ga, \fa}(t) = \M'_{\ga, \fa}$, where the latter has the predictable quadratic covariations 
 \begin{equation}\label{eq:M-prime-bracket}
\big\langle \M'_{\ga, \fa}(\bigcdot, x), \M'_{\ga, \fa}(\bigcdot, x') \big\rangle_t = 2 \varkappa_{\ga, 2} \delta^{(\eps)}_{x,x'} \int_{\tau_{\ga, \fa}}^{t} \Bigl( 1 - \delta \sigma'_{\ga, \fa} \Bigl(\frac{s}{\alpha}, \frac{x}{\eps}\Bigr) X'_{\ga, \fa}(s, x) \Bigr) \d s,
\end{equation}
with $\delta^{(\eps)}_{x,x'}$ defined below \eqref{eq:M-bracket} and $\varkappa_{\ga, 2}$ is defined in \eqref{constant:isingkac_one}. Then for $t \geq \tau_{\ga, \fa}$ we have 
\begin{equation}\label{eq:M-a-variation}
\big\langle \M_{\ga, \fa}(\bigcdot, x), \M_{\ga, \fa}(\bigcdot, x') \big\rangle_t = \big\langle \M_\ga(\bigcdot, x), \M_\ga(\bigcdot, x') \big\rangle_{\tau_{\ga, \fa}} + \big\langle \M'_{\ga, \fa}(\bigcdot, x), \M'_{\ga, \fa}(\bigcdot, x') \big\rangle_t.
\end{equation}
We define $X_{\ga, \fa}$ as the solution of an analogue of equation \eqref{eq:IsingKacEqn}, driven by these new martingales
\begin{align}\label{eq:IsingKacEqn-periodic}
	X_{\ga, \fa}(t, x) &= P^\ga_t X^0_\ga(x) + \sqrt 2\,  Y_{\ga, \fa}(t, x) \\
	&\qquad + \int_0^t \widetilde{P}^{\ga}_{t-s} \Bigl( -\frac{\be^3}{3} X^3_{\ga, \fa} + \bigl(\fC_\ga + A\bigr) X_{\ga, \fa} + E_{\ga, \fa} \Bigr)(s, x)\, \d s, \nonumber
\end{align}
where 
\begin{equation*}
	Y_{\ga, \fa} (t, x) := \frac{1}{\sqrt 2} \eps^3 \sum_{y \in \Le} \int_0^t \widetilde{P}^{\ga}_{t-s}(x-y)\, \d \M_{\ga, \fa}(s, y),
\end{equation*}
and the error term $E_{\ga, \fa}$ is defined in the same way as $E_{\ga}$ in \eqref{eq:expr_error_term}, but via the process $X_{\ga, \fa}$. For $t \leq \tau_{\ga, \fa}$ we have $X_{\ga, \fa}(t) = X_{\ga}(t)$, and for $t > \tau_{\ga, \fa}$ we have $X_{\ga, \fa}(t) = X'_{\ga, \fa}(t)$.

Working with the process $X_{\ga, \fa}$ is advantageous, because we can use the a priori bounds provided by the stopping times \eqref{eq:tau-1} and \eqref{eq:tau-2}, which guarantees convergence of the martingales and their lift to a discrete model (see Proposition~\ref{prop:models-converge}). To prove Theorem~\ref{thm:main}, we will first prove the respective convergence result for $X_{\ga, \fa}$ and then we will take the limit $\fa \to \infty$. In order to show that $\tau_{\ga, \fa}$ almost surely diverges in these limits, we will prove that this stopping time is close to a stopping time of the limiting process $X$, and the latter is almost surely infinite. 

\subsection{Periodic extensions}

We are going to write equation \eqref{eq:IsingKacEqn-periodic} in the framework of regularity structures. For this, we need to write this equation on the whole domain $\Lattice$ rather than on the torus $\Le$. To do this, we denote by $G^\ga_t : \Lattice \to \R$ the discrete heat kernel, which solves equation \eqref{eq:kernels-P} on $\Lattice$ (one can see that for $\ga$ small enough, the discrete operator $\Delta_\ga$ is naturally extended to functions on $\Lattice$). Then we have the identity
\begin{equation}\label{eq:From-P-to-G}
 \eps^3 \sum_{x \in \Le} P^\ga_t(x) f(x) = \eps^3 \sum_{x \in \Lattice} G^\ga_t(x) f(x),
\end{equation}
for any $f : \Le \to \R$, where on the right-hand side we extended $f$ periodically to $\Lattice$. We define respectively
\begin{equation}\label{eq:From-P-to-G-tilde}
\widetilde{G}^{\ga}_{t}(x) := \eps^3 \sum_{y \in \Lattice} G^{\ga}_{t}(x - y)K_\ga(y).
\end{equation}
Then equation \eqref{eq:IsingKacEqn-periodic} may be written as
\begin{align}\label{eq:IsingKacEqn-new}
	X_{\ga, \fa}(t, x) &= G^\ga_t X^0_\ga(x) + \sqrt 2\, Y_{\ga, \fa}(t, x) \\
&\qquad + \int_0^t \widetilde{G}^{\ga}_{t-s} \Bigl( -\frac{\be^3}{3} X^3_{\ga, \fa} + \bigl(\fC_{\ga} + A\bigr) X_{\ga, \fa} + E_{\ga, \fa} \Bigr)(s, x)\, \d s, \nonumber
\end{align}
where we extended all the involved processes periodically to $\Lattice$.

\section{The dynamical $\Phi^4_3$ model} 
\label{sec:phi4Section}

In this section we recall the notion of solution to the $\Phi^4$ equation \eqref{eq:Phi43} on the three-dimensional torus. Following \cite{Regularity}, we describe the solution in the framework of regularity structures. Throughout the section we are going to use singular modelled distribution and their basic properties, which can be found in \cite{Regularity}. However, we prefer to duplicate some of the definitions here to have a better motivation for the setting of Sections~\ref{sec:discreteRegStruct} and \ref{sec:discrete-solution}.

\subsection{A model space}
\label{sec:model-space-cont}

In this section we introduce an infinite set $\fT$ and a finite-dimensional regularity structure $\ST = (\CA, \CT, \CG)$ such that $\CT \subset \fT$ and that is required to describe equation \eqref{eq:Phi43}. 

To define the space $\fT$, it is convenient to use some ``abstract symbols'' as its basis elements. Namely, $\blueXi$ will represent the driving noise in \eqref{eq:Phi43}, the integration map $\CI$ will represent the space-time convolution with the heat kernel, i.e. the Green's function of the parabolic operator $\partial_t - \Delta$ on $\R^3$. The symbols $\X_i$, $i = 0, \ldots ,3$, will represent the time and space variables, and for $\ell = (\ell_0, \ldots, \ell_3) \in \N_0^4$ we will use the shorthand $\X^\ell = \X_0^{\ell_0} \X_1^{\ell_1} \X_2^{\ell_2} \X_3^{\ell_3}$, with the special unit symbol $\blueOne := \X^0$. We define $\CW_{\poly} := \{\X^\ell : \ell \in \N_0^4\}$ to be the set of all monomials.

Then we define the minimal sets $\CV$ and $\CU$ of formal expressions such that $\blueXi \in \CV$, $\CW_{\poly} \subset \CV \cap \CU$ and the following implications hold:
\begin{subequations}\label{eqs:rules}
\begin{align}
    \tau \in \CV \quad &\Rightarrow \quad \CI(\tau) \in \CU, \label{eq:rule1}\\
    \tau_1, \tau_2, \tau_3 \in \CU \quad &\Rightarrow \quad \tau_1 \tau_2 \tau_3 \in \CV, \label{eq:rule2}
\end{align}
\end{subequations}
where the product of symbols is commutative with the convention $\blueOne \tau = \tau$. We postulate $\CI(\X^\ell) = 0$ and do not include such zero elements into $\CU$ and $\CV$. The set $\CU$ contains the elements needed to describe the solution of \eqref{eq:Phi43}, while $\CV$ contains the elements to describe the expression on the right-hand side of this equation. Namely, the relation \eqref{eq:rule1} means that the elements of $\CU$ are obtained by integrating the elements on the right-hand side of the equation. The rule \eqref{eq:rule2} means that the right-hand side of \eqref{eq:IsingKacEqn-new} contains the third power of the solution (we note that since $\blueOne \in \CU$, the set $\CV$ also contains the symbols $\tau_1$ and $\tau_1 \tau_2$ for all $\tau_1, \tau_2 \in \CU$). 
\medskip

We set $\fW := \CU \cup \CV$, and for a fixed $\kappa \in (0, \frac{1}{14})$ we define the homogeneity $|\bigcdot| : \fW \to \R$ of each element of $\fW$ by the recurrent relations 
\begin{subequations}\label{eqs:hom}
\begin{align}
    |\X^\ell| &= |\ell|_\s, \label{eq:hom1}\\
    |\blueXi| &= -\frac{5}{2} - \kappa,\label{eq:hom2} \\
    |\tau_1 \tau_2| &= |\tau_2| + |\tau_2|, \label{eq:hom3}\\
    |\CI(\tau)| &= |\tau| + 2, \quad \tau \notin \CW_{\poly}. \label{eq:hom4}
\end{align}
\end{subequations}
The definition \eqref{eq:hom1} takes into account the parabolic scaling of space-time; \eqref{eq:hom2} is the regularity of the space-time white noise; \eqref{eq:hom4} is motivated by the Schauder estimate, i.e. a convolution with the heat kernel increases regularity by $2$. One can readily see that for any $\kappa < \frac{1}{14}$ and for any $\zeta \in \R$ the set $\{\tau \in \fW : |\tau| < \zeta\}$ is finite. The restriction $\kappa < \frac{1}{14}$ will be useful later in Section~\ref{sec:discrete-solution} and it is explained in Remark~\ref{rem:regularity-explanation}. 

We define $\fT$ to contain all finite linear combinations of the elements in $\fW$, and we view $\CI$ as a linear map $\tau \mapsto \CI(\tau)$, defined on the subspace generated by $\{ \blueXi, \CI(\blueXi)^2, \CI(\blueXi)^3 \}$. Our definition of this map implies that it can be considered as ``an abstract integration map'' from \cite{Regularity}. The set $\fA$ contains the homogeneities $|\tau|$ for all $\tau \in \fW$.

In order to solve equation \eqref{eq:Phi43}, it is enough to consider the elements in $\fW$ with negative homogeneities to describe the right-hand side, while the solution of this equation is described by the elements of homogeneities not exceeding $1 + 3\kappa$. Hence, we define
\begin{equation}\label{eq:basis}
     \CW := \{\tau \in \CV : |\tau| \leq 0,\, \tau \neq \blueXi\} \cup \{\tau \in \CU : |\tau| \leq 1 + 3 \kappa\}.
\end{equation}
This is the minimal set of the basis elements of a regularity structure, which will allow us to solve the equation \eqref{eq:abstract_equation}, an abstract version of \eqref{eq:Phi43}. We will see in Section~\ref{sec:lift}, that the element $\blueXi$ plays a special role; namely, $\blueXi$ corresponds to a distribution (a time derivative of a martingale), while the other elements correspond to functions. That is why it will be convenient to remove $\blueXi$ from the regularity structure. 

As we will see, the set $\CV$ contains the elements describing the right-hand side of \eqref{eq:abstract_equation}, except the noise element $\blueXi$ which we prefer to exclude. In order to get the right-hand side of \eqref{eq:Phi43} after reconstruction of the right-hand side of \eqref{eq:AbstractDiscreteEquation}, it is enough to use the elements of $\CV$ with positive homogeneities. This explains why we use only the elements $\{\tau \in \CV : |\tau| \leq 0\}$ in \eqref{eq:basis}. As we explained above, we use the elements $\{\tau \in \CU : |\tau| \leq 1 + 3 \kappa\}$, because we are going to solve equation \eqref{eq:abstract_equation} in a space of modelled distributions of regularity $1 + 3 \kappa$. 

We define $\CT$ to be the linear span of the elements in $\CW$, and the set $\CA$ contains the homogeneities $|\tau|$ for all elements $\tau \in \CW$.

It is convenient to represent the elements of $\CW$ as trees. Namely, we denote $\blueXi$ by a node \<oneNode>\,. When a map $\CI$ is applied to a symbol $\tau$, we draw an edge from the root of the tree representing this symbol $\tau$. For example, the symbol $\CI(\blueXi)$ is represented by the diagram $\<1b>$. The product of symbols $\tau_1, \ldots, \tau_n$ is represented by the tree, obtained from the trees of these symbols by drawing them from the same root. For example, $\<2b>$ and $\<3b>$ are the diagrams for $\CI(\blueXi)^2$ and $\CI(\blueXi)^3$ respectively. We use the symbols for the polynomials as before. In Table~\ref{tab:symbols-cont} we provide the elements of $ \CW$ and their homogeneities.

\begin{table}[h]
\centering
\begingroup
\setlength{\tabcolsep}{10pt} 
\renewcommand{\arraystretch}{1.2}
	\subfloat{
    	\begin{tabular}{cc}
		\hline
		\textbf{Element} & \textbf{Homogeneity} \\
		\hline
		$\blueOne$ & $0$ \\		
		$\X_i$, $i = 1, 2, 3$    & $1$ \\
		$\<1b>$              & $-\frac{1}{2}-\kappa$ \\
		$\<2b>$                   & $-1-2\kappa$ \\
		$\<2b> \X_i$, $i = 1, 2, 3$     & $-2\kappa$ \\
		$\<3b>$                   & $-\frac{3}{2}-3\kappa$
	\end{tabular}  \hspace{0.3cm}}
	\subfloat{\hspace{0.3cm}
	\begin{tabular}{cc}
		\hline
		\textbf{Element} & \textbf{Homogeneity} \\
		\hline
		$\<3b> \X_i$, $i = 1, 2, 3$     & $-\frac{1}{2}-3\kappa$ \\
		$\<20b>$ & $\ 1-2\kappa$ \\
		$\<30b>$ & $\frac{1}{2}-3\kappa$ \\
		 $\<22b>$  & $ -4\kappa$ \\
		 $\<31b>$ & $ -4\kappa$ \\		
		$\<32b>$                  & $-\frac{1}{2}-5\kappa$ \\
	\end{tabular}}
\endgroup
	\caption{The elements of $\CW$ and their homogeneities. \label{tab:symbols-cont}}
\end{table}

 Every element $f \in \fT$ can be uniquely written as $f = \sum_{\tau \in \fW} f_\tau \tau$ for $f_\tau \in \R$, and we define 
\begin{equation}\label{eq:RS-norm}
    |f|_\alpha := \sum_{\tau \in \fW : |\tau| = \alpha} |f_\tau|,
\end{equation}
postulating $|f|_\alpha=0$ if the sum runs over the empty set. We also introduce the projections 
\begin{equation}\label{eq:RS-projections}
    \CQ_{< \alpha} f := \sum_{\tau \in \fW : |\tau| < \alpha} f_\tau \tau, \qquad\qquad \CQ_{\leq \alpha} f := \sum_{\tau \in \fW : |\tau| \leq \alpha} f_\tau \tau.
\end{equation}
Let the model space $\fT_{< \alpha}$ contain all the elements $f \in \fT$ satisfying $f = \CQ_{< \alpha} f$. All these definitions can be immediately projected to $\CT$.

\subsection{A structure group}
\label{sec:structure-group-cont}

In order to use the results of \cite{Regularity}, we need to define a structure group $\CG$. For this, we need to introduce another set of basis elements $\CW_+$, containing $\blueOne$, $\X_i$ for $i = 1, 2, 3$, and the elements of $\CW$ of the form $\CI(\tau)$ for $\tau \neq \blueXi$. Then we define $\CT_+$ to be the free commutative algebra generated by the elements of $\CW_+$.

We define a linear map $\Delta : \CT \to \CT \otimes \CT_+$ by the identities 
\begin{subequations}\label{eqs:coproduct}
\begin{equation}
\Delta \blueOne = \blueOne \otimes \blueOne, \qquad \Delta \X_i = \X_i \otimes \blueOne + \blueOne \otimes \X_i,
\end{equation} 
and then recursively by (we denote by $I$ the identity operator on $\CT_+$) 
\begin{align}
\Delta \tau_1 \tau_2 &= (\Delta \tau_1) (\Delta \tau_2), \\
\Delta \CI(\blueXi) &= \CI(\blueXi) \otimes \blueOne, \label{eq:Delta3} \\
\Delta \CI(\tau) &= (\CI \otimes I) \Delta \tau + \blueOne \otimes \CI (\tau), \quad \tau \neq \blueXi, \label{eq:Delta4}
\end{align}
\end{subequations} 
for respective elements $\tau_i, \tau, \bar \tau \in \CW$. In Table~\ref{tab:Delta-cont} we write $\Delta \tau$ for all $\tau \in \CW$.

\begin{table}[h]
\centering
\begingroup
\setlength{\tabcolsep}{10pt} 
\renewcommand{\arraystretch}{1.2}
	\subfloat{
    	\begin{tabular}{c} 
		$\Delta \blueOne = \blueOne \otimes \blueOne$ \\
		$\Delta \X_i = \X_i \otimes \blueOne + \blueOne \otimes \X_i$ \\
		$\Delta \<1b> = \<1b> \otimes \blueOne$ \\
		$\Delta \<2b> = \<2b> \otimes \blueOne$ \\
		$\Delta \<2b> \X_i = \<2b> \X_i \otimes \blueOne + \<2b> \otimes \X_i$ \\
		$\Delta \<3b> = \<3b> \otimes \blueOne$
	\end{tabular} \hspace{0.3cm}}
	\subfloat{\hspace{0.5cm}
	\begin{tabular}{c}
		$\Delta \<3b> \X_i = \<3b> \X_i \otimes \blueOne + \<3b> \otimes \X_i$ \\
		$\Delta \<20b> = \<20b> \otimes \blueOne + \blueOne \otimes \<20b>$ \\
		$\Delta\<30b> = \<30b> \otimes \blueOne + \blueOne \otimes \<30b>$ \\
		 $\Delta\<22b> = \<22b> \otimes \blueOne + \<2b> \otimes \<20b> $ \\
		 $\Delta\<31b> = \<31b> \otimes \blueOne + \<1b> \otimes \<30b>$ \\
		$\Delta\<32b> = \<32b> \otimes \blueOne + \<2b> \otimes \<30b>$
	\end{tabular}}
\endgroup
	\caption{The image of the operator $\Delta$. \label{tab:Delta-cont}}
\end{table}

\begin{remark}
Since we restricted the set of basis elements \eqref{eq:basis}, our definition of the map $\Delta$ looks much easier than in \cite[Eq.~8.8b]{Regularity}. More precise, the general definition of $\Delta \CI(\tau)$ should be 
\begin{equation*}
\Delta \CI(\tau) = (\CI \otimes I) \Delta \tau + \sum_{\substack{k, \ell \in \N_0^4 \\ |k + \ell|_s < |\tau| + 2}} \frac{\X^k}{k!} \otimes \frac{\X^\ell}{\ell!} \CI_{k + \ell} (\tau),
\end{equation*}
where $\CI_{k + \ell}$ are new auxiliary symbols. Our definition \eqref{eq:basis} implies that there is at most one term in this sum, which yields \eqref{eq:Delta3} and \eqref{eq:Delta4}. 
\end{remark}

For any linear functional $f : \CT_+ \to \R$ we define the map $\Gamma_f : \CT \to \CT$ as 
\begin{equation}\label{eq:Gamma-general}
\Gamma_{\!f} \tau := (I \otimes f) \Delta \tau.
\end{equation}
Then the structure group $\CG$ is defined as $\CG := \{\Gamma_{\!f} : f \in \CG_+\}$, where $\CG_+$ contains all linear functionals $f : \CT_+ \to \R$ satisfying $f(\blueOne) = 1$. In general $f$ are assumed to be multiplicative \cite{Regularity}, i.e. $f(\tau \bar \tau) = f(\tau) f(\bar \tau)$ for $\tau, \bar \tau \in \CT_+$, but our set $\CT_+$ may involve only the products of the form $\blueOne \tau$ which simplifies this identity. 

Since the model space $\CT$ is generated by a small number of elements listed in Table~\ref{tab:symbols-cont}, we can describe the structure group $\CG$ explicitly. More precisely, $\CG$ contains all the transformations listed in Table~\ref{tab:linear_transformations} for any real constants $a_i$, for $i=0, \ldots, 3$, $b$ and $c$. 
\begin{table}[h]
\centering
\begingroup
\setlength{\tabcolsep}{10pt} 
\renewcommand{\arraystretch}{1.2}
	\subfloat{
    	\begin{tabular}{cc}
		\hline
		\textbf{Element} & \textbf{Image} \\
		\hline
		$\blueOne$ & $\blueOne$ \\		
		$\X_i$, $i = 1, 2, 3$ & $\X_i + a_i \blueOne$ \\
		$\<1b>$              & $\<1b>$ \\
		$\<2b>$                   & $\<2b>$ \\
		$\<2b> \X_i$, $i = 1, 2, 3$     & $\<2b> \X_i + a_i \<2b>$ \\
		$\<3b>$                   & $\<3b>$
	\end{tabular} \hspace{0.3cm}}
	\subfloat{\hspace{0.5cm}
	\begin{tabular}{cc}
		\hline
		\textbf{Element} & \textbf{Image} \\
		\hline
		$\<3b> \X_i$, $i = 1, 2, 3$     &  $\<3b> \X_i + a_i \<3b>$ \\
		$\<20b>$ & $\<20b> + b \blueOne$ \\
		$\<30b>$ & $\<30b> + c \blueOne$ \\
		 $\<22b>$  & $\<22b> + b \<2b>$ \\
		 $\<31b>$ & $\<31b> + c \<1b>$ \\		
		$\<32b>$                  & $\<32b> + c \<2b>$
	\end{tabular}}
\endgroup
	\caption{Linear transformations in $\CG$ of the elements in $\CW$. \label{tab:linear_transformations}}
\end{table}
The bijection between these constants and the functionals $f \in \CG_+$ is given by
\begin{equation*}
a_i = f(\X_i), \qquad b = f\bigl(\,\<20b>\,\bigr), \qquad c = f\bigl(\,\<30b>\,\bigr).
\end{equation*}

In the rest of this section we use the framework of \cite{Regularity} to work with the regularity structure $\ST = (\CA, \CT, \CG)$ just introduced. 

\subsection{A solution map}

Let $G$ be the heat kernel, i.e. the Green's function of the parabolic operator $\partial_t - \Delta$ on $\R^3$. As in \cite[Sec.~5]{Regularity}, we write it as $G = \SK + \SR$, where $\SR$ is smooth and $\SK$ is singular, compactly supported. Let furthermore, $Z = (\Pi, \Gamma)$ be the model on the regularity structure $\ST$ for the equation \eqref{eq:Phi43}, defined in \cite[Sec.~10.5]{Regularity} with respect to the kernel $\SK$. Using the value $\kappa$ from \eqref{eq:hom2}, we define the abstract integration operator 
\begin{equation}\label{eq:abstract-integral}
\CP := \CK_{\kappa} + R_{1 + 3 \kappa} \CR,
\end{equation}
where the operator $\CK_{\kappa}$ is defined in \cite[Eq.~5.15]{Regularity} via the kernel $\SK$ for the values $\beta = 2$ and $\gamma = \kappa$, the operator $R_{1 + 3 \kappa}$ is defined in \cite[Eq.~7.7]{Regularity} as a Taylor's expansion of the function $\SR$ up to the order $1 + 3 \kappa$, and $\CR$ is the reconstruction map for the model $Z$ defined in \cite[Thm.~3.10]{Regularity}. The choice of the values $\kappa$ and $1 + 3\kappa$ in \eqref{eq:abstract-integral} is motivated as follows. We are going to solve an abstract version of equation \eqref{eq:Phi43} for a modelled distribution $U \in \CD^{\zeta, \eta}$ with $\zeta = 1 + 3 \kappa$ being the minimal regularity such that the theory can be applied. Then the non-linearity $U^3$ of the equation is an element of the space $U \in \CD^{\zeta + 2 |\CI(\blueXi)|, \bar \eta}$, for $|\CI(\blueXi)|$ being the regularity of the sector in which $U$ takes values. Since $\zeta + 2 |\CI(\blueXi)| = \kappa$ (see Table~\ref{tab:symbols-cont}), the map $\CP$ should act on elements of $\CD^{\kappa, \bar \eta}$.

Using this integral operator, we define the modelled distribution 
\begin{equation}\label{eq:W-def}
    W(z) := \CP \1_+ (\blueXi)(z),
\end{equation}
where $\1_+$ is the projection of modelled distributions to $\R_+$ in the time variable. Using the polynomial lift of the convolution $G X^0$, defined in \cite[Lem.~7.5]{Regularity}, we consider the abstract equation
\begin{equation}\label{eq:abstract_equation}
U = \CQ_{< \zeta} \Bigl(G X^0 + \CP \1_+ F(U) + \sqrt 2\, W\Bigr),
\end{equation}
where $U \in \CD^{\zeta, \eta} (Z)$ is a modelled distribution, for $\zeta = 1 + 3 \kappa$ and $\eta \in \R$, and where the non-linearity $F$ is given by
\begin{equation*}
F(U) := \CQ_{\leq 0} \Bigl(- \frac{1}{3} U^3 + A U\Bigr).
\end{equation*}
We note that the product $U^3$ is in general an element of $\fT$ and may contain terms which are not included into the model space $\CT$. The aim of applying the projection $\CQ_{\leq 0}$ is to remove such terms. Respectively, the right-hand side of \eqref{eq:abstract_equation} may contain elements with homogeneities higher than $1$, but we consider only the projection to the homogeneities not exceeding $\zeta$.

Let us now consider a mollified noise $\xi_{\de} = \varrho_\de \star \xi$, where the mollifier $\varrho_\de$ is defined in \eqref{eq:rho} for $\de > 0$. Let us define $X_{\de}^{0} := \psi_\de * X^0$, where the mollifier $\psi_\de(x) := \frac{1}{\de^{3}} \psi (\frac{x}{\de})$ is defined for a smooth compactly supported function $\psi : \R^3 \to \R$, satisfying $\int_{\R^3} \psi(x) \d x = 1$. Let furthermore $U^{(\de)}$ be the solution of equation \eqref{eq:abstract_equation}, defined with respect to the initial condition $X_{\de}^{0}$ and  the model $Z^{(\de)} = (\Pi^{(\de)}, \Gamma^{(\de)})$, defined in \cite[Sec.~10.5]{Regularity} via the mollified noise $\xi_{\de}$. Then from \cite[Sec.~9.4]{Regularity} we conclude that the process $X_{\de} = \CR^{(\de)} U^{(\de)}$, where $\CR^{(\de)}$ is the reconstruction map for the model $Z^{(\de)}$ from \cite[Thm.~3.10]{Regularity}, is the classical solution of the SPDE
\begin{equation}\label{eq:Phi43-delta}
\bigl(\partial_t - \Delta\bigr) X_{\de} = - \frac{1}{3} X_{\de}^3 + \bigl(\fC_\de + A\bigr) X_{\de} + \sqrt 2\, \xi_{\de}, 
\end{equation}
with the initial condition $X_{\de}^{0}$ at time $0$. The renormalisation constant $\fC^{(\de)} \sim \de^{-1}$ is defined in \cite{Regularity} and is such that the solution of \eqref{eq:Phi43-delta} converges as $\de \to 0$ in a suitable space of distributions.

\begin{theorem}\label{thm:Phi-solution}
For $\zeta = 1 + 3 \kappa$ and for $\eta$ as in Theorem~\ref{thm:main}, equation \eqref{eq:abstract_equation} has a unique local in time solution $U \in \CD^{\zeta, \eta}(Z)$, and the solution map $U = \CS(X^0, Z)$ is locally Lipschitz continuous with respect to the initial state $X^0 \in \CC^\eta(\T)$ and the model $Z$.

Then the solution of \eqref{eq:Phi43} is defined as $X = \CR U$, where $\CR$ is the reconstruction map associated to the model $Z$ by \cite[Thm.~3.10]{Regularity}. Moreover, for any $T > 0$ and $p \geq 1$ one has 
\begin{equation*}
\E \biggl[\sup_{t \in [0, T]} \| X(t) \|^p_{\CC^\eta}\biggr] < \infty,
\end{equation*}
and the same bound holds for $\| (X - \sqrt 2\, \CR W)(t) \|_{\CC^{3 / 2 + 3 \eta}}$, where $W$ is defined in \eqref{eq:W-def}.

Finally, let $X_{\de}$ be the solution of \eqref{eq:Phi43-delta}. Then there exists $\theta > 0$ such that for any $T > 0$, $p \geq 1$ and for some $C > 0$, depending on $T$ and $p$, one has
\begin{equation}\label{eq:solution-continuity}
\E \biggl[\sup_{t \in [0, T]} \| (X - X_{\de})(t) \|^p_{\CC^\eta}\biggr] \leq C \de^{\theta p}
\end{equation}
uniformly over $\de \in (0,1]$.
\end{theorem}

\begin{proof}
Existence of a local solution and its continuity was proved in \cite[Prop.~9.10]{Regularity}. From \cite[Thm.~1.1]{from-infinity} we obtain the moment bounds on the processes $X$ and $X - \sqrt 2\, \CR W$. 
\end{proof}

One can readily see that the solution $U$ has the following expansion:
 \begin{equation}\label{eq:U-expansion} 
 U(z) = \sqrt 2\, \<1b> + v(z) \blueOne - \frac{2 \sqrt 2}{3}\, \<30b> - 2 v(z) \<20b> + \sum_{i = 1, 2,3} v^i(z) \X_i, 
\end{equation}
 for some functions $v, v^i : \R_+ \times \R^3 \to \R$. Indeed, this identity follows by writing the integration operator in \eqref{eq:abstract_equation} explicitly as
 \begin{equation*} 
 U(z) = \CI \Bigl(- \frac{1}{3} U(z)^3 + A U(z) + \sqrt 2\, \blueXi\Bigr) + v(z) \blueOne + \sum_{i = 1, 2,3} v^i(z) \X_i,
 \end{equation*}
 repeating the iterative approximation of the solution several times and truncating all terms with homogeneities strictly bigger than $1$. The function $v$ may be written as $v = X - \sqrt 2 Y$, where $X = \CR U$ and $Y = \CR W$, with $W$ defined in \eqref{eq:W-def},  and it solves the ``remainder equation''
 \begin{equation}\label{eq:remainder-equation}
\bigl(\partial_t - \Delta\bigr) v = - \frac{1}{3} \bigl(v + \sqrt 2\, Y\bigl)^3 + A \bigl(v + \sqrt 2\, Y\bigr),
 \end{equation}
 with the initial condition $X^0$ at time $t = 0$. Interpretation of the functions $v^i$ is more complicated, and we do not provide it here. Theorem~\ref{thm:Phi-solution} implies that for any $p \geq 1$ and $T > 0$ we have 
\begin{equation*}
\E \biggl[\sup_{t \in [0, T]} \| v(t) \|^p_{\CC^{3 / 2 + 3 \eta}}\biggr] < \infty.
\end{equation*}

\section{A regularity structure for the discrete equation} 
\label{sec:discreteRegStruct}

Proving convergence of the Ising-Kac model requires solving equation \eqref{eq:IsingKacEqn-new} using the theory of regularity structures. For this we are going to use the framework \cite{erhard2017discretisation}, which is suitable for solving approximate stochastic PDEs. A less general framework developed in \cite{HairerMatetski} could also be applied.

We would like to stress very clearly that the regularity structure for equation \eqref{eq:IsingKacEqn-new} is very similar to the one used to solve the $\Phi^4_3$ equation, except for the fact that in our setting we need to describe the additional error term $E_\ga$ defined in \eqref{eq:expr_error_term}. As we shall see, the local description of this error term involves the local description of the fifth power of the solution of our equation; this is the only reason why we need to introduce new trees which would not appear in the classical $\Phi^4_3$ solution theory.

In the following section we are going to define a regularity structure $\ST^\ex = (\CA^\ex, \CT^\ex, \CG^\ex)$ which extends the regularity structure $\ST$, defined in Section~\ref{sec:phi4Section}, by adding several basis elements. Throughout this section we are going to use the notation from Section~\ref{sec:phi4Section}.

\subsection{A model space}
\label{sec:model-space}

In addition to the integration map $\CI$ we introduce a new map $\CE$ which will represent the multiplication operator by $\emezo^2 \approx \gamma^{6}$. Then we define the minimal sets $\CV^\ex$ and $\CU^\ex$ of formal expressions by the implications \eqref{eqs:rules} and 
\begin{align}
    \tau_1, \ldots, \tau_5 \in \CU^\ex \quad &\Rightarrow \quad \CE(\tau_1 \cdots \tau_5) \in \CV^\ex, \label{eq:rule3}
\end{align}
where we postulate $\CE(\X^\ell) = 0$ and do not include such zero elements into $\CV^\ex$. The rule \eqref{eq:rule3} describes the remainder \eqref{eq:expr_error_term}, in the Taylor expansion of which the first non-vanishing element is proportional to $\gamma^6 X_\ga(t,x)^5$: in fact, the trees coming out from the rule \eqref{eq:rule3} are those which will allow a local description of the error tern $E_\ga$ (see also Remark~\ref{rem:regularity-explanation}).

We define the set of elements $\fW^\ex := \CU^\ex \cup \CV^\ex$ with the homogeneity $|\bigcdot| : \fW^\ex \to \R$ defined by \eqref{eqs:hom} and 
\begin{align}
    |\CE(\tau_1 \cdots \tau_5)| &= |\tau_1| + \cdots + |\tau_5| + 2, \quad \tau_1 \cdots \tau_5 \notin \CW_{\poly}. \label{eq:hom5}
\end{align}
The increase of homogeneity by $2$ in \eqref{eq:hom5} comes from the multiplier $\ga^6 \approx \emezo^2$. 

The set $\fT^\ex$ contains all finite linear combinations of the elements in $\fW^\ex$, and we view $\CI$ and $\CE$ as linear maps $\tau \mapsto \CI(\tau)$ and $\bar \tau \mapsto \CE(\bar\tau)$, defined on the subspaces generated by $\{ \blueXi, \CI(\blueXi)^2, \CI(\blueXi)^3 \}$ and $\{ \CI(\blueXi)^4, \CI(\blueXi)^5 \}$ respectively. Our definitions of these maps imply that they have the same properties (but, as just stated, different domains), and both of them can be considered as ``abstract integration maps'' from \cite{Regularity}. The set $\fA^\ex$ contains the homogeneities $|\tau|$ for all $\tau \in \fW^\ex$.

By analogy with \eqref{eq:basis} we define 
\begin{equation}\label{eq:basis-discrete}
     \CW^\ex := \{\tau \in \CV^\ex : |\tau| \leq 0,\, \tau \neq \blueXi\} \cup \{\tau \in \CU^\ex : |\tau| \leq 1 + 3 \kappa\} \cup \{\tau : \CE(\tau) \in \CV^\ex, |\tau| \leq -2\},
\end{equation}
where we also add to $\fW^\ex$ those $\tau$ such that $\CE(\tau) \in \CV^\ex$. This is the minimal set of the basis elements of a regularity structure, which will allow us to solve the equation \eqref{eq:AbstractDiscreteEquation}, an abstract version of \eqref{eq:IsingKacEqn-new}. We need the elements $\{\tau : \CE(\tau) \in \CV^\ex, |\tau| \leq -2\}$ to be able to reconstruct the non-linearity \eqref{eq:E2-ga}.

As before, we define $\CT^\ex$ to be the linear span of the elements in $\CW^\ex$, and the set $\CA^\ex$ contains the homogeneities $|\tau|$ for all elements $\tau \in \CW^\ex$. We obviously have $\CW \subset \CW^\ex$ and $\CT \subset \CT^\ex$ for the sets defined in Section~\ref{sec:model-space-cont}.

 As in Section~\ref{sec:model-space-cont}, we use the graphical representation of the elements of $\CW^\ex$, where application of the map $\CE$ is represented by the double edge $\<eb>$. For example, the diagram $\<40eb>$ represents the symbol $\CE(\CI(\blueXi)^4)$. Table~\ref{tab:symbols} contains those elements of $\CW^\ex$ which are not included in Table~\ref{tab:symbols-cont}. This setting is very similar to the one of the $\Phi^4_3$ solution theory, except that here we have an extra ``integration map''~$\CE$.

\begin{table}[h]
\centering
\begingroup
\setlength{\tabcolsep}{10pt} 
\renewcommand{\arraystretch}{1.2}
	\subfloat{
    	\begin{tabular}{cc}
		\hline
		\textbf{Element} & \textbf{Homogeneity} \\
		\hline
		$\<4b>$                  & $-2-4\kappa$ \\
		$\<5b>$                  & $-\frac{5}{2}-5\kappa$
	\end{tabular}  \hspace{0.3cm}}
	\subfloat{\hspace{0.3cm}
	\begin{tabular}{cc}
		\hline
		\textbf{Element} & \textbf{Homogeneity} \\
		\hline
		$\<40eb>$                 & $-4\kappa$ \\
		 $\<50eb>$                 & $-\frac{1}{2}-5\kappa$
	\end{tabular}}
\endgroup
	\caption{The elements of $\CW^\ex$ and their homogeneities which are not included into Table~\ref{tab:symbols-cont}. \label{tab:symbols}}
\end{table}

We are going the same notations for the norms and projections for the elements in $\fW^\ex$ as in \eqref{eq:RS-norm} and \eqref{eq:RS-projections}.

\subsection{A structure group}
\label{sec:structure-group}

We introduce the set of basis elements $\CW^\ex_+$, containing $\blueOne$, $\X_i$ for $i = 1, 2, 3$, and the elements of $\CW^\ex$ of the form $\CI(\tau)$ and $\CE(\bar \tau)$, for $\tau \neq \blueXi$. Then we define $\CT^\ex_+$ to be the free commutative algebra generated by the elements of $\CW^\ex_+$. The linear map $\Delta : \CT^\ex \to \CT^\ex \otimes \CT^\ex_+$ is define by \eqref{eqs:coproduct} and 
\begin{align}
\Delta \CE(\bar \tau) &= (\CE \otimes I) \Delta \bar \tau, \label{eq:Delta5}
\end{align}
for $\bar \tau \in \{\<4b>, \<5b>\}$. Then the action of $\Delta$ on the elements from $\CW$ is provided in Table~\ref{tab:Delta-cont} and the action on the other elements in $\CW^\ex$ is trivial and is provided in Table~\ref{tab:Delta}.

\begin{table}[h]
\centering
\begingroup
\setlength{\tabcolsep}{10pt} 
\renewcommand{\arraystretch}{1.2}
	\subfloat{
    	\begin{tabular}{c} 
		$\Delta\<4b> = \<4b> \otimes \blueOne$ \\
		$\Delta\<5b> = \<5b> \otimes \blueOne$
	\end{tabular} \hspace{0.3cm}}
	\subfloat{\hspace{0.5cm}
	\begin{tabular}{c}
		$\Delta\<40eb> = \<40eb> \otimes \blueOne$ \\
		 $\Delta\<50eb> = \<50eb> \otimes \blueOne$
	\end{tabular}}
\endgroup
	\caption{The image of the operator $\Delta$ for the elements in $\CW^\ex$ not provided in Table~\ref{tab:Delta-cont}. \label{tab:Delta}}
\end{table}

The structure group $\CG^\ex$ is defined as $\CG^\ex := \{\Gamma_{\!f} : f \in \CG^\ex_+\}$, where $\Gamma_{\!f}$ is given by \eqref{eq:Gamma-general} and $\CG^\ex_+$ contains all linear functionals $f : \CT^\ex_+ \to \R$ satisfying $f(\blueOne) = 1$. One can readily see that the elements of $\CG^\ex$ ant on $\CW$ as described in Table~\ref{tab:linear_transformations}, and they act on the other elements of $\CW^\ex$ as the identity maps.

We will use the framework of \cite{erhard2017discretisation} to work with the regularity structure $\ST^\ex = (\CA^\ex, \CT^\ex, \CG^\ex)$ just introduced on the discrete lattice $\Lattice$. 

\subsection{Discrete models}
\label{sec:DiscreteModels}

Let $\CB^2_\s$ be the set of all \emph{test functions} $\varphi \in \CC^2(\R^4)$, compactly supported in the ball of radius $1$ around the origin (with respect to the parabolic distance $\| \bigcdot \|_\s$ defined in Section~\ref{sec:notation}), and satisfying $\| \varphi \|_{\CC^2} \leq 1$. By analogy with \eqref{eq:rescaled-function}, for $\varphi \in \CB^2_\s$, $\lambda \in (0, 1]$ and $(s, y) \in \R^4$ we define a rescaled and recentered function 
\begin{equation}\label{eq:rescaled-function-general}
\varphi^\lambda_{(s, y)} (t, x) := \frac{1}{\lambda^5} \varphi \Bigl( \frac{t-s}{\lambda^2}, \frac{x-y}{\lambda} \Bigr). 
\end{equation}
In the rest of the paper we use the time-space domain $D_\eps := \R \times \Lattice$, where the spatial grid $\Lattice$ is defined in Section~\ref{sec:notation}.

In order to use the results of \cite{erhard2017discretisation}, we need to define a \emph{discretisation} for the regularity structure $\ST^\ex$ according to \cite[Def.~2.1]{erhard2017discretisation}. 

\begin{definition}\label{def:discretisation}
\begin{enumerate}[leftmargin=0.5cm]
\item We define the space $\CX_\eps := L^\infty (D_\eps)$, and we extend the operator \eqref{eq:iota} to $\iota_\eps: \CX_\eps \hookrightarrow L^\infty \bigl(\R, \SD'(\R^3)\bigr)$ as 
\[
(\iota_\eps f)(t, \bigcdot) := \bigl(\iota_\eps f(t)\bigr)(\bigcdot)
\]
 for $f \in \CX_\eps$. For any smooth compactly supported function $\varphi : \R^4 \to \R$ it will be convenient to write 
\begin{equation}\label{eq:iota-general}
(\iota_\eps f)(\varphi) := \eps^{3} \sum_{x \in \Lattice} \int_{\R} f(t, x) \varphi(t, x)\, \d t.
\end{equation}
\item For any $\zeta \in \R$, $z \in D_\eps$ and a compact set $K_\emezo \subset \R^4$ of diameter at most $2 \emezo$, we define the following seminorm $f \in \CX_\eps$:
\begin{equation}\label{eq:local-norm}
\Vert f \Vert_{\zeta; K_\emezo; z; \emezo} := \emezo^{-\zeta} \sup_{z \in K_\emezo \cap D_\eps} | f(z) |.
\end{equation}
Obviously, this seminorm is local in the sense that if $f, g \in \CX_\eps$ and $(\iota_\eps f)(\varphi) = (\iota_\eps g)(\varphi)$ for every $\varphi \in \CC^2$ supported in $K_\emezo$, then $\Vert f - g \Vert_{\zeta; K_\emezo; z; \emezo} = 0$.
\item Let the function $\varphi^\emezo_{z}$ be defined by \eqref{eq:rescaled-function-general} with $\lambda = \emezo$, and let $[\varphi^\emezo_{z}]$ denote its support. Then from the definition \eqref{eq:local-norm} we readily get the bound
\[
| (\iota_\eps f)(\varphi^\emezo_{z}) | \leq \biggl( \sup_{\bar z \in [\varphi^\emezo_{z}] \cap D_\eps} | f(\bar z) | \biggr) \eps^3 \sum_{x \in \Lattice} \int_{\R} | \varphi^\emezo_{z}(t, x) | \d t \lesssim \emezo^\zeta \Vert f \Vert_{\zeta; [\varphi^\emezo_z]; z; \emezo},
\]
uniformly over $f \in \CX_\eps$, $z \in D_\eps$, $\zeta \in \R$, and $\varphi \in \CB^2_\s$. 
\item\label{it:discretisation-md} For any function $\Gamma : D_\eps \times D_\eps \to \CG^\ex$, any compact set $K \subset \R^4$ and any $\zeta \in \R$, we define the following seminorm on the functions $f: D_\eps \to \CT^\ex_{< \zeta}$: 
\begin{equation}\label{eq:choice_distr_norm}
\$ f \$_{\zeta; K; \emezo} := \sup_{\substack{z, \bar z \in K \cap D_\eps \\ \Vert z - \bar z \Vert_{\s} \leq \emezo}} \sup_{m<\zeta} \emezo^{m -\zeta} | f(z) - \Gamma_{\!z \bar z} f(\bar z) |_{m}.
\end{equation}
For a second function $\bar \Gamma : D_\eps \times D_\eps \to \CG^\ex$ and for $\bar f: D_\eps \to \CT^\ex_{< \zeta}$ we also define 
\begin{equation}\label{eq:distance_mod_distr_eps}
\$ f; \bar{f} \$_{\zeta; K; \emezo} := \sup_{\substack{z, \bar z \in K \cap D_\eps \\ \Vert z - \bar z \Vert_\s \leq \emezo}} \sup_{m<\zeta} \emezo^{m -\zeta} \bigl| f(z) - \Gamma_{\!z \bar z} f(\bar z) - \bar{f}(z) + \bar{\Gamma}_{\!z \bar z} \bar{f}(\bar z) \bigr|_{m}.
\end{equation}
Both seminorms depend only on the values of $f$ and $\bar f$ in a neighbourhood of size $c \emezo$ around $K$, for a fixed constant $c > 0$.
\end{enumerate}
\end{definition}

\begin{remark}
The seminorms \eqref{eq:choice_distr_norm} and \eqref{eq:distance_mod_distr_eps} depends on the functions $\Gamma$ and $\bar \Gamma$. However, we prefer not to indicate it to have a lighter notation. The choice of these functions will be always clear from the context. 
\end{remark}

\begin{remark}\
Our definitions correspond to the ``semidiscrete'' case in \cite[Sec.~2]{erhard2017discretisation}.
\end{remark}

Following \cite[Def.~2.5]{erhard2017discretisation}, we can define a discrete model on the regularity structure~$\ST^\ex$.

\begin{definition}\label{def:model}
A \emph{discrete model} $(\Pi^\ga, \Gamma^\ga)$ on the regularity structure $\ST^\ex$ consists of a collection of maps $D_\eps \ni z \mapsto \Pi^\ga_z \in \CL(\CT^\ex, \CX_\eps)$ and $D_\eps \times D_\eps \ni (z, \bar z) \mapsto \Gamma^\ga_{\!z \bar z} \in \CG^\ex$ with the following properties: 
\begin{enumerate}
\item $\Gamma^\ga_{\!z z} = \id$ (where $\id$ is the identity operator), and $\Gamma^\ga_{\!z \bar z} \Gamma^\ga_{\!\bar z \tilde z} = \Gamma^\ga_{\!z \tilde z}$ for all $z, \bar z, \tilde z \in D_\eps$,
\item $\Pi^\ga_{\bar z} = \Pi^\ga_{z} \Gamma^\ga_{\!z \bar z}$ for all $z, \bar z \in D_\eps$.
\end{enumerate}
Furthermore, for any compact set $K \subset \R^4$ the following bounds hold
\begin{subequations}\label{eqs:model-bounds}
\begin{equation}\label{eq:Pi-bounds}
\sup_{\varphi \in \CB^2_\s} \sup_{z \in K \cap D_\eps} \bigl| \bigl(\iota_\eps \Pi^\ga_{z} \tau\bigr) (\varphi^\lambda_{z})\bigr| \lesssim \lambda^{|\tau|}, \qquad\qquad \sup_{K_\emezo \subset K} \sup_{z \in K \cap D_\eps} \| \Pi^\ga_{z} \tau \|_{|\tau|; K_\emezo; z; \emezo} \lesssim 1,
\end{equation}
uniformly over $\lambda \in [\emezo, 1]$ and $\tau \in \CW^\ex \setminus \{\<5b>\}$, where the supremum in the second bound is over compact sets $K_\emezo \subset K$ with the diameter not exceeding $2\emezo$. For the element $\tau = \<5b>$ we assume 
\begin{equation}\label{eq:Pi-bounds-bad}
\sup_{\varphi \in \CB^2_\s} \sup_{z \in K \cap D_\eps} \bigl| \bigl(\iota_\eps \Pi^\ga_{z} \tau\bigr) (\varphi^\lambda_{z})\bigr| \lesssim \gamma^{-1} \lambda^{|\tau| + \frac{1}{3}}, \qquad\qquad \sup_{K_\emezo \subset K} \sup_{z \in K \cap D_\eps} \| \Pi^\ga_{z} \tau \|_{|\tau| + \frac{1}{3}; K_\emezo; z; \emezo} \lesssim \gamma^{-1},
\end{equation}
uniformly over the same quantities. For the function $f^{\tau, \Gamma^\ga}_{\bar z}(z) := \Gamma^\ga_{\!z \bar z} \tau - \tau$ one has
\begin{equation}\label{eq:Gamma-bounds}
| \Gamma^\ga_{\!z \bar z} \tau |_m \lesssim \| z - \bar z \|_\s^{|\tau| - m}, \qquad\qquad \sup_{\bar z \in D_\eps} \$ f^{\tau, \Gamma^\ga}_{\bar z} \$_{|\tau|; K; \emezo} \lesssim 1,
\end{equation}
\end{subequations}
uniformly over $\tau \in \CW^\ex$, $m < |\tau|$ and $z, \bar z \in K \cap D_\eps$ such that $\| z - \bar z \|_\s \in [\emezo, 1]$. In the second bound in \eqref{eq:Gamma-bounds} we consider the seminorm \eqref{eq:choice_distr_norm} with respect to the map $\Gamma^\ga$.
\end{definition}

\begin{remark}
The first bounds in \eqref{eqs:model-bounds} control the model on the scale above $\emezo$ similarly to continuous models in \cite{Regularity}, and the second bounds in \eqref{eqs:model-bounds} control the model on the scale below $\emezo$. 
\end{remark}

\begin{remark}
We need to assume the much weaker bounds \eqref{eq:Pi-bounds-bad} for the element $\<5b>$, since in our definition in Section~\ref{sec:lift} $\Pi^\ga_{z} \<5b>$ is an approximation of an element of the fifth Wiener chaos. The latter is undefined in three dimensions because its correlation kernel is not integrable, which prevents us from imposing the uniform bounds \eqref{eq:Pi-bounds} (see Section~\ref{sec:fifth-chaos} for more details). This element is multiplied by $\gamma^6$ in the definition of solution in Section~\ref{sec:discrete-solution}, and the multiplier compensates the divergence assumed in \eqref{eq:Pi-bounds-bad}. We do not need to distinguish this element in \eqref{eq:Gamma-bounds} because $\Gamma^\ga_{\!z \bar z}$ acts trivially on it. 
\end{remark}

As we explained in Section~\ref{sec:model-space-cont}, we cannot define a model on the symbol $\blueXi$, because it corresponds to a distribution (a time derivative of the martingale) which is not an element of the space $\CX_\eps$ introduced in Definition~\ref{def:discretisation}.

We denote by $\Vert \Pi^\ga \Vert^{(\emezo)}_{K}$ and $\Vert \Gamma^\ga \Vert^{(\emezo)}_{K}$ the smallest proportionality constants such that the bounds \eqref{eq:Pi-bounds} and \eqref{eq:Gamma-bounds} hold respectively. Then for the model $Z^{\ga} = (\Pi^\ga, \Gamma^\ga)$ we set 
\begin{equation*}
\$ Z^{\ga} \$^{(\emezo)}_{K} := \Vert \Pi^\ga \Vert^{(\emezo)}_{K} + \Vert \Gamma^\ga \Vert^{(\emezo)}_{K}. 
\end{equation*}
For a second model $\bar Z^{\ga} = (\bar\Pi^\ga, \bar \Gamma^\ga)$ we define the ``distance'' 
\begin{equation*}
\$ Z^{\ga}; \bar Z^{\ga} \$^{(\emezo)}_{K} := \Vert \Pi^\ga - \bar \Pi^\ga \Vert^{(\emezo)}_{K} + \Vert \Gamma^\ga; \bar \Gamma^\ga \Vert^{(\emezo)}_{K},
\end{equation*}
where $\Vert \Gamma^\ga; \bar \Gamma^\ga \Vert^{(\emezo)}_{K}$ is the smallest proportionality constant such that the following bounds hold
\begin{equation*}
\| \bigl(\Gamma^\ga_{\!z \bar z} - \bar \Gamma^\ga_{\!z \bar z}\bigr) \tau \|_m \lesssim \| z - \bar z \|_\s^{|\tau| - m}, \qquad \sup_{\bar z \in D_\eps} \$ f^{\tau, \Gamma^\ga}_{\bar z}; f^{\tau, \bar \Gamma^\ga}_{\bar z} \$_{|\tau|; K; \emezo} \lesssim 1,
\end{equation*}
uniformly over the same quantities as in \eqref{eq:Gamma-bounds}, where in the second bound we consider the distance \eqref{eq:distance_mod_distr_eps} with respect to $\Gamma^\ga$ and $\bar \Gamma^\ga$. 

\begin{remark}\label{rem:model-no-K}
We will often work with models on the set $K = [-T, T] \times [-1,1]^3$. In this case we prefer to remove the set $K$ from the notation and write $\Vert \Pi^\ga \Vert^{(\emezo)}_{T}$, $\Vert \Gamma^\ga \Vert^{(\emezo)}_{T}$, etc. 
\end{remark}

\subsection{Modelled distributions}
\label{sec:DiscreteModelledDistributions}

By analogy with \cite[Sec.~6]{Regularity}, we are going to define a weighted norm for $\CT^\ex$-valued functions with a weight at time $0$. For this we define the following quantities for $z, \bar z \in \R^4$:
\begin{equation*}
\| z \|_0 := |t|^{\frac{1}{2}} \wedge 1, \qquad \| z, \bar z \|_0 := \| z \|_0 \wedge \| \bar z \|_0,
\end{equation*}
where $z = (t,x)$ with $t \in \R$. We also set $\| z, \bar z \|_\emezo := \| z, \bar z \|_0 \vee \emezo$.

For $\zeta, \eta \in \R$ and for a compact set $K \subset \R^4$, we define in the context of Definition~\ref{def:discretisation}(\ref{it:discretisation-md}) the following quantities (see \cite[Eqs.~3.21, 3.22]{erhard2017discretisation}):
\begin{equation}\label{eq:md-small-scale}
\$ f \$_{\zeta, \eta; K; \emezo} := \sup_{\substack{z \in K \cap D_\eps \\ \Vert z\Vert_{\s} \leq \emezo}} \sup_{m<\zeta} \frac{| f(z) |_m}{\emezo^{(\eta - m) \wedge 0}} + \sup_{\substack{z, \bar z \in K \cap D_\eps \\ \Vert z - \bar z \Vert_{\s} \leq \emezo}} \sup_{m<\zeta} \frac{| f(z) - \Gamma_{\!z \bar z} f(\bar z) |_{m}}{\emezo^{\zeta - m} \Vert z, \bar z\Vert_{\emezo}^{\eta - \zeta}},
\end{equation}
and 
\begin{align}\label{eq:md-distance-small-scale}
\$ f; \bar f \$_{\zeta, \eta; K; \emezo} &:= \sup_{\substack{z \in K \cap D_\eps \\ \Vert z\Vert_{\s} \leq \emezo}} \sup_{m<\zeta} \frac{| f(z) - \bar f(z) |_m}{\Vert z\Vert_{\s}^{(\eta - m) \wedge 0}} \\
&\qquad + \sup_{\substack{z, \bar z \in K \cap D_\eps \\ \Vert z - \bar z \Vert_{\s} \leq \emezo}} \sup_{m<\zeta} \frac{| f(z) - \Gamma_{\!z \bar z} f(\bar z) - \bar f(z) + \bar \Gamma_{\!z \bar z} \bar f(\bar z) |_{m}}{\emezo^{\zeta - m} \Vert z, \bar z\Vert_{\emezo}^{\eta - \zeta}}. \nonumber
\end{align}

Let us now take a discrete model $Z^{\ga} = (\Pi^\ga, \Gamma^\ga)$. A \emph{discrete modelled distribution} is an element of the space $\CD^{\zeta, \eta}_\emezo(\Gamma^\ga)$, containing the maps $f: D_\eps \to \CT^\ex_{< \zeta}$ such that, for any compact set $K \subseteq \R^4$,
\begin{align}\label{eq:def_dgamma_norm}
	\$ f \$^{(\emezo)}_{\zeta, \eta; K} &:= \sup_{\substack{z \in K \cap D_\eps \\ \Vert z\Vert_{\s} > \emezo}} \sup_{m<\zeta} \frac{| f(z) |_m}{\Vert z\Vert_{\s}^{(\eta - m) \wedge 0}} \\
	&\qquad + \sup_{\substack{z, \bar z \in K \cap D_\eps \\ \Vert z - \bar z \Vert_{\s} > \emezo}} \sup_{m<\zeta} \frac{| f(z) - \Gamma^\ga_{\!z \bar z} f(\bar z) |_{m}}{\| z - \bar z \|_\s^{\zeta - m} \Vert z, \bar z\Vert_{\emezo}^{\eta - \zeta}} + \$ f \$_{\zeta, \eta; K; \emezo} < \infty, \nonumber
\end{align}
where the last term is define by \eqref{eq:md-small-scale} via $\Gamma^\ga$. Sometimes it will be convenient to write $\CD^{\zeta, \eta}_\emezo(Z^{\ga})$ for $\CD^{\zeta, \eta}_\emezo(\Gamma^\ga)$, and when the model is clear from the context we will omit it from the notation and will simply write $\CD^{\zeta, \eta}_\emezo$. Observe that the first two terms in \eqref{eq:def_dgamma_norm} are the same as in the definition of the modelled distributions in \cite[Def.~6.2]{Regularity}, except that we look at the scale above $\emezo$. The last term measures regularity of $f$ on scale below~$\emezo$.

		For another discrete model $\bar Z^{\ga} = (\bar \Pi^\ga, \bar \Gamma^\ga)$ and a modelled distribution $\bar{f} \in \CD^\zeta_\emezo(\bar Z^{\ga})$, we set
\begin{equation*}
\begin{aligned}
	\$ f; \bar{f} \$^{(\emezo)}_{\zeta, \eta; K} &:= \sup_{\substack{z \in K \cap D_\eps \\ \Vert z\Vert_{\s} > \emezo}} \sup_{m<\zeta} \frac{| f(z) - \bar f(z) |_m}{\Vert z\Vert_{\s}^{(\eta - m) \wedge 0}} \\
	&\qquad + \sup_{\substack{z, \bar z \in K \cap D_\eps \\ \Vert z - \bar z \Vert_{\s} > \emezo}} \sup_{m<\zeta} \frac{| f(z) - \Gamma^\ga_{\!z \bar z} f(\bar z) - \bar f(z) + \bar \Gamma^\ga_{\!z \bar z} \bar f(\bar z) |_{m}}{\| z - \bar z \|_\s^{\zeta - m} \Vert z, \bar z\Vert_{\emezo}^{\eta - \zeta}} + \$ f; \bar f \$_{\zeta; K; \emezo},
\end{aligned}
\end{equation*}
where the last term is defined by \eqref{eq:md-distance-small-scale} via $\Gamma^\ga$ and $\bar \Gamma^\ga$.

\begin{remark}\label{rem:T-space}
When we work on the compact set $K = [-T, T] \times [-1,1]^3$, we simply write $\$ f \$^{(\emezo)}_{\zeta, \eta; T}$ and $\$ f; \bar{f} \$^{(\emezo)}_{\zeta, \eta; T}$. The space of modelled distributions, restricted to this set $K$ we denote by $\CD^{\zeta, \eta}_{\emezo, T}$.
\end{remark}

\subsection{The reconstruction theorem}

For a discrete model $(\Pi^\ga, \Gamma^\ga)$ and for a modelled distribution $f \in \CD^{\zeta, \eta}_\emezo$, we would like to define a \emph{reconstruction map} $\CR^\ga: \CD^{\zeta, \eta}_\emezo \to \CX_\eps$, which behaves around each point $z$ as $\Pi^\ga_{z} f(z)$. Following the idea of \cite[Def.~4.5]{HairerMatetski}, we define it as 
\begin{equation}\label{eq:def_rec_op}
	(\CR^\ga f)(z) := \bigl( \Pi^\ga_{z} f(z) \bigr) (z).
\end{equation}
In the case $\eta = \zeta$, i.e. when there is no weights in the definition \eqref{eq:def_dgamma_norm}, we have the following ``\emph{reconstruction theorem},'' where we use the short notation $\CD^{\zeta}_\emezo := \CD^{\zeta, \zeta}_\emezo$.

\begin{proposition}
	For a discrete model $(\Pi^\ga, \Gamma^\ga)$, a modelled distribution $f \in  \CD^{\zeta}_\emezo(\Gamma^\ga)$ with $\zeta > 0$, and compact set $K \subset \R^4$ one has
	\begin{equation}\label{eq:ReconThm1}
		\bigl| \iota_\eps \bigl( \CR^\ga f - \Pi^\ga_z f(z) \bigr) (\varphi^\lambda_z) \bigr| \lesssim (\lambda \vee \emezo)^\zeta \Vert \Pi^\ga \Vert^{(\emezo)}_{\bar{K}} \$ f \$^{(\emezo)}_{\zeta; [\varphi^\lambda_z]},
	\end{equation}
	uniformly over $\varphi \in \CB^2_\s$, $\lambda \in (0, 1]$, $z \in D_\eps$, and $\emezo \in (0, 1]$. Here, $\bar{K}$ is the $1$-fattening of $K$, $[\varphi^\lambda_z]$ is the support of $\varphi^\lambda_z$, and we used the map \eqref{eq:iota-general}.

	Let $(\bar{\Pi}^\ga, \bar{\Gamma}^\ga)$ be another discrete model with the respective reconstruction map $\bar{\CR}^\ga$, defined by \eqref{eq:def_rec_op}. Then for any $\bar f \in \CD^{\zeta}_{\emezo}(\bar \Gamma^\ga)$ one has
	\begin{align}\label{eq:ReconThm2}
		\bigl| \iota_\eps \big( \CR^\ga f - \Pi^\ga_z f(z) &- \bar{\CR}^\ga \bar{f} + \bar{\Pi}^\ga_z \bar{f}(z) \big) (\varphi^\lambda_z) \bigr| \\
		&\qquad \lesssim (\lambda \vee \emezo)^\zeta \Big( \Vert \bar{\Pi}^\ga \Vert^{(\emezo)}_{\bar K} \$ f; \bar{f} \$^{(\emezo)}_{\zeta; [\varphi^\lambda_z]} + \Vert \Pi^\ga - \bar{\Pi}^\ga \Vert^{(\emezo)}_{\bar K} \$ f \$^{(\emezo)}_{\zeta; [\varphi^\lambda_z]} \Big),\nonumber
	\end{align}
	uniformly over the same quantities as in \eqref{eq:ReconThm1}.
\end{proposition}

\begin{proof}
For any compact set $K_\emezo \subset \R^4$ of diameter smaller than $2 \emezo$ and for any $z \in D_\eps$, from the properties of the model and modelled distribution we get
\begin{equation}\label{eq:reconstruction-proof}
\begin{aligned}
	\bigl\Vert \CR^\ga f - \Pi^\ga_z f(z) \bigr\Vert_{\zeta; K_\emezo; z; \emezo} 
	&= \emezo^{-\zeta} \sup_{\bar z \in K_\emezo \cap D_\eps} \bigl| \Pi^\ga_{\bar z} \big( f(\bar z) - \Gamma^\ga_{\bar z z} f(z) \big)(\bar z) \bigr| \\
	&\lesssim \sup_{\be < \zeta}  \sup_{\bar z \in K_\emezo \cap D_\eps} \emezo^{\be - \zeta} \Vert \Pi^\ga \Vert^{(\emezo)}_{\bar K_\emezo} \bigl| f(\bar z) - \Gamma^\ga_{\bar z z}f(z) \big|_\be. 
\end{aligned}
\end{equation}
Using \eqref{eq:md-small-scale}, the latter yields 
\begin{equation}\label{eq:small_scale_condition} 
\bigl\Vert \CR^\ga f - \Pi^\ga_z f(z) \bigr\Vert_{\zeta; K_\emezo; z; \emezo} \lesssim \Vert \Pi^\ga \Vert^{(\emezo)}_{\bar K_\emezo} \$ f \$_{\zeta; K_\emezo; \emezo}.
\end{equation} 
Then \eqref{eq:ReconThm1} follows from \cite[Thm.~3.5]{erhard2017discretisation} and this bound.

The estimate \eqref{eq:ReconThm2} follows again from \cite[Thm.~3.5]{erhard2017discretisation} and from the following bound, which can be proved similarly to \eqref{eq:small_scale_condition},
\begin{align*}
\bigl\| \CR^\ga f - \Pi^\ga_z f(z) &- \bar{\CR}^\ga \bar{f} + \bar{\Pi}^\ga_z \bar{f}(z) \bigr\|_{\zeta; K_\emezo; z; \emezo} \\
&\qquad\lesssim \Vert \bar{\Pi}^\ga \Vert^{(\emezo)}_{\bar{K}_\emezo} \$ f; \bar{f} \$_{\zeta; K_\emezo; \emezo} + \Vert \Pi^\ga - \bar{\Pi}^\ga \Vert^{(\emezo)}_{\bar{K}_\emezo} \$ f \$_{\zeta; K_\emezo; \emezo},
\end{align*}
uniformly over the involved quantities. 
\end{proof}

Respectively, we can show that the reconstruction theorem \cite[Thm.~3.13]{erhard2017discretisation} holds in our case. The required Assumpsions~3.6 and 3.12 in \cite{erhard2017discretisation} follow readily from our definitions and estimates similar to \eqref{eq:reconstruction-proof}. We prefer not to duplicate full statement of this theorem, and we provide only the estimates which we are going to use later. 

\begin{proposition}\label{prop:ReconThm_v2}
In the described context, \cite[Thm.~3.13]{erhard2017discretisation} holds. In particular, let $(\Pi^\ga, \Gamma^\ga)$ be a discrete model and let $f \in \CD^{\zeta, \eta}_\emezo(\Gamma^\ga)$ be a modelled distribution, taking values in a sector of regularity $\alpha \leq 0$ and such that $\zeta > 0$, $\eta \leq \zeta$ and $\alpha \wedge \eta > -2$. Then for any compact set $K \subset \R^4$ one has 
\begin{equation*}
\bigl| \iota_\eps \bigl( \CR^\ga f \bigr) (\varphi^\lambda_z) \bigr| \lesssim (\lambda \vee \emezo)^{\alpha \wedge \eta} \Vert \Pi^\ga \Vert^{(\emezo)}_{\bar{K}} \$ f \$^{(\emezo)}_{\zeta; [\varphi^\lambda_z]},
\end{equation*}
uniformly over the same quantities as in \eqref{eq:ReconThm1}.

For a second discrete model $(\bar{\Pi}^\ga, \bar{\Gamma}^\ga)$ and for $\bar f \in \CD^{\zeta}_{\emezo}(\bar \Gamma^\ga)$ one has
\begin{equation*}
\bigl| \iota_\eps \bigl( \CR^\ga f - \bar \CR^\ga f\bigr) (\varphi^\lambda_z) \bigr| \lesssim (\lambda \vee \emezo)^{\alpha \wedge \eta} \Big( \Vert \bar{\Pi}^\ga \Vert^{(\emezo)}_{\bar K} \$ f; \bar{f} \$^{(\emezo)}_{\zeta; [\varphi^\lambda_z]} + \Vert \Pi^\ga - \bar{\Pi}^\ga \Vert^{(\emezo)}_{\bar K} \$ f \$^{(\emezo)}_{\zeta; [\varphi^\lambda_z]} \Big),
\end{equation*}
uniformly over the same quantities. 
\end{proposition}

\section{A renormalised lift of martingales}
\label{sec:lift}

Now we will construct a discrete model $Z^{\ga, \fa}_{\lift} = (\Pi^{\ga, \fa}, \Gamma^{\ga, \fa})$ which will be used to write equation \eqref{eq:IsingKacEqn-new} on the regularity structure $\ST^\ex$ (as in \cite{Regularity}, we call this model a ``lift'' of the random driving noise $\M_{\ga, \fa}$). For this, we are going to use the martingales from \eqref{eq:IsingKacEqn-new}, such that we have the a priori bounds on the solution provided by the stopping time \eqref{eq:tau}. 

Since we have only few basis elements in the regularity structure, we prefer to define $Z^{\ga, \fa}_{\lift}$ as a renormalised model, as opposed to \cite[Sec.~8.2]{Regularity}, where renormalisation of a canonical lift was done separately.

It will be convenient to use the following short notation:
\begin{equation*}
\int_{D_\eps} \varphi(z)\, \d z := \eps^3 \sum_{x \in \Lattice} \int_{\R} \varphi(t,x)\, \d t.
\end{equation*}
Throughout this section we will use the decomposition $\widetilde{G}^\ga = \mywidetilde{\SK}^\ga + \mywidetilde{\SR}^\ga$ of the discrete kernel \eqref{eq:From-P-to-G-tilde}, defined in Appendix~\ref{sec:decompositions}.

\subsection{Definition of the map $\PPi^{\ga, \fa}$}
\label{sec:PPi}

In order to define the model $(\Pi^{\ga, \fa}, \Gamma^{\ga, \fa})$, we first introduce an auxiliary map $\PPi^{\ga, \fa} \in \SL(\CT^\ex, \CX_\eps)$, and then we will use the results from \cite[Sec.~8.3]{Regularity}. 

It will be convenient to extend the martingales $\M_{\ga, \fa}(t, x)$ to all $t \in \R$. For this, we denote by $\widetilde{X}_{\ga, \fa}(t, x)$ an independent copy of $X_{\ga, \fa}(t, x)$, defined in Section~\ref{sec:a-priori}. Then $\widetilde{X}_{\ga, \fa}$ solves equation \eqref{eq:IsingKacEqn-periodic} driven by a martingale $\widetilde{\M}_{\ga, \fa}(t, x)$. We define the extension of $\M_{\ga, \fa}(t, x)$ to $t < 0$ as 
\begin{equation}\label{eq:martingale-extension}
\M_{\ga, \fa}(t, x) = \widetilde{\M}_{\ga, \fa}(-t, x).
\end{equation}
This extension does not affect equation \eqref{eq:IsingKacEqn-new} in any way, and is a technical trick which simplifies the following formulas. In particular, it allows to define time integrals in \eqref{eq:Pi-Xi} and later on whole $\R$ rather than $\R_+$. In what follows, the martingales $\M_{\ga, \fa}(t, x)$ are extended periodically to $x \in \Lattice$.

Using the map \eqref{eq:iota-general}, we start with making the following definition:
\begin{equation}\label{eq:Pi-Xi}
\iota_\eps \bigl(\PPi^{\ga, \fa} \blueXi\bigr)(\varphi) = \frac{1}{\sqrt 2}\eps^3 \sum_{x \in \Lattice} \int_{\R} \varphi(t,x)\, \d \M_{\ga, \fa}(t, x), 
\end{equation}
for any smooth, compactly supported function $\varphi : \R^4 \to \R$, and for the processes $\M_{\ga, \fa}$ just introduced. The stochastic integral is defined with respect to the martingale $t \mapsto \M_{\ga, \fa}(t, x)$, which is well defined in the Stieltjes sense since the function $\varphi$ is smooth. We need to use the factor $\frac{1}{\sqrt 2}$ in order to have convergence of $\PPi^{\ga, \fa} \blueXi$ to a white noise (as follows from \eqref{eq:M-a-variation}, the martingale $\M_{\ga, \fa}$ converges to a cylindrical Wiener process with diffusion $2$). For monomials we set 
\begin{equation*}
\bigl(\PPi^{\ga, \fa} \blueOne\bigr) (t,x) = 1, \qquad \bigl(\PPi^{\ga, \fa} \X_i\bigr) (t,x) = x_i.
\end{equation*}
Furthermore, we use the kernel $\mywidetilde{\SK}^\ga$, defined in the beginning of this section, and set
\begin{equation*}
	\bigl(\PPi^{\ga, \fa} \<1b>\bigr)(t, x) = \frac{1}{\sqrt 2} \eps^3 \sum_{y \in \Lattice} \int_{\R} \mywidetilde{\SK}^\ga_{t-s}(x - y)\, \d \M_{\ga, \fa}(s, y), 
\end{equation*}
as well as
\begin{equation}\label{eq:cherry}
\begin{aligned}
	\bigl(\PPi^{\ga, \fa} \<2b>\bigr)(t, x) &= \bigl( \PPi^{\ga, \fa} \<1b>\bigr)(t, x)^2 - \fc_\ga - \fc_\ga',\\
	\bigl(\PPi^{\ga, \fa} \<3b>\bigr)(t, x) &= \bigl( \PPi^{\ga, \fa} \<1b> \bigr)(t, x)^3 - 3 \fc_\ga \bigl(\PPi^{\ga, \fa} \<1b>\bigr)(t, x),
\end{aligned}
\end{equation}
where
\begin{equation}\label{eq:renorm-constant1}
	\fc_\ga := \int_{D_\eps} \mywidetilde{\SK}^\ga(z)^2\, \d z
\end{equation}
is a diverging renormalisation constant (we show in Lemma~\ref{lem:renorm-constants} that the divergence speed is $\emezo^{-1}$),
and 
\begin{equation}\label{eq:renorm-constant3}
\fc_{\ga}' := - \be \varkappa_{\ga, 3} \ga^6 \un{\fC}_\ga \fc_\ga
\end{equation}
is a renormalisation constant which is bounded uniformly in $\ga$, as follows from Lemmas~\ref{lem:renorm-constants} and \ref{lem:renorm-constant-underline}. We used in \eqref{eq:renorm-constant3} the constants $\be$, $\varkappa_{\ga, 3}$ and $\un{\fC}_\ga$ defined in \eqref{eq:beta}, \eqref{eq:c-gamma-2} and \eqref{eq:c-under} respectively. 

We prefer to separate the two renormalisation constants in \eqref{eq:cherry}, because they have different origins. More precisely, the constant $\fc_\ga$ would be the same if the driving noise was Gaussian, while $\fc_{\ga}'$ comes from the renormalisation of the non-linearity of the bracket process \eqref{eq:M-bracket}. The necessity of such renormalisation will be clear from Section~\ref{sec:second-symbol}.

Let $H_n : \R \times \R_+ \to \R$ be the $n$-th Hermite polynomial, defined for $n \in \N$ and a real constant $c > 0$ in the following recursive way:
\begin{equation}\label{eq:def_Hermite}
	H_1(u, c) = u, \qquad H_{n+1}(u, c) = u H_n(u, c) - c H'_{n}(u, c) ~\text{ for any }~ n \geq 1,
\end{equation}
with $H'_{n}$ denoting the derivative of the polynomial $H_{n}$ with respect to the variable $u$. In particular, the first several Hermite polynomials are $H_1(u, c) = u$, $H_2(u, c) = u^2 - c$, $H_3(u, c) = u^3 - 3 c u$, $H_4(u, c) = u^4 - 6 c u^2 + 3 c^2$ and $H_5(u, c) = u^5 - 10 c u^3 + 15 c^2 u$.

Observe then that we have the identities $\bigl(\PPi^{\ga, \fa} \<2b>\bigr)(t, x) = H_2 \bigl((\PPi^{\ga, \fa} \<1b>)(t, x), \fc_\ga + \fc_\ga'\bigr)$ and $\bigl(\PPi^{\ga, \fa} \<3b>\bigr)(t, x) = H_3\bigl((\PPi^{\ga, \fa} \<1b>)(t, x), \fc_\ga\bigr)$, which correspond to the Wick renormalisation of models in the case of a Gaussian noise \cite[Sec.~10]{Regularity}. Hence, in the same spirit we can define $\PPi^{\ga, \fa} \<4b>$ and $\PPi^{\ga, \fa} \<5b>$ in terms of the Hermite polynomials:
\begin{equation}\label{eq:model-Hermite}
\begin{aligned}
	\bigl(\PPi^{\ga, \fa} \<4b>\bigr)(t, x) &= H_4\bigl(( \PPi^{\ga, \fa} \<1b>)(t, x), \fc_\ga \bigr) \\
	&= \bigl( \PPi^{\ga, \fa} \<1b>\bigr)(t, x)^4 - 6 \fc_\ga \bigl( \PPi^{\ga, \fa} \<1b>\bigr)(t, x)^2 + 3 \fc_\ga^2, \\
	\bigl(\PPi^{\ga, \fa} \<5b>\bigr)(t, x) &= H_5\bigl(( \PPi^{\ga, \fa} \<1b>)(t, x), \fc_\ga \bigr) \\
	&= \bigl( \PPi^{\ga, \fa} \<1b>\bigr)(t, x)^5 - 10 \fc_\ga \bigl( \PPi^{\ga, \fa} \<1b>\bigr)(t, x)^3 + 15 \fc_\ga^2 \bigl( \PPi^{\ga, \fa} \<1b>\bigr)(t, x).
\end{aligned}
\end{equation}
For the elements of the form $\tau \X_i \in \CW^\ex$ we set 
\begin{equation*}
	\bigl(\PPi^{\ga, \fa} \tau \X_i\bigr) (t,x) = \bigl(\PPi^{\ga, \fa} \tau\bigr) (t,x) \bigl(\PPi^{\ga, \fa} \X_i\bigr) (t,x).
\end{equation*}
For each element $\CE(\tau) \in \CW^\ex$ we define 
 \begin{equation*}
 \bigl(\PPi^{\ga, \fa} \CE(\tau)\bigr) (t,x) = \ga^6 \bigl(\PPi^{\ga, \fa} \tau\bigr) (t,x),
 \end{equation*}
 and for each element $\CI(\tau) \in \CW^\ex$ we set 
 \begin{equation*}
 \bigl(\PPi^{\ga, \fa} \CI(\tau)\bigr) (t,x) = \eps^3 \sum_{y \in \Lattice} \int_{\R}  \mywidetilde{\SK}^\ga_{t-s}(x - y) \bigl(\PPi^{\ga, \fa} \tau\bigr) (s,y)\, \d s.
 \end{equation*}
 
 One can see that this recursive definition of the map $\PPi^{\ga, \fa}$ gives its action on all the elements from Tables~\ref{tab:symbols-cont} and \ref{tab:symbols}, except the three diagrams $\<22b>$, $\<31b>$ and $\<32b>$. So, it is left to define the map $\PPi^{\ga, \fa}$ for these three elements. For the element $\<22b>$ we set 
 \begin{equation*}
  \bigl(\PPi^{\ga, \fa} \<22b>\bigr) (t,x) = \bigl(\PPi^{\ga, \fa} \<20b>\bigr) (t,x) \bigl(\PPi^{\ga, \fa} \<2b>\bigr) (t,x) - \fc_\ga'',
 \end{equation*}
 where
 \begin{equation}\label{eq:renorm-constant2}
 \fc_\ga'' := 2 \int_{D_\eps} \int_{D_\eps} \int_{D_\eps} \mywidetilde{\SK}^\ga(z)\, \mywidetilde{\SK}^\ga(z_1)\, \mywidetilde{\SK}^\ga(z_2)\, \mywidetilde{\SK}^\ga(z_1 - z)\, \mywidetilde{\SK}^\ga(z_2 - z)\, \d z \d z_1 \d z_2
 \end{equation}
 is a new diverging renormalisation constant (we show in Section~\ref{sec:renormalisation} that the divergence order is $\log \emezo$). Finally, we define
  \begin{align*}
  \bigl(\PPi^{\ga, \fa} \<31b>\bigr) (t,x) &= \bigl(\PPi^{\ga, \fa} \<30b>\bigr) (t,x) \bigl(\PPi^{\ga, \fa} \<1b>\bigr) (t,x), \\
  \bigl(\PPi^{\ga, \fa} \<32b>\bigr) (t,x) &= \bigl(\PPi^{\ga, \fa} \<30b>\bigr) (t,x) \bigl(\PPi^{\ga, \fa} \<2b>\bigr) (t,x) - 3 \fc_\ga'' \bigl(\PPi^{\ga, \fa} \<1b>\bigr) (t,x). 
 \end{align*}
 
 \subsection{Definition of the model}
 \label{sec:model-lift}
 
 Having $\PPi^{\ga, \fa}$ defined on the basis elements $\CW^\ex$, we extend it linearly to $\CT^\ex$, which yields the map $\PPi^{\ga, \fa} \in \SL(\CT^\ex, \CX_\eps)$. As we pointed above, we had to exclude the symbol $\blueXi$ from $\CT^\ex$, because our definition \eqref{eq:Pi-Xi} suggests that $\PPi^{\ga, \fa} \blueXi$ does not belong to $\CX_\eps$. A discrete model $Z^{\ga, \fa}_{\lift} = (\Pi^{\ga, \fa}, \Gamma^{\ga, \fa})$ on $\ST^\ex$ is defined via this map $\PPi^{\ga, \fa}$ as in \cite[Sec.~8.3]{Regularity}. More precisely, we define 
 \begin{equation*}
f^{\ga, \fa}_z(\blueOne) = - 1, \qquad f^{\ga, \fa}_z(\X_i) = - x_i, \qquad f^{\ga, \fa}_z(\CI(\tau)) = - \bigl(\PPi^{\ga, \fa} \CI(\tau)\bigr) (z) \quad\text{for}~ \tau \neq \blueXi.
 \end{equation*}
We extend this function linearly to $f^{\ga, \fa}_z : \CT^\ex_+ \to \R$, where $\CT^\ex_+$ is defined in Section~\ref{sec:structure-group}, and we can use \eqref{eq:Gamma-general} to define
 \begin{equation*}
F^{\ga, \fa}_{\!z} := \Gamma_{\!f^{\ga, \fa}_z}.
 \end{equation*}
 Since $F^{\ga, \fa}_{\!z}$ is an element of the group $\CG^\ex$, it has the inverse $(F^{\ga, \fa}_{\!z})^{-1}$. Then the discrete model $(\Pi^{\ga, \fa}, \Gamma^{\ga, \fa})$ is defined as 
 \begin{equation}\label{eq:Pi-Gamma-lift}
 \Pi^{\ga, \fa}_z \tau = \bigl(\PPi^{\ga, \fa} \otimes f^{\ga, \fa}_z\bigr) \Delta \tau, \qquad\qquad \Gamma^{\ga, \fa}_{\! z \bar z} =  (F^{\ga, \fa}_{\!z})^{-1} \circ F^{\ga, \fa}_{\!\bar z},
 \end{equation}
 where the operator $\Delta$ is defined in Section~\ref{sec:structure-group}. All the properties in Definition~\ref{def:model} follow from the definition of the model $Z^{\ga, \fa}_{\lift}$. However, showing that the bounds \eqref{eqs:model-bounds} hold uniformly in $\ga > 0$ is non-trivial and we prove these bounds in Section~\ref{sec:convergence-of-models}.
 
 Since the operator $\Delta$ is simple in our case, we can write the map $\Pi^{\ga, \fa}$ explicitly. Namely, we have $\bigl(\Pi^{\ga, \fa}_z \blueOne\bigr) (\bar z) = 1$ and $\bigl(\Pi^{\ga, \fa}_z \X_i\bigr) (\bar z) = \bar x_i - x_i$, for $ z = ( t,  x)$ and $\bar z = (\bar t, \bar x)$. Using the same space-time points, we furthermore have 
\begin{equation}\label{eq:lift-hermite}
\begin{aligned}
	\bigl(\Pi^{\ga, \fa}_z \<1b>\bigr)(\bar z) &= \frac{1}{\sqrt 2} \eps^3 \sum_{y \in \Lattice} \int_{\R} \mywidetilde{\SK}^\ga_{\bar t-s}(\bar x - y)\, \d\M_{\ga, \fa}(s, y), \\
	\bigl(\Pi^{\ga, \fa}_z \<2b>\bigr)(\bar z) &= H_2\bigl((\Pi^{\ga, \fa}_z \<1b>)(\bar z), \fc_\ga + \fc_\ga' \bigr), \\
	 \bigl(\Pi^{\ga, \fa}_z \<1b>^n\bigr)(\bar z) &= H_n\bigl((\Pi^{\ga, \fa}_z \<1b>)(\bar z), \fc_\ga \bigr) \quad \text{for $n = 3, 4, 5$}.
\end{aligned}
\end{equation}
For $\tau \in \{\<2b>, \<3b>\}$ we have $\bigl(\Pi^{\ga, \fa}_z \tau \X_i\bigr) (\bar z) = \bigl(\Pi^{\ga, \fa}_z \tau\bigr) (\bar z) \bigl(\Pi^{\ga, \fa}_z \X_i\bigr) (\bar z)$, and for $\tau \in \{\<4b>, \<5b>\}$ we have 
\begin{equation}\label{eq:Pi-E}
\bigl(\Pi^{\ga, \fa}_z \CE(\tau)\bigr) (\bar z) = \ga^6 \bigl(\Pi^{\ga, \fa}_z \tau\bigr) (\bar z).
\end{equation}
For the elements $\tau \in \{\<2b>, \<3b>\}$ the following formulas hold:
\begin{equation}\label{eq:Pi-I}
 	\bigl(\Pi^{\ga, \fa}_z \CI(\tau)\bigr) (\bar z) = \int_{D_\eps} \bigl(\mywidetilde{\SK}^\ga(\bar z - \tilde{z}) - \mywidetilde{\SK}^\ga(z - \tilde{z})\bigr) \bigl(\Pi^{\ga, \fa}_z \tau\bigr) (\tilde{z})\, \d \tilde{z}.
\end{equation}
Finally, we have the identities
\begin{align}
\bigl(\Pi^{\ga, \fa}_z \<22b>\bigr) (\bar z) &= \bigl(\Pi^{\ga, \fa}_z \<2b>\bigr) (\bar z)\, \int_{D_\eps} \bigl(\mywidetilde{\SK}^\ga(\bar z - \tilde{z}) - \mywidetilde{\SK}^\ga(z - \tilde{z})\bigr) \bigl(\Pi^{\ga, \fa}_z \<2b>\bigr) (\tilde{z})\, \d \tilde{z} - \fc_\ga'', \nonumber\\
\bigl(\Pi^{\ga, \fa}_z \<31b>\bigr) (\bar z) &= \bigl(\Pi^{\ga, \fa}_z \<1b>\bigr) (\bar z)\, \int_{D_\eps} \bigl(\mywidetilde{\SK}^\ga(\bar z - \tilde{z}) - \mywidetilde{\SK}^\ga(z - \tilde{z})\bigr) \bigl(\Pi^{\ga, \fa}_z \<3b>\bigr) (\tilde{z})\, \d \tilde{z}, \label{eq:Pi-rest} \\
\bigl(\Pi^{\ga, \fa}_z \<32b>\bigr) (\bar z) &= \bigl(\Pi^{\ga, \fa}_z \<2b>\bigr) (\bar z)\, \int_{D_\eps} \bigl(\mywidetilde{\SK}^\ga(\bar z - \tilde{z}) - \mywidetilde{\SK}^\ga(z - \tilde{z})\bigr) \bigl(\Pi^{\ga, \fa}_z \<3b>\bigr) (\tilde{z})\, \d \tilde{z} - 3 \fc_\ga'' \bigl(\Pi^{\ga, \fa}_z \<1b>\bigr) (\bar z).\nonumber
\end{align}

Once the map $\Pi^{\ga, \fa}$ is defined, we can also write the map $\Gamma^{\ga, \fa}$ explicitly. The latter can be easily obtained from the identity $\Pi^{\ga, \fa}_{\bar z} = \Pi^{\ga, \fa}_{z} \Gamma^{\ga, \fa}_{\!z \bar z}$, which is a part of Definition~\ref{def:model}. Namely, for fixed $z, \bar z \in D_\eps$, we have that $\Gamma^{\ga, \fa}_{\!z \bar z}$ is a linear map on $\CT^\ex$, whose action on the elements of $\CW^\ex$ is given in Table~\ref{tab:linear_transformations} with the constants
\begin{equation}\label{eq:Gamma-lift}
	a_i = - \bigl(\Pi^{\ga, \fa}_{z} \X_i\bigr) (\bar z), \qquad b = - \bigl(\Pi^{\ga, \fa}_{z} \<20b>\bigr)(\bar z), \qquad c = - \bigl(\Pi^{\ga, \fa}_{z} \<30b>\bigr)(\bar z).
\end{equation}

\begin{remark}\label{rem:positive-vanish}
From the definition of the discrete model $Z^{\ga, \fa}_{\lift}$ and the definition of he respective reconstruction map $\CR^{\ga, \fa}$ in \eqref{eq:def_rec_op}, we can see that $\CR^{\ga, \fa} \tau \equiv 0$ if $|\tau| > 0$.
\end{remark}

\begin{remark}\label{rem:Eps-reconstruct}
For an element $\CE(\tau)$ we obviously have $\CR^{\ga, \fa} \CE(\tau) = \ga^6 \CR^{\ga, \fa} \tau$.
\end{remark}

\begin{remark}\label{rem:extension-to-Xi}
We note that in \eqref{eq:Pi-Xi} we defined the action of the map $\PPi^{\ga, \fa}$ also on the symbol $\blueXi$. This allows to extend the maps \eqref{eq:Pi-Gamma-lift} on this symbol as 
\begin{equation*}
\Gamma^{\ga, \fa}_{\! z \bar z} \blueXi = \blueXi, \qquad \qquad \iota_\eps \bigl(\Pi^{\ga, \fa}_z \blueXi\bigr)(\varphi) = \frac{1}{\sqrt 2} \eps^3 \sum_{x \in \Lattice} \int_{\R} \varphi(t,x)\, \d \M_{\ga, \fa}(t, x), 
\end{equation*}
for any smooth, compactly supported function $\varphi : \R^4 \to \R$, and for the martingales $\M_{\ga, \fa}$ as in \eqref{eq:Pi-Xi}. We see however that $\Pi^{\ga, \fa}_z \blueXi$ is not a function, which explains why we excluded the symbol $\blueXi$ from the domain of discrete models in Definition~\ref{def:model}. We can also extend the reconstruction map as 
\begin{equation*}
\iota_\eps \bigl(\CR^{\ga, \fa} \blueXi\bigr)(\varphi) = \frac{1}{\sqrt 2} \eps^3 \sum_{x \in \Lattice} \int_{\R} \varphi(t,x)\, \d \M_{\ga, \fa}(t, x).
\end{equation*}
\end{remark}

\subsection{Asymptotics of the renormalisation constants}

We can show precise speeds of divergence of the renormalisation constants.

\begin{lemma}\label{lem:renorm-constants}
Let $\fc_\ga$ and $\fc_\ga''$ be defined in \eqref{eq:renorm-constant1} and \eqref{eq:renorm-constant2} respectively. The constants $\fc_\ga^{(2)}$ and $\fc_\ga^{(1)}$ are well-defined by \eqref{eq:renorm-constants-main} for all $\ga > 0$ small enough. Moreover, $\fc_\ga^{(2)} \sim \emezo^{-1}$ and $\fc_\ga^{(1)} \sim \log \emezo$, and the expressions $\fc_\ga - \frac{1}{2}\fc_\ga^{(2)}$ and $\fc_\ga'' - \frac{1}{4} \fc_\ga^{(1)}$ converge as $\ga \to 0$. This in particular implies the asymptotics of the renormalisation constant $\fC_\ga$ stated in Theorem~\ref{thm:main}.
\end{lemma}

\begin{proof}
The kernel $\mywidetilde{\SK}^\ga$, involved in the definitions \eqref{eq:renorm-constant1} and \eqref{eq:renorm-constant2}, is supported in a ball of radius $c \geq 1$ (see Appendix~\ref{sec:decompositions}). Let $D_{c, \eps} := [0, c] \times \Le $. Then, without any harm, we can replace the integration domains $D_\eps$ by $D_{c, \eps}$ in these definitions. 

We define new constants 
\begin{subequations}\label{eqs:renorm-constants-new}
\begin{align}
\tilde{\fc}_\ga &= \int_{D_{c, \eps}} \widetilde{P}^\ga(z)^2 \d z, \label{eq:renorm-constant1-new}\\
\tilde{\fc}_\ga'' &= 2 \int_{D_{c, \eps}} \int_{D_{c, \eps}} \int_{D_{c, \eps}} \widetilde{P}^\ga(z) \widetilde{P}^\ga(z_1) \widetilde{P}^\ga(z_2) \widetilde{P}^\ga(z_1 - z) \widetilde{P}^\ga(z_2 - z) \d z \d z_1 \d z_2. \label{eq:renorm-constant2-new}
\end{align}
\end{subequations}
We note that if we replace at least one instance of the singular kernel $\mywidetilde{\SK}^\ga$ by a smooth function in the definitions of the renormalisation constants \eqref{eq:renorm-constant1} and \eqref{eq:renorm-constant2}, then we obtain convergent constants (as $\gamma \to 0$). This follows from the properties of $\mywidetilde{\SK}^\ga$ from Appendix~\ref{sec:decompositions}. Since $\mywidetilde{\SK}^\ga$ is a singular part of the discrete heat kernel, this implies that the limits $\lim_{\ga \to 0} (\tilde\fc_\ga - \fc_\ga)$ and $\lim_{\ga \to 0} (\tilde\fc_\ga'' - \fc_\ga'')$ exist and are finite. Hence, to prove this lemma, we need to show that the required asymptotic behaviours hold if we replace $\fc_\ga$ and $\fc_\ga''$ by $\tilde{\fc}_\ga$ and $\tilde{\fc}_\ga''$ respectively. 

It will be convenient to write the constants \eqref{eqs:renorm-constants-new} in a different form. Applying Parseval's identity \eqref{eq:Parseval} in the spatial variable in \eqref{eq:renorm-constant1-new}, we get 
\begin{equation*}
\tilde{\fc}_\ga = \int_0^c \eps^3 \sum_{x \in \Le} \widetilde{P}^\ga_t(x)^2 \d t = \frac{1}{8} \int_0^c \sum_{|\om|_{\sinfty} \leq N} \exp \Bigl( 2 \varkappa_{\ga, 3}^2 \big( \widehat{K}_\ga(\om) - 1 \big) \frac{t}{\alpha} \Bigr) |\widehat{K}_\ga(\om)|^2 \d t,
\end{equation*}
where we used the Fourier transform \eqref{eq:tildeP}. From the properties of the function \eqref{eq:K-gamma} we have $\eps^3 \sum_{x \in \Le} K_\ga(x) = 1$, which yields $\widehat{K}_\ga(0) = 1$. Furthermore, from Lemma~\ref{lem:Kg} we can conclude that there exists $\ga_0 > 0$ such that $\widehat{K}_\ga(\om) \neq 1$ for $\ga < \ga_0$ and all $\om \in \Z^3$ satisfying $0 < |\om|_{\sinfty} \leq N$. Then we have 
\begin{align*}
\tilde{\fc}_\ga &= \frac{c}{8} + \frac{1}{8} \int_0^c \sum_{0 < |\om|_{\sinfty} \leq N} \exp \Bigl( 2 \varkappa_{\ga, 3}^2 \big( \widehat{K}_\ga(\om) - 1 \big) \frac{t}{\alpha} \Bigr) |\widehat{K}_\ga(\om)|^2 \d t \\
&= \frac{c}{8} + \frac{1}{16} \sum_{0 < |\om|_{\sinfty} \leq N} \frac{\alpha |\widehat{K}_\ga(\om)|^2}{\varkappa_{\ga, 3}^2 \big(1 -  \widehat{K}_\ga(\om) \big)} \biggl(1 - \exp \Bigl( 2 \varkappa_{\ga, 3}^2 \big( \widehat{K}_\ga(\om) - 1 \big) \frac{c}{\alpha} \Bigr)\biggr).
\end{align*}
Let us write $\tilde{\fc}_\ga = \frac{c}{8} + \tilde{\fc}^{(2)}_\ga - \tilde{\fc}^{(1)}_\ga$, where 
\begin{align*}
\tilde{\fc}^{(2)}_\ga &= \frac{1}{16} \sum_{0 < |\om|_{\sinfty} \leq N} \frac{\alpha |\widehat{K}_\ga(\om)|^2}{\varkappa_{\ga, 3}^2 \big(1 -  \widehat{K}_\ga(\om) \big)}, \\
 \tilde{\fc}^{(1)}_\ga &= \frac{1}{16} \sum_{0 < |\om|_{\sinfty} \leq N} \frac{\alpha |\widehat{K}_\ga(\om)|^2}{\varkappa_{\ga, 3}^2 \big(1 -  \widehat{K}_\ga(\om) \big)}\exp \Bigl( 2 \varkappa_{\ga, 3}^2 \big( \widehat{K}_\ga(\om) - 1 \big) \frac{c}{\alpha} \Bigr).
\end{align*}
From Lemma~\ref{lem:Kg} for $0 < \ga < \ga_0$ we have the bounds $1 -\widehat{K}_\ga(\om) \geq C_1 \big( |\ga^3 \om|^2  \wedge 1 \big)$, $|\widehat{K}_\ga(\om)| \leq 1$ for $|\om| \leq \ga^{-3}$, and $|\widehat{K}_\ga(\om)| \leq C_2 |\ga^3 \om|^{-k}$ for $|\om| \geq \ga^{-3}$ and for any $k \in \N$. Using these bounds and \eqref{eq:c-gamma-2}, we conclude that $\tilde{\fc}^{(2)}_\ga$ diverges as $\ga \to 0$ with the rate $\emezo^{-1}$. Moreover, the constant $\varkappa_{\ga, 3}$ in the definition of $\tilde{\fc}^{(2)}_\ga$ can be replaced by $1$, which produces a convergent error, i.e. $\tilde{\fc}^{(2)}_\ga - \frac{1}{2} \fc^{(2)}_\ga$ has a finite limit as $\ga \to 0$, where the constant $\fc_\ga^{(2)}$ is defined in \eqref{eq:renorm-constants-main}. Similarly, we can conclude that $\tilde{\fc}^{(1)}_\ga$ is bounded uniformly in $0 < \ga < \ga_0$, and moreover it converges as $\ga \to 0$. Thus, we have that $\fc_\ga - \frac{1}{2}\fc_\ga^{(2)}$ converges as $\ga \to 0$, which finishes the proof of the asymptotic behaviours from the statement of this lemma which involve $\fc_\ga$ and $\fc_\ga^{(2)}$.

The constant $\fc_\ga''$ is analysed in a similar way. Namely, we can show that 
\begin{equation*}
\fc_\ga'' - \frac{1}{32} \sum_{0 < |\om_1|_\sinfty, |\om_2|_\sinfty \leq N} \frac{\alpha^3 |\widehat{K}_\ga(\om_1)|^2 |\widehat{K}_\ga(\om_2)|^2}{(1 - \widehat{K}_\ga(\om_1)) (1 - \widehat{K}_\ga(\om_2))} \frac{ \widehat{K}_\ga(\om_1 + \om_2)}{1 - \widehat{K}_\ga(\om_1) - \widehat{K}_\ga(\om_2) + \widehat{K}_\ga(\om_1 + \om_2)}
\end{equation*}
converges as $\ga \to 0$. This expression equals $\fc_\ga'' - \frac{1}{4} \fc_\ga^{(1)}$ and the double sum diverges with the rate $\log \emezo$.
\end{proof}

Similarly, we can study the asymptotic behaviour of the renormalisation constant \eqref{eq:c-under}

\begin{lemma}\label{lem:renorm-constant-underline}
The constant \eqref{eq:c-under} satisfies $|\un{\fC}_\ga| \lesssim \emezo^{-1}$.
\end{lemma}

\begin{proof}
Applying identities \eqref{eq:Parseval} and \eqref{eq:Fourier-of-product} to \eqref{eq:c-under} we get
\begin{equation*}
\un{\fC}_\ga = \frac{\varkappa_{\ga, 2}}{4} \int_0^\infty \sum_{|\om|_{\sinfty} \leq N} \exp \Bigl( 2 \varkappa_{\ga, 3}^2 \big( \widehat{K}_\ga(\om) - 1 \big) \frac{t}{\alpha} \Bigr) \widehat{\un{K}}_\ga(\om) \widehat{K}_\ga(\om)\, \d t,
\end{equation*}
where we used the Fourier transform \eqref{eq:tildeP}. As in the proof of Lemma~\ref{lem:renorm-constants}, for all $\ga > 0$ small enough we can compute the integral, which yields  
\begin{equation}\label{eq:un-fc-bound}
\un{\fC}_\ga = \frac{\varkappa_{\ga, 2}}{8} \sum_{0 < |\om|_{\sinfty} \leq N} \frac{\alpha \widehat{\un{K}}_\ga(\om) \widehat{K}_\ga(\om)}{\varkappa_{\ga, 3}^2 \big(1 -  \widehat{K}_\ga(\om) \big)}.
\end{equation}
From Lemma~\ref{lem:Kg} for all $\ga > 0$ small enough we have $1 -\widehat{K}_\ga(\om) \geq C_1 \big( |\ga^3 \om|^2  \wedge 1 \big)$, $|\widehat{K}_\ga(\om)| \leq 1$ for $|\om| \leq \ga^{-3}$, and $|\widehat{K}_\ga(\om)| \leq C_2 |\ga^3 \om|^{-k}$ for $|\om| \geq \ga^{-3}$ and for any $k \in \N$. Since the kernel $\un{K}_\ga$ has the same properties as $K_\ga$, except that it is rescaled by $\ga^{3 + \un\kappa}$ rather than $\ga^3$, we have the respective bounds $|\widehat{\un{K}}_\ga(\om)| \leq C_3$ for $|\om| \leq \ga^{-3 - \un \kappa}$, and $|\widehat{\un{K}}_\ga(\om)| \leq C_4 |\ga^{3 + \un\kappa} \om|^{-k}$ for $|\om| \geq \ga^{-3 - \un \kappa}$ and for any $k \in \N$. Then the part of the sum in \eqref{eq:un-fc-bound} running over $0 < |\om|_{\sinfty} \leq \ga^{-3}$ is bounded by a constant multiple of 
\begin{equation*}
\int_{0 < |\om|_{\sinfty} \leq \ga^{-3}} \alpha |\ga^3 \om|^{-2} \d \om \lesssim \ga^{-3},
\end{equation*}
where we made use of \eqref{eq:c-gamma-2}. The part of the sum running over $\ga^{-3} < |\om|_{\sinfty} \leq N$ is bounded by a constant times 
\begin{equation*}
\int_{\ga^{-3} < |\om|_{\sinfty} \leq N} \alpha |\ga^3 \om|^{-k} \d \om \lesssim \ga^{3}.
\end{equation*}
Hence, we have the required bound $|\un{\fC}_\ga| \lesssim \emezo^{-1}$.
\end{proof}

\section{Properties of the martingales and auxiliary results}
\label{sec:martingales}

In this section we collect several result which will be used to prove moment bounds for the discrete models constructed above. We first show that the martingales $\bigl(\M_{\ga, \fa}(\bigcdot, x)\bigr)_{x \in \Le}$ satisfy Assumption~1 in \cite{Martingales}.

\subsection{Properties of the martingales}
\label{sec:martingales-properties}

The required properties of the predictable quadratic covariations, stated in Assumption~1(1) in \cite{Martingales}, follow from \eqref{eq:M-a-variation} and \eqref{eq:M-bracket}: $\big\langle \M_{\ga, \fa}(\bigcdot, x), \M_{\ga, \fa}(\bigcdot, x') \big\rangle_t = 0$ for $x \neq x'$, and in the case $x = x'$ we have 
\begin{equation}\label{eq:M-gamma-bracket}
\big\langle \M_{\ga, \fa}(\bigcdot, x) \big\rangle_t = \eps^{-3} \int_0^t \bC_{\ga, \fa} (s,x)\, \d s,
\end{equation}
with an adapted process $t \mapsto \bC_{\ga, \fa} (t,x)$, given by
\begin{equation}\label{eq:bC}
\bC_{\ga, \fa} (t,x) := 
\begin{cases}
	2 \varkappa_{\ga, 2} \Bigl( 1 - \sigma \bigl(\frac{t}{\alpha}, \frac{x}{\eps}\bigr) \tanh \bigl( \be \delta X_\ga(t, x) \bigr)\Bigr) &\text{for}~~ t < \tau_{\ga, \fa}, \\
	2 \varkappa_{\ga, 2} \Bigl( 1 - \delta \sigma'_{\ga, \fa} \bigl(\frac{t}{\alpha}, \frac{x}{\eps}\bigr) X'_{\ga, \fa}(t, x) \Bigr) &\text{for}~~ t \geq \tau_{\ga, \fa},
\end{cases}
\end{equation}
where $X'_{\ga, \fa}$ is defined as in \eqref{eq:X-gamma} via the process $\sigma'_{\ga, \fa}$ and where the constant $\varkappa_{\ga, 2}$, which is closed to $1$, was introduced in \eqref{constant:isingkac_one}. We observe that the inequality $|\bC_{\ga, \fa} (s,x)| \leq 2$ holds uniformly over $\ga \in (0,1)$ and $x \in \Le$.

Assumption~1(2) in \cite{Martingales} follows readily from the definition of the martingales, because every time a spin of the Ising-Kac model flips, only one of the martingales $\M_{\ga, \fa}(\bigcdot, x)$ changes its value, while the others stay unchanged. 

We see that the martingale $\M_{\ga, \fa}(\bigcdot, x)$, for any fixed $x \in \Le$, has jumps of size $2 \ga^{-3}$, because the martingale $\fm_{\ga}(\bigcdot, x)$ from \eqref{eq:equation-for-sigma} has jumps of size $2$. Therefore, given that $\eps \approx \ga^4$, Assumption~1(3) in \cite{Martingales} holds with any value of the constant $\kone$ bigger than $\frac{3}{4}$.

For a \cadlag process $f$, we denote by $f(t^-)$ its left limit at time $t$ and we define the jump size at time $t$ as 
\begin{equation}\label{eq:jump}
\Delta_t f := f(t) - f(t^-).
\end{equation}

The process $t \mapsto \sigma_t(k)$ is pure jump, and from equation \eqref{eq:equation-for-sigma} we have $\Delta_t \sigma(k) = \Delta_t \fm_\ga (k)$. Moreover, from \eqref{eq:generator} and \eqref{eq:rates} we have $\SL_\ga \sigma(k) = \tanh \big( \be h_\ga(\sigma, k) \big) - \sigma(k)$. Hence, using the definition \eqref{eq:sigma-stopped} and rescaling equation \eqref{eq:equation-for-sigma} we get
\begin{equation} \label{eq:mrt_decomp}
	\M_{\ga, \fa}(t, x) = J_{\ga, \fa}(t, x) + \eps^{\kone - 3} \int_0^t \Cd_{\ga, \fa}(s, x) \d s,
\end{equation}
where $t \mapsto J_{\ga, \fa}(t, x) = \sum_{0 \leq s \leq t} \Delta_s \M_{\ga, \fa}(x)$ is a pure jump process and 
\begin{equation}\label{eq:C-gamma}
\Cd_{\ga, \fa}(t, x) =
\begin{cases}
\sigma \bigl( \frac{t}{\alpha}, \frac{x}{\eps} \bigr) - \tanh \bigl( \be \delta X_{\ga}( t, x) \bigr) &\text{for}~~ t < \frac{\tau_{\ga, \fa}}{\alpha}, \\
\sigma'_{\ga, \fa} \bigl( \frac{t}{\alpha}, \frac{x}{\eps} \bigr) - \delta X'_{\ga, \fa}( t, x) &\text{for}~~ t \geq \frac{\tau_{\ga, \fa}}{\alpha}.
\end{cases}
\end{equation}
The process $t \mapsto \Cd_{\ga, \fa}(t, x)$ is adapted and is bounded uniformly in $x$ and $t$, and Assumption~1(4) in \cite{Martingales} is satisfied.

\subsection{Besov spaces of distributions}
\label{sec:Besov}

In this section we recall the definition of the Besov spaces using the Littlewood-Paley theory. 

According to \cite[Prop.~2.10]{BookChemin} there exist two smooth functions $\widetilde \chi, \chi : \R^3 \to \R$, taking values in $[0,1]$, such that $\widetilde \chi$ is supported on $B(0, \frac{4}{3})$, $\chi$ is supported on $B(0, \frac{8}{3}) \setminus B(0, \frac{3}{4})$, and for every $\om \in \R^3$ they satisfy
\begin{equation*}
\widetilde \chi(\om) + \sum_{k = 0}^\infty \chi(2^{-k} \om) = 1.
\end{equation*}
Then we define $\chi_{-1} (\om) := \widetilde \chi(\om)$ and $\chi_{k} (\om) := \chi(2^{-k} \om)$ for $k \geq 0$, and set $\rho_{k} := \SF^{-1} \chi_{k}$, where $\SF^{-1}$ is the inverse Fourier transform on $\R^3$. Then for $k \geq 0$ we have $\rho_k(\om) = 2^{3k} \rho(2^k \om)$ where $\rho := \rho_0$. The $k$-th \emph{Littlewood-Paley block} of a function or tempered distribution $f$ is defined as 
\begin{equation}\label{eq:LittlewoodPaleyBlock}
\delta_k f := \rho_k * f = \SF^{-1} \bigl(\chi_{k}\, \SF f\bigr).
\end{equation}
Then one can show that $f = \sum_{k \geq -1} \delta_k f$ in the sense of distributions for any tempered distribution $f$.

For $\eta \in \R$ and $p, q \in [1, \infty]$ the Besov space $\CB^\eta_{p, q}(\T^3)$ is defined as a completion of the the space of smooth functions $f : \T^3 \to \R$ under the norm 
\begin{equation*}
\| f \|_{\CB^\eta_{p, q}} := \Bigl\| \Bigl(2^{\eta k} \|\delta_k f\|_{L^p}\Bigr)_{k \geq -1} \Bigr\|_{\ell^q},
\end{equation*}
where we extended $f$ periodically on the right-hand side and where we write $\|(a_k)_{k \geq -1}\|_{\ell^q}$ for the $\ell^q$ norm of the sequence $(a_k)_{k \geq -1}$.

It is not hard to see that $\CB^\eta_{\infty, \infty}(\T^3)$ coincides with the space $\CC^\eta(\T^3)$ defined in Section~\ref{sec:notation}.

\subsection{Controlling the processes $\widehat{X}_{\ga}$ and $S_\ga$}

We need to prove some auxiliary bounds which will be used in the proof of Proposition~\ref{prop:Pi-bounds}. The following result provides bounds on the high frequency Fourier modes of the process $X_{\ga}$.

\begin{lemma}\label{lem:X-Fourier-a-priori}
For any $\bar \kappa > 0$ and $M > 0$, there is a non-random constant $C > 0$, such that 
\begin{equation}\label{eq:X-Fourier-a-priori}
| \widehat{X}_{\ga} (t, \om) | \leq C \ga^M,
\end{equation}
uniformly in $t \in \R_+$, $\ga^{-3 - \bar \kappa} \leq |\om|_\sinfty \leq N$ and $\ga \in (0,1)$.
\end{lemma}

\begin{proof}
Using \eqref{eq:X-gamma} and \eqref{eq:S-def}, we may write $X_\ga = K_\ga *_\eps S_\ga$. Then Parseval's identity \eqref{eq:Parseval} then yields $\widehat{X}_{\ga} (t, \om) = \widehat{K}_\ga(\om) \widehat{Z}_{\ga} (t, \om)$. Using the trivial bound $|S_{\ga} (t, x)| \leq \ga^{-3}$ we get $|\widehat{Z}_{\ga} (t, \om)| \lesssim \ga^{-3}$ and the absolute value of $\widehat{X}_{\ga} (t, \om)$ is bounded by $C_1 \ga^{-3} |\widehat{K}_\ga(\om)|$. Furthermore, we use \eqref{eq:K3} to bound it by $C_2 \ga^{-3} |\ga^3\om|^{-m}$ for any integer $m \geq 0$, where the constant $C_2$ depends on $m$. Hence, for any $\bar \kappa > 0$ and $|\om|_\sinfty \geq \ga^{-3 - \bar \kappa}$ we have $| \widehat{X}_{\ga} (t, \om) | \leq C_3 \ga^{\bar \kappa m - 3}$, which is the required bound \eqref{eq:X-Fourier-a-priori} with $M = \bar \kappa m - 3$.
\end{proof}

The following result show that the a priori bound, provided by the stopping time \eqref{eq:tau-1}, yields a bound on the process $S_{\ga}$ defined in \eqref{eq:S-def}. The bound on $S_{\ga}$ is however slightly worse than for the process $X_{\ga}$. Namely, while we consider the average values of $X_\ga$ on the scales above $\emezo$ (see the definition \eqref{eq:eps-norm-1} of the seminorm), we bound $S_\ga$ on strictly larger scales.

\begin{lemma}\label{lem:Z-bound}
Let $\eta$ be as in the statement of Theorem~\ref{thm:main}, let us fix any $\tilde\kappa \in (0,1)$ and let $r$ be the smallest integer satisfying $r > \frac{1 + \eta}{\tilde \kappa} - \eta$ and $r \geq 2$. Then there exist non-random $\ga_0 > 0$ and $C > 0$ such that
\begin{equation*}
\sup_{t \in [0, \tau^{(1)}_{\ga, \fa}]} \sup_{x \in \Lattice} \bigl| \bigl( \iota_\eps S_{\ga} (t)\bigr) (\phi_x^\lambda) \bigr| \leq C \fa \lambda^{\eta},
\end{equation*}
uniformly over $\lambda \in [\emezo^{1 - \tilde\kappa}, 1]$, $\phi \in \CB^{r}$ and $\ga \in (0, \ga_0)$. We recall that the stopping time $\tau^{(1)}_{\ga, \fa}$ is defined in \eqref{eq:tau-1}. 
\end{lemma}

\begin{proof} 
From the definitions \eqref{eq:kernel-K} and \eqref{eq:K-gamma} we have $\SF_{\!\!\eps}K_\ga(\om)= \varkappa_{\ga, 1} \SF_{\!\!\ga}\fK(\eps \om / \ga)$, where $\SF_{\!\!\eps}$ is the discrete Fourier transform defined in \eqref{eq:Fourier} and $\SF_{\!\!\ga}$ is defined by replacing $\eps$ with $\ga$. The first property in \eqref{eq:K-moments} implies that there are $a, c > 0$ such that $\SF \fK(\om) \geq a$ for $|\om|_\sinfty \leq c$, and for all $\ga > 0$ small enough we have $\SF_{\!\!\ga} \fK(\om) \geq a/2$ for $|\om|_\sinfty \leq c$. Hence, $\SF_{\!\!\eps}K_\ga(\om) \geq a \varkappa_{\ga, 1}/2$ for $|\om|_\sinfty \leq c \ga^{-3}$, and \eqref{eq:c-gamma-2} yields $a \varkappa_{\ga, 1}/2 \geq a / 4$ for all $\ga > 0$ small enough. 
We define the function $\psi_\ga(x)$ by its discrete Fourier transform $\widehat{\psi}_\ga(\om) = 1 / \widehat{K}_\ga(\om)$ for $|\om|_\sinfty \leq c \ga^{-3}$ and $\widehat{\psi}_\ga(\om) = 0$ for $|\om|_\sinfty > c \ga^{-3}$. Then $|\widehat{\psi}_\ga(\om)| \leq 4/a$ for $|\om|_\sinfty \leq c \ga^{-3}$, which implies that $\psi_\ga$ is a rescaled function with the scaling parameter $\gamma$ (in the sense of \eqref{eq:rescaled-function}).

Let $n_0$ be the smallest integer such that $2^{n_0} > c \ga^{-3}$. Then we use the Littlewood-Paley blocks, defined in Appendix~\ref{sec:Besov}, to write $\bigl( \iota_\eps  S_{\ga} (t)\bigr) (\phi_x^\lambda) = \bigl( \iota_\eps  S^{(1)}_{\ga} (t)\bigr) (\phi_x^\lambda) + \bigl( \iota_\eps S^{(2)}_{\ga} (t)\bigr) (\phi_x^\lambda)$,  where
\begin{equation*}
\bigl( \iota_\eps S^{(1)}_{\ga} (t)\bigr) (\phi_x^\lambda) = \sum_{k \geq n_0} \bigl(\delta_k \iota_\eps  S_{\ga} (t)\bigr) (\phi_x^\lambda), \qquad \bigl( \iota_\eps S^{(2)}_{\ga} (t)\bigr) (\phi_x^\lambda) = \sum_{-1 \leq k < n_0} \bigl(\delta_k \iota_\eps  S_{\ga} (t)\bigr) (\phi_x^\lambda).
\end{equation*}

We first bound the process $S^{(1)}_{\ga}$. We note that we can write $\bigl(\delta_k \iota_\eps S_{\ga} (t)\bigr) (\delta_k \phi_x^\lambda)$ in the sum. From the definition \eqref{eq:S-def} we have $|S_\ga(t,x)| \leq \ga^{-3}$. Let us fix any integer $r > \frac{1 + \eta}{\tilde \kappa} - \eta$ such that $r \geq 2$ and $\kappa = r + \eta - \frac{1 + \eta}{\tilde \kappa}$. Then we have $0 < \kappa < r$. Moreover, if we take $\phi \in \CB^{r}$, then $\| \delta_k \phi_x^\lambda \|_{L^1} \lesssim (\lambda 2^{k})^{-r + \kappa}$, because $\CB^{r}$ is embedded into the Besov space $\CB^{r - \kappa}_{\infty, \infty}$. Then we have
\begin{equation*}
\bigl|\bigl( \iota_\eps S^{(1)}_{\ga} (t)\bigr) (\phi_x^\lambda)\bigr| \lesssim \ga^{-3} \sum_{k \geq n_0} (\lambda 2^{k})^{-r + \kappa} \lesssim \ga^{-3} (\lambda 2^{n_0})^{-r + \kappa} \lesssim \lambda^{-r + \kappa} \emezo^{r - 1 - \kappa}.
\end{equation*}
If $\lambda \geq \emezo^{1 - \tilde \kappa}$, then the latter is bounded by $\lambda^\eta$.

Now, we will bound $S^{(2)}_{\ga}$. We note that for $k < n_0$ we can express $S_{\ga} (t)$ in terms of $\psi_\ga *_\eps X_{\ga} (t)$, and using \eqref{eq:LittlewoodPaleyBlock} we may write 
\begin{equation}\label{eq:Z2-expansion}
\bigl( \iota_\eps S^{(2)}_{\ga} (t)\bigr) (\phi_x^\lambda) = \bigl( \iota_\eps X_{\ga} (t)\bigr) \bigl(\Phi_x^{\ga, \lambda}\bigr),
\end{equation}
where 
\begin{equation}\label{eq:Phi-function-def}
\Phi_x^{\ga, \lambda}(y) =  \int_{\R^3} \eps^3 \sum_{z \in \Lattice} \phi_{x}^\lambda(z) \psi_\ga(z - v) \sum_{-1 \leq k < n_0} \rho_k(v - y)\, \d v.
\end{equation}
This is a convolution of three rescaled functions and hence it can be viewed as a function rescaled by $\lambda$. Then the definition \eqref{eq:tau-1} yields $|( \iota_\eps X_{\ga} (t)) (\Phi_x^{\ga, \lambda})| \lesssim \fa 2^{- k \eta}$ for $t \in [0, \tau^{(1)}_{\ga, \fa}]$. We note that the function $\Psi_x^{\ga, \lambda}$ is not compactly supported, as required in the definition of the seminorm \eqref{eq:eps-norm-1}. This however does not play any role since the process $X_{\ga} (t)$ is periodic and the function has a fast decay at infinity (because the function $\phi$ involved in the definition \eqref{eq:Phi-function-def} is compactly supported and the functions $\psi_\ga$ and $\rho_k$ have fast decays at infinity). Then \eqref{eq:Z2-expansion} is absolutely bounded by a constant multiple of $\fa \la^{\eta}$, as required. 
\end{proof}

\subsection{Controlling the bracket process of the martingales}

In Section~\ref{sec:second-symbol} we need to analyse the process
\begin{equation*}
Q_{\ga, \fa}(t, x) := \eps^3 \sum_{y \in \Lattice} \int_{0}^{\tau_{\ga, \fa}} \!\!\!\mywidetilde{\SK}^\ga_{t-s}( x-y)^2 S_\ga (s, y) X_\ga(s, y) \d s,
\end{equation*}
where we used the stopping time \eqref{eq:tau}. In the following lemma we estimate the error after replacing $S_\ga$ by its local average, i.e. we estimate how close $Q_{\ga, \fa}$ is to
	\begin{equation*}
		\un{Q}_{\ga, \fa}(t, x) := \eps^3 \sum_{y \in \Lattice} \int_0^{\tau_{\ga, \fa}} \mywidetilde{\SK}^\ga_{t-s} (x-y)^2 \un{X}_{\ga} (s, y) X_\ga(s, y) \d s,
	\end{equation*}
where the process $\un{X}_{\ga}$ is defined in \eqref{eq:X-under}.

\begin{lemma}\label{lem:Y-approximation}
	For every $T > 0$ there exist deterministic constants $\ga_0 > 0$ and $C > 0$, depending also on the constant $\un{\kappa}$ fixed in \eqref{eq:X-under}, such that
	\begin{equation}\label{eq:Y-approximation}
	\sup_{t \in [0, T]} \sup_{x \in \Lattice} |Q_{\ga, \fa}(t, x) - \un{Q}_{\ga, \fa}(t, x)| \leq C \fa \ga^{3 (\eta-1)},
	\end{equation}
	uniformly in $\ga \in (0, \ga_0)$. The value $\eta$ is as in the statement of Theorem~\ref{thm:main}.
\end{lemma}

\begin{proof}
As we stated in the beginning of Section \ref{sec:lift}, $\widetilde{G}^\ga = \mywidetilde{\SK}^\ga + \widetilde{\SR}^\ga$. Replacing $\mywidetilde{\SK}^\ga$ in the definitions of $Q_{\ga, \fa}$ by $\widetilde{G}^\ga$, and using \eqref{eq:X-apriori} and the trivial bound $|S_\ga(s, y)| \leq \ga^{-3}$, we obtain
\begin{equation*}
Q_{\ga, \fa}(t, x) = \eps^3 \sum_{y \in \Lattice} \int_0^{\tau_{\ga, \fa}} \widetilde{G}^\ga_{t-s}(x-y)^2 S_\gamma(s, y) X_\ga(s, y)\, \d s + \CO\bigl(\ga^{3 (\eta-1)}\bigr),
\end{equation*}
and an analogous formula holds for $\un{Q}_{\ga, \fa}$. Here, we made use of the estimates 
\begin{equation*}
\eps^3 \sum_{y \in \Lattice} \int_0^{\tau_{\ga, \fa}} \mywidetilde{\SK}^\ga_{t-s}(y) \widetilde{\SR}^\ga_{t-s}(y)\, \d s \lesssim 1, \qquad \eps^3 \sum_{y \in \Lattice} \int_0^{\tau_{\ga, \fa}} \widetilde{\SR}^\ga_{t-s}(y)^2\, \d s \lesssim 1
\end{equation*}
which follow from smoothness of the function $\SR^\ga$ and an integrable singularity of $\mywidetilde{\SK}^\ga$ (see Appendix~\ref{sec:kernels}). We note that the value of the stopping time $\tau_{\ga, \fa}$ does not play a role in this bound, because the kernel $\widetilde{G}^\ga_{t-s}$ vanishes for $s \geq t$ and the integration interval is contained in $[0, t]$. That is why the error term $\CO\bigl(\ga^{3 (\eta-1)}\bigr)$ is bounded uniformly in $t \in [0, T]$ and $x \in \Lattice$.

Using the spatial periodicity of the processes we can write furthermore  
\begin{equation*}
Q_{\ga, \fa}(t, x) = \eps^3 \sum_{y \in \Le} \int_0^{\tau_{\ga, \fa}} \Phi^\ga_{t-s}(x-y) S_\gamma(s, y) X_\ga(s, y)\, \d s + \CO\bigl(\ga^{3 (\eta-1)}\bigr),
\end{equation*}
where by analogy with \eqref{eq:From-P-to-G} the function $\Phi^\ga$ is defined by
\begin{equation}\label{eq:Phi-def}
 \eps^3 \sum_{x \in \Le} \Phi^\ga_t(x) f(x) = \eps^3 \sum_{x \in \Lattice} \widetilde{G}^\ga_{t}(x)^2 f(x),
\end{equation}
for any $f : \Le \to \R$, where on the right-hand side we extended $f$ periodically to $\Lattice$. Then we can write $Q_{\ga, \fa}(t, x) = \un{Q}_{\ga, \fa}(t, x) + E_{\ga, \fa}(t, x) + \CO(\ga^{3 (\eta-1)})$ with the error term
\begin{equation*}
E_{\ga, \fa}(t, x) = \eps^3 \sum_{y \in \Le} \int_0^{\tau_{\ga, \fa}} \Phi^\ga_{t-s}(x-y) \bigl(S_\gamma - \un{X}_{\ga}\bigr) (s, y) X_\ga(s, y)\, \d s,
\end{equation*}
and we need to show that this error term is absolutely bounded by the right-hand side of \eqref{eq:Y-approximation}. Applying Parseval's identity \eqref{eq:Parseval} we get 
\begin{equation*}
E_{\ga, \fa}(t, x) = \frac{1}{8} \int_0^{\tau_{\ga, \fa}} \sum_{|\om|_{\sinfty} \leq N} \SF_{\!\!\eps} \Phi^\ga_{t-s} (\om)\, \SF_{\!\!\eps}\bigl((S_\gamma - \un{X}_{\ga}) X_\ga \bigr)(s, \om)\, e^{-\pi i \om \cdot x}\, \d s.
\end{equation*}
We expect that the high frequency Fourier modes of $\Phi^\ga_{t-s}$ decay very fast, which allows to have a good control of the whole expression in the integral. To separate low and high Fourier modes of the function $\Phi^\ga_{t-s}$, we take $\kappa_1 > 0$, whose precise value will be fixed later, and write 
\begin{align*}
E_{\ga, \fa}(t, x) &= \frac{1}{8} \int_0^{\tau_{\ga, \fa}} \sum_{\gamma^{-3 - \kappa_1} < |\om|_{\sinfty} \leq N} \SF_{\!\!\eps} \Phi^\ga_{t-s} (\om)\, \SF_{\!\!\eps}\bigl((S_\gamma - \un{X}_{\ga}) X_\ga \bigr)(s, \om)\, e^{-\pi i \om \cdot x} \,\d s \\
&\qquad + \frac{1}{8} \int_0^{\tau_{\ga, \fa}} \sum_{|\om|_{\sinfty} \leq \gamma^{-3 - \kappa_1}} \SF_{\!\!\eps} \Phi^\ga_{t-s} (\om)\, \SF_{\!\!\eps}\bigl((S_\gamma - \un{X}_{\ga}) X_\ga \bigr)(s, \om)\, e^{-\pi i \om \cdot x} \,\d s.
\end{align*}
We denote these two terms by $E^{(1)}_{\ga, \fa}(t, x)$ and $E^{(2)}_{\ga, \fa}(t, x)$ respectively. 

We start with analysing the term $E^{(1)}_{\ga, \fa}$. The processes $S_\gamma$ and $\un{X}_{\ga}$ can be uniformly bounded by $\ga^{-3}$, while for the process $X_\ga$ we have the bound \eqref{eq:X-apriori}. Then 
\begin{equation*}
|E^{(1)}_{\ga, \fa}(t, x)| \lesssim \fa \gamma^{3 (\eta - 1)} \int_0^{t} \sum_{\gamma^{-3 - \kappa_1} < |\om|_{\sinfty} \leq N} |\SF_{\!\!\eps} \Phi^\ga_{s} (\om)|\, \d s,
\end{equation*}
where we used the property $\Phi^\ga_{s} \equiv 0$ for $s < 0$ to extend the integral to $[0, t]$. The definition \eqref{eq:Phi-def}, the Poisson summation formula and the identity \eqref{eq:tildeP} yield 
\begin{align*}
\SF_{\!\!\eps} \Phi^\ga_{s} (\om) &= \sum_{|\om'|_{\sinfty} \leq N} \SF_{\!\!\eps}{\widetilde{P}^\ga_{s}}(\om - \om') \SF_{\!\!\eps}{\widetilde{P}^\ga_{s}}(\om') \\
& = \sum_{|\om'|_{\sinfty} \leq N} \exp \Bigl( \ga^{-6} \varkappa_{\ga, 3}^2 \bigl( \widehat{K}_\ga(\om - \om') + \widehat{K}_\ga(\om') - 2 \bigr) s \Bigr) \widehat{K}_\ga(\om - \om') \widehat{K}_\ga(\om').
\end{align*}
from which we readily get
\begin{equation}\label{eq:double-heat-sum-1}
\int_0^t |\SF_{\!\!\eps} \Phi^\ga_{s} (\om)|\, \d s \leq t \ga^{6} \varkappa_{\ga, 3}^{-2} \sum_{|\om'|_{\sinfty} \leq N} |\widehat{K}_\ga(\om - \om') \widehat{K}_\ga(\om')|,
\end{equation}
where we made use of \eqref{eq:K2} to bound the exponential by $1$. Estimating $\widehat{K}_\ga$ by \eqref{eq:K1B} and \eqref{eq:K3}, for any $k \geq 4$ we bound the preceding expression by a constant multiple of 
\begin{equation}\label{eq:double-heat-sum}
\ga^{6} \sum_{|\om'|_{\sinfty} \leq N} (| \ga^3 (\om - \om')|\vee 1)^{-2k} (| \ga^3 \om'| \vee 1)^{-2k}.
\end{equation}
Then we have 
\begin{align}
&|E_{\ga, \fa}^{(1)}(t, x)| \lesssim \fa \gamma^{3 (\eta + 1)} \sum_{\gamma^{-3 - \kappa_1} < |\om|_{\sinfty} \leq N} \sum_{|\om'|_{\sinfty} \leq N} (| \ga^3 (\om - \om')|\vee 1)^{-2k} (| \ga^3 \om'| \vee 1)^{-2k} \nonumber \\
&\qquad \lesssim \fa \gamma^{3 (\eta - 5)} \int_{\gamma^{- \kappa_1} < |\om|_{\sinfty} \leq \ga^3 N} \int_{|\om'|_{\sinfty} \leq \ga^3 N} (| \om - \om'| \vee 1)^{-2k} (| \om'| \vee 1)^{-2k} \d \om' \d \om. \label{eq:E1-integral}
\end{align}
In order to bound this integral, we split the domain of integration into two subdomains. 

If $|\om'|_{\sinfty} \leq |\om|_{\sinfty} / 2$, then $| \om - \om'|_{\sinfty} \geq |\om|_{\sinfty} / 2$. We also have $|\om|_{\sinfty} \geq 1$. Then the part of the double integral \eqref{eq:E1-integral}, in which the integration variables satisfy $|\om'|_{\sinfty} \leq |\om|_{\sinfty} / 2$, is bounded by a constant times
\begin{equation*}
\int_{\gamma^{ - \kappa_1} < |\om|_{\sinfty} \leq \ga^3 N} \int_{|\om'|_{\sinfty} \leq |\om|_{\sinfty} / 2} | \om|_{\sinfty}^{-2k} (| \om'|_{\sinfty} \vee 1)^{-2k}\d \om' \d \om \lesssim \int_{\gamma^{- \kappa_1} < |\om|_{\sinfty} \leq \ga^3 N} |\om|_{\sinfty}^{-2k} \d \om,
\end{equation*}
and the latter is of order $\ga^{(2k - 3) \kappa_1}$. Taking $k$ large enough, we can make the power of $\ga$ arbitrarily big.

If $|\om'|_{\sinfty} > |\om|_{\sinfty} / 2$, then we simply bound $| \om - \om'|_{\sinfty} \vee 1 \geq 1$, and the respective part of the double integral \eqref{eq:E1-integral} is bounded by 
\begin{align*}
\int_{\gamma^{ - \kappa_1} < |\om|_{\sinfty} \leq \ga^3 N} \int_{|\om|_{\sinfty} / 2 < |\om'|_{\sinfty} \leq \ga^3 N} | \om'|_{\sinfty}^{-2k} \d \om' \d \om \lesssim \int_{\gamma^{- \kappa_1} < |\om|_{\sinfty} \leq \ga^3 N} | \om|_{\sinfty}^{3 -2k} \d \om,
\end{align*}
which is of order $\ga^{(2k - 6) \kappa_1}$. Combining the preceding bounds, we get $|E_{\ga, \fa}^{(1)}(t, x)| \lesssim \fa$.

Now, we will analyse the term $E^{(2)}_{\ga, \fa}$. We have
\begin{equation}\label{eq:E2-gamma-bound}
E^{(2)}_{\ga, \fa}(t, x) = \int_0^{\tau_{\ga, \fa}} \sum_{|\om|_\sinfty \leq \gamma^{-3 - \kappa_1}} \SF_{\!\!\eps}\Phi^\ga_{t-s} (\om)\, \SF_{\!\!\eps} \bigl((S_\gamma - \un{X}_{\ga}) X_\ga \bigr)(s, \om)\, e^{-\pi i \om \cdot x} \,\d s.
\end{equation}
Furthermore, \eqref{eq:Fourier-of-product} yields 
\begin{align*}
\SF_{\!\!\eps} \bigl((S_\gamma - \un{X}_{\ga}) X_\ga \bigr)(s, \om) &= \sum_{|\om'|_{\sinfty} \leq N} \bigl(\widehat{Z}_\ga - \widehat{\un{X}}_\ga \bigr)(s, \om') \widehat{X}_\ga(s, \om - \om') \\
&= \sum_{|\om'|_{\sinfty} \leq N} \bigl(1 - \widehat{\un{K}}_\ga(\om)\bigr) \widehat{Z}_\ga(s, \om') \widehat{X}_\ga(s, \om - \om').
\end{align*}
We assumed in Section~\ref{sec:a-priori} that $\SF \un{\fK} (\om) = 1$ for all $\om \in \R^3$ such that $|\om|_{\sinfty} \leq 1$, from which we conclude that the terms in the preceding sum may be non-vanishing only for $|\om'|_{\sinfty} > \ga^{-3 - \un \kappa}$. Then the variables in these sums satisfy $|\om - \om'|_{\sinfty} \geq c \gamma^{-3 - \un \kappa}$ for some $c > 0$, if we take $\kappa_1 = \un{\kappa}/2$. From Lemma~\ref{lem:X-Fourier-a-priori} we have $|\widehat{X}_\ga(s, \om - \om')| \lesssim \gamma^M$ for any $M > 0$, where the proportionality constant depends on $\un{\kappa}$ and $M$. Applying the preceding estimate to \eqref{eq:E2-gamma-bound}, we get
\begin{align*}
|E^{(2)}_{\ga, \fa}(t, x)| \lesssim \ga^{M -3} \int_0^{t} \sum_{|\om|_\sinfty \leq \gamma^{-3 - \kappa_1}} \bigl|\SF_{\!\!\eps} \Phi^\ga_{s} (\om) \bigr|\, \d s,
\end{align*}
where as before we used the bound $|\widehat{Z}_\ga(s, \om')| \lesssim \ga^{-3}$ and extended the integral to the interval $[0, t]$. Using \eqref{eq:double-heat-sum-1} and \eqref{eq:double-heat-sum}, this expression is bounded as 
\begin{align*}
|E^{(2)}_{\ga, \fa}(t, x)| &\lesssim \ga^{M + 3} \sum_{|\om|_\sinfty \leq  \gamma^{-3 - \kappa_1}} \sum_{|\om'|_\sinfty \leq N} (| \ga^3 (\om - \om')|_\sinfty\vee 1)^{-2k} (| \ga^3 \om'|_\sinfty \vee 1)^{-2k} \\
&\lesssim \ga^{M -15} \int_{|\om|_\sinfty \leq \ga^3  \gamma^{-3 - \kappa_1}} \int_{|\om'|_\sinfty \leq \ga^3 N} (| \om - \om'|_\sinfty \vee 1)^{-2k} (| \om'|_\sinfty \vee 1)^{-2k} \d \om' \d \om.
\end{align*}
This expression is of order $\ga^{M -15}$ which can be made arbitrarily small by taking $M$ large.
\end{proof}

\subsection{Controlling the process $X'_{\ga, \fa}$}
\label{sec:X-prime}

We recall that $X'_{\ga, \fa}$ is defined below \eqref{eq:sigma-stopped} via the spin field $\sigma'_{\ga, \fa}$, and let us define 
\begin{equation}\label{eq:Z-prime-def}
S'_{\ga, \fa}(t,x) := \frac{1}{\delta} \sigma'_{\ga, \fa} \Bigl( \frac{t}{\alpha}, \frac{x}{\eps} \Bigr) \qquad \text{for}~~ x \in \Le,~ t \geq 0.
\end{equation}
We need to control these two processes.

\begin{lemma}\label{lem:X-prime-bound}
Let $\eta$ be as in Theorem~\ref{thm:main}. There exists $\gamma_0 > 0$ such that for every $p \geq 1$ and $T > 0$ one has 
\begin{equation}\label{eq:X-prime-bound}
\E \biggl[ \sup_{t \in [\tau_{\ga, \fa}, T]} \bigl| \bigl( \iota_\eps X'_{\ga, \fa} (t)\bigr) (\phi_x^\lambda) \bigr|^p\biggr] \leq C \fa^p (\lambda \vee \emezo)^{\eta p},
\end{equation}
uniformly over $\ga \in (0, \ga_0)$, $\phi \in \CB^{1}$, $x \in \Lattice$ and $\lambda \in (0, 1]$. The constant $C$ depends only on $p$, $T$ and $\ga_0$.
\end{lemma}

\begin{proof}
By the definition in Section~\ref{sec:a-priori} we have $X'_{\ga, \fa}(\tau_{\ga, \fa}) = X_{\ga}(\tau_{\ga, \fa})$, and in the same way as we derived equation \eqref{eq:IsingKacEqn}, we get for $t \geq \tau_{\ga, \fa}$
\begin{equation}\label{eq:X-prime-equation-P}
	X'_{\ga, \fa}(t, x) = \bigl(P^\ga_{t - \tau_{\ga, \fa}} X_{\ga}\bigr)(\tau_{\ga, \fa}, x) + \eps^3 \sum_{y \in \Le} \int_{\tau_{\ga, \fa}}^t \widetilde{P}^{\ga}_{t - s}(x-y) \,\d \M'_{\ga, \fa}(s, y).
\end{equation}
Extending the processes periodically to $x \in \Lattice$ and using \eqref{eq:From-P-to-G}, we replace $P^{\ga}$, $\widetilde{P}^{\ga}$ and $\Le$ in the preceding equation by $G^\ga$, $\widetilde{G}^{\ga}$ and $\Lattice$ respectively. Then, for a test function $\phi \in \CB^{1}$ we have
\begin{equation}\label{eq:X-prime-equation}
\bigl( \iota_\eps X'_{\ga, \fa} (t)\bigr) (\phi_x^\lambda) = \bigl( \iota_\eps G^\ga_{t - \tau_{\ga, \fa}} X_{\ga}(\tau_{\ga, \fa})\bigr) (\phi_x^\lambda) + \eps^3 \sum_{y \in \Lattice} \int_{\tau_{\ga, \fa}}^t \bigl(\widetilde{G}^{\ga}_{t -s} *_\eps \phi_x^\lambda\bigr)(y) \,\d \M'_{\ga, \fa}(s, y).
\end{equation}
We denote the two terms on the right-hand side by $A_{\ga, \lambda}(t)$ and $B_{\ga, \lambda}(t)$ respectively. Then the first term may be written as 
\begin{equation*}
A_{\ga, \lambda}(t) = \eps^3 \sum_{y \in \Lattice} G^\ga_{t - \tau_{\ga, \fa}}(y) \bigl( \iota_\eps X_{\ga}(\tau_{\ga, \fa})\bigr) (\phi_{x - y}^\lambda).
\end{equation*}
Using the a priori bound on $X_{\ga}$, provided by the stopping time \eqref{eq:tau-1}, we get $|\bigl( \iota_\eps X_{\ga}(\tau_{\ga, \fa})\bigr) (\phi_{x - y}^\lambda)| \lesssim \fa (\lambda \vee \emezo)^\eta$ where we used the definition of the seminorm \eqref{eq:eps-norm-1}. Then since the kernel $G^\ga_t$ integrates to $1$, we get
\begin{equation}\label{eq:A-bound}
\bigl|A_{\ga, \lambda}(t)\bigr| \lesssim \fa (\lambda \vee \emezo)^\eta,
\end{equation}
with a proportionality constant independent of the involved values. Here, we used the fact that the discrete heat kernel $G^\ga_t$ is absolutely summable over $\Lattice$ and the sum is bounded uniformly in $\ga$ and $t$, which follows from Lemma~\ref{lem:Pgt}.

Now, we will bound the last term in \eqref{eq:X-prime-equation}. For this, we define 
\begin{equation*}
B_{\ga, \lambda}(t', t) := \eps^3 \sum_{y \in \Lattice} \int_{\tau_{\ga, \fa}}^{t'} \bigl(\widetilde{G}^{\ga}_{t - s} *_\eps \phi_x^\lambda\bigr)(y) \,\d \M'_{\ga, \fa}(s, y),
\end{equation*}
so that $B_{\ga, \lambda}(t) = B_{\ga, \lambda}(t, t)$ and the process $t' \mapsto B_{\ga, \lambda}(t', t)$ is a martingale on $[\tau_{\ga, \fa}, t]$. In order to apply the Burkholder-Davis-Gundy inequality \cite[Prop.~A.2]{Martingales} to this martingale, we need to bound its jumps and bracket process. The jump times of $B_{\ga, \lambda}$ coincide with those of $\M'_{\ga, \fa}$, and we get
\begin{equation*}
\bigl|\Delta_s B_{\ga, \lambda}(\bigcdot, t)\bigr| \leq \eps^3 \sum_{y \in \Lattice} \bigl|\bigl(\widetilde{G}^{\ga}_{t - s} *_\eps \phi_x^\lambda\bigr)(y) \bigr| |\Delta_s \M'_{\ga, \fa}(\bigcdot, y)|,
\end{equation*}
for $s \in [\tau_{\ga, \fa}, t]$, where we use the jump of the martingale $\Delta_s \M'_{\ga, \fa}$ defined in \eqref{eq:jump}. Moreover, the jump size of $\M'_{\ga, \fa}$ is bounded by $2 \ga^{-3}$ and if $\M'_{\ga, \fa}(s, y)$ has a jump, it happens almost surely at the points $\{y_* + k : k \in \Z^3\}$ for a unique $y_* \in \Le$ (recall Section~\ref{sec:martingales-properties} and periodicity of the martingale). Thus, we get almost surely 
\begin{equation}\label{eq:B-jump-bound}
\bigl|\Delta_s B_{\ga, \lambda}(\bigcdot, t)\bigr| \leq 2 \ga^{-3} \eps^3 \sup_{y_* \in \Le} \sum_{k \in \Z^3} \bigl|\bigl(\widetilde{G}^{\ga}_{t - s} *_\eps \phi_x^\lambda\bigr)(y_* + k) \bigr| \lesssim \ga^{-3} \eps^3 \lesssim \ga^9.
\end{equation}
The sum is bounded, because the discrete heat kernel decays very fast at infinity (see Lemma~\ref{lem:Pgt}).

Recalling \eqref{eq:M-prime-bracket}, the bracket process of $B_{\ga, \lambda}(t', t)$ equals 
\begin{equation*}
\big\langle B_{\ga, \lambda}(\bigcdot, t) \big\rangle_{t'} = 2 \varkappa_{\ga, 2} \sum_{y \in \Lattice} \int_{\tau_{\ga, \fa}}^{t'} \bigl(\widetilde{G}^{\ga}_{t - s} *_\eps \phi_x^\lambda\bigr)(y)^2 \Bigl( 1 - \delta \sigma'_{\ga, \fa} \Bigl(\frac{s}{\alpha}, \frac{x}{\eps}\Bigr) X'_{\ga, \fa}(s, x) \Bigr) \d s.
\end{equation*}
The process in the parentheses is bounded by a constant, and the definition \eqref{eq:scalings} yields 
\begin{equation*}
\bigl|\big\langle B_{\ga, \lambda}(\bigcdot, t) \big\rangle_{t'}\bigr| \lesssim \eps^3 \sum_{y \in \Lattice} \int_{0}^{t} \bigl(\widetilde{G}^{\ga}_{s} *_\eps \phi_x^\lambda\bigr)(y)^2 \d s,
\end{equation*}
where we used $t' \leq t$. Similarly to how we estimated \eqref{eq:renorm-constant1}, we can show that 
\begin{equation}\label{eq:B-bracket-bound}
\bigl|\big\langle B_{\ga, \lambda}(\bigcdot, t) \big\rangle_{t'}\bigr| \lesssim (\lambda \vee \emezo)^{-1}.
\end{equation}

Applying the Burkholder-Davis-Gundy inequality \cite[Prop.~A.2]{Martingales} and using the bounds \eqref{eq:B-jump-bound} and \eqref{eq:B-bracket-bound}, we get
\begin{equation}\label{eq:B-bound}
\Bigl(\E \Bigl[ \sup_{t \in [\tau_{\ga, \fa}, T]} \bigl| B_{\ga, \lambda}(t) \bigr|^p\Bigr]\Bigr)^{\frac{1}{p}} \lesssim (\lambda \vee \emezo)^{-\frac{1}{2}} + \ga^9.
\end{equation}
Using then the Minkowski inequality and the bounds \eqref{eq:A-bound} and \eqref{eq:B-bound}, we obtain from \eqref{eq:X-prime-equation} the required result \eqref{eq:X-prime-bound}.
\end{proof}

Using the preceding result, the following one is proved in exactly the same way as Lemma~\ref{lem:Z-bound}.

\begin{lemma}\label{lem:Z-prime-bound}
For any $\tilde\kappa \in (0,1)$ there exist $\ga_0 > 0$ such that for any $\ga \in (0, \ga_0)$, $T > 0$, $\phi \in \CB^{r}$ and $\lambda \in [\emezo^{1 - \tilde\kappa}, 1]$ one has
\begin{equation*}
\sup_{t \in [\tau_{\ga, \fa}, T]} \sup_{x \in \Lattice} \bigl| \bigl( \iota_\eps S'_{\ga, \fa} (t)\bigr) (\phi_x^\lambda) \bigr| \leq C \fa \lambda^{\eta},
\end{equation*}
where the values $\eta$ and $r$ is the same as in the statement of Lemma~\ref{lem:Z-bound}. The non-random proportionality constant $C$ depends on $T$ and is independent of $\ga$, $\phi$ and $\lambda$.
\end{lemma}

\subsection{Controlling the process $\un{X}'_{\ga, \fa}$}

Let us define by analogy with \eqref{eq:c-under} the renormalisation term, which is a function of the time variable,
\begin{equation}\label{eq:C-gamma-function}
\un{\fC}_{\ga}(t) := 2 \varkappa_{\ga, 2} \int_{0}^t \eps^3 \sum_{x \in \Le} \un{P}^{\ga}_s(x) \widetilde{P}^\ga_s(x) \,\d s,
\end{equation}
where $\un{P}^{\ga}_{t} := P^{\ga}_{t} *_\eps \un{K}_\ga$ and $\varkappa_{\ga, 2}$ was defined in \eqref{constant:isingkac_one}. The following result will be useful later.

\begin{lemma}\label{lem:renorm-function-underline}
The constant \eqref{eq:c-under} and the function \eqref{eq:C-gamma-function} satisfy $|\un{\fC}_\ga - \un{\fC}_\ga(t)| \lesssim t^{-c/2} \emezo^{c -1}$ for any $c \in [0, 1)$.
\end{lemma}

\begin{proof}
The proof of the bound goes along the lines of the proof of Lemma~\ref{lem:renorm-constant-underline}. More precisely, as in \eqref{eq:un-fc-bound} we get 
\begin{equation*}
\un{\fC}_\ga - \un{\fC}_\ga(t) = \frac{\varkappa_{\ga, 2}}{8} \sum_{0 < |\om|_{\sinfty} \leq N} \frac{\alpha \widehat{\un{K}}_\ga(\om) \widehat{K}_\ga(\om)}{\varkappa_{\ga, 3}^2 \big(1 -  \widehat{K}_\ga(\om) \big)} \exp \Bigl( 2 \varkappa_{\ga, 3}^2 \big( \widehat{K}_\ga(\om) - 1 \big) \frac{t}{\alpha} \Bigr).
\end{equation*}
Since the power of the exponential is negative, we can use the simple bound $e^{-x} \lesssim x^{-c/2}$ for any $x > 0$ and any $c > 0$, to estimate 
\begin{equation*}
|\un{\fC}_\ga - \un{\fC}_\ga(t)| \lesssim \sum_{0 < |\om|_{\sinfty} \leq N} \frac{\alpha |\widehat{\un{K}}_\ga(\om)| |\widehat{K}_\ga(\om)|}{1 -  \widehat{K}_\ga(\om)} \Bigl(\big(1 -  \widehat{K}_\ga(\om) \big) \frac{t}{\alpha} \Bigr)^{-c/2}.
\end{equation*}
Proceeding as in the proof of Lemma~\ref{lem:renorm-constant-underline}, we get the desired bound.
\end{proof}

Let us define the process $\un{X}'_{\ga, \fa}$ as in \eqref{eq:X-under}, but via the spin field  $\sigma'_{\ga, \fa}$. The following result will be used in Section~\ref{sec:second-symbol}.

\begin{lemma}\label{lem:X2-under-prime}
Let $\eta$ be as in Theorem~\ref{thm:main} and let $\un\kappa$ and $\un\emezo$ be as in \eqref{eq:tau-2}. There exists $\gamma_0 > 0$ such that for every $p \geq 1$ and $T > 0$ one has 
\begin{equation}\label{eq:X2-prime-bound}
\E \biggl[ \sup_{t \in [\tau_{\ga, \fa}, T]} (t - \tau_{\ga, \fa})^{-\frac{\eta p}{2}} \bigl\| \un{X}'_{\ga, \fa} (t) X'_{\ga, \fa}(t) - \un{\fC}_\ga (t-\tau_{\ga, \fa}) \bigr\|_{L^\infty}^p\biggr] \leq C \fa^{2 p} \un{\emezo}^{\eta p},
\end{equation}
uniformly over $\ga \in (0, \ga_0)$. The constant $C$ depends only on $p$, $T$, $\ga_0$ and $\un\kappa$.
\end{lemma}

\begin{proof}
Let $I_{\ga, \fa}(t,x) := \un{X}'_{\ga, \fa} (t, x) X'_{\ga, \fa}(t, x) - \un{\fC}_\ga (t-\tau_{\ga, \fa})$ be the function, which we need to bound. From the proof of Lemma~\ref{lem:X-prime-bound} we know that $X'_{\ga, \fa}$ solves equation \eqref{eq:X-prime-equation-P}. Similarly, we can show that 
\begin{equation}\label{eq:X-under-prime-equation}
\un{X}'_{\ga, \fa} (t, x) = \bigl(P^\ga_{t - \tau_{\ga, \fa}} \un{X}_{\ga}\bigr)(\tau_{\ga, \fa}, x) + \eps^3 \sum_{y \in \Le} \int_{\tau_{\ga, \fa}}^t \un{P}^{\ga}_{t - s}(x-y) \,\d \M'_{\ga, \fa}(s, y).
\end{equation}
Here, we need to take $\ga$ small enough so that the radius of the support of the function $\un{K}_\ga$ gets smaller than one. Let us denote by $Y'_{\ga, \fa}(t, x)$ and $\un{Y}'_{\ga, \fa} (t, x)$ the last terms in \eqref{eq:X-prime-equation-P} and \eqref{eq:X-under-prime-equation} respectively, and let us define 
\begin{equation}\label{eq:Y-definitions}
\begin{aligned}
Y'_{\ga, \fa}(r, t, x) &:= \eps^3 \sum_{y \in \Le} \int_{\tau_{\ga, \fa}}^{r} \widetilde{P}^{\ga}_{t - s}(x-y) \,\d \M'_{\ga, \fa}(s, y), \\
\un{Y}'_{\ga, \fa}(r, t, x) &:= \eps^3 \sum_{y \in \Le} \int_{\tau_{\ga, \fa}}^{r} \un{P}^{\ga}_{t - s}(x-y) \,\d \M'_{\ga, \fa}(s, y).
\end{aligned}
\end{equation}
Then these two processes are \cadlag martingales in $r \in [\tau_{\ga, \fa}, t]$, and $Y'_{\ga, \fa}(t, x) = Y'_{\ga, \fa}(t, t, x)$ and $\un{Y}'_{\ga, \fa}(t, x) = \un{Y}'_{\ga, \fa}(t, t, x)$. Since these martingales have finite total variation, their quadratic covariation may be written as (see \cite{JS03})
\begin{equation}\label{eq:Y-jumps}
\bigl[Y'_{\ga, \fa}(\bigcdot, t, x), \un{Y}'_{\ga, \fa}(\bigcdot, t, x)\bigr]_r = \sum_{\tau_{\ga, \fa} \leq s \leq r} \Delta_s Y'_{\ga, \fa}(\bigcdot, t, x)\, \Delta_s \un{Y}'_{\ga, \fa}(\bigcdot, t, x),
\end{equation}
where $\Delta_s Y'_{\ga, \fa}(\bigcdot, t, y)$ is the jumps size of the martingale at time $s$. Moreover, the process 
\begin{equation}\label{eq:N-prime}
\fN'_{\ga, \fa}(r, t, x) := \bigl[Y'_{\ga, \fa}(\bigcdot, t, x), \un{Y}'_{\ga, \fa}(\bigcdot, t, x)\bigr]_r - \langle Y'_{\ga, \fa}(\bigcdot, t, x), \un{Y}'_{\ga, \fa}(\bigcdot, t, x)\rangle_r
\end{equation}
is a martingale for $r \in [\tau_{\ga, \fa}, t]$, where from \eqref{eq:M-gamma-bracket} we have 
\begin{equation*}
\langle Y'_{\ga, \fa}(\bigcdot, t, x), \un{Y}'_{\ga, \fa}(\bigcdot, t, x)\rangle_r = \eps^3 \sum_{y \in \Le} \int_{\tau_{\ga, \fa}}^r \un{P}^{\ga}_{t - s}(x-y) \widetilde{P}^{\ga}_{t - s}(x-y) \bC_{\ga, \fa} (s,y) \,\d s.
\end{equation*}
We denote $\fN'_{\ga, \fa}(t, x) = \fN'_{\ga, \fa}(t, t, x)$. Then we multiply \eqref{eq:X-prime-equation-P} and \eqref{eq:X-under-prime-equation}, to get 
\begin{align}
I_{\ga, \fa}(t, x) &= \bigl(P^\ga_{t - \tau_{\ga, \fa}} \un{X}_{\ga}\bigr)(\tau_{\ga, \fa}, x)\, X'_{\ga, \fa}(t, x) + \un{Y}'_{\ga, \fa}(t, x)\, \bigl(P^\ga_{t - \tau_{\ga, \fa}} X_{\ga}\bigr)(\tau_{\ga, \fa}, x) + \fN'_{\ga, \fa}(t, x) \nonumber\\
&\qquad + \biggl(\eps^3 \sum_{y \in \Le} \int_{\tau_{\ga, \fa}}^t \un{P}^{\ga}_{t - s}(x-y) \widetilde{P}^{\ga}_{t - s}(x-y) \bC_{\ga, \fa} (s,y) \,\d s - \un{\fC}_\ga (t-\tau_{\ga, \fa})\biggr). \label{eq:X-under-X-prime-expansion}
\end{align}
We denote the four terms on the right-hand side by $I_{\ga, \fa}^{(i)}(t,x)$, for $i = 1, \ldots, 4$, and we will bound them one by one.

Expanding the discrete kernel as in Appendix~\ref{sec:decompositions} and using the a priori bound provided by the stopping time \eqref{eq:tau-1}, we obtain from Lemma~\ref{lem:P-convolved-with-f}
\begin{equation*}
| (P^\ga_{t - \tau_{\ga, \fa}} X_{\ga})(\tau_{\ga, \fa}, x)| \lesssim \fa (t - \tau_{\ga, \fa})^{\eta/2}, \qquad | (P^\ga_{t - \tau_{\ga, \fa}}  \un{X}_{\ga})(\tau_{\ga, \fa}, x)| \lesssim \fa (t - \tau_{\ga, \fa})^{\eta/2}.
\end{equation*}
Then the first term in \eqref{eq:X-under-X-prime-expansion} we bound as 
\begin{equation*}
\E \biggl[ \sup_{t \in [\tau_{\ga, \fa}, T]} (t - \tau_{\ga, \fa})^{-\frac{\eta p}{2}} \bigl| I_{\ga, \fa}^{(1)}(t,x) \bigr|^p\biggr] \leq \fa^p \E \biggl[ \sup_{t \in [\tau_{\ga, \fa}, T]} \bigl| X'_{\ga, \fa}(t, x) \bigr|^{p}\biggr].
\end{equation*}
Applying Lemma~\ref{eq:X-prime-bound}, the preceding expression is bounded by a constant times $\fa^{2 p} \emezo^{\eta p}$. The term $I_{\ga, \fa}^{(2)}(t,x)$ can be bounded similarly. Indeed, $\un{Y}'_{\ga, \fa}$ coincides with $\un{X}'_{\ga, \fa}$, when the initial condition is $0$, and Lemma~\ref{eq:X-prime-bound} holds for $\un{X}'_{\ga, \fa}$ where $\un{\emezo}$ is used in place of $\emezo$. Hence, we have 
\begin{equation*}
\E \biggl[ \sup_{t \in [\tau_{\ga, \fa}, T]} (t - \tau_{\ga, \fa})^{-\frac{\eta p}{2}} \bigl| I_{\ga, \fa}^{(2)}(t,x) \bigr|^p\biggr] \lesssim \fa^{p} \un{\emezo}^{\eta p}.
\end{equation*}

To bound the third term in \eqref{eq:X-under-X-prime-expansion}, we use the Burkholder-Davis-Gundy inequality and get
\begin{equation}\label{eq:I-3-bound}
\E \biggl[ \sup_{t \in [\tau_{\ga, \fa}, T]} \bigl| I_{\ga, \fa}^{(3)}(t,x) \bigr|^p\biggr] \lesssim \biggl(\E \Bigl[ \bigl[ \fN'_{\ga, \fa}(\bigcdot, t, x) \bigr]_t\Bigr] \biggr)^{\frac{p}{2}},
\end{equation}
where the quadratic variation is computed for the martingale \eqref{eq:N-prime}. From the definition of the martingale, we get 
\begin{equation}\label{eq:N-prime-bracket}
\bigl[ \fN'_{\ga, \fa}(\bigcdot, t, x) \bigr]_t = \sum_{\tau_{\ga, \fa} \leq s \leq t} \bigl(\Delta_s \fN'_{\ga, \fa}(\bigcdot, t, x) \bigr)^2.
\end{equation}
Moreover, \eqref{eq:Y-jumps} yields $\Delta_s \fN'_{\ga, \fa}(\bigcdot, t, x) = \Delta_s Y'_{\ga, \fa}(\bigcdot, t, x)\, \Delta_s \un{Y}'_{\ga, \fa}(\bigcdot, t, x)$. Furthermore, the definitions \eqref{eq:Y-definitions} allow to bound the jumps of $Y'_{\ga, \fa}$ and $\un{Y}'_{\ga, \fa}$ in terms of jumps of $\M'_{\ga, \fa}$. Since the jumps size of the latter is bounded by $2 \ga^{-3}$ (as follows from the scaling \eqref{eq:martingales}) and almost surely $\M'_{\ga, \fa}(s, y)$ has a jump at a unique point $y$, we get 
\begin{equation*}
\bigl|\Delta_s Y'_{\ga, \fa}(\bigcdot, t, x)\bigr| \leq 2 \ga^{-3} \eps^3 \bigl\|\widetilde{P}^{\ga}_{t - s}\bigr\|_{L^\infty}, \qquad \bigl|\Delta_s \un{Y}'_{\ga, \fa}(\bigcdot, t, x)\bigr| \leq 2 \ga^{-3} \eps^3 \bigl\|\un{P}^{\ga}_{t - s}\bigr\|_{L^\infty}.
\end{equation*}
From Lemma~\ref{lem:tilde-P-bound} we have $\bigl\|\widetilde{P}^{\ga}_{t - s}\bigr\|_{L^\infty} \lesssim (t-s + \emezo^2)^{-3/2}$ and $\bigl\|\un{P}^{\ga}_{t - s}\bigr\|_{L^\infty} \lesssim (t-s + \un{\emezo}^2)^{-3/2}$. Using these bounds in \eqref{eq:N-prime-bracket} yields 
\begin{equation*}
\bigl[ \fN'_{\ga, \fa}(\bigcdot, t, x) \bigr]_t \lesssim \ga^{18} \sum_{\tau_{\ga, \fa} \leq s \leq t} (t-s + \un{\emezo}^2)^{-3} \mathbbm{1}_{\{ s: \, \M'_{\ga, \fa}(s, x) - \M'_{\ga, \fa}(s-, x) \neq 0 \}},
\end{equation*}
where $\mathbbm{1}$ is the indicator function and so the sum runs over the jump times of the martingales $\M'_{\ga, \fa}$. The moments of the number of jumps of the martingales are of order $\ga^{-6}$, and hence the $p$-th moment of the preceding expression is bounded by a constant times
\begin{equation*}
\ga^{12} \int_{\tau_{\ga, \fa}}^{t} (t-s + \un{\emezo}^2)^{-3} \d s \lesssim \ga^{12} \un{\emezo}^{-4} \lesssim \ga^{-4\un\kappa}.
\end{equation*}
Then the right-hand side of \eqref{eq:I-3-bound} is bounded by a constant multiple of $\ga^{-2 \un \kappa p}$.

It is left to bound the last term in \eqref{eq:X-under-X-prime-expansion}. Using \eqref{eq:bC} and \eqref{eq:C-gamma-function}, we have 
\begin{equation*}
I_{\ga, \fa}^{(4)}(t,x) = - \frac{2 \eps^6}{\alpha} \sum_{y \in \Le} \int_{\tau_{\ga, \fa}}^t \un{P}^{\ga}_{t - s}(x-y) \widetilde{P}^{\ga}_{t - s}(x-y) S'_{\ga, \fa} (s, y) X'_{\ga, \fa}(s, y) \,\d s.
\end{equation*}
Let $I_{\ga, \fa}^{(5)}(t,x)$ be defined by this formula, where we replace $S'_{\ga, \fa}$ by $\un{X}'_{\ga, \fa}$. From Lemma~\ref{lem:X-prime-bound} we have $|X'_{\ga, \fa}(s, y)| \lesssim \emezo^{\eta}$ and we have $|S'_{\ga, \fa}(s, y)| \lesssim \emezo^{-1}$. Then, is we replace the kernels $\un{P}^{\ga}$ and $\widetilde{P}^{\ga}$ in $I_{\ga, \fa}^{(5)}$ by $\un{\SK}^{\ga}$ and $\mywidetilde{\SK}^{\ga}$, we get an error of order $\emezo^{1 + \eta}$. Then Lemma~\ref{lem:Y-approximation} yields $|I_{\ga, \fa}^{(4)}(t,x) - I_{\ga, \fa}^{(5)}(t,x)| \lesssim \ga^{3 \eta}$ uniformly in $x$ and locally uniformly in $t$. To bound $I_{\ga, \fa}^{(5)}$ we write 
\begin{align*}
I_{\ga, \fa}^{(5)}(t,x) &= - \frac{2 \eps^6}{\alpha} \sum_{y \in \Le} \int_{\tau_{\ga, \fa}}^t \un{P}^{\ga}_{t - s}(x-y) \widetilde{P}^{\ga}_{t - s}(x-y) I_{\ga, \fa}(s, y) \,\d s \\
&\qquad + \frac{2 \eps^6}{\alpha} \sum_{y \in \Le} \int_{\tau_{\ga, \fa}}^t \un{P}^{\ga}_{t - s}(x-y) \widetilde{P}^{\ga}_{t - s}(x-y) \un{\fC}_\ga (s-\tau_{\ga, \fa}) \,\d s,
\end{align*}
and we denote these two terms by $I_{\ga, \fa}^{(6)}(t,x)$ and $I_{\ga, \fa}^{(7)}(t,x)$. Since $|\un{\fC}_\ga (s)| \lesssim \emezo^{-1}$ (what follows from Lemma~\ref{lem:renorm-constant-underline}), we get $|I_{\ga, \fa}^{(7)}(t,x)| \lesssim 1$.  Furthermore, we have $|I_{\ga, \fa}^{(6)}(t,x)| \lesssim \emezo \sup_{s \in [\tau_{\ga, \fa}, t]}\| I_{\ga, \fa}(s)\|_{L^\infty}$. 

Combining all the previous bounds, we get
\begin{align*}
\E \biggl[ \sup_{t \in [\tau_{\ga, \fa}, T]} (t - \tau_{\ga, \fa})^{-\frac{\eta p}{2}} \bigl\| I_{\ga, \fa}(t) \bigr\|_{L^\infty}^p\biggr] &\lesssim \fa^{2 p} \emezo^{\eta p} + \fa^{p} \un{\emezo}^{\eta p} + \ga^{-2 \un \kappa p} + \ga^{3 \eta p}\\
&\qquad + \emezo^p \E \biggl[ \sup_{t \in [\tau_{\ga, \fa}, T]} (t - \tau_{\ga, \fa})^{-\frac{\eta p}{2}} \bigl\| I_{\ga, \fa}(t) \bigr\|_{L^\infty}^p\biggr].
\end{align*}
Taking $\emezo$ small enough, we get the required bound \eqref{eq:X2-prime-bound}.
\end{proof}

\section{Moment bounds for the discrete models}
\label{sec:convergence-of-models}

Let $Z^{\ga, \fa}_{\lift}$ be the discrete model defined in Section \ref{sec:discreteRegStruct}. In this section, we prove that this model is bounded uniformly in $\ga$. Moreover, we introduce a new discrete model $Z^{\ga, \de, \fa}_{\lift}$, defined as $Z^{\ga, \fa}_{\lift}$ but via mollified martingales. Then we show that the distance between these two models vanishes as $\de \to 0$, uniformly in $\ga$.

Let $\varrho : \R^4 \to \R$ be a symmetric smooth function, supported on the ball of radius $1$ (with respect to the parabolic distance $\| \bigcdot \|_\s$) and satisfying $\int_{\R^4} \varrho(z) \d z = 1$. For any $\de \in (0,1)$ we define its rescaling
\begin{equation}\label{eq:rho}
\varrho_\de(t, x) := \frac{1}{\de^{5}} \varrho \Bigl( \frac{t}{\de^2}, \frac{x}{\de}\Bigr). 
\end{equation}
We need to modify this function in a way that its integral over $D_\eps$ becomes $1$. For this, we approximate the function by its local averages as
\begin{equation}\label{eq:rho-gamma}
\varrho_{\ga, \de}(t, x) := \eps^{-3} \int_{y \in \R^3: |y - x|_\infty \leq \eps/2} \varrho_\de(t, y) \d y,
\end{equation}
which satisfies $\int_{D_\eps} \varrho_{\ga, \de}(z) \d z = 1$. We regularise the martingales in the following way:
\begin{equation}\label{eq:xi-gamma-delta}
	\xi_{\ga, \de, \fa}(t,x) := \frac{1}{\sqrt 2} \eps^{3} \sum_{y \in \Lattice} \int_{\R} \varrho_{\ga, \de} (t - s,x - y)\, \d \M_{\ga, \fa}(s, y).
\end{equation}
Then the process $\xi_{\ga, \de, \fa}(t,x)$ is defined on $(t,x) \in \R \times \Le$, but it is not a martingale anymore. On the other hand, a convolution with this process can be interpreted as a stochastic integral. For example, a convolution with the kernel $\mywidetilde{\SK}^\ga$ may be written as 
\begin{equation*}
\bigl(\mywidetilde{\SK}^\ga \star_\eps \xi_{\ga, \de, \fa}\bigr)(t,x) = \frac{1}{\sqrt 2} \eps^3 \sum_{y \in \Lattice} \int_{\R} \mywidetilde{\SK}^{\ga, \de}_{t-s}(x - y)\, \d \M_{\ga, \fa}(s, y),
\end{equation*}
where $\star_\eps$ is the convolution on $D_\eps$ and $\mywidetilde{\SK}^{\ga, \de} := \mywidetilde{\SK}^{\ga} \star_\eps \varrho_{\ga, \de}$. Then we can easily compare the two kernels as 
\begin{equation*}
\bigl(\mywidetilde{\SK}^{\ga, \de} - \mywidetilde{\SK}^{\ga}\bigr)(z) = \int_{D_\eps} \mywidetilde{\SK}^{\ga}(z - \bar z)\bigl(\varrho_{\ga, \de}(\bar z) - 1\bigr) \d \bar z,
\end{equation*}
which is the main reason to mollify the noise using the function \eqref{eq:rho}.

Using $\xi_{\ga, \de, \fa}$, we make the following definitions 
\begin{equation*}
\bigl(\PPi^{\ga, \de, \fa} \blueXi\bigr)(z) = \xi_{\ga, \de, \fa}(z), \qquad \bigl(\PPi^{\ga, \de, \fa} \<1b>\bigr)(z) = \bigl(\mywidetilde{\SK}^\ga \star_\eps \xi_{\ga, \de, \fa}\bigr)(z).
\end{equation*}
After that we define the linear map $\PPi^{\ga, \de, \fa}$ on $\CT$ by the same recursive definitions as in Section~\ref{sec:PPi}, but using the following renormalisation constants in place of \eqref{eq:renorm-constant1}, \eqref{eq:renorm-constant3} and \eqref{eq:renorm-constant2} respectively:
\begin{equation}\label{eq:convolved_constants}
\begin{aligned}
\fc_{\ga, \de} &:= \int_{D_\eps} \mywidetilde{\SK}^{\ga, \de}(z)^2\, \d z, \\
\fc_{\ga, \de}' &:= -\be \varkappa_{\ga, 3} \ga^6 \un{\fC}_\ga \fc_{\ga, \de}, \\
\fc_{\ga, \de}'' &:= 2 \int_{D_\eps} \int_{D_\eps} \int_{D_\eps} \mywidetilde{\SK}^\ga(z) \mywidetilde{\SK}^{\ga, \de}(z_1) \mywidetilde{\SK}^{\ga, \de}(z_2) \mywidetilde{\SK}^{\ga, \de}(z_1 - z) \mywidetilde{\SK}^{\ga, \de}(z_2 - z)\, \d z\, \d z_1\, \d z_2.
\end{aligned}
\end{equation}
As we did in Section~\ref{sec:model-lift}, we define a discrete model $Z^{\ga, \de, \fa}_{\lift} = (\Pi^{\ga, \de, \fa}, \Gamma^{\ga, \de, \fa})$ from the map $\PPi^{\ga, \de, \fa}$. In the following proposition we provide moment bounds for this model.

\begin{proposition}\label{prop:models-converge}
Let the constants $\kappa$ and $\un \kappa$, used in \eqref{eqs:hom} and \eqref{eq:tau-2}, satisfy $\kappa \geq \un{\kappa}$. Then for the discrete models $Z^{\ga, \fa}_{\lift}$ and $Z^{\ga, \de, \fa}_{\lift}$, there exist $\ga_0 > 0$ and $\theta > 0$ for which the following holds: for any $p \geq 1$ and $T > 0$ there is $C > 0$ such that
\begin{equation}\label{eq:prop:models-converge}
\sup_{\ga \in (0, \ga_0)} \E \Bigl[\bigl( \$ Z^{\ga, \fa}_{\lift}\$_{T}^{(\emezo)}\bigr)^p\Bigr] \leq C, \qquad \sup_{\ga \in (0, \ga_0)} \E \Bigl[\bigl( \$ Z^{\ga, \fa}_{\lift}; Z^{\ga, \de, \fa}_{\lift}\$_T^{(\emezo)}\bigr)^p\Bigr] \leq C \de^{\theta p},
\end{equation}
for any $\de \in (0,1)$. Here, we use the metrics for the discrete models, defined in Remark~\ref{rem:model-no-K}. 
\end{proposition}

We prove this proposition in Section~\ref{sec:convergence-of-models-proof}. For this, we use the framework developed in \cite{Martingales}, which provides moment bounds on multiple stochastic integrals with respect to a quiet general class of martingales. We showed in Section~\ref{sec:martingales-properties} that the martingales $\M_{\ga, \fa}$, introduced in Section~\ref{sec:a-priori}, have the required properties.

\subsection{Bounds on the discrete model} 
\label{sec:Models_bounds}

The basis elements of the regularity structure are listed in Tables~\ref{tab:symbols-cont} and \ref{tab:symbols}, and in this section we are going to prove bounds only on the map $\Pi^{\ga, \fa}$ from the discrete model $Z^{\ga, \fa}_{\lift}$ on the basis elements with negative homogeneities, which do not contain the symbols $\CE$ and $\X_i$. More precisely, we consider the set
\begin{equation*}
\bar{\CW} = \bigl\{\<1b>, \<2b>, \<3b>,  \<22b>, \<31b>, \<32b>, \<4b>, \<5b>\bigr\},
\end{equation*}
and prove the following bounds for its elements. We use in the statement of this proposition and in its proof the notation from Section~\ref{sec:DiscreteModels}.

\begin{proposition}\label{prop:Pi-bounds}
Let the constants $\kappa$ and $\un \kappa$, used in \eqref{eqs:hom} and \eqref{eq:tau-2}, satisfy $\kappa \geq \un{\kappa}$. Then there are constants $\bar \kappa > 0$, $\ga_0 > 0$ and $\theta > 0$, such that for any $\tau \in \bar{\CW}$, $p \geq 1$ and $T > 0$ there is $C > 0$ for which we have the bounds
	\begin{align}
		\Bigl(\E \bigl| \iota_\eps \bigl(\Pi^{\ga, \fa}_z \tau\bigr)(\varphi^\lambda_z)\bigr|^p\Bigr)^{\frac{1}{p}} &\leq C (\lambda \vee \emezo)^{|\tau| + \bar \kappa}, \label{eq:model_bound} \\
		\Bigl(\E \bigl| \iota_\eps \bigl(\Pi^{\ga, \fa}_z \tau - \Pi^{\ga, \de, \fa}_z \tau\bigr) (\varphi^\lambda_z) \bigr|^p\Bigr)^{\frac{1}{p}} &\leq C \de^{\theta} (\lambda \vee \emezo)^{|\tau| + \bar \kappa - \theta},\label{eq:model_bound-delta}
	\end{align}
	uniformly in $z \in D_\eps$, $\lambda \in (0,1]$, $\varphi \in \CB^2_\s$ and $\ga \in (0, \ga_0)$. 
\end{proposition}

\noindent The rest of this section is devoted to the proof of this result. We are going to prove the bounds \eqref{eq:model_bound}-\eqref{eq:model_bound-delta} for any $p$ sufficiently large, and the bound for any $p \geq 1$ follow then by H\"{o}lder inequality. 

For every symbol $\tau \in \bar{\CW}$, we use the definition of the discrete model in Section~\ref{sec:model-lift} and the expansion \cite[Eq.~2.14]{Martingales} to write $\iota_\eps \bigl(\Pi^{\ga, \fa}_{z}\tau\bigr)(\phi_z^\lambda)$ as a sum of terms of the form
\begin{align}\label{eq:terms-in-Pi}
	&\int_{D_\eps} \phi_z^\lambda(\bar z) \biggl(\int_{D_\eps^{ n}} F_{\bar z} (z_1, \ldots, z_n) \, \d \bM^{ n}_{\ga, \fa} (z_1, \ldots, z_n)\biggr) \d \bar z \\
	&\hspace{3cm} = \int_{D_\eps^{ n}} \biggl( \int_{D_\eps} \phi_z^\lambda(\bar z) F_{\bar z} (z_1, \ldots, z_n) \, \d \bar z \biggr) \d \bM^{ n}_{\ga, \fa} (z_1, \ldots, z_n), \nonumber
\end{align}
where the measure $\bM^{ n}_{\ga, \fa}$ is defined in Section~2.1 in \cite{Martingales} for the martingales $\M_{\ga, \fa}$, and a function $F$ of $n$ space-time variables. Similarly, we write $\iota_\eps(\Pi_{z}^{\ga, \de, \fa}\tau)(\phi_z^\lambda)$ as a sum of terms of the form
\begin{align}\label{eq:terms-in-Pi-de}
	&\int_{D_\eps} \phi_z^\lambda(\bar z) \biggl(\int_{D_\eps^{ n}} F_{\bar z} (z_1, \ldots, z_n) \, \d \bM^{ n}_{\ga, \fa, (\de)} (z_1, \ldots, z_n)\biggr) \d \bar z \\ 
	&\hspace{3cm} = \int_{D_\eps^{ n}} \biggl( \int_{D_\eps} \phi_z^\lambda(\bar z) \big( F_{\bar z} \star_\eps \rho_{\ga, \de} \big) (z_1, \ldots, z_n) \, \d \bar z \biggr) \d \bM^{ n}_{\ga, \fa} (z_1, \ldots, z_n), \nonumber
\end{align}
where $\bM^{ n}_{\ga, \fa, (\de)} (z_1, \ldots, z_n)$ stays for the product measure associated to the regularised martingales $\xi_{\ga, \de, \fa}$, defined in \eqref{eq:xi-gamma-delta}.  The functions $F$ will be typically defined in terms of the singular part $\mywidetilde{\SK}^\ga$ of the decomposition $\widetilde{G}^\ga = \mywidetilde{\SK}^\ga + \mywidetilde{\SR}^\ga$ done in Appendix~\ref{sec:decompositions}, or in terms of the function $\mywidetilde{\SK}^{\ga, \de} := \varrho_{\ga, \de} \star_\eps \mywidetilde{\SK}^\ga$ where $\varrho_{\ga, \de}$ is the mollifier from \eqref{eq:xi-gamma-delta}.

To bound the terms \eqref{eq:terms-in-Pi} and their difference with those in \eqref{eq:terms-in-Pi-de}, we are going to use Corollary~5.6 in \cite{Martingales}. For this, it is convenient to use graphical notation to represent the function $F$ and integrals, where nodes represent variables and arrows represent kernels. In what follows, the vertex ``\,\tikz[baseline=-3] \node [root] {};\,'' labelled with $z$ represents the basis point $z \in D_\eps$; the arrow ``\,\tikz[baseline=-0.1cm] \draw[testfcn] (1,0) to (0,0);\,'' represents a test function $\phi^\lambda_z$; the arrow ``\,\tikz[baseline=-0.1cm] \draw[keps] (0,0) to (1,0);\,'' represents the discrete kernel $\mywidetilde{\SK}^\ga$, and we will write two labels $(a_e, r_e)$ on this arrow, which correspond to the labels on graphs as described in \cite[Sec.~5]{Martingales}. More precisely, since the kernel $\mywidetilde{\SK}^\ga$ satisfies \cite[Assum.~4]{Martingales} with $a_e=3$ (see Lemma~\ref{lem:Pgt}), we draw ``\,\tikz[baseline=-0.1cm] \draw[keps] (0,0) to node[labl,pos=0.45] {\tiny 3,0} (1,0);\,''. 

Each variable $z_i$, integrated with respect to the measure $\bM^{ n}_{\ga, \fa}$ with $n \geq 2$ is denoted by a node ``\,\tikz[baseline=-3] \node [var_very_blue] {};\,''; the variable integrated with respect to the martingale $\M_{\ga, \fa}$ we denote by ``\,\tikz[baseline=-3] \node [var_blue] {};\,''. By the node ``\,\tikz[baseline=-3] \node [dot] {};\,'' we denote a variable integrated out in $D_\eps$.

\subsubsection{The element $\tau = \protect\<1b>$} 

We represent the function $\Pi^{\ga, \fa}_{z}\tau$, defined in \eqref{eq:lift-hermite}, diagrammatically as
\begin{equation*}
\iota_\eps \bigl(\Pi^{\ga, \fa}_{z}\tau\bigr)(\phi_z^\lambda)
\;=\; 
\begin{tikzpicture}[scale=0.35, baseline=0cm]
	\node at (0.9,0.2)  [root] (root) {};
	\node at (0.9,0.2) [rootlab] {$z$};
	\node at (-2.3, 0.2)  [dot] (int) {};
	\node at (-5.5,0.2)  [var_blue] (left) {};	
	\draw[testfcn] (int) to (root);	
	\draw[keps] (left) to node[labl,pos=0.45] {\tiny 3,0} (int);
\end{tikzpicture}\;.
\end{equation*}
This diagram is in the form \eqref{eq:terms-in-Pi} with $n=1$, where in this case, in the inner integral, we have the generalised convolution $\CK^{\lambda, \emezo}_{\CCG, z}$ given by (as in \cite[Eq.~5.13]{Martingales})
\begin{equation*}
\CK^{\lambda, \emezo}_{\CCG, z}(z^{\var}) = \int_{D_\eps} \!\!\phi_z^\lambda(\bar z)\, \mywidetilde{\SK}^\ga(\bar z - z^{\var})\, \d \bar z.
\end{equation*}
One can check that \cite[Assum.~3]{Martingales} is satisfied for this diagram with a trivial contraction and the bound \cite[Eq.~5.16]{Martingales} holds with the sets $\tilde \CCV_{\!\var} = \Gamma = \{ 1 \}$ and labeling $\L = \{ \nil \}$. The set $B$ in this bound has to be $\emptyset$, while $A$ might be either $\{ 1 \}$ or $\emptyset$. From the diagram we see that $| \tilde \CCV_{\!\var}| = 1$ and $|\hat \CCV_{\!\bar\star} \setminus \hat \CCV^\uparrow_{\!\star}| = 1$; therefore, the value of the constant $\nu_\gamma$ in \cite[Eq.~5.15]{Martingales} is $-\frac{1}{2}$. Applying \cite[Cor.~5.6]{Martingales}, we get that, for any $\bar\kappa > 0$ and any $p \geq 2$ large enough:
\begin{equation*}
	\Bigl( \E \bigl| \iota_\eps \bigl(\Pi^{\ga, \fa}_{z}\tau\bigr)(\phi_z^\lambda) \bigr|^p\Bigr)^{\frac{1}{p}} \lesssim (\lambda \vee \emezo)^{- \frac{1}{2}} \Bigl( 1 + \eps^{\frac{9}{4} - \bar\kappa} \emezo^{-\frac52} \Bigr).
\end{equation*}
Since $\emezo \approx \ga^3$ and $\eps \approx \ga^4$, this expression is bounded by a multiple of $(\lambda \vee \emezo)^{- \frac{1}{2}}$ as required in \eqref{eq:model_bound} (recall that $|\tau| = -\frac{1}{2}-\kappa$). 

In what follows, we use the notation and terminology from \cite[Sec.~5]{Martingales} in the same way as we did for this element $\tau$, and we prefer not to make references every time.

\subsubsection{The element $\tau = \protect\<2b>$}
\label{sec:second-symbol}

Using the definition \eqref{eq:lift-hermite} and the expansion \cite[Eq.~2.11]{Martingales}, the function $\Pi^{\ga, \fa}_{z}\tau$ can be represented by the diagrams
\begin{equation}\label{eq:Pi2}
\iota_\eps \bigl(\Pi^{\ga, \fa}_{z}\tau\bigr)(\phi_z^\lambda)
\;=\; 
\begin{tikzpicture}[scale=0.35, baseline=0cm]
	\node at (0,-2.2)  [root] (root) {};
	\node at (0,-2.2) [rootlab] {$z$};
	\node at (0,-2.5) {$$};
	\node at (0,0)  [dot] (int) {};
	\node at (-1.5,2.5)  [var_blue] (left) {};
	\node at (1.5,2.5)  [var_blue] (right) {};
	
	\draw[testfcn] (int) to (root);
	
	\draw[keps] (left) to node[labl,pos=0.45] {\tiny 3,0} (int);
	\draw[keps] (right) to node[labl,pos=0.45] {\tiny 3,0} (int);
\end{tikzpicture}
\; +\;
\begin{tikzpicture}[scale=0.35, baseline=0cm]
	\node at (0,-2.2)  [root] (root) {};
	\node at (0,-2.2) [rootlab] {$z$};
	\node at (0,0)  [dot] (int) {};
	\node at (0,2.5)  [var_very_blue] (left) {};
	
	\draw[testfcn] (int) to (root);
	
	\draw[keps] (left) to[bend left=60] node[labl,pos=0.45] {\tiny 3,0} (int);
	\draw[keps] (left) to[bend left=-60] node[labl,pos=0.45] {\tiny 3,0} (int);
\end{tikzpicture}
\; - \; (\fc_\ga + \fc_\ga' ) \, 
\begin{tikzpicture}[scale=0.35, baseline=-0.5cm]
	\node at (0,-2.2)  [root] (root) {};
	\node at (0,-2.2) [rootlab] {$z$};
	\node at (0,0)  [dot] (int) {};
	
	\draw[testfcn] (int) to (root);
\end{tikzpicture}\;.
\end{equation}
Let us denote by ``\,\tikz[baseline=-3] \node [var_red_triangle] {};\,'' the integration against the family of martingales given by the predictable quadratic variation $x \mapsto \langle \M_{\ga, \fa}(x) \rangle$, and by ``\,\tikz[baseline=-3] \node [var_red_square] {};\,'' the integration in the family of martingales $x \mapsto [\M_{\ga, \fa}(x)] - \langle \M_{\ga, \fa}(x) \rangle$. Then we can write \eqref{eq:Pi2} as
\begin{equation}\label{eq:Pi2-new}
\iota_\eps \bigl(\Pi^{\ga, \fa}_{z}\tau\bigr)(\phi_z^\lambda)
\;=\;
\begin{tikzpicture}[scale=0.35, baseline=0cm]
	\node at (0,-2.2)  [root] (root) {};
	\node at (0,-2.2) [rootlab] {$z$};
	\node at (0,-2.5) {$$};
	\node at (0,0)  [dot] (int) {};
	\node at (-1.5,2.5)  [var_blue] (left) {};
	\node at (1.5,2.5)  [var_blue] (right) {};
	
	\draw[testfcn] (int) to (root);
	
	\draw[keps] (left) to node[labl,pos=0.45] {\tiny 3,0} (int);
	\draw[keps] (right) to node[labl,pos=0.45] {\tiny 3,0} (int);
\end{tikzpicture}
\; +\;
\begin{tikzpicture}[scale=0.35, baseline=0cm]
	\node at (0,-2.2)  [root] (root) {};
	\node at (0,-2.2) [rootlab] {$z$};
	\node at (0,0)  [dot] (int) {};
	\node at (0,2.5)  [var_red_square] (left) {};
	
	\draw[testfcn] (int) to (root);
	
	\draw[keps] (left) to[bend left=60] node[labl,pos=0.45] {\tiny 3,0} (int);
	\draw[keps] (left) to[bend left=-60] node[labl,pos=0.45] {\tiny 3,0} (int);
\end{tikzpicture}
\; +\; \left(
\begin{tikzpicture}[scale=0.35, baseline=0cm]
	\node at (0,-2.2)  [root] (root) {};
	\node at (0,-2.2) [rootlab] {$z$};
	\node at (0,0)  [dot] (int) {};
	\node at (0,2.5)  [var_red_triangle] (left) {};
	
	\draw[testfcn] (int) to (root);
	
	\draw[keps] (left) to[bend left=60] node[labl,pos=0.45] {\tiny 3,0} (int);
	\draw[keps] (left) to[bend left=-60] node[labl,pos=0.45] {\tiny 3,0} (int);
\end{tikzpicture}
\; - \; (\fc_\ga + \fc_\ga' ) \, 
\begin{tikzpicture}[scale=0.35, baseline=-0.5cm]
	\node at (0,-2.2)  [root] (root) {};
	\node at (0,-2.2) [rootlab] {$z$};
	\node at (0,0)  [dot] (int) {};
	
	\draw[testfcn] (int) to (root);
\end{tikzpicture} \right).
\end{equation} 
Let us denote the first two of these diagrams by $\iota_\eps \bigl(\Pi_{z}^{\ga, 1}\tau\bigr)(\phi_z^\lambda)$ and $\iota_\eps \bigl(\Pi_{z}^{\ga, 2}\tau\bigr)(\phi_z^\lambda)$ respectively, and let $\iota_\eps \bigl(\Pi_{z}^{\ga, 3}\tau\bigr)(\phi_z^\lambda)$ denote the expression in the brackets in \eqref{eq:Pi2-new}.

Let us analyse the first diagram in \eqref{eq:Pi2-new}. \cite[Assum.~3]{Martingales} is satisfied for it with a trivial contraction, and the bound \cite[Eq.~5.16]{Martingales} holds with the sets $\tilde \CCV_{\!\var} = \Gamma = \{1, 2\}$ and labeling $\L = \{ \nil, \nil \}$. The sets $B$ in \cite[Eq.~5.16]{Martingales} needs to be $\emptyset$, while $A$ can be $\emptyset$, $\{1\}$, $\{2\}$ or $\{1, 2\}$. Furthermore, we have $| \hat \CCV_{\!\var}| = 2$ and $|\hat \CCV_{\!\bar\star} \setminus \hat \CCV^\uparrow_{\!\star}| = 2$ and the value of the constant $\nu_\gamma$ in \cite[Eq.~5.16]{Martingales} is $-1$. Applying \cite[Cor.~5.6]{Martingales} to this diagram, we get for any $\bar \kappa > 0$ and for any $p \geq 2$ large enough 
\begin{equation*}
 \Bigl(\E \bigl|\iota_\eps \bigl(\Pi_{z}^{\ga, 1}\tau\bigr)(\phi_z^\lambda)\bigr|^p\Bigr)^{\frac{1}{p}} \lesssim (\lambda \vee \emezo)^{-1} \Bigl( 1 + \eps^{\frac94 - \bar \kappa} \emezo^{-\frac52} + \eps^{\frac92 - \bar \kappa} \emezo^{-5} \Bigr).
\end{equation*}
Recalling that $|\tau| = -1-2\kappa$, we get the required bound \eqref{eq:model_bound}.

For the second diagram in \eqref{eq:Pi2-new} we have $\tilde \CCV_{\!\var} = \{1\}$, $\Gamma = \emptyset$ and the labeling $\L = \{ \diamond \}$. However, the graph does not satisfy \cite[Assum.~3]{Martingales}. To resolve this problem, we note that multiplication of a kernel by $\emezo^{3-a}$ with $a > 0$ ``improves'' its regularity by $3-a$, meaning that the singularity of the kernel now diverges like $\emezo^a$ instead of like $\emezo^3$. Then for $0 < a < \frac{5}{2}$ we can write 
\begin{equation}\label{eq:renorm-of-2}
\iota_\eps \bigl(\Pi_{z}^{\ga, 2}\tau\bigr)(\phi_z^\lambda) = \emezo^{2 (a - 3)} 
\begin{tikzpicture}[scale=0.35, baseline=0cm]
	\node at (0,-2.2)  [root] (root) {};
	\node at (0,-2.2) [rootlab] {$z$};
	\node at (0,0)  [dot] (int) {};
	\node at (0,2.5)  [var_red_square] (left) {};
	
	\draw[testfcn] (int) to (root);
	
	\draw[keps] (left) to[bend left=60] node[labl,pos=0.45] {\tiny a,0} (int);
	\draw[keps] (left) to[bend left=-60] node[labl,pos=0.45] {\tiny a,0} (int);
\end{tikzpicture}\;,
\end{equation}
and \cite[Assum.~3]{Martingales} is satisfied. Then for any $\bar \kappa > 0$ and any $p \geq 2$ large enough \cite[Cor.~5.6]{Martingales} yields
\begin{equation*}
	\Bigl(\E \bigl|\iota_\eps \bigl(\Pi_{z}^{\ga, 2}\tau\bigr)(\phi_z^\lambda)\bigr|^p\Bigr)^{\frac{1}{p}} \lesssim \emezo^{2 (a - 3)} (\lambda \vee \emezo)^{\frac{5}{2} - 2 a} \Bigl( \eps^{\frac94} + \eps^{\frac92 - \bar \kappa} \emezo^{-\frac{5}{2}} \Bigr).
\end{equation*}
For $\frac{3}{2} < a \leq \frac{7}{4}$ and $\bar \kappa > 0$ small enough the right-hand side is bounded by $c_\gamma (\lambda \vee \emezo)^{-1}$, where $c_\gamma$ vanishes as $\ga \to 0$.

The term $\iota_\eps \bigl(\Pi_{z}^{\ga, 3}\tau\bigr)(\phi_z^\lambda)$ requires a more complicated analysis. Using the quadratic covariation \eqref{eq:bC} and the definition of the renormalisation constants \eqref{eq:renorm-constant1} and \eqref{eq:renorm-constant3}, we can write
\begin{align}
&\iota_\eps \bigl(\Pi_{z}^{\ga, 3}\tau\bigr)(\phi_z^\lambda) = \frac{1}{2} \int_{D_\eps} \phi^\lambda_z (\bar z) \biggl( \eps^3 \sum_{\tilde{y} \in \Lattice} \int_0^\infty \mywidetilde{\SK}^\ga_{\bar t - \tilde{s}}(\bar x - \tilde{y})^2 \bigl(\bC_{\ga, \fa}(\tilde{s}, \tilde{y}) - 2 + 2 \be \varkappa_{\ga, 3} \ga^6 \un{\fC}_\ga\bigr) \d \tilde{s} \biggr) \d \bar z \nonumber \\
&\quad + \frac{1}{2} \int_{D_\eps} \phi^\lambda_z (\bar z) \biggl( \eps^3 \sum_{\tilde{y} \in \Lattice} \int_{-\infty}^0 \mywidetilde{\SK}^\ga_{\bar t - \tilde{s}}(\bar x - \tilde{y})^2 \bigl(\widetilde{\bC}_{\ga, \fa}(-\tilde{s}, \tilde{y}) - 2 + 2 \be \varkappa_{\ga, 3} \ga^6 \un{\fC}_\ga\bigr) \d \tilde{s} \biggr) \d \bar z \label{eq:Pi-3}
\end{align}
for $\bar z = (\bar t, \bar x)$ and where $\widetilde{\bC}_{\ga, \fa}$ is the bracket process \eqref{eq:bC} for the martingale $\widetilde{\M}_{\ga, \fa}$ used in \eqref{eq:martingale-extension}. This is the definition \eqref{eq:martingale-extension} which requires us to consider the two integrals: for positive and negative times. Since the two terms in \eqref{eq:Pi-3} are bounded in the same way, we will derive below only a bound on the first term.

Using the rescaled process $S_\ga$, defined in \eqref{eq:S-def}, from formula \eqref{eq:bC} we then get
\begin{equation*}
\bC_{\ga, \fa}(\tilde{s}, \tilde{y}) - 2 = 
\begin{cases}
- 2 \varkappa_{\ga, 3} \ga^3 S_\ga (\tilde{s}, \tilde{y}) \tanh \bigl( \be \delta X_\ga(\tilde{s}, \tilde{y}) \bigr) + 2 (1 - \varkappa_{\ga, 3}) &\text{for}~~ \tilde{s} < \tau_{\ga, \fa}, \\
- 2 \varkappa_{\ga, 3} \ga^6 S'_{\ga, \fa} (\tilde{s}, \tilde{y}) X'_{\ga, \fa}(\tilde{s}, \tilde{y}) + 2 (1 - \varkappa_{\ga, 3}) &\text{for}~~ \tilde{s} \geq \tau_{\ga, \fa}.
\end{cases}
\end{equation*}
From \eqref{eq:c-gamma-2} we have $1 - \varkappa_{\ga, 3} = \CO(\ga^{4})$. Moreover, the function $\tanh$ can be approximated by its first-order Taylor polynomial: $\tanh(x) = x + \CO(x^3)$, and \eqref{eq:X-apriori} yields $\| X_\ga(\cdot) \|_{L^\infty} \lesssim \emezo^{\eta}$ almost surely. Hence, $|\tanh \bigl( \be \delta X_\ga(\tilde{s}, \tilde{y}) \bigr) - \be \delta X_\ga(\tilde{s}, \tilde{y})| \lesssim \delta^3 \emezo^{3\eta} \lesssim \ga^{9 (1 + \eta)}$ almost surely uniformly in $\tilde{y}$ and $\tilde{s} < \tau_{\ga, \fa}$. Then the preceding expression equals
\begin{equation}\label{eq:bC-new}
\bC_{\ga, \fa}(\tilde{s}, \tilde{y}) - 2 = 
\begin{cases}
- 2 \varkappa_{\ga, 3} \be \ga^6 S_\ga (\tilde{s}, \tilde{y}) X_\ga(\tilde{s}, \tilde{y}) + \Err_{\ga, \fa}(\tilde{s}, \tilde{y}) &\text{for}~~ \tilde{s} < \tau_{\ga, \fa}, \\
- 2 \varkappa_{\ga, 3} \ga^6 S'_{\ga, \fa} (\tilde{s}, \tilde{y}) X'_{\ga, \fa}(\tilde{s}, \tilde{y}) + \Err_{\ga, \fa}'(\tilde{s}, \tilde{y}) &\text{for}~~ \tilde{s} \geq \tau_{\ga, \fa},
\end{cases}
\end{equation}
where the error terms are almost surely uniformly bounded on the respective time intervals by $|\Err_{\ga, \fa}(\tilde{s}, \tilde{y})| \lesssim \ga^{9 (1 + \eta) \wedge 4}$ and $|\Err_{\ga, \fa}'(\tilde{s}, \tilde{y})| \lesssim \ga^{4}$. Using \eqref{eq:bC-new}, we then write the first term in \eqref{eq:Pi-3} as
\begin{align}
	&-\be \varkappa_{\ga, 3} \ga^6 \int_{D_\eps} \phi^\lambda_z (\bar z) \biggl( \eps^3 \sum_{\tilde{y} \in \Lattice} \int_{0}^{\tau_{\ga, \fa}} \!\!\!\mywidetilde{\SK}^\ga_{\bar t-\tilde{s}}(\bar x-\tilde{y})^2 \bigl(S_\ga (\tilde{s}, \tilde{y}) X_\ga(\tilde{s}, \tilde{y}) - \un{\fC}_\ga\bigr) \d \tilde{s} \biggr) \d \bar z \nonumber\\
	&\qquad -\be \varkappa_{\ga, 3} \ga^6 \int_{D_\eps} \phi^\lambda_z (\bar z) \biggl( \eps^3 \sum_{\tilde{y} \in \Lattice} \int_{\tau_{\ga, \fa}}^\infty \!\!\!\mywidetilde{\SK}^\ga_{\bar t-\tilde{s}}(\bar x-\tilde{y})^2 \bigl(S'_{\ga, \fa} (\tilde{s}, \tilde{y}) X'_{\ga, \fa}(\tilde{s}, \tilde{y}) - \un{\fC}_\ga\bigr) \d \tilde{s} \biggr) \d \bar z \nonumber \\
	&\qquad + \Err^\lambda_{\ga, \fa}(z), \label{eq:renorm-of-2-last}
\end{align}
where $|\Err^\lambda_{\ga, \fa}(z)| \lesssim \ga^{6 + 9 \eta}$ almost surely, uniformly in $\ga \in (0,1]$ and $z \in D_\eps$. In the bound on the error term we used the bounds on the error terms in \eqref{eq:bC-new}, the assumption $-\frac{4}{7} < \eta < -\frac{1}{2}$ in Theorem~\ref{thm:main}, and the bound $\int_{D_\eps} \mywidetilde{\SK}^\ga(z)^2\, \d z \lesssim \emezo^{-1}$ which follows from Lemma~\ref{lem:renorm-constants} and the definition \eqref{eq:renorm-constant1}. We note that the assumptions on $\eta$ imply $6 + 9 \eta > 0$, and hence all moments of the error term $\Err^\lambda_{\ga, \fa}(z)$ vanish as $\ga \to 0$.

We denote the first two terms in \eqref{eq:renorm-of-2-last} by $\iota_\eps \bigl(\Pi_{z}^{\ga, 4}\tau\bigr)(\phi_z^\lambda)$ and $\iota_\eps \bigl(\Pi_{z}^{\ga, 5}\tau\bigr)(\phi_z^\lambda)$ respectively, and we start with bounding the first of them.

We first show that the rescaled spin field $S_{\ga}$ can be replaced in this expression by its local average. After that we can work with the product of two spin fields in \eqref{eq:renorm-of-2-last} similarly to how we work with $X_\ga^2$. We can now write
\begin{align}
	&\iota_\eps \bigl(\Pi_{z}^{\ga, 4}\tau\bigr)(\phi_z^\lambda) \label{eq:tree2-2-withXX} \\
	& = - \be \varkappa_{\ga, 3} \ga^6 \int_{D_\eps} \phi^\lambda_z (\bar z) \biggl( \eps^3 \sum_{\tilde{y} \in \Lattice} \int_{0}^{\tau_{\ga, \fa}} \!\!\!\mywidetilde{\SK}^\ga_{\bar t-\tilde{s}}(\bar x-\tilde{y})^2 \bigl(\un{X}_\ga(\tilde{s}, \tilde{y}) X_\ga(\tilde{s}, \tilde{y}) - \un{\fC}_\ga\bigr) \d \tilde{s} \biggr) \d \bar z \nonumber \\ 
	&- \be \varkappa_{\ga, 3} \ga^6 \int_{D_\eps} \phi^\lambda_z (\bar z) \biggl( \eps^3 \sum_{\tilde{y} \in \Lattice} \int_{0}^{\tau_{\ga, \fa}} \!\!\!\mywidetilde{\SK}^\ga_{\bar t-\tilde{s}}(\bar x-\tilde{y})^2 \big( S_\ga(\tilde{s}, \tilde{y}) X_\ga(\tilde{s}, \tilde{y}) - \un{X}_\ga(\tilde{s}, \tilde{y}) X_\ga(\tilde{s}, \tilde{y}) \big) \d \tilde{s} \biggr) \d \bar z. \nonumber
\end{align}
Using Lemma~\ref{lem:Y-approximation}, the last term is absolutely bounded by a constant times $\ga^{3 + 3 \eta}$, and it vanishes as $\ga \to 0$. Using the a priori bound, provided by the stopping time \eqref{eq:tau-2}, the first term in \eqref{eq:tree2-2-withXX} is absolutely bounded by a constant times
\begin{equation}
	(\lambda \vee \emezo)^{-1- \un\kappa} \emezo^{\un\kappa/2 - 1} \ga^6 \int_0^{\infty} \eps^3 \sum_{\tilde y \in \Lattice} \mywidetilde{\SK}^\ga_{\tilde s}(\tilde y)^2 \d \tilde s \lesssim (\lambda \vee \emezo)^{-1- \un\kappa} \emezo^{\un\kappa/2 - 2} \ga^6 \lesssim (\lambda \vee \emezo)^{-1- \un\kappa} \emezo^{\un\kappa/2}.\label{eq:some-term-again} 
\end{equation}
Here, we used Lemma~\ref{lem:renorm-constants} to bound the integral, because it coincides with the renormalisation constant \eqref{eq:renorm-constant1}. Since we assumed $\kappa \geq \un{\kappa}$, the preceding expression is bounded by $(\lambda \vee \emezo)^{-1 - \kappa} \emezo^{\un\kappa/2}$. 

It is left to bund the secdon term in \eqref{eq:renorm-of-2-last}. As in \eqref{eq:tree2-2-withXX}, we get that $\iota_\eps \bigl(\Pi_{z}^{\ga, 5}\tau\bigr)(\phi_z^\lambda)$ equals 
\begin{equation*}
	 -\be \ga^6 \int_0^{\infty} \eps^3 \sum_{y \in \Lattice} \mywidetilde{\SK}^\ga_{s}(y)^2 \biggl( \int_{\tau_{\ga, \fa}}^\infty \eps^3 \sum_{\bar x \in \Lattice} \phi^\lambda_z (\bar t + s, \bar x + y) \Bigl( \un{X}'_{\ga, \fa}(\bar t, \bar x) X'_{\ga, \fa}(\bar t, \bar x) - \un{\fC}_\ga \Bigr) \d \bar t \biggr) \d s
\end{equation*}
up to an error, vanishing as $\ga \to 0$. Here, the process $\un{X}'_{\ga, \fa}$ is defined as in \eqref{eq:X-under} but via the spin field $\sigma'_{\ga, \fa}$. Furthermore, we replace the constant $\un{\fC}_\ga$ by the function \eqref{eq:C-gamma-function} and get
\begin{align}
	& -\be \ga^6 \int_0^{\infty} \eps^3 \sum_{y \in \Lattice} \mywidetilde{\SK}^\ga_{s}(y)^2 \biggl( \int_{\tau_{\ga, \fa}}^\infty \eps^3 \sum_{\bar x \in \Lattice} \phi^\lambda_z (\bar t + s, \bar x + y) \bigl( \un{\fC}_\ga(\bar t - \tau_{\ga, \fa}) - \un{\fC}_\ga\bigr) \d \bar t \biggr) \d s \nonumber \\
	 &\qquad -\be \ga^6 \int_0^{\infty} \eps^3 \sum_{y \in \Lattice} \mywidetilde{\SK}^\ga_{s}(y)^2 \biggl( \int_{\tau_{\ga, \fa}}^\infty \eps^3 \sum_{\bar x \in \Lattice} \phi^\lambda_z (\bar t + s, \bar x + y) \label{eq:some-bound} \\
	 &\hspace{6cm}\times \Bigl( \un{X}'_{\ga, \fa}(\bar t, \bar x) X'_{\ga, \fa}(\bar t, \bar x) - \un{\fC}_\ga(\bar t - \tau_{\ga, \fa}) \Bigr) \d \bar t \biggr) \d s.  \nonumber
\end{align}
Applying Lemma~\ref{lem:renorm-function-underline} with any $c \in (0,1)$, the absolute value of the first term in \eqref{eq:some-bound} is bounded by a constant times 
\begin{equation}\label{eq:term-1}
	\ga^6 \emezo^{c -1} \int_0^{\infty} \eps^3 \sum_{y \in \Lattice} \mywidetilde{\SK}^\ga_{s}(y)^2 \biggl( \int_{\tau_{\ga, \fa}}^\infty \eps^3 \sum_{\bar x \in \Lattice} |\phi^\lambda_z (\bar t + s, \bar x + y)| (\bar t - \tau_{\ga, \fa})^{-c/2} \d \bar t \biggr) \d s.
\end{equation}
Using the scaling properties of the involved functions, this expression is of order $\ga^6 \emezo^{c -2} (\lambda \vee \emezo)^{-c}$. Recalling that $\emezo \approx \ga^3$, it vanishes as $\ga \to 0$.

Now, we consider the second term in \eqref{eq:some-bound}. Multiplying and dividing the random process in the brackets by $(t - \tau_{\ga, \fa})^{-\frac{\eta}{2}}$, we estimate the absolute value of this expression by
\begin{align*}
	 -\be \ga^6 \int_0^{\infty} \eps^3 \sum_{y \in \Lattice} \mywidetilde{\SK}^\ga_{s}(y)^2 &\biggl( \int_{\tau_{\ga, \fa}}^\infty \eps^3 \sum_{\bar x \in \Lattice} |\phi^\lambda_z (\bar t + s, \bar x + y)| (t - \tau_{\ga, \fa})^{\frac{\eta}{2}} \d \bar t \biggr) \d s \\
	 &\times \biggl( \sup_{\bar t \in [\tau_{\ga, \fa}, T]} (t - \tau_{\ga, \fa})^{-\frac{\eta}{2}} \Bigl| \un{X}'_{\ga, \fa}(\bar t, \bar x) X'_{\ga, \fa}(\bar t, \bar x) - \un{\fC}_\ga(\bar t - \tau_{\ga, \fa}) \Bigr| \biggr) 
\end{align*}
for a sufficiently large $T$. We can restrict the variable $\bar t$ by $T$ in this formula because $\mywidetilde{\SK}^\ga$ and $\phi^\lambda_z$ are compactly supported. Applying the H\"{o}lder inequality and Lemma~\ref{lem:X2-under-prime}, the $p$-th moment of this expression is bounded by a constant multiple of 
\begin{equation*}
	 \ga^6 \un{\emezo}^{\eta} \biggl(\E \biggl[ \int_0^{\infty} \eps^3 \sum_{y \in \Lattice} \mywidetilde{\SK}^\ga_{s}(y)^2 \biggl( \int_{\tau_{\ga, \fa}}^\infty \eps^3 \sum_{\bar x \in \Lattice} |\phi^\lambda_z (\bar t + s, \bar x + y)| (\bar t - \tau_{\ga, \fa})^{\frac{\eta}{2}} \d \bar t \biggr) \d s \biggr]^{2 p}\biggr)^{\frac{1}{2p}}.
\end{equation*}
As in \eqref{eq:term-1}, this expression is bounded by a constant times $\ga^6 \un{\emezo}^{\eta} \emezo^{-1} (\lambda \vee \emezo)^{\eta}$. Recalling that $\un\emezo = \emezo \ga^{\un\kappa}$ and $\emezo \approx \ga^3$, this expression vanishes as $\ga \to 0$.

The analysis which we performed the renormalised contraction of two vertices in \eqref{eq:Pi2} will be used many times for the other diagrams below. In order to draw less diagrams, we prefer to introduce a new vertex 
\begin{equation*}
\begin{tikzpicture}[scale=0.35, baseline=0.4cm]
	\node at (0,0)  [dot] (int) {};
	\node at (0,2.5)  [var_very_pink] (left) {};
	
	\draw[keps] (left) to[bend left=60] node[labl,pos=0.45] {\tiny 3,0} (int);
	\draw[keps] (left) to[bend left=-60] node[labl,pos=0.45] {\tiny 3,0} (int);
\end{tikzpicture}
\; := \;
\begin{tikzpicture}[scale=0.35, baseline=0.4cm]
	\node at (0,0)  [dot] (int) {};
	\node at (0,2.5)  [var_very_blue] (left) {};
	
	\draw[keps] (left) to[bend left=60] node[labl,pos=0.45] {\tiny 3,0} (int);
	\draw[keps] (left) to[bend left=-60] node[labl,pos=0.45] {\tiny 3,0} (int);
\end{tikzpicture}
\; - \; \fc_\ga.
\end{equation*}

\subsubsection{The element $\tau = \protect\<3b>$}
\label{sec:third-element}

The definition \eqref{eq:lift-hermite} of the renormalized model and the expansion \cite[Eq.~2.11]{Martingales} yield a diagrammatical representation of the map $\Pi^{\ga, \fa}_{z}\tau$:
\begin{equation}\label{eq:Psi3}
\iota_\eps \bigl(\Pi^{\ga, \fa}_{z}\tau\bigr)(\phi_z^\lambda)
\;=\; 
\begin{tikzpicture}[scale=0.35, baseline=0cm]
	\node at (0,-2.2)  [root] (root) {};
	\node at (0,-2.2) [rootlab] {$z$};
	\node at (0,-2.5) {$$};
	\node at (0,0)  [dot] (int) {};
	\node at (-2,2)  [var_blue] (left) {};
	\node at (0,2.9)  [var_blue] (cent) {};
	\node at (2,2)  [var_blue] (right) {};
	
	\draw[testfcn] (int) to (root);
	
	\draw[keps] (left) to node[labl,pos=0.45] {\tiny 3,0} (int);
	\draw[keps] (right) to node[labl,pos=0.45] {\tiny 3,0} (int);
	\draw[keps] (cent) to node[labl,pos=0.4] {\tiny 3,0} (int);
\end{tikzpicture}
\; +\;3
\begin{tikzpicture}[scale=0.35, baseline=0cm]
	\node at (0,-2.2)  [root] (root) {};
	\node at (0,-2.2) [rootlab] {$z$};
	\node at (0,0)  [dot] (int) {};
	\node at (-1.3,2.3)  [var_very_pink] (left) {};
	\node at (2.1,2.3)  [var_blue] (right) {};
	
	\draw[testfcn] (int) to (root);
	
	\draw[keps] (left) to[bend left=60] node[labl,pos=0.45] {\tiny 3,0} (int);
	\draw[keps] (left) to[bend left=-60] node[labl,pos=0.45] {\tiny 3,0} (int);
	\draw[keps] (right) to[bend left=30] node[labl,pos=0.45] {\tiny 3,0} (int);
\end{tikzpicture}
\; +\;
\begin{tikzpicture}[scale=0.35, baseline=0cm]
	\node at (0,-2.2)  [root] (root) {};
	\node at (0,-2.2) [rootlab] {$z$};
	\node at (0,0)  [dot] (int) {};
	\node at (0,2.9)  [var_very_blue] (left) {};
	
	\draw[testfcn] (int) to (root);
	
	\draw[keps] (left) to[bend left=90] node[labl,pos=0.6] {\tiny 3,0} (int);
	\draw[keps] (left) to[bend left=-90] node[labl,pos=0.6] {\tiny 3,0} (int);
	\draw[keps] (left) to node[labl,pos=0.4] {\tiny 3,0} (int);
\end{tikzpicture}\;.
\end{equation}
Using \cite[Cor.~5.6]{Martingales}, for any $\bar \kappa > 0$ and for any $p \geq 2$ large enough, we bound the $p$-th moment of first diagram in \eqref{eq:Psi3} by a constant times
\[
	(\lambda \vee \emezo)^{-\frac32} \Bigl( 1 + \eps^{\frac{9}{4} - \bar \kappa} \emezo^{-\frac52} + \eps^{\frac{9}{2}- \bar \kappa} \emezo^{-5} + \eps^{\frac{27}{4}- \bar \kappa} \emezo^{-\frac{15}{2}} \Bigr),
\]
which is the required bound \eqref{eq:model_bound} with $|\tau| = - \frac{3}{2}-3\kappa$.

We demonstrate once again how to analyse renormalised contraction of two vertices in the second diagram in \eqref{eq:Psi3}, and we prefer not to repeat analogous computation in what follows. As in \eqref{eq:Pi2-new} and \eqref{eq:renorm-of-2}, we write 
\[
	\begin{tikzpicture}[scale=0.35, baseline=0cm]
		\node at (0,-2.2)  [root] (root) {};
		\node at (0,-2.2) [rootlab] {$z$};
		\node at (0,0)  [dot] (int) {};
		\node at (-1.3,2.3)  [var_very_pink] (left) {};
		\node at (2.1,2.3)  [var_blue] (right) {};
		
		\draw[testfcn] (int) to (root);
		
		\draw[keps] (left) to[bend left=60] node[labl,pos=0.45] {\tiny 3,0} (int);
		\draw[keps] (left) to[bend left=-60] node[labl,pos=0.45] {\tiny 3,0} (int);
		\draw[keps] (right) to[bend left=30] node[labl,pos=0.45] {\tiny 3,0} (int);
	\end{tikzpicture}
	\; = \emezo^{2 (a - 3)} \;
	\begin{tikzpicture}[scale=0.35, baseline=0cm]
		\node at (0,-2.2)  [root] (root) {};
		\node at (0,-2.2) [rootlab] {$z$};
		\node at (0,0)  [dot] (int) {};
		\node at (-1.3,2.3)  [var_red_square] (left) {};
		\node at (2.1,2.3)  [var_blue] (right) {};
		
		\draw[testfcn] (int) to (root);
		
		\draw[keps] (left) to[bend left=60] node[labl,pos=0.45] {\tiny a,0} (int);
		\draw[keps] (left) to[bend left=-60] node[labl,pos=0.45] {\tiny a,0} (int);
		\draw[keps] (right) to[bend left=30] node[labl,pos=0.45] {\tiny 3,0} (int);
	\end{tikzpicture}
	\; + \left(
	\begin{tikzpicture}[scale=0.35, baseline=0cm]
		\node at (0,-2.2)  [root] (root) {};
		\node at (0,-2.2) [rootlab] {$z$};
		\node at (0,0)  [dot] (int) {};
		\node at (-1.3,2.3)  [var_red_triangle] (left) {};
		\node at (2.1,2.3)  [var_blue] (right) {};
		
		\draw[testfcn] (int) to (root);
		
		\draw[keps] (left) to[bend left=60] node[labl,pos=0.45] {\tiny 3,0} (int);
		\draw[keps] (left) to[bend left=-60] node[labl,pos=0.45] {\tiny 3,0} (int);
		\draw[keps] (right) to[bend left=30] node[labl,pos=0.45] {\tiny 3,0} (int);
	\end{tikzpicture}
	\; - \; \fc_\ga \, 
	\begin{tikzpicture}[scale=0.35, baseline=0cm]
		\node at (1.5,2)  [var_blue] (right) {};
		\node at (0,-2.2)  [root] (root) {};
		\node at (0,-2.2) [rootlab] {$z$};
		\node at (0,0)  [dot] (int) {};
		
		\draw[testfcn] (int) to (root);
		\draw[keps] (right) to node[labl,pos=0.45] {\tiny 3,0} (int);
	\end{tikzpicture}\right),
\]
for any $0 < a < \frac{5}{2}$. \cite[Assum.~3]{Martingales} is satisfied for the two preceding diagrams, with $\tilde \CCV_{\!\var} = \{ 1, 2 \}$, $\Gamma = \{2\}$ and the labeling is $\L = \{ \diamond, \nil \}$ for the first diagram and $\L = \{ \triangledown, \nil \}$ for the second. Applying \cite[Cor.~5.6]{Martingales} to the first diagram, we get, for any $\bar \kappa > 0$ and for any $p \geq 2$ large enough, a bound on the $p$-th moment of the order
\[
	\emezo^{2 (a - 3)} (\lambda \vee \emezo)^{2 - 2a} \Bigl(\eps^{\frac94} + \eps^{\frac92 - \bar \kappa} \emezo^{-\frac{5}{2}}\Bigr).
\]
For $\frac{3}{2} < a \leq \frac{7}{4}$ and $\bar \kappa > 0$ small enough the right-hand side is bounded by $c_\gamma (\lambda \vee \emezo)^{-3/2}$, where $c_\gamma$ vanishes as $\ga \to 0$. The second diagram is analyzed similarly to the third diagram in \eqref{eq:Pi2-new}, and it can be also bounded by $c_\gamma (\lambda \vee \emezo)^{-3/2}$. 

We now look at the last diagram in \eqref{eq:Psi3}. By using equation \eqref{eq:mrt_decomp}, we can write 
\begin{align*}
&\eps^{9} \sum_{\tilde x \in \Le} \sum_{\tilde s \in \R} \mywidetilde{\SK}^\ga_{\bar t - \tilde s}(\bar x - \tilde x)^3 \big(\Delta_{\tilde s} \M_{\ga, \fa}(\tilde x) \big)^{3} = 4 \eps^{9} \ga^{-6} \sum_{\tilde x \in \Le} \sum_{\tilde s \in \R} \mywidetilde{\SK}^\ga_{\bar t - \tilde s}(\bar x - \tilde x)^3 \Delta_{\tilde s} \M_{\ga, \fa}(\tilde x) \\
&\qquad = 4 \eps^{6} \ga^{-6} \int_{D_\eps} \mywidetilde{\SK}^\ga(\bar z - \tilde z)^3 \d \M_{\ga, \fa}(\tilde z) + 4 \eps^{\frac{15}{4}} \ga^{-6} \int_{D_\eps} \mywidetilde{\SK}^\ga(\bar z - \tilde z)^3 \Cd_{\ga, \fa}(\tilde z) \d \tilde z.
\end{align*}
As we explained at the beginning of this section, the kernel $\mywidetilde{\SK}^\ga$ satisfies \cite[Assum.~4]{Martingales} with $a_e=3$, and hence \cite[Lem.~5.2]{Martingales} yields the bound $|\mywidetilde{\SK}^\ga(z)| \lesssim (\Vert z\Vert_{\s} \vee \emezo)^{-3}$. Then we have $|\mywidetilde{\SK}^\ga(z)^3| \lesssim (\Vert z\Vert_{\s} \vee \emezo)^{-9}$ and \cite[Lem.~5.2]{Martingales} implies that $\mywidetilde{\SK}^\ga(z)^3$ satisfies \cite[Assum.~4]{Martingales} with $a_e=9$. This allows to write the last diagram in \eqref{eq:Psi3} as
\begin{equation}\label{eq:Psi3-last}
4 \eps^{6} \ga^{-6} \emezo^{a - 9}
\begin{tikzpicture}[scale=0.35, baseline=0cm]
	\node at (0,-2.2)  [root] (root) {};
	\node at (0,-2.2) [rootlab] {$z$};
	\node at (0,0)  [dot] (int) {};
	\node at (0,2.9)  [var_blue] (left) {};
	
	\draw[testfcn] (int) to (root);
	
	\draw[keps] (left) to node[labl,pos=0.4] {\tiny a,0} (int);
\end{tikzpicture} 
\;+\;
4 \eps^{\frac{15}{4}} \ga^{-6} \int_{D_\eps} \int_{D_\eps} \varphi_z^\lambda(\bar z)\, \mywidetilde{\SK}^\ga(\bar z - \tilde z)^3 \Cd_{\ga, \fa}(\tilde z) \d \tilde z \d \bar z,
\end{equation}
where we used the same trick as in \eqref{eq:renorm-of-2} to ``improve'' the singularity of the kernel. For $a < 5$, the first diagram in \eqref{eq:Psi3-last} satisfies \cite[Assum.~3]{Martingales}, and for any $\bar \kappa > 0$ and $p \geq 2$ large enough, \cite[Cor.~5.6]{Martingales} allows to bound its $p$-th moment by a constant multiple of 
\begin{equation*}
\eps^{6} \ga^{-6} \emezo^{a - 9} (\lambda \vee \emezo)^{\frac{5}{2} - a} \Bigl(1 + \eps^{\frac94 - \bar \kappa} \emezo^{-\frac{5}{2}}\Bigr).
\end{equation*}

For $a > 3$, this expression is of order $c_\ga (\lambda \vee \emezo)^{\frac{5}{2} - a}$, where $c_\ga \to 0$ as $\ga \to 0$, which is the required bound \eqref{eq:model_bound}.

Now we will analyse the second term in \eqref{eq:Psi3-last}. Because of our extension of the martingales \eqref{eq:martingale-extension}, we need to bound separately the part of \eqref{eq:Psi3-last} with positive and negative times. Because the bounds in the two cases are the same, we will write only the analysis for positive times. Using \eqref{eq:C-gamma} and changing the integration variables, we can write it as a constant multiple of
\begin{equation}\label{eq:Psi3-last-as-two}
\begin{aligned}
&\ga^{12} \int_{D_\eps} \mywidetilde{\SK}^\ga(\bar z)^3 \int_{\R_+ \times \Lattice} \varphi_{z - \bar z}^\lambda(\tilde z) \Cd_{\ga, \fa}(\tilde z) \d \tilde z \d \bar z \\
&\qquad = \ga^{12} \int_{D_\eps} \mywidetilde{\SK}^\ga(\bar z)^3 \int_{[0, \tau_{\ga, \fa}] \times \Lattice} \varphi_{z - \bar z}^\lambda(\tilde z) \Bigl( S_\ga(\tilde z) - \ga^{-3}\tanh \big( \be \ga^3 X_\ga(\tilde z) \big) \Bigr) \d \tilde z \d \bar z \\
&\qquad\qquad + \ga^{12} \int_{D_\eps} \mywidetilde{\SK}^\ga(\bar z)^3 \int_{[\tau_{\ga, \fa}, \infty) \times \Lattice} \varphi_{z - \bar z}^\lambda(\tilde z) \bigl(S'_{\ga, \fa}(\tilde z) - X'_{\ga, \fa}(\tilde z) \bigr) \d \tilde z \d \bar z,
\end{aligned}
\end{equation}
where we used the rescaled spin field \eqref{eq:S-def} and where $S'_{\ga, \fa}$ is defined by \eqref{eq:S-def} for the spin field $\sigma'_{\ga, \fa}$. 

Let us bound the first term in \eqref{eq:Psi3-last-as-two}. From the decomposition of the kernel $\mywidetilde{\SK}^\ga$, provided in the beginning of Section~\ref{sec:lift}, we get $| \mywidetilde{\SK}^\ga(z)| \lesssim (\| z \|_\s\vee \emezo)^{-3}$. Then from \cite[Lem.~7.3]{HairerMatetski} we get $\int_{D_\eps} \mywidetilde{\SK}^\ga(\bar z)^3 \d \bar z \lesssim \emezo^{-4}$. Approximating the function $\tanh$ by its Taylor expansion and using \eqref{eq:beta}, we write the first term in \eqref{eq:Psi3-last-as-two} as
\begin{equation}\label{eq:Psi3-very-last}
\ga^{12} \int_{D_\eps}  \mywidetilde{\SK}^\ga(\bar z)^3 \int_{[0, \tau_{\ga, \fa}] \times \Lattice} \varphi_{z - \bar z}^\lambda(\tilde z) \bigl( S_\ga(\tilde z) - X_\ga(\tilde z) \bigr) \d \tilde z \d \bar z + \Err_{\gamma, \lambda},
\end{equation}
where the error term $\Err_{\gamma, \lambda}$ is absolutely bounded by a constant times 
\begin{align*}
\ga^{18} \int_{D_\eps} &\mywidetilde{\SK}^\ga(\bar z)^3 \int_{[0, \tau_{\ga, \fa}] \times \Lattice} |\varphi_{z - \bar z}^\lambda(\tilde z)| \Bigl( \big( \CGG + A \big) \| X_\ga(\tilde t)\|_{L^\infty}  + \| X_\ga(\tilde t) \|^3_{L^\infty} \Bigr) \d \tilde z \d \bar z \\
&\qquad \lesssim \ga^{18} \emezo^{-4} \sup_{t \geq 0} \Bigl( \big( \CGG + A \big) \| X_\ga( t)\|_{L^\infty}  + \| X_\ga(t) \|^3_{L^\infty} \Bigr),
\end{align*}
with $\tilde t$ being the time variable in $\tilde z$. Here, we used $\int_{D_\eps} |\varphi_{z - \bar z}^\lambda(\tilde z)| \d \tilde z \lesssim 1$. The a priori bound \eqref{eq:X-apriori} allows to estimate the preceding expression by $\ga^{18} \emezo^{3 \eta-4} \lesssim \ga^{6 + 9 \eta}$, which vanishes as $\gamma \to 0$ because $\eta > -\frac{2}{3}$ in the assumptions of Theorem~\ref{thm:main}.

Now, we will bound the first term in \eqref{eq:Psi3-very-last}. From the definitions \eqref{eq:X-gamma} and \eqref{eq:S-def} we conclude that $X_\ga(t, x) = \eps^3 \sum_{y \in \Lattice} K_\ga(x - y) S_\ga(t, y)$. Then we can write
\begin{equation}\label{eq:Z-is-X}
\int_{[0, \tau_{\ga, \fa}] \times \Lattice} \varphi_{z - \bar z}^\lambda(\tilde z) \bigl( S_\ga(\tilde z) - X_\ga(\tilde z) \bigr) \d \tilde z = \int_{[0, \tau_{\ga, \fa}] \times \Lattice} \psi_{z - \bar z}^\lambda(\tilde z) S_\ga(\tilde z) \d \tilde z,
\end{equation}
with $\psi_{z - \bar z}^\lambda(\tilde t, \tilde x) = \varphi_{z - \bar z}^\lambda(\tilde t, \tilde x) - \eps^3 \sum_{y \in \Lattice} \varphi_{z - \bar z}^\lambda(\tilde t, y) K_\ga(y - \tilde x)$. This function can be viewed as a rescaled test function, which for any $\kappa_1 \in [0, 1)$ and any $k \in \N_0^4$ satisfies $\| D^k \psi_{z - \bar z}^\lambda \|_{L^\infty} \lesssim \emezo^{\kappa_1} \lambda^{-5 - |k|_\s - \kappa_1}$. Then for any $\tilde \kappa > 0$, Lemma~\ref{lem:Z-bound} yields 
\begin{equation*}
\biggl| \int_{D_\eps} \psi_{z - \bar z}^\lambda(\tilde z) S_\ga(\tilde z) \d \tilde z\biggr| \lesssim \fa \emezo^{\kappa_1} \lambda^{\eta - \kappa_1}
 \end{equation*}
 uniformly in $\lambda \in [\emezo^{1- \tilde \kappa}, 1]$. For $\lambda < \emezo^{1- \tilde \kappa}$ we can use $|S_\ga(\tilde z)| \leq \ga^{-3}$ and estimate the left-hand side by a constant multiple of $\emezo^{-1}$. Since $-1 < \eta < -\frac{1}{2}$, from the two preceding bounds we conclude that
  \begin{equation}\label{eq:3-bound}
\biggl| \int_{D_\eps} \psi_{z - \bar z}^\lambda(\tilde z) S_\ga(\tilde z) \d \tilde z\biggr| \lesssim \fa \emezo^{\kappa_1} (\lambda \vee \emezo)^{-\frac{1 + \kappa_1}{1 - \tilde \kappa}}.
 \end{equation}
 Moreover, as above we have $\int_{D_\eps} \mywidetilde{\SK}^\ga(\bar z)^3 \d \bar z \lesssim \emezo^{-4}$. Hence, the first term in \eqref{eq:Psi3-very-last} is absolutely bounded by a constant times $\fa \emezo^{\kappa_1} (\lambda \vee \emezo)^{-\frac{1 + \kappa_1}{1 - \tilde \kappa}}$. If we take $\tilde \kappa = \frac{1 - 2 \kappa_1}{3}$, we get an estimate by $\fa \emezo^{\kappa_1} (\lambda \vee \emezo)^{-\frac{3}{2}}$, which vanishes as $\gamma \to 0$. Taking $\kappa_1$ close to $0$ we make $\tilde \kappa$ close to $\frac{1}{2}$,  and Lemma~\ref{lem:Z-bound} suggests that the test functions may be taken from $\CB^2_\s$.
 
 It is left to bound the last term in \eqref{eq:Psi3-last-as-two}. Identity \eqref{eq:Z-is-X} with the time interval $[\tau_{\ga, \fa}, \infty)$ holds for the processes $S'_{\ga, \fa}$ and $X'_{\ga, \fa}$. Applying Lemma~\ref{lem:Z-prime-bound}, we get the same bound as \eqref{eq:3-bound}, i.e.
 \begin{equation*}
 \biggl| \int_{[\tau_{\ga, \fa}, \infty) \times \Lattice} \psi_{z - \bar z}^\lambda(\tilde z) S'_{\ga, \fa}(\tilde z) \d \tilde z \biggr| \lesssim \emezo^{\kappa_1} (\lambda \vee \emezo)^{-\frac{3}{2}}
 \end{equation*}
 uniformly in $\lambda \in (0, 1]$.

\subsubsection{The element $\tau = \protect\<22b>$}

Using the definition \eqref{eq:Pi-rest} and the expansion \cite[Eq.~2.11]{Martingales}, we can represent the map $\Pi^{\ga, \fa}_{z}\tau$ diagrammatically as
\begin{align} \nonumber
\iota_\eps \bigl(\Pi^{\ga, \fa}_{z}\tau\bigr)(\phi_z^\lambda)
&\;=\; 
\begin{tikzpicture}[scale=0.35, baseline=-0.5cm]
	\node at (0,-5.1)  [root] (root) {};
	\node at (0,-5.1) [rootlab] {$z$};
	\node at (0,0) [dot] (int) {};
	\node at (-2,2)  [var_blue] (left) {};
	\node at (2,2)  [var_blue] (right) {};
	\node at (0,-2.9) [dot] (cent1) {};
	\node at (2,-0.9)  [var_blue] (right1) {};
	\node at (-2,-0.9)  [var_blue] (left1) {};
	\draw[testfcn] (cent1) to (root);
	\draw[keps] (left) to node[labl,pos=0.45] {\tiny 3,0} (int);
	\draw[keps] (right) to node[labl,pos=0.45] {\tiny 3,0} (int);
	\draw[keps] (int) to node[labl,pos=0.45] {\tiny 3,1} (cent1);
	\draw[keps] (right1) to node[labl,pos=0.45] {\tiny 3,0} (cent1);
	\draw[keps] (left1) to node[labl,pos=0.45] {\tiny 3,0} (cent1);
\end{tikzpicture}
\;+\;2\;
\begin{tikzpicture}[scale=0.35, baseline=-0.5cm]
	\node at (0,-5.1)  [root] (root) {};
	\node at (0,-5.1) [rootlab] {$z$};
	\node at (0,0)  [dot] (int) {};
	\node at (-2,2)  [var_blue] (left) {};
	\node at (0,-2.9) [dot] (cent1) {};
	\node at (2.5,-1.45)  [var_very_blue] (right1) {};
	\node at (-2,-0.9)  [var_blue] (left1) {};
	\draw[testfcn] (cent1) to (root);
	\draw[keps] (left) to node[labl,pos=0.45] {\tiny 3,0} (int);
	\draw[keps] (right1) to node[labl,pos=0.45] {\tiny 3,0} (int);
	\draw[keps] (int) to node[labl,pos=0.45] {\tiny 3,1} (cent1);
	\draw[keps] (right1) to node[labl,pos=0.45] {\tiny 3,0} (cent1);
	\draw[keps] (left1) to node[labl,pos=0.45] {\tiny 3,0} (cent1);
\end{tikzpicture}
\; +\;
\begin{tikzpicture}[scale=0.35, baseline=-0.5cm]
	\node at (0,-5.1)  [root] (root) {};
	\node at (0,-5.1) [rootlab] {$z$};
	\node at (0,0)  [dot] (int) {};
	\node at (0,2.9)  [var_very_pink] (left) {};
	\node at (0,-2.9) [dot] (cent1) {};
	\node at (2,-0.9)  [var_blue] (right1) {};
	\node at (-2,-0.9)  [var_blue] (left1) {};
	\draw[testfcn] (cent1) to (root);
	\draw[keps] (left) to[bend left=60] node[labl,pos=0.45] {\tiny 3,0} (int);
	\draw[keps] (left) to[bend left=-60] node[labl,pos=0.45] {\tiny 3,0} (int);
	\draw[keps] (int) to node[labl,pos=0.45] {\tiny 3,1} (cent1);
	\draw[keps] (right1) to node[labl,pos=0.45] {\tiny 3,0} (cent1);
	\draw[keps] (left1) to node[labl,pos=0.45] {\tiny 3,0} (cent1);
\end{tikzpicture}
\; +\;
\begin{tikzpicture}[scale=0.35, baseline=-0.5cm]
	\node at (0,-5.1)  [root] (root) {};
	\node at (0,-5.1) [rootlab] {$z$};
	\node at (0,0) [dot] (int) {};
	\node at (-2,2)  [var_blue] (left) {};
	\node at (2,2)  [var_blue] (cent) {};
	\node at (0,-2.9) [dot] (cent1) {};
	\node at (2.9,-1.7)  [var_very_pink] (right1) {};
	\draw[testfcn] (cent1) to (root);
	\draw[keps] (left) to node[labl,pos=0.45] {\tiny 3,0} (int);
	\draw[keps] (cent) to node[labl,pos=0.45] {\tiny 3,0} (int);
	\draw[keps] (int) to node[labl,pos=0.4] {\tiny 3,1} (cent1);
	\draw[keps] (right1) to[bend left=-30] node[labl,pos=0.45] {\tiny 3,0} (cent1);
	\draw[keps] (right1) to[bend left=30] node[labl,pos=0.45] {\tiny 3,0} (cent1);
\end{tikzpicture} 
 \\[0.2cm] \label{eq:22b}
&\qquad \;+\;
\begin{tikzpicture}[scale=0.35, baseline=-0.5cm]
	\node at (0,-5.1)  [root] (root) {};
	\node at (0,-5.1) [rootlab] {$z$};
	\node at (0,0)  [dot] (int) {};
	\node at (-1.3,2.3)  [var_very_pink] (left) {};
	\node at (0,-2.9) [dot] (cent1) {};
	\node at (2.9,-1.7) [var_very_pink] (right1) {};
	\draw[testfcn] (cent1) to (root);
	\draw[keps] (left) to[bend left=60] node[labl,pos=0.45] {\tiny 3,0} (int);
	\draw[keps] (left) to[bend left=-60] node[labl,pos=0.45] {\tiny 3,0} (int);
	\draw[keps] (int) to node[labl,pos=0.45] {\tiny 3,1} (cent1);
	\draw[keps] (right1) to[bend left=-30] node[labl,pos=0.45] {\tiny 3,0} (cent1);
	\draw[keps] (right1) to[bend left=30] node[labl,pos=0.45] {\tiny 3,0} (cent1);
\end{tikzpicture}
 \;+\;2\,
\begin{tikzpicture}[scale=0.35, baseline=-0.5cm]
	\node at (0,-5.1)  [root] (root) {};
	\node at (0,-5.1) [rootlab] {$z$};
	\node at (0,0) [dot] (int) {};
	\node at (-2,2) [var_blue] (left) {};
	\node at (0,-2.9) [dot] (cent1) {};
	\node at (3,-1.45)  [var_very_blue] (right1) {};
	\draw[testfcn] (cent1) to (root);
	\draw[keps] (left) to node[labl,pos=0.45] {\tiny 3,0} (int);
	\draw[keps] (right1) to[bend left=-30] node[labl,pos=0.45] {\tiny 3,0} (int);
	\draw[keps] (int) to node[labl,pos=0.45] {\tiny 3,1} (cent1);
	\draw[keps] (right1) to[bend left=-30] node[labl,pos=0.45] {\tiny 3,0} (cent1);
	\draw[keps] (right1) to[bend left=30] node[labl,pos=0.45] {\tiny 3,0} (cent1);
\end{tikzpicture} 
\;+\;2\,
\begin{tikzpicture}[scale=0.35, baseline=-0.5cm]
	\node at (0,-5.1)  [root] (root) {};
	\node at (0,-5.1) [rootlab] {$z$};
	\node at (0,0)  [dot] (int) {};
	\node at (-1.3,2.3)  [var_very_blue] (left) {};
	\node at (0,-2.9) [dot] (cent1) {};
	\node at (2,-1.45)  [var_blue] (right1) {};
	\draw[testfcn] (cent1) to (root);
	\draw[keps] (left) to[bend left=60] node[labl,pos=0.45] {\tiny 3,0} (int);
	\draw[keps] (left) to[bend left=-60] node[labl,pos=0.45] {\tiny 3,0} (int);
	\draw[keps] (int) to node[labl,pos=0.45] {\tiny 3,1} (cent1);
	\draw[keps] (right1) to node[labl,pos=0.45] {\tiny 3,0} (cent1);
	\draw[keps] (left) to[bend left=-80] node[labl,pos=0.55] {\tiny 3,0} (cent1);
\end{tikzpicture} \\[0.2cm]
&\qquad 
\; +\;
\begin{tikzpicture}[scale=0.35, baseline=-0.5cm]
	\node at (0,-5.1)  [root] (root) {};
	\node at (0,-5.1) [rootlab] {$z$};
	\node at (0,0)  [dot] (int) {};
	\node at (0,2.9)  [var_very_blue] (left) {};
	\node at (0,-2.9) [dot] (cent1) {};
	\draw[testfcn] (cent1) to (root);
	\draw[keps] (left) to[bend left=70] node[labl,pos=0.45] {\tiny 3,0} (int);
	\draw[keps] (left) to[bend left=-70] node[labl,pos=0.45] {\tiny 3,0} (int);
	\draw[keps] (int) to node[labl,pos=0.45] {\tiny 3,1} (cent1);
	\draw[keps] (left) to[bend left=100] node[labl,pos=0.55] {\tiny 3,0} (cent1);
	\draw[keps] (left) to[bend right=100] node[labl,pos=0.55] {\tiny 3,0} (cent1);
\end{tikzpicture}
\; + \left(\;2\;
\begin{tikzpicture}[scale=0.35, baseline=-0.5cm]
	\node at (0,-5.1)  [root] (root) {};
	\node at (0,-5.1) [rootlab] {$z$};
	\node at (0,1) [dot] (int) {};
	\node at (0,-2.9) [dot] (cent1) {};
	\node at (2,-0.9)  [var_very_blue] (right1) {};
	\node at (-2,-0.9)  [var_very_blue] (left1) {};
	\draw[testfcn] (cent1) to (root);
	\draw[keps] (left1) to node[labl,pos=0.45] {\tiny 3,0} (int);
	\draw[keps] (right1) to node[labl,pos=0.45] {\tiny 3,0} (int);
	\draw[keps] (int) to node[labl,pos=0.45] {\tiny 3,1} (cent1);
	\draw[keps] (right1) to node[labl,pos=0.45] {\tiny 3,0} (cent1);
	\draw[keps] (left1) to node[labl,pos=0.45] {\tiny 3,0} (cent1);
\end{tikzpicture}
 \;-\; \fc_\ga''\,
 \begin{tikzpicture}[scale=0.35, baseline=-0.5cm]
	\node at (0,-2.2)  [root] (root) {};
	\node at (0,-2.2) [rootlab] {$z$};
	\node at (0,0)  [dot] (int) {};
	\draw[testfcn] (int) to (root);
\end{tikzpicture}\right), \nonumber
\end{align}
where the renormalisation constant $\fc_\ga''$ is defined in \eqref{eq:renorm-constant2}, and where the edge with the label ``$3,1$'' represents the kernel in \eqref{eq:Pi-rest}, where ``$1$'' refers to the positive renormalisation (see \cite[Sec.~5]{Martingales}). 

Applying \cite[Cor.~5.6]{Martingales}, the high moments of the first and second diagrams are bounded by a constant multiplier of $(\lambda \vee \emezo)^{-\kappa/2}$. Analysing contractions of vertices in the same way as we did in \eqref{eq:Pi2} and \eqref{eq:Psi3}, the diagrams number $3, \ldots, 7$ are bounded by $c_\ga (\lambda \vee \emezo)^{-\kappa/2}$ with a constant $c_\ga$ vanishing as $\ga \to 0$.

Regarding the eight tree, for any $\kappa > 0$, we first rewrite it as (as before we use \cite[Lem.~5.2]{Martingales} to show that a product of singular kernels again satisfies \cite[Assum.~4]{Martingales})
\begin{equation}
	\begin{tikzpicture}[scale=0.35, baseline=-0.5cm]
		\node at (0,-5.1)  [root] (root) {};
		\node at (0,-5.1) [rootlab] {$z$};
		\node at (0,0)  [dot] (int) {};
		\node at (0,2.9)  [var_very_blue] (left) {};
		\node at (0,-2.9) [dot] (cent1) {};
		\draw[testfcn] (cent1) to (root);
		\draw[keps] (left) to[bend left=70] node[labl,pos=0.45] {\tiny 3,0} (int);
		\draw[keps] (left) to[bend left=-70] node[labl,pos=0.45] {\tiny 3,0} (int);
		\draw[keps] (int) to node[labl,pos=0.45] {\tiny 3,1} (cent1);
		\draw[keps] (left) to[bend left=100] node[labl,pos=0.55] {\tiny 3,0} (cent1);
		\draw[keps] (left) to[bend right=100] node[labl,pos=0.55] {\tiny 3,0} (cent1);
	\end{tikzpicture} \; = \emezo^{-2 - 2\kappa} \;
	\begin{tikzpicture}[scale=0.35, baseline=-0.5cm]
		\node at (0,-5.1)  [root] (root) {};
		\node at (0,-5.1) [rootlab] {$z$};
		\node at (0,0)  [dot] (int) {};
		\node at (0,2.9)  [var_very_blue] (left) {};
		\node at (0,-2.9) [dot] (cent1) {};
		\draw[testfcn] (cent1) to (root);
		\draw[keps] (left) to node[labl,pos=0.45] {\tiny 5-$\kappa$,0} (int);
		\draw[keps] (int) to node[labl,pos=0.45] {\tiny 3,1} (cent1);
		\draw[keps] (left) to[bend left=100] node[labl,pos=0.55] {\tiny 5-$\kappa$,0} (cent1);
	\end{tikzpicture} \, ,
\end{equation}
where we multiplied some kernels by positive powers of $\emezo$ in order to satisfy all the hypotheses of \cite[Cor.~5.6]{Martingales}. Once we apply it, we then get the bound
\[
	\emezo^{-2 - 2\kappa} ( \lambda \vee \emezo )^{-3+2\kappa} \big( \eps^{\frac92 - \tilde \kappa} + \eps^{9 - \tilde \kappa} \emezo^{-\frac52} + \eps^{9 - \tilde \kappa} \emezo^{-5} \big) \lesssim ( \lambda \vee \emezo )^{-3+2\kappa} \emezo^{4-2\kappa},
\]
which vanishes as $\ga \to 0$.

Recalling the definition of the positively renormalised kernel in \eqref{eq:Pi-rest}, the expression in the brackets in \eqref{eq:22b} may be written as
\begin{equation}\label{eq:22b-renorm}
\left(2\;
\begin{tikzpicture}[scale=0.35, baseline=-0.5cm]
	\node at (0,-4.2)  [root] (root) {};
	\node at (0,-4.2) [rootlab] {$z$};
	\node at (0,-4.5) {$$};
	\node at (0,1) [dot] (int) {};
	\node at (0,-2) [dot] (cent1) {};
	\node at (2,-0.5)  [var_very_blue] (right1) {};
	\node at (-2,-0.5)  [var_very_blue] (left1) {};
	\draw[testfcn] (cent1) to (root);
	\draw[keps] (left1) to node[labl,pos=0.45] {\tiny 3,0} (int);
	\draw[keps] (right1) to node[labl,pos=0.45] {\tiny 3,0} (int);
	\draw[keps] (int) to node[labl,pos=0.45] {\tiny 3,0} (cent1);
	\draw[keps] (right1) to node[labl,pos=0.45] {\tiny 3,0} (cent1);
	\draw[keps] (left1) to node[labl,pos=0.45] {\tiny 3,0} (cent1);
\end{tikzpicture}
\;-\; \fc_\ga''\,
 \begin{tikzpicture}[scale=0.35, baseline=-0.5cm]
	\node at (0,-2.2)  [root] (root) {};
	\node at (0,-2.2) [rootlab] {$z$};
	\node at (0,0)  [dot] (int) {};
	\draw[testfcn] (int) to (root);
\end{tikzpicture}\right)
\;-\; 2\;
\begin{tikzpicture}[scale=0.35, baseline=-0.5cm]
	\node at (0,-4.2)  [root] (root) {};
	\node at (0,-4.2) [rootlab] {$z$};
	\node at (0,-4.5) {$$};
	\node at (0,1) [dot] (int) {};
	\node at (0,-2) [dot] (cent1) {};
	\node at (2,-0.5)  [var_very_blue] (right1) {};
	\node at (-2,-0.5)  [var_very_blue] (left1) {};
	\draw[testfcn] (cent1) to (root);
	\draw[keps] (left1) to node[labl,pos=0.45] {\tiny 3,0} (int);
	\draw[keps] (right1) to node[labl,pos=0.45] {\tiny 3,0} (int);
	\draw[keps] (right1) to node[labl,pos=0.45] {\tiny 3,0} (cent1);
	\draw[keps] (left1) to node[labl,pos=0.45] {\tiny 3,0} (cent1);
	\draw [keps] (int) to[out=0,in=90] (3,-1.4) to [out=-90, in=0] node[labl,pos=0] {\tiny 3,0} (root);
\end{tikzpicture}\;.
\end{equation}
The last diagram in \eqref{eq:22b-renorm} is readily bounded using \cite[Cor.~5.6]{Martingales} by a multiple of $(\lambda \vee \emezo)^{-\kappa / 2}$, while the expression in the brackets requires some work. Using the notation from \eqref{eq:Pi2-new}, the expression in the brackets in \eqref{eq:22b-renorm} may be written as
\begin{equation}\label{eq:22b-last}
 2\;
\begin{tikzpicture}[scale=0.35, baseline=-0.5cm]
	\node at (0,-4.2)  [root] (root) {};
	\node at (0,-4.2) [rootlab] {$z$};
	\node at (0,-4.5) {$$};
	\node at (0,1) [dot] (int) {};
	\node at (0,-2) [dot] (cent1) {};
	\node at (2,-0.5)  [var_red_square] (right1) {};
	\node at (-2,-0.5)  [var_red_square] (left1) {};
	\draw[testfcn] (cent1) to (root);
	\draw[keps] (left1) to node[labl,pos=0.45] {\tiny 3,0} (int);
	\draw[keps] (right1) to node[labl,pos=0.45] {\tiny 3,0} (int);
	\draw[keps] (int) to node[labl,pos=0.45] {\tiny 3,0} (cent1);
	\draw[keps] (right1) to node[labl,pos=0.45] {\tiny 3,0} (cent1);
	\draw[keps] (left1) to node[labl,pos=0.45] {\tiny 3,0} (cent1);
\end{tikzpicture}
\; + 4 \;
\begin{tikzpicture}[scale=0.35, baseline=-0.5cm]
	\node at (0,-4.2)  [root] (root) {};
	\node at (0,-4.2) [rootlab] {$z$};
	\node at (0,-4.5) {$$};
	\node at (0,1) [dot] (int) {};
	\node at (0,-2) [dot] (cent1) {};
	\node at (2,-0.5)  [var_red_square] (right1) {};
	\node at (-2,-0.5)  [var_red_triangle] (left1) {};
	\draw[testfcn] (cent1) to (root);
	\draw[keps] (left1) to node[labl,pos=0.45] {\tiny 3,0} (int);
	\draw[keps] (right1) to node[labl,pos=0.45] {\tiny 3,0} (int);
	\draw[keps] (int) to node[labl,pos=0.45] {\tiny 3,0} (cent1);
	\draw[keps] (right1) to node[labl,pos=0.45] {\tiny 3,0} (cent1);
	\draw[keps] (left1) to node[labl,pos=0.45] {\tiny 3,0} (cent1);
\end{tikzpicture}\;
\; + \; \left( 2\;
\begin{tikzpicture}[scale=0.35, baseline=-0.5cm]
	\node at (0,-4.2)  [root] (root) {};
	\node at (0,-4.2) [rootlab] {$z$};
	\node at (0,-4.5) {$$};
	\node at (0,1) [dot] (int) {};
	\node at (0,-2) [dot] (cent1) {};
	\node at (2,-0.5)  [var_red_triangle] (right1) {};
	\node at (-2,-0.5)  [var_red_triangle] (left1) {};
	\draw[testfcn] (cent1) to (root);
	\draw[keps] (left1) to node[labl,pos=0.45] {\tiny 3,0} (int);
	\draw[keps] (right1) to node[labl,pos=0.45] {\tiny 3,0} (int);
	\draw[keps] (int) to node[labl,pos=0.45] {\tiny 3,0} (cent1);
	\draw[keps] (right1) to node[labl,pos=0.45] {\tiny 3,0} (cent1);
	\draw[keps] (left1) to node[labl,pos=0.45] {\tiny 3,0} (cent1);
\end{tikzpicture}
\;-\; \fc_\ga''\,
 \begin{tikzpicture}[scale=0.35, baseline=-0.5cm]
	\node at (0,-2.2)  [root] (root) {};
	\node at (0,-2.2) [rootlab] {$z$};
	\node at (0,0)  [dot] (int) {};
	\draw[testfcn] (int) to (root);
\end{tikzpicture}\right),
\end{equation}
where the first two diagram are bounded using \cite[Cor.~5.6]{Martingales} by $c_\ga (\lambda \vee \emezo)^{-\kappa/2}$ with a constant $c_\ga$ vanishing as $\ga \to 0$.

It is left to bound the expression in the brackets in \eqref{eq:22b-last}. For this, let us define the random kernel
\begin{equation}\label{eq:CG-def}
\CG_{\ga} (z_1, z_2) = 
 \begin{tikzpicture}[scale=0.35, baseline=-0.2cm]
	\node at (0,1) [root] (int) {};
	\node at (0,1.6) {\scriptsize $z_2$};
	\node at (0,-2) [root] (cent1) {};
	\node at (0,-2) [rootlab] {$z_1$};
	\node at (2,-0.5)  [var_red_triangle] (right1) {};
	\node at (-2,-0.5)  [var_red_triangle] (left1) {};
	\draw[keps] (left1) to node[labl,pos=0.45] {\tiny 3,0} (int);
	\draw[keps] (right1) to node[labl,pos=0.45] {\tiny 3,0} (int);
	\draw[keps] (int) to node[labl,pos=0.45] {\tiny 3,0} (cent1);
	\draw[keps] (right1) to node[labl,pos=0.45] {\tiny 3,0} (cent1);
	\draw[keps] (left1) to node[labl,pos=0.45] {\tiny 3,0} (cent1);
\end{tikzpicture}\;, 
\end{equation}
which may be written explicitly as 
\begin{align}\label{eq:CG-explicit}
\CG_{\ga} (z_1, z_2) &= \int_{D_\eps} \int_{D_\eps} \mywidetilde{\SK}^\ga(z_1 - z_3) \mywidetilde{\SK}^\ga(z_1 - z_4) \mywidetilde{\SK}^\ga(z_1 - z_2) \\
&\hspace{2cm} \times \mywidetilde{\SK}^\ga(z_2 - z_3) \mywidetilde{\SK}^\ga(z_2 - z_4) \bC_{\ga, \fa}(z_3) \bC_{\ga, \fa}(z_4) \d z_3 \d z_4.\nonumber
\end{align}
where we used the bracket process \eqref{eq:bC}. Then the expression in the brackets in \eqref{eq:22b-last} is absolutely bounded by
\begin{equation}\label{eq:last-term}
\int_{D_\eps} |\phi_z^\lambda(z_1)| \left| 2 \int_{D_\eps} \CG_{\ga} (z_1, z_2) \d z_2 - \fc_\ga'' \right| \d z_1 \leq \sup_{z_1 \in D_\eps} \left| 2 \int_{D_\eps} \CG_{\ga} (z_1, z_2) \d z_2 - \fc_\ga'' \right|.
\end{equation}
We need the following bounds on the kernel $\CG_{\ga}$.

\begin{lemma}\label{lem:CG-bounds}
There exists a non-random constant $C > 0$, independent of $\ga$, such that 
\begin{equation}\label{eq:CG-bound1}
\bigl| \CG_{\ga} (z_1, z_2)\bigr| \leq C \bigl(\| z_1 - z_2 \|_\s \vee \emezo\bigr)^{-5},
\end{equation}
uniformly in $z_1 \neq z_2$. Moreover, for any $\theta \in (0, 1)$ we have
\begin{equation}\label{eq:CG-bound2}
\left| 2 \int_{D_\eps} \CG_{\ga} (z_1, z_2) \d z_2 - \fc_\ga'' \right| \leq C \emezo^\theta,
\end{equation}
uniformly over $z_1$.
\end{lemma}

\begin{proof}
As we state in \eqref{eq:bC}, we can uniformly bound $\bC_{\ga, \fa}$. Moreover, from the decomposition of the kernel $\mywidetilde{\SK}^\ga$, provided in Appendix~\ref{sec:decompositions}, we can conclude $| \mywidetilde{\SK}^\ga(z)| \lesssim (\| z \|_\s\vee \emezo)^{-3}$. Then the bound \eqref{eq:CG-bound1} follows from \cite[Lem.~7.3]{HairerMatetski}.

Using the definitions \eqref{eq:CG-explicit} and \eqref{eq:renorm-constant2}, the expression in the absolute value in \eqref{eq:CG-bound2} may be written explicitly as
\begin{align*}
&2 \int_{D_\eps} \int_{D_\eps} \int_{D_\eps} \mywidetilde{\SK}^\ga(z_1 - z_3) \mywidetilde{\SK}^\ga(z_1 - z_4) \mywidetilde{\SK}^\ga(z_1 - z_2) \\
&\hspace{3cm} \times \mywidetilde{\SK}^\ga(z_2 - z_3) \mywidetilde{\SK}^\ga(z_2 - z_4) \bigl(\bC_{\ga, \fa}(z_3) - 2\bigr) \bC_{\ga, \fa}(z_4) \d z_2 \d z_3 \d z_4 \\
&\qquad + 4 \int_{D_\eps} \int_{D_\eps} \int_{D_\eps} \mywidetilde{\SK}^\ga(z_1 - z_3) \mywidetilde{\SK}^\ga(z_1 - z_4) \SK^\ga(z_1 - z_2) \\
&\hspace{4cm} \times \mywidetilde{\SK}^\ga(z_2 - z_3) \mywidetilde{\SK}^\ga(z_2 - z_4) \bigl(\bC_{\ga, \fa}(z_4) - 2\bigr) \d z_2 \d z_3 \d z_4.
\end{align*}
 Moreover, we can write the difference $\bC_{\ga, \fa} - 2$ as in \eqref{eq:bC-new}. After that, we apply Lemma~\ref{lem:Y-approximation} to replace the product $S_\ga X_\ga$ by $\un{X}_\ga X_\ga$, up to an error term. Then the preceding expression equals 
 \begin{equation}\label{eq:22b-very-last}
\begin{aligned}
&- 4 \be \ga^6 \int_{D_\eps} \int_{D_\eps} \int_{D_\eps} \mywidetilde{\SK}^\ga(z_1 - z_3) \mywidetilde{\SK}^\ga(z_1 - z_4) \mywidetilde{\SK}^\ga(z_1 - z_2) \\
&\hspace{2cm} \times \mywidetilde{\SK}^\ga(z_2 - z_3) \mywidetilde{\SK}^\ga(z_2 - z_4) \un{X}_\ga (z_3) X_\ga(z_3) \bC_{\ga, \fa}(z_4) \d z_2 \d z_3 \d z_4  \\
&\qquad - 8 \be \ga^6 \int_{D_\eps} \int_{D_\eps} \int_{D_\eps} \mywidetilde{\SK}^\ga(z_1 - z_3) \mywidetilde{\SK}^\ga(z_1 - z_4) \mywidetilde{\SK}^\ga(z_1 - z_2) \\
&\hspace{3cm} \times \mywidetilde{\SK}^\ga(z_2 - z_3) \mywidetilde{\SK}^\ga(z_2 - z_4) \un{X}_\ga (z_4) X_\ga(z_4) \d z_2 \d z_3 \d z_4 + \Err_{\gamma, \lambda}, 
\end{aligned}
\end{equation}
where the error term satisfies 
\begin{align*}
|\Err_{\gamma, \lambda}| &\lesssim \ga^{6 + 3 \eta} \,
\left|\; \begin{tikzpicture}[scale=0.35, baseline=-0.6cm]
	\node at (0,0) [dot] (int) {};
	\node at (0,-2.9) [root] (cent1) {};
	\node at (0,-2.9) [rootlab] {$0$};
	\node at (2.5,-1.45)  [dot] (right1) {};
	\node at (-2.5,-1.45)  [dot] (left1) {};
	\draw[keps] (left1) to node[labl,pos=0.45] {\tiny 3,0} (int);
	\draw[keps] (right1) to node[labl,pos=0.45] {\tiny 3,0} (int);
	\draw[keps] (int) to node[labl,pos=0.45] {\tiny 3,0} (cent1);
	\draw[keps] (right1) to node[labl,pos=0.45] {\tiny 3,0} (cent1);
	\draw[keps] (left1) to node[labl,pos=0.45] {\tiny 3,0} (cent1);
\end{tikzpicture}\; \right|\;.
\end{align*}
Using the bounds on singular kernels derived in \cite[Lem.~7.3]{HairerMatetski}, we obtain $|\Err_{\gamma, \lambda}| \lesssim \ga^{6 + 3 \eta - \bar{\kappa}}$ for any $\bar{\kappa} > 0$. From \eqref{eq:tau-2} we have the a priori bound $\| \un{X}_\ga(t)  X_\ga(t)\|_{L^\infty} \lesssim \emezo^{\un\kappa/2-2}$ for $t < \tau_{\ga, \fa}$, which allows to bound the first two terms in \eqref{eq:22b-very-last} by a constant multiple of $\gamma^{6 - 3(2 - \un{\kappa}/2) - \bar{\kappa}} = \gamma^{3 {\un\kappa}/2 - \bar{\kappa}}$. Taking $\bar{\kappa}$ sufficiently small, this gives the required bound \eqref{eq:CG-bound2}.
\end{proof}

Applying \eqref{eq:CG-bound2}, we bound \eqref{eq:last-term} by a positive power of $\gamma$. This finishes the proof of the required bound \eqref{eq:model_bound} for the element $\tau$.

\subsubsection{The element $\tau = \protect\<31b>$}

The definition \eqref{eq:Pi-rest} and the expansion \cite[Eq.~2.15]{Martingales} allow to represent the map $\Pi^{\ga, \fa}_{z}\tau$ diagrammatically as
\begin{align}\label{eq:fourth-symbol}
\iota_\eps \bigl(\Pi^{\ga, \fa}_{z}\tau\bigr)(\phi_z^\lambda)
&\;=\; 
\begin{tikzpicture}[scale=0.35, baseline=-0.5cm]
	\node at (0,-5.1)  [root] (root) {};
	\node at (0,-5.1) [rootlab] {$z$};
	\node at (0,0) [dot] (int) {};
	\node at (-2,2) [var_blue] (left) {};
	\node at (0,2.9)  [var_blue] (cent) {};
	\node at (2,2)  [var_blue] (right) {};
	\node at (0,-2.9) [dot] (cent1) {};
	\node at (2,-0.9)  [var_blue] (right1) {};
	\draw[testfcn] (cent1) to (root);
	\draw[keps] (left) to node[labl,pos=0.45] {\tiny 3,0} (int);
	\draw[keps] (right) to node[labl,pos=0.45] {\tiny 3,0} (int);
	\draw[keps] (cent) to node[labl,pos=0.45] {\tiny 3,0} (int);
	\draw[keps] (int) to node[labl,pos=0.45] {\tiny 3,1} (cent1);
	\draw[keps] (right1) to node[labl,pos=0.45] {\tiny 3,0} (cent1);
\end{tikzpicture}
\; +\;3\,
\begin{tikzpicture}[scale=0.35, baseline=-0.5cm]
	\node at (0,-5.1)  [root] (root) {};
	\node at (0,-5.1) [rootlab] {$z$};
	\node at (0,0) [dot] (int) {};
	\node at (-2,2) [var_blue] (left) {};
	\node at (0,2.9)  [var_blue] (cent) {};
	\node at (0,-2.9) [dot] (cent1) {};
	\node at (2.5,-1.45)  [var_very_blue] (right1) {};
	\draw[testfcn] (cent1) to (root);
	\draw[keps] (left) to node[labl,pos=0.5] {\tiny 3,0} (int);
	\draw[keps] (cent) to node[labl,pos=0.4] {\tiny 3,0} (int);
	\draw[keps] (int) to node[labl,pos=0.45] {\tiny 3,1} (cent1);
	\draw[keps] (right1) to node[labl,pos=0.45] {\tiny 3,0} (int);
	\draw[keps] (right1) to node[labl,pos=0.45] {\tiny 3,0} (cent1);
\end{tikzpicture}
\; +\;3\,
\begin{tikzpicture}[scale=0.35, baseline=-0.5cm]
	\node at (0,-5.1)  [root] (root) {};
	\node at (0,-5.1) [rootlab] {$z$};
	\node at (0,0)  [dot] (int) {};
	\node at (-1.3,2.3)  [var_very_pink] (left) {};
	\node at (2.1,2.3)  [var_blue] (right) {};
	\node at (0,-2.9) [dot] (cent1) {};
	\node at (2,-0.9)  [var_blue] (right1) {};
	\draw[testfcn] (cent1) to (root);
	\draw[keps] (left) to[bend left=60] node[labl,pos=0.45] {\tiny 3,0} (int);
	\draw[keps] (left) to[bend left=-60] node[labl,pos=0.45] {\tiny 3,0} (int);
	\draw[keps] (right) to[bend left=30] node[labl,pos=0.45] {\tiny 3,0} (int);
	\draw[keps] (int) to node[labl,pos=0.45] {\tiny 3,1} (cent1);
	\draw[keps] (right1) to node[labl,pos=0.45] {\tiny 3,0} (cent1);
\end{tikzpicture}
 \\[0.2cm]
& \nonumber
\qquad \; +\;3\,
\begin{tikzpicture}[scale=0.35, baseline=-0.5cm]
	\node at (0,-5.1)  [root] (root) {};
	\node at (0,-5.1) [rootlab] {$z$};
	\node at (0,0)  [dot] (int) {};
	\node at (-1.3,2.3)  [var_very_pink] (left) {};
	\node at (0,-2.9) [dot] (cent1) {};
	\node at (2.5,-1.45)  [var_very_blue] (right1) {};
	\draw[testfcn] (cent1) to (root);
	\draw[keps] (left) to[bend left=60] node[labl,pos=0.45] {\tiny 3,0} (int);
	\draw[keps] (left) to[bend left=-60] node[labl,pos=0.45] {\tiny 3,0} (int);
	\draw[keps] (right1) to node[labl,pos=0.45] {\tiny 3,0} (int);
	\draw[keps] (int) to node[labl,pos=0.45] {\tiny 3,1} (cent1);
	\draw[keps] (right1) to node[labl,pos=0.45] {\tiny 3,0} (cent1);
\end{tikzpicture}
\; +\;3\,
\begin{tikzpicture}[scale=0.35, baseline=-0.5cm]
	\node at (0,-5.1)  [root] (root) {};
	\node at (0,-5.1) [rootlab] {$z$};
	\node at (0,0)  [dot] (int) {};
	\node at (-1.3,2.3)  [var_very_blue] (left) {};
	\node at (2,2)  [var_blue] (right) {};
	\node at (0,-2.9) [dot] (cent1) {};
	\draw[testfcn] (cent1) to (root);
	\draw[keps] (left) to[bend left=60] node[labl,pos=0.45] {\tiny 3,0} (int);
	\draw[keps] (left) to[bend left=-60] node[labl,pos=0.45] {\tiny 3,0} (int);
	\draw[keps] (right) to[bend left=30] node[labl,pos=0.45] {\tiny 3,0} (int);
	\draw[keps] (int) to node[labl,pos=0.45] {\tiny 3,1} (cent1);
	\draw[keps] (left) to[bend left=-80] node[labl,pos=0.45] {\tiny 3,0} (cent1);
\end{tikzpicture} 
\; +\;
\begin{tikzpicture}[scale=0.35, baseline=-0.5cm]
	\node at (0,-5.1)  [root] (root) {};
	\node at (0,-5.1) [rootlab] {$z$};
	\node at (0,0)  [dot] (int) {};
	\node at (0,2.9)  [var_very_blue] (left) {};
	\node at (0,-2.9) [dot] (cent1) {};
	\node at (2,-0.9)  [var_blue] (right1) {};
	\draw[testfcn] (cent1) to (root);
	\draw[keps] (left) to[bend left=90] node[labl,pos=0.6] {\tiny 3,0} (int);
	\draw[keps] (left) to[bend left=-90] node[labl,pos=0.6] {\tiny 3,0} (int);
	\draw[keps] (left) to node[labl,pos=0.4] {\tiny 3,0} (int);
	\draw[keps] (int) to node[labl,pos=0.45] {\tiny 3,1} (cent1);
	\draw[keps] (right1) to node[labl,pos=0.45] {\tiny 3,0} (cent1);
\end{tikzpicture}
\; +\;
\begin{tikzpicture}[scale=0.35, baseline=-0.5cm]
	\node at (0,-5.1)  [root] (root) {};
	\node at (0,-5.3) {$$};
	\node at (0,-5.1) [rootlab] {$z$};
	\node at (0,0)  [dot] (int) {};
	\node at (0,2.9)  [var_very_blue] (left) {};
	\node at (0,-2.9) [dot] (cent1) {};
	\draw[testfcn] (cent1) to (root);
	\draw[keps] (left) to[bend left=90] node[labl,pos=0.6] {\tiny 3,0} (int);
	\draw[keps] (left) to[bend left=-90] node[labl,pos=0.6] {\tiny 3,0} (int);
	\draw[keps] (left) to node[labl,pos=0.4] {\tiny 3,0} (int);
	\draw[keps] (int) to node[labl,pos=0.45] {\tiny 3,1} (cent1);
	\draw[keps] (left) to[bend left=100] node[labl,pos=0.55] {\tiny 3,0} (cent1);
\end{tikzpicture},
\end{align}
where as before the arrow ``\,\tikz[baseline=-0.1cm] \draw[keps] (0,0) to node[labl,pos=0.45] {\tiny 3,1} (1,0);\,'' represents the positively renormalised kernel in \eqref{eq:Pi-rest}.

\cite[Cor.~5.6]{Martingales} allows to bound the moments of the first two diagrams in \eqref{eq:fourth-symbol} by a constant multiple of $(\lambda \vee \emezo)^{-\kappa}$, which yields the required bound \eqref{eq:model_bound}. Analysing the contractions of two and three vertices as before, all the other diagrams, except the last one, are bounded using \cite[Cor.~5.6]{Martingales} by $c_\ga (\lambda \vee \emezo)^{-\kappa}$ for a vanishing $c_\ga$ as $\ga \to 0$.

To bound the last diagram in \eqref{eq:fourth-symbol}, we use a positive power of $\emezo$ to improve the singularity of the kernel:
\begin{equation*}
\begin{tikzpicture}[scale=0.35, baseline=-0.5cm]
	\node at (0,-5.1)  [root] (root) {};
	\node at (0,-5.1) [rootlab] {$z$};
	\node at (0,0)  [dot] (int) {};
	\node at (0,2.9)  [var_very_blue] (left) {};
	\node at (0,-2.9) [dot] (cent1) {};
	\draw[testfcn] (cent1) to (root);
	\draw[keps] (left) to[bend left=90] node[labl,pos=0.6] {\tiny 3,0} (int);
	\draw[keps] (left) to[bend left=-90] node[labl,pos=0.6] {\tiny 3,0} (int);
	\draw[keps] (left) to node[labl,pos=0.4] {\tiny 3,0} (int);
	\draw[keps] (int) to node[labl,pos=0.45] {\tiny 3,1} (cent1);
	\draw[keps] (left) to[bend left=100] node[labl,pos=0.55] {\tiny 3,0} (cent1);
\end{tikzpicture}
\;=\; \emezo^{3 a - 9}
\begin{tikzpicture}[scale=0.35, baseline=-0.5cm]
	\node at (0,-5.1)  [root] (root) {};
	\node at (0,-5.1) [rootlab] {$z$};
	\node at (0,0)  [dot] (int) {};
	\node at (0,2.9)  [var_very_blue] (left) {};
	\node at (0,-2.9) [dot] (cent1) {};
	\draw[testfcn] (cent1) to (root);
	\draw[keps] (left) to[bend left=90] node[labl,pos=0.6] {\tiny a,0} (int);
	\draw[keps] (left) to[bend left=-90] node[labl,pos=0.6] {\tiny a,0} (int);
	\draw[keps] (left) to node[labl,pos=0.4] {\tiny a,0} (int);
	\draw[keps] (int) to node[labl,pos=0.45] {\tiny 3,1} (cent1);
	\draw[keps] (left) to[bend left=100] node[labl,pos=0.55] {\tiny 3,0} (cent1);
\end{tikzpicture},
\end{equation*}
for $a < \frac{5}{3}$. Then \cite[Cor.~5.6]{Martingales} allows to bound the right-hand side by a constant times
\begin{equation*}
\emezo^{3a - 9} (\lambda \vee \emezo)^{4 - 3 a} \Bigl( \eps^{9-\bar\kappa} \emezo^{-5} + \eps^{\frac92-\bar\kappa} + \eps^{\frac{27}{4}} \Bigr), 
\end{equation*}
for any $\bar\kappa >0$. Choosing appropriate values of $a$ and $\bar \kappa$, we can estimate this by $c_\ga (\lambda \vee \emezo)^{-\kappa}$, where $c_\ga$ vanishes as $\ga \to 0$.

\subsubsection{The element $\tau = \protect\<32b>$}

Using \eqref{eq:Pi-rest} and \cite[Eq.~2.15]{Martingales} we can write
\begin{align} \label{eq:sixth-symbol}
&\iota_\eps \bigl(\Pi^{\ga, \fa}_{z}\tau\bigr)(\phi_z^\lambda)
\;=\; 
\begin{tikzpicture}[scale=0.35, baseline=-0.5cm]
	\node at (0,-5.1)  [root] (root) {};
	\node at (0,-5.1) [rootlab] {$z$};
	\node at (0,0) [dot] (int) {};
	\node at (-2,2) [var_blue] (left) {};
	\node at (0,2.9)  [var_blue] (cent) {};
	\node at (2,2)  [var_blue] (right) {};
	\node at (0,-2.9) [dot] (cent1) {};
	\node at (2,-0.9)  [var_blue] (right1) {};
	\node at (-2,-0.9)  [var_blue] (left1) {};
	\draw[testfcn] (cent1) to (root);
	\draw[keps] (left) to node[labl,pos=0.45] {\tiny 3,0} (int);
	\draw[keps] (right) to node[labl,pos=0.45] {\tiny 3,0} (int);
	\draw[keps] (cent) to node[labl,pos=0.45] {\tiny 3,0} (int);
	\draw[keps] (int) to node[labl,pos=0.45] {\tiny 3,1} (cent1);
	\draw[keps] (right1) to node[labl,pos=0.45] {\tiny 3,0} (cent1);
	\draw[keps] (left1) to node[labl,pos=0.45] {\tiny 3,0} (cent1);
\end{tikzpicture}
\;+\;6\;
\begin{tikzpicture}[scale=0.35, baseline=-0.5cm]
	\node at (0,-5.1)  [root] (root) {};
	\node at (0,-5.1) [rootlab] {$z$};
	\node at (0,0)  [dot] (int) {};
	\node at (2.5,-1.45)  [var_very_blue] (right1) {};
	\node at (-2,-0.9)  [var_blue] (left1) {};
	\node at (-2,2) [var_blue] (left) {};
	\node at (0,2.9)  [var_blue] (cent) {};
	\draw[testfcn] (cent1) to (root);
	\draw[keps] (left) to node[labl,pos=0.45] {\tiny 3,0} (int);
	\draw[keps] (cent) to node[labl,pos=0.45] {\tiny 3,0} (int);
	\draw[keps] (right1) to node[labl,pos=0.45] {\tiny 3,0} (int);
	\draw[keps] (int) to node[labl,pos=0.45] {\tiny 3,1} (cent1);
	\draw[keps] (right1) to node[labl,pos=0.45] {\tiny 3,0} (cent1);
	\draw[keps] (left1) to node[labl,pos=0.45] {\tiny 3,0} (cent1);
\end{tikzpicture}
\; +\;3\,
\begin{tikzpicture}[scale=0.35, baseline=-0.5cm]
	\node at (0,-5.1)  [root] (root) {};
	\node at (0,-5.1) [rootlab] {$z$};
	\node at (0,0)  [dot] (int) {};
	\node at (-1.3,2.3)  [var_very_pink] (left) {};
	\node at (2.1,2.3)  [var_blue] (right) {};
	\node at (0,-2.9) [dot] (cent1) {};
	\node at (2,-0.9)  [var_blue] (right1) {};
	\node at (-2,-0.9)  [var_blue] (left1) {};
	\draw[testfcn] (cent1) to (root);
	\draw[keps] (left) to[bend left=60] node[labl,pos=0.45] {\tiny 3,0} (int);
	\draw[keps] (left) to[bend left=-60] node[labl,pos=0.45] {\tiny 3,0} (int);
	\draw[keps] (right) to[bend left=30] node[labl,pos=0.45] {\tiny 3,0} (int);
	\draw[keps] (int) to node[labl,pos=0.45] {\tiny 3,1} (cent1);
	\draw[keps] (right1) to node[labl,pos=0.45] {\tiny 3,0} (cent1);
	\draw[keps] (left1) to node[labl,pos=0.45] {\tiny 3,0} (cent1);
\end{tikzpicture} 
\;+\; 
\begin{tikzpicture}[scale=0.35, baseline=-0.5cm]
	\node at (0,-5.1)  [root] (root) {};
	\node at (0,-5.1) [rootlab] {$z$};
	\node at (0,0) [dot] (int) {};
	\node at (-2,2) [var_blue] (left) {};
	\node at (0,2.9) [var_blue] (cent) {};
	\node at (2,2) [var_blue] (right) {};
	\node at (0,-2.9) [dot] (cent1) {};
	\node at (2.9,-1.7) [var_very_pink] (right1) {};
	\draw[testfcn] (cent1) to (root);
	\draw[keps] (left) to node[labl,pos=0.45] {\tiny 3,0} (int);
	\draw[keps] (right) to node[labl,pos=0.45] {\tiny 3,0} (int);
	\draw[keps] (cent) to node[labl,pos=0.45] {\tiny 3,0} (int);
	\draw[keps] (int) to node[labl,pos=0.45] {\tiny 3,1} (cent1);
	\draw[keps] (right1) to[bend left=-30] node[labl,pos=0.45] {\tiny 3,0} (cent1);
	\draw[keps] (right1) to[bend left=30] node[labl,pos=0.45] {\tiny 3,0} (cent1);
\end{tikzpicture} \\[0.2cm] \nonumber
&\qquad 
\;+\;3\,
\begin{tikzpicture}[scale=0.35, baseline=-0.5cm]
	\node at (0,-5.1)  [root] (root) {};
	\node at (0,-5.1) [rootlab] {$z$};
	\node at (0,0)  [dot] (int) {};
	\node at (-1.3,2.3)  [var_very_pink] (left) {};
	\node at (2.1,2.3)  [var_blue] (right) {};
	\node at (0,-2.9) [dot] (cent1) {};
	\node at (2.9,-1.7) [var_very_pink] (right1) {};
	\draw[testfcn] (cent1) to (root);
	\draw[keps] (left) to[bend left=60] node[labl,pos=0.45] {\tiny 3,0} (int);
	\draw[keps] (left) to[bend left=-60] node[labl,pos=0.45] {\tiny 3,0} (int);
	\draw[keps] (int) to node[labl,pos=0.45] {\tiny 3,1} (cent1);
	\draw[keps] (right1) to[bend left=-30] node[labl,pos=0.45] {\tiny 3,0} (cent1);
	\draw[keps] (right1) to[bend left=30] node[labl,pos=0.45] {\tiny 3,0} (cent1);
	\draw[keps] (right) to[bend left=30] node[labl,pos=0.45] {\tiny 3,0} (int);
\end{tikzpicture}
\; +\; 3 \left(
2\;
\begin{tikzpicture}[scale=0.35, baseline=-0.5cm]
	\node at (0,-5.1)  [root] (root) {};
	\node at (0,-5.1) [rootlab] {$z$};
	\node at (0,0) [dot] (int) {};
	\node at (0,2.9)  [var_blue] (cent) {};
	\node at (0,-2.9) [dot] (cent1) {};
	\node at (2.5,-1.45)  [var_very_blue] (right1) {};
	\node at (-2.5,-1.45)  [var_very_blue] (left1) {};
	\draw[testfcn] (cent1) to (root);
	\draw[keps] (left1) to node[labl,pos=0.45] {\tiny 3,0} (int);
	\draw[keps] (right1) to node[labl,pos=0.45] {\tiny 3,0} (int);
	\draw[keps] (cent) to node[labl,pos=0.45] {\tiny 3,0} (int);
	\draw[keps] (int) to node[labl,pos=0.45] {\tiny 3,1} (cent1);
	\draw[keps] (right1) to node[labl,pos=0.45] {\tiny 3,0} (cent1);
	\draw[keps] (left1) to node[labl,pos=0.45] {\tiny 3,0} (cent1);
\end{tikzpicture}
\; -\; \fc_\ga'' \;
\begin{tikzpicture}[scale=0.35, baseline=-0.7cm]
	\node at (0,-5.1)  [root] (root) {};
	\node at (0,-5.1) [rootlab] {$z$};
	\node at (0,0)  [var_blue] (cent) {};
	\node at (0,-2.9) [dot] (cent1) {};
	\draw[testfcn] (cent1) to (root);
	\draw[keps] (cent) to node[labl,pos=0.45] {\tiny 3,0} (cent1);
\end{tikzpicture} \right)
\; +\;6\,
\begin{tikzpicture}[scale=0.35, baseline=-0.5cm]
	\node at (0,-5.1)  [root] (root) {};
	\node at (0,-5.1) [rootlab] {$z$};
	\node at (0,0)  [dot] (int) {};
	\node at (-1.3,2.3)  [var_very_pink] (left) {};
	\node at (0,-2.9) [dot] (cent1) {};
	\node at (2.5,-1.45)  [var_very_blue] (right1) {};
	\node at (-2,-0.9)  [var_blue] (left1) {};
	\draw[testfcn] (cent1) to (root);
	\draw[keps] (left) to[bend left=60] node[labl,pos=0.45] {\tiny 3,0} (int);
	\draw[keps] (left) to[bend left=-60] node[labl,pos=0.45] {\tiny 3,0} (int);
	\draw[keps] (right1) to node[labl,pos=0.45] {\tiny 3,0} (int);
	\draw[keps] (int) to node[labl,pos=0.45] {\tiny 3,1} (cent1);
	\draw[keps] (right1) to node[labl,pos=0.45] {\tiny 3,0} (cent1);
	\draw[keps] (left1) to node[labl,pos=0.45] {\tiny 3,0} (cent1);
\end{tikzpicture} \\[0.2cm] \nonumber
&\qquad
\; +\;
\begin{tikzpicture}[scale=0.35, baseline=-0.5cm]
	\node at (0,-5.1)  [root] (root) {};
	\node at (0,-5.1) [rootlab] {$z$};
	\node at (0,0)  [dot] (int) {};
	\node at (0,2.9)  [var_very_blue] (left) {};
	\node at (0,-2.9) [dot] (cent1) {};
	\node at (2,-0.9)  [var_blue] (right1) {};
	\node at (-2,-0.9)  [var_blue] (left1) {};
	\draw[testfcn] (cent1) to (root);
	\draw[keps] (left) to[bend left=90] node[labl,pos=0.4] {\tiny 3,0} (int);
	\draw[keps] (left) to[bend left=-90] node[labl,pos=0.4] {\tiny 3,0} (int);
	\draw[keps] (left) to node[labl,pos=0.6] {\tiny 3,0} (int);
	\draw[keps] (int) to node[labl,pos=0.45] {\tiny 3,1} (cent1);
	\draw[keps] (right1) to node[labl,pos=0.45] {\tiny 3,0} (cent1);
	\draw[keps] (left1) to node[labl,pos=0.45] {\tiny 3,0} (cent1);
\end{tikzpicture}
\; +\;3\,
\begin{tikzpicture}[scale=0.35, baseline=-0.5cm]
	\node at (0,-5.1)  [root] (root) {};
	\node at (0,-5.1) [rootlab] {$z$};
	\node at (0,-4.5) {$$};
	\node at (0,0) [dot] (int) {};
	\node at (-2,2) [var_blue] (left) {};
	\node at (0,2.9)  [var_blue] (cent) {};
	\node at (0,-2.9) [dot] (cent1) {};
	\node at (3,-1.45)  [var_very_blue] (right1) {};
	\draw[testfcn] (cent1) to (root);
	\draw[keps] (left) to node[labl,pos=0.45] {\tiny 3,0} (int);
	\draw[keps] (right1) to[bend left=-30] node[labl,pos=0.45] {\tiny 3,0} (int);
	\draw[keps] (cent) to node[labl,pos=0.45] {\tiny 3,0} (int);
	\draw[keps] (int) to node[labl,pos=0.45] {\tiny 3,1} (cent1);
	\draw[keps] (right1) to[bend left=-30] node[labl,pos=0.45] {\tiny 3,0} (cent1);
	\draw[keps] (right1) to[bend left=30] node[labl,pos=0.45] {\tiny 3,0} (cent1);
\end{tikzpicture}
\; +\;6\,
\begin{tikzpicture}[scale=0.35, baseline=-0.5cm]
	\node at (0,-5.1)  [root] (root) {};
	\node at (0,-5.1) [rootlab] {$z$};
	\node at (0,0)  [dot] (int) {};
	\node at (-1.3,2.3)  [var_very_blue] (left) {};
	\node at (2.1,2.3)  [var_blue] (right) {};
	\node at (0,-2.9) [dot] (cent1) {};
	\node at (2,-0.9)  [var_blue] (right1) {};
	\draw[testfcn] (cent1) to (root);s
	\draw[keps] (left) to[bend left=60] node[labl,pos=0.45] {\tiny 3,0} (int);
	\draw[keps] (left) to[bend left=-60] node[labl,pos=0.45] {\tiny 3,0} (int);
	\draw[keps] (right) to[bend left=30] node[labl,pos=0.45] {\tiny 3,0} (int);
	\draw[keps] (int) to node[labl,pos=0.45] {\tiny 3,1} (cent1);
	\draw[keps] (left) to[bend left=-80] node[labl,pos=0.45] {\tiny 3,0} (cent1);
	\draw[keps] (right1) to node[labl,pos=0.45] {\tiny 3,0} (cent1);
\end{tikzpicture} 
\; +\;
\begin{tikzpicture}[scale=0.35, baseline=-0.5cm]
	\node at (0,-5.1)  [root] (root) {};
	\node at (0,-5.1) [rootlab] {$z$};
	\node at (0,0)  [dot] (int) {};
	\node at (0,2.9)  [var_very_blue] (left) {};
	\node at (0,-2.9) [dot] (cent1) {};
	\node at (2.9,-1.7) [var_very_pink] (right1) {};
	\draw[testfcn] (cent1) to (root);
	\draw[keps] (left) to[bend left=90] node[labl,pos=0.4] {\tiny 3,0} (int);
	\draw[keps] (left) to[bend left=-90] node[labl,pos=0.4] {\tiny 3,0} (int);
	\draw[keps] (left) to node[labl,pos=0.6] {\tiny 3,0} (int);
	\draw[keps] (int) to node[labl,pos=0.45] {\tiny 3,1} (cent1);
	\draw[keps] (right1) to[bend left=-30] node[labl,pos=0.45] {\tiny 3,0} (cent1);
	\draw[keps] (right1) to[bend left=30] node[labl,pos=0.45] {\tiny 3,0} (cent1);
\end{tikzpicture} \\[0.2cm] \nonumber
&\qquad\; +\;6\,
\begin{tikzpicture}[scale=0.35, baseline=-0.5cm]
	\node at (0,-5.1)  [root] (root) {};
	\node at (0,-5.1) [rootlab] {$z$};
	\node at (0,0)  [dot] (int) {};
	\node at (-1.3,2.3)  [var_very_blue] (left) {};
	\node at (0,-2.9) [dot] (cent1) {};
	\node at (2.5,-1.45)  [var_very_blue] (right1) {};
	\draw[testfcn] (cent1) to (root);
	\draw[keps] (left) to[bend left=60] node[labl,pos=0.45] {\tiny 3,0} (int);
	\draw[keps] (left) to[bend left=-60] node[labl,pos=0.45] {\tiny 3,0} (int);
	\draw[keps] (right1) to node[labl,pos=0.45] {\tiny 3,0} (int);
	\draw[keps] (int) to node[labl,pos=0.45] {\tiny 3,1} (cent1);
	\draw[keps] (right1) to node[labl,pos=0.45] {\tiny 3,0} (cent1);
	\draw[keps] (left) to[bend left=-80] node[labl,pos=0.55] {\tiny 3,0} (cent1);
\end{tikzpicture}
\; +\;3\,
\begin{tikzpicture}[scale=0.35, baseline=-0.5cm]
	\node at (0,-5.1)  [root] (root) {};
	\node at (0,-5.1) [rootlab] {$z$};
	\node at (0,0)  [dot] (int) {};
	\node at (-1.3,2.3)  [var_very_pink] (left) {};
	\node at (0,-2.9) [dot] (cent1) {};
	\node at (3,-1.45) [var_very_blue] (right1) {};
	\draw[testfcn] (cent1) to (root);
	\draw[keps] (left) to[bend left=60] node[labl,pos=0.45] {\tiny 3,0} (int);
	\draw[keps] (left) to[bend left=-60] node[labl,pos=0.45] {\tiny 3,0} (int);
	\draw[keps] (right1) to[bend left=-20] node[labl,pos=0.45] {\tiny 3,0} (int);
	\draw[keps] (int) to node[labl,pos=0.45] {\tiny 3,1} (cent1);
	\draw[keps] (right1) to[bend left=-30] node[labl,pos=0.45] {\tiny 3,0} (cent1);
	\draw[keps] (right1) to[bend left=30] node[labl,pos=0.45] {\tiny 3,0} (cent1);
\end{tikzpicture}
\; +\; 2\,
\begin{tikzpicture}[scale=0.35, baseline=-0.5cm]
	\node at (0,-5.1)  [root] (root) {};
	\node at (0,-5.1) [rootlab] {$z$};
	\node at (0,0)  [dot] (int) {};
	\node at (0,2.9)  [var_very_blue] (left) {};
	\node at (0,-2.9) [dot] (cent1) {};
	\node at (2,-0.9)  [var_blue] (right1) {};
	\draw[testfcn] (cent1) to (root);
	\draw[keps] (left) to[bend left=90] node[labl,pos=0.4] {\tiny 3,0} (int);
	\draw[keps] (left) to[bend left=-90] node[labl,pos=0.4] {\tiny 3,0} (int);
	\draw[keps] (left) to node[labl,pos=0.6] {\tiny 3,0} (int);
	\draw[keps] (int) to node[labl,pos=0.45] {\tiny 3,1} (cent1);
	\draw[keps] (left) to[bend right=100] node[labl,pos=0.55] {\tiny 3,0} (cent1);
	\draw[keps] (right1) to node[labl,pos=0.45] {\tiny 3,0} (cent1);
\end{tikzpicture}
\; +\;
\begin{tikzpicture}[scale=0.35, baseline=-0.5cm]
	\node at (0,-5.1)  [root] (root) {};
	\node at (0,-5.1) [rootlab] {$z$};
	\node at (0,0)  [dot] (int) {};
	\node at (0,2.9)  [var_very_blue] (left) {};
	\node at (0,-2.9) [dot] (cent1) {};
	\draw[testfcn] (cent1) to (root);
	\draw[keps] (left) to[bend left=90] node[labl,pos=0.4] {\tiny 3,0} (int);
	\draw[keps] (left) to[bend left=-90] node[labl,pos=0.4] {\tiny 3,0} (int);
	\draw[keps] (left) to node[labl,pos=0.6] {\tiny 3,0} (int);
	\draw[keps] (int) to node[labl,pos=0.45] {\tiny 3,1} (cent1);
	\draw[keps] (left) to[bend left=100] node[labl,pos=0.55] {\tiny 3,0} (cent1);
	\draw[keps] (left) to[bend right=100] node[labl,pos=0.55] {\tiny 3,0} (cent1);
\end{tikzpicture}.
\end{align}
Using \cite[Cor.~5.6]{Martingales}, the moments of the first two diagrams are bounded by a constant multiple of $(\lambda \vee \emezo)^{-\bar \kappa}$ for any $\bar \kappa > 0$. Analysing contracted vertices as before, all the other diagrams, except the expression in the brackets, are bounded by $c_\ga (\lambda \vee \emezo)^{-\bar \kappa}$ for any $\bar \kappa > 0$ and for vanishing $c_\ga$ as $\ga \to 0$. Here, the contraction of five vertices is analysed in the same way as a contraction of three, with the only difference in the powers of $\emezo$ in multipliers.

Now, we will bound the expression in the brackets in \eqref{eq:sixth-symbol}. Recalling the definition of the kernel in \eqref{eq:Pi-rest}, we can write 
\begin{align*}
2\;
\begin{tikzpicture}[scale=0.35, baseline=-0.5cm]
	\node at (0,-5.1)  [root] (root) {};
	\node at (0,-5.1) [rootlab] {$z$};
	\node at (0,0) [dot] (int) {};
	\node at (0,2.9)  [var_blue] (cent) {};
	\node at (0,-2.9) [dot] (cent1) {};
	\node at (2.5,-1.45)  [var_very_blue] (right1) {};
	\node at (-2.5,-1.45)  [var_very_blue] (left1) {};
	\draw[testfcn] (cent1) to (root);
	\draw[keps] (left1) to node[labl,pos=0.45] {\tiny 3,0} (int);
	\draw[keps] (right1) to node[labl,pos=0.45] {\tiny 3,0} (int);
	\draw[keps] (cent) to node[labl,pos=0.45] {\tiny 3,0} (int);
	\draw[keps] (int) to node[labl,pos=0.45] {\tiny 3,1} (cent1);
	\draw[keps] (right1) to node[labl,pos=0.45] {\tiny 3,0} (cent1);
	\draw[keps] (left1) to node[labl,pos=0.45] {\tiny 3,0} (cent1);
\end{tikzpicture}
\; -\; \fc_\ga'' \;
\begin{tikzpicture}[scale=0.35, baseline=-0.7cm]
	\node at (0,-5.1)  [root] (root) {};
	\node at (0,-5.1) [rootlab] {$z$};
	\node at (0,0)  [var_blue] (cent) {};
	\node at (0,-2.9) [dot] (cent1) {};
	\draw[testfcn] (cent1) to (root);
	\draw[keps] (cent) to node[labl,pos=0.45] {\tiny 3,0} (cent1);
\end{tikzpicture}
\; = \; \left(
2\;
\begin{tikzpicture}[scale=0.35, baseline=-0.5cm]
	\node at (0,-5.1)  [root] (root) {};
	\node at (0,-5.1) [rootlab] {$z$};
	\node at (0,0) [dot] (int) {};
	\node at (0,2.9)  [var_blue] (cent) {};
	\node at (0,-2.9) [dot] (cent1) {};
	\node at (2.5,-1.45)  [var_very_blue] (right1) {};
	\node at (-2.5,-1.45)  [var_very_blue] (left1) {};
	\draw[testfcn] (cent1) to (root);
	\draw[keps] (left1) to node[labl,pos=0.45] {\tiny 3,0} (int);
	\draw[keps] (right1) to node[labl,pos=0.45] {\tiny 3,0} (int);
	\draw[keps] (cent) to node[labl,pos=0.45] {\tiny 3,0} (int);
	\draw[keps] (int) to node[labl,pos=0.45] {\tiny 3,0} (cent1);
	\draw[keps] (right1) to node[labl,pos=0.45] {\tiny 3,0} (cent1);
	\draw[keps] (left1) to node[labl,pos=0.45] {\tiny 3,0} (cent1);
\end{tikzpicture}
\; -\; \fc_\ga'' \;
\begin{tikzpicture}[scale=0.35, baseline=-0.7cm]
	\node at (0,-5.1)  [root] (root) {};
	\node at (0,-5.1) [rootlab] {$z$};
	\node at (0,0)  [var_blue] (cent) {};
	\node at (0,-2.9) [dot] (cent1) {};
	\draw[testfcn] (cent1) to (root);
	\draw[keps] (cent) to node[labl,pos=0.45] {\tiny 3,0} (cent1);
\end{tikzpicture} \right)
\; -\; 2\;
\begin{tikzpicture}[scale=0.35, baseline=-0.5cm]
	\node at (0,-5.1)  [root] (root) {};
	\node at (0,-5.1) [rootlab] {$z$};
	\node at (0,0) [dot] (int) {};
	\node at (0,2.9)  [var_blue] (cent) {};
	\node at (0,-2.9) [dot] (cent1) {};
	\node at (2.5,-1.45)  [var_very_blue] (right1) {};
	\node at (-2.5,-1.45)  [var_very_blue] (left1) {};
	\draw[testfcn] (cent1) to (root);
	\draw[keps] (left1) to node[labl,pos=0.45] {\tiny 3,0} (int);
	\draw[keps] (right1) to node[labl,pos=0.45] {\tiny 3,0} (int);
	\draw[keps] (cent) to node[labl,pos=0.45] {\tiny 3,0} (int);
	\draw[keps] (right1) to node[labl,pos=0.45] {\tiny 3,0} (cent1);
	\draw[keps] (left1) to node[labl,pos=0.45] {\tiny 3,0} (cent1);
	\draw [keps] (int) to[out=0,in=90] (3.5,-2.5) to [out=-90, in=0] node[labl,pos=0] {\tiny 3,0} (root);
\end{tikzpicture}\;.
\end{align*}
Applying \cite[Cor.~5.6]{Martingales}, the moments of the last diagram are bounded by a constant multiple of $(\lambda \vee \emezo)^{-\bar \kappa}$ for any $\bar \kappa > 0$. Similarly to \eqref{eq:22b-last}, we can write the expression in the brackets as
\begin{equation}\label{eq:sixth-symbol-last}
2\;
\begin{tikzpicture}[scale=0.35, baseline=-0.5cm]
	\node at (0,-5.1)  [root] (root) {};
	\node at (0,-5.1) [rootlab] {$z$};
	\node at (0,0) [dot] (int) {};
	\node at (0,2.9)  [var_blue] (cent) {};
	\node at (0,-2.9) [dot] (cent1) {};
	\node at (2.5,-1.45)  [var_red_square] (right1) {};
	\node at (-2.5,-1.45)  [var_red_square] (left1) {};
	\draw[testfcn] (cent1) to (root);
	\draw[keps] (left1) to node[labl,pos=0.45] {\tiny 3,0} (int);
	\draw[keps] (right1) to node[labl,pos=0.45] {\tiny 3,0} (int);
	\draw[keps] (cent) to node[labl,pos=0.45] {\tiny 3,0} (int);
	\draw[keps] (int) to node[labl,pos=0.45] {\tiny 3,0} (cent1);
	\draw[keps] (right1) to node[labl,pos=0.45] {\tiny 3,0} (cent1);
	\draw[keps] (left1) to node[labl,pos=0.45] {\tiny 3,0} (cent1);
\end{tikzpicture}
\;+\; 4\;
\begin{tikzpicture}[scale=0.35, baseline=-0.5cm]
	\node at (0,-5.1)  [root] (root) {};
	\node at (0,-5.1) [rootlab] {$z$};
	\node at (0,0) [dot] (int) {};
	\node at (0,2.9)  [var_blue] (cent) {};
	\node at (0,-2.9) [dot] (cent1) {};
	\node at (2.5,-1.45)  [var_red_square] (right1) {};
	\node at (-2.5,-1.45)  [var_red_triangle] (left1) {};
	\draw[testfcn] (cent1) to (root);
	\draw[keps] (left1) to node[labl,pos=0.45] {\tiny 3,0} (int);
	\draw[keps] (right1) to node[labl,pos=0.45] {\tiny 3,0} (int);
	\draw[keps] (cent) to node[labl,pos=0.45] {\tiny 3,0} (int);
	\draw[keps] (int) to node[labl,pos=0.45] {\tiny 3,0} (cent1);
	\draw[keps] (right1) to node[labl,pos=0.45] {\tiny 3,0} (cent1);
	\draw[keps] (left1) to node[labl,pos=0.45] {\tiny 3,0} (cent1);
\end{tikzpicture}
\;+\; \left(2\;
\begin{tikzpicture}[scale=0.35, baseline=-0.5cm]
	\node at (0,-5.1)  [root] (root) {};
	\node at (0,-5.1) [rootlab] {$z$};
	\node at (0,0) [dot] (int) {};
	\node at (0,2.9)  [var_blue] (cent) {};
	\node at (0,-2.9) [dot] (cent1) {};
	\node at (2.5,-1.45)  [var_red_triangle] (right1) {};
	\node at (-2.5,-1.45)  [var_red_triangle] (left1) {};
	\draw[testfcn] (cent1) to (root);
	\draw[keps] (left1) to node[labl,pos=0.45] {\tiny 3,0} (int);
	\draw[keps] (right1) to node[labl,pos=0.45] {\tiny 3,0} (int);
	\draw[keps] (cent) to node[labl,pos=0.45] {\tiny 3,0} (int);
	\draw[keps] (int) to node[labl,pos=0.45] {\tiny 3,0} (cent1);
	\draw[keps] (right1) to node[labl,pos=0.45] {\tiny 3,0} (cent1);
	\draw[keps] (left1) to node[labl,pos=0.45] {\tiny 3,0} (cent1);
\end{tikzpicture}
\; -\; \fc_\ga'' \;
\begin{tikzpicture}[scale=0.35, baseline=-0.7cm]
	\node at (0,-5.1)  [root] (root) {};
	\node at (0,-5.1) [rootlab] {$z$};
	\node at (0,0)  [var_blue] (cent) {};
	\node at (0,-2.9) [dot] (cent1) {};
	\draw[testfcn] (cent1) to (root);
	\draw[keps] (cent) to node[labl,pos=0.45] {\tiny 3,0} (cent1);
\end{tikzpicture}
\right),
\end{equation}
where the first two diagram are bounded using \cite[Cor.~5.6]{Martingales} by $c_\ga (\lambda \vee \emezo)^{-\kappa/2}$ with a constant $c_\ga$ vanishing as $\ga \to 0$.

Now, we will bound the expression in the brackets in \eqref{eq:sixth-symbol-last}. We write 
\begin{equation}\label{eq:sixth-symbol-last-again}
2\;
\begin{tikzpicture}[scale=0.35, baseline=-0.5cm]
	\node at (0,-5.1)  [root] (root) {};
	\node at (0,-5.1) [rootlab] {$z$};
	\node at (0,0) [dot] (int) {};
	\node at (0,2.9)  [var_blue] (cent) {};
	\node at (0,-2.9) [dot] (cent1) {};
	\node at (2.5,-1.45)  [var_red_triangle] (right1) {};
	\node at (-2.5,-1.45)  [var_red_triangle] (left1) {};
	\draw[testfcn] (cent1) to (root);
	\draw[keps] (left1) to node[labl,pos=0.45] {\tiny 3,0} (int);
	\draw[keps] (right1) to node[labl,pos=0.45] {\tiny 3,0} (int);
	\draw[keps] (cent) to node[labl,pos=0.45] {\tiny 3,0} (int);
	\draw[keps] (int) to node[labl,pos=0.45] {\tiny 3,0} (cent1);
	\draw[keps] (right1) to node[labl,pos=0.45] {\tiny 3,0} (cent1);
	\draw[keps] (left1) to node[labl,pos=0.45] {\tiny 3,0} (cent1);
\end{tikzpicture}
\; -\; \fc_\ga'' \;
\begin{tikzpicture}[scale=0.35, baseline=-0.7cm]
	\node at (0,-5.1)  [root] (root) {};
	\node at (0,-5.1) [rootlab] {$z$};
	\node at (0,0)  [var_blue] (cent) {};
	\node at (0,-2.9) [dot] (cent1) {};
	\draw[testfcn] (cent1) to (root);
	\draw[keps] (cent) to node[labl,pos=0.45] {\tiny 3,0} (cent1);
\end{tikzpicture}
\; =\;
\left(2\; \begin{tikzpicture}[scale=0.35, baseline=-0.5cm]
	\node at (0,-5.1)  [root] (root) {};
	\node at (0,-5.1) [rootlab] {$z$};
	\node at (0,0) [dot] (int) {};
	\node at (0,2.9)  [var_blue] (cent) {};
	\node at (0,-2.9) [dot] (cent1) {};
	\node at (2.5,-1.45)  [dot] (right1) {};
	\node at (-2.5,-1.45)  [dot] (left1) {};
	\draw[testfcn] (cent1) to (root);
	\draw[keps] (left1) to node[labl,pos=0.45] {\tiny 3,0} (int);
	\draw[keps] (right1) to node[labl,pos=0.45] {\tiny 3,0} (int);
	\draw[keps] (cent) to node[labl,pos=0.45] {\tiny 3,0} (int);
	\draw[keps] (int) to node[labl,pos=0.45] {\tiny 3,0} (cent1);
	\draw[keps] (right1) to node[labl,pos=0.45] {\tiny 3,0} (cent1);
	\draw[keps] (left1) to node[labl,pos=0.45] {\tiny 3,0} (cent1);
\end{tikzpicture}
\; -\; \fc_\ga'' \;
\begin{tikzpicture}[scale=0.35, baseline=-0.7cm]
	\node at (0,-5.1)  [root] (root) {};
	\node at (0,-5.1) [rootlab] {$z$};
	\node at (0,0)  [var_blue] (cent) {};
	\node at (0,-2.9) [dot] (cent1) {};
	\draw[testfcn] (cent1) to (root);
	\draw[keps] (cent) to node[labl,pos=0.45] {\tiny 3,0} (cent1);
\end{tikzpicture}
\right)
\; + \;
\Err_{\ga, \lambda}(z)
\end{equation}
where the error term $\Err_{\ga, \lambda}$ is defined via random kernel \eqref{eq:CG-def} and can be bounded as
\begin{equation*}
|\Err_{\ga, \lambda}(z)| \lesssim \sup_{z_1 \in D_\eps} \left| 2 \int_{D_\eps} \CG_{\ga} (z_1, z_2) \d z_2 - \fc_\ga'' \right|\; \sup_{z_2 \in D_\eps} \Bigl| \begin{tikzpicture}[scale=0.35, baseline=0cm]
	\node at (-2.3, 0.2)  [root] (root) {};
	\node at (-2.3, 0.2) [rootlab] {$z_2$};
	\node at (-5.5,0.2)  [var_blue] (left) {};	
	\draw[keps] (left) to node[labl,pos=0.45] {\tiny 3,0} (root);
\end{tikzpicture} \Bigr|
\end{equation*}
(we recall the renormalisation constant \eqref{eq:renorm-constant2}). Using \cite[Cor.~5.6]{Martingales} we can bound the high moments of the last supremum by a constant multiple of $\emezo^{-\frac{1}{2}}$, while Lemma~\ref{lem:CG-bounds} allows to bound the first supremum by a constant multiple of $\emezo^\theta$ with $\theta \in (\frac{1}{2}, 1)$. Hence, all high moments of the error term vanish as $\ga \to 0$.

It is left to bound the expression in the brackets in \eqref{eq:sixth-symbol-last-again}. For this, we define the kernel 
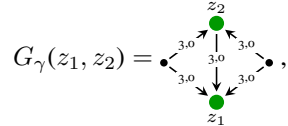
\begin{equation*}
G_{\ga} (z_1, z_2) = 
 \begin{tikzpicture}[scale=0.35, baseline=-0.2cm]
	\node at (0,1) [root] (int) {};
	\node at (0,1.6) {\scriptsize $z_2$};
	\node at (0,-2) [root] (cent1) {};
	\node at (0,-2) [rootlab] {$z_1$};
	\node at (2,-0.5)  [dot] (right1) {};
	\node at (-2,-0.5)  [dot] (left1) {};
	\draw[keps] (left1) to node[labl,pos=0.45] {\tiny 3,0} (int);
	\draw[keps] (right1) to node[labl,pos=0.45] {\tiny 3,0} (int);
	\draw[keps] (int) to node[labl,pos=0.45] {\tiny 3,0} (cent1);
	\draw[keps] (right1) to node[labl,pos=0.45] {\tiny 3,0} (cent1);
	\draw[keps] (left1) to node[labl,pos=0.45] {\tiny 3,0} (cent1);
\end{tikzpicture}\;, 
\end{equation*}
and we define for any smooth, compactly supported function $\psi : \R^4 \times \R^4 \to \R$ its ``negative renormalisation''
\begin{equation*}
\bigl(\SR_\ga G_\ga\bigr)(\psi) := \int_{D_\eps} \int_{D_\eps} G_\ga(z_1, z_2) \bigl( \psi(z_1, z_2) - \psi(z_1, z_1) \bigr) \d z_1 \d z_2.
\end{equation*}
This identity defined $\SR_\ga G_\ga$ as a distribution on $\R^4 \times \R^4$ (more precisely, $\SR_\ga G_\ga$ is a function in the first variable and a distribution in the second one). Then the expression in the brackets in \eqref{eq:sixth-symbol-last-again} may be written as
\begin{equation*}
\int_{D_\eps} \int_{D_\eps} \int_{D_\eps} \phi_z^\lambda(z_1) \bigl(\SR_\ga G_\ga\bigr)(z_1, z_2) \mywidetilde{\SK}^\ga(z_3 - z_2) \d \M_{\ga, \fa}(z_3) \d z_1 \d z_2.
\end{equation*}
We note that this expression is well defined, because the distribution $\SR_\ga G_\ga$ is convolved with smooth functions. It will be convenient to represent this expression as a diagram. For this, we denote the random kernel $G_\ga$ by an edge ``\,\tikz[baseline=-0.1cm] \draw[kernelBig] (0,0) to node[labl,pos=0.45] {\tiny 5,0} (1,0);\,'', and we denote $\SR_\ga G_\ga$ by ``\,\tikz[baseline=-0.1cm] \draw[kernelBig] (0,0) to node[labl,pos=0.45] {\tiny 5,-1} (1,0);\,''. Here, the label ``$5$'' refers to the order of singularity of $G_\ga$ (which can be proved similarly to \eqref{eq:CG-bound1}), and the label ``$-1$'' refers to the order of negative renormalisation (see \cite[Sec.~5]{Martingales}). Then the preceding expression can be represented as
\begin{equation*}
\begin{tikzpicture}[scale=0.35, baseline=0cm]
	\node at (0,0)  [root] (root) {};
	\node at (0,0) [rootlab] {$z$};
	\node at (2.2,0) [dot] (cent1) {};
	\node at (5.1,0) [dot] (int) {};
	\node at (7.7,0)  [var_blue] (cent) {};
	\draw[testfcn] (cent1) to (root);
	\draw[keps] (cent) to node[labl,pos=0.45] {\tiny 3,0} (int);
	\draw[kernelBig] (int) to node[labl,pos=0.45] {\tiny 5,-1} (cent1);
\end{tikzpicture}\;.
\end{equation*}
Applying \cite[Cor.~5.6]{Martingales}, the high enough moments of this expression are bounded by constant multiples of $(\lambda \vee \emezo)^{-\bar \kappa}$ for any $\bar \kappa > 0$.

\subsubsection{The element $\tau = \protect\<4b>$}

The definition \eqref{eq:model-Hermite} and the expansion \cite[Eq.~2.15]{Martingales} yield a diagrammatical representation of the map $\Pi^{\ga, \fa}_{z}\tau$:
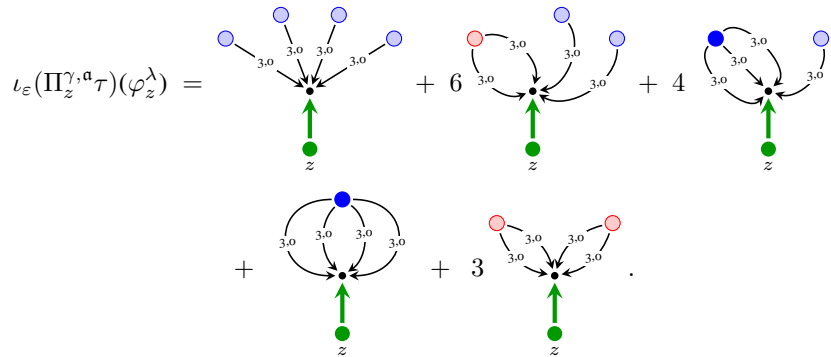
\begin{align} \nonumber
\iota_\eps \bigl(\Pi^{\ga, \fa}_{z}\tau\bigr)(\phi_z^\lambda) &\;=\;  
\begin{tikzpicture}[scale=0.35, baseline=0cm]
	\node at (0,-2.2)  [root] (root) {};
	\node at (0,-2.2) [rootlab] {$z$};
	\node at (0,-2.5) {$$};
	\node at (0,0)  [dot] (int) {};
	\node at (-3.2,2)  [var_blue] (left) {};
	\node at (-1.1,2.9)  [var_blue] (cent_left) {};
	\node at (1.1,2.9)  [var_blue] (cent_right) {};
	\node at (3.2,2)  [var_blue] (right) {};
	\draw[testfcn] (int) to (root);
	\draw[keps] (left) to node[labl,pos=0.45] {\tiny 3,0} (int);
	\draw[keps] (cent_left) to node[labl,pos=0.45] {\tiny 3,0} (int);
	\draw[keps] (cent_right) to node[labl,pos=0.45] {\tiny 3,0} (int);
	\draw[keps] (right) to node[labl,pos=0.4] {\tiny 3,0} (int);
\end{tikzpicture}
\; + \;6\,
\begin{tikzpicture}[scale=0.35, baseline=0cm]
	\node at (0,-2.2)  [root] (root) {};
	\node at (0,-2.2) [rootlab] {$z$};
	\node at (0,0)  [dot] (int) {};
	\node at (-2.2,2)  [var_very_pink] (left) {};
	\node at (1.1,2.9)  [var_blue] (cent) {};
	\node at (3.2,2)  [var_blue] (right) {};
	\draw[testfcn] (int) to (root);
	\draw[keps] (left) to[bend left=60] node[labl,pos=0.45] {\tiny 3,0} (int);
	\draw[keps] (left) to[bend left=-60] node[labl,pos=0.45] {\tiny 3,0} (int);
	\draw[keps] (cent) to[bend left=60] node[labl,pos=0.4] {\tiny 3,0} (int);
	\draw[keps] (right) to[bend left=60] node[labl,pos=0.4] {\tiny 3,0} (int);
\end{tikzpicture}
\; + \;4\,
\begin{tikzpicture}[scale=0.35, baseline=0cm]
	\node at (0,-2.2)  [root] (root) {};
	\node at (0,-2.2) [rootlab] {$z$};
	\node at (0,0)  [dot] (int) {};
	\node at (-2,2)  [var_very_blue] (left) {};
	\node at (2,2)  [var_blue] (right) {};
	\draw[testfcn] (int) to (root);
	\draw[keps] (left) to[bend left=90] node[labl,pos=0.45] {\tiny 3,0} (int);
	\draw[keps] (left) to[bend left=-90] node[labl,pos=0.45] {\tiny 3,0} (int);
	\draw[keps] (left) to node[labl,pos=0.4] {\tiny 3,0} (int);
	\draw[keps] (right) to[bend left=60] node[labl,pos=0.4] {\tiny 3,0} (int);
\end{tikzpicture}
\\[0.2cm]
&\qquad \label{eq:Psi4}
\; +\;
\begin{tikzpicture}[scale=0.35, baseline=0cm]
	\node at (0,-2.2)  [root] (root) {};
	\node at (0,-2.2) [rootlab] {$z$};
	\node at (0,0)  [dot] (int) {};
	\node at (0,2.9)  [var_very_blue] (left) {};
	\draw[testfcn] (int) to (root);
	\draw[keps] (left) to[out=0, in=0, distance=2.5cm] node[labl,pos=0.55] {\tiny 3,0} (int);
	\draw[keps] (left) to[out=180, in=180, distance=2.5cm] node[labl,pos=0.55] {\tiny 3,0} (int);
	\draw[keps] (left) to[out=-35, in=35] node[labl,pos=0.45] {\tiny 3,0} (int);
	\draw[keps] (left) to[out=225, in=135] node[labl,pos=0.45] {\tiny 3,0} (int);
\end{tikzpicture}
\; +\; \;3\,
\begin{tikzpicture}[scale=0.35, baseline=0cm]
	\node at (0,-2.2)  [root] (root) {};
	\node at (0,-2.2) [rootlab] {$z$};
	\node at (0,0)  [dot] (int) {};
	\node at (-2.2,2)  [var_very_pink] (left) {};
	\node at (2.2,2)  [var_very_pink] (right) {};
	\draw[testfcn] (int) to (root);
	\draw[keps] (left) to[bend left=30] node[labl,pos=0.45] {\tiny 3,0} (int);
	\draw[keps] (left) to[bend left=-30] node[labl,pos=0.45] {\tiny 3,0} (int);
	\draw[keps] (right) to[bend left=30] node[labl,pos=0.45] {\tiny 3,0} (int);
	\draw[keps] (right) to[bend left=-30] node[labl,pos=0.45] {\tiny 3,0} (int);
\end{tikzpicture}\;. \nonumber
\end{align}
The high enough moments of the first diagram are bounded using \cite[Cor.~5.6]{Martingales} by a constant multiplier of $(\lambda \vee \emezo)^{-2}$, which is the required bound \eqref{eq:model_bound}. Reducing the singularity of kernels in the same way as we did above, \cite[Cor.~5.6]{Martingales} allows to bound the moments of the other diagrams by $c_\gamma (\lambda \vee \emezo)^{-2 - \bar{\kappa}}$, for any $\bar{\kappa} > 0$ and where $c_\gamma$ vanishes as $\ga \to 0$. 

\subsubsection{The element $\tau = \protect\<5b>$} 
\label{sec:fifth-chaos}

Similarly to the previous element, we can write
\begin{align*}
	&\iota_\eps \bigl(\Pi^{\ga, \fa}_{z}\tau\bigr)(\phi_z^\lambda)
	\;=\; 
	\begin{tikzpicture}[scale=0.35, baseline=0cm]
		\node at (0,-2.2)  [root] (root) {};
		\node at (0,-2.2) [rootlab] {$z$};
		\node at (0,0)  [dot] (int) {};
		\node at (0,3.2)  [var_blue] (cent) {};
		\node at (-2.8,0.8)  [var_blue] (left) {};
		\node at (-2,2.5)  [var_blue] (cent_left) {};
		\node at (2,2.5)  [var_blue] (cent_right) {};
		\node at (2.8,0.8)  [var_blue] (right) {};
		\draw[testfcn] (int) to (root);
		\draw[keps] (left) to node[labl,pos=0.45] {\tiny 3,0} (int);
		\draw[keps] (cent_left) to node[labl,pos=0.45] {\tiny 3,0} (int);
		\draw[keps] (cent_right) to node[labl,pos=0.45] {\tiny 3,0} (int);
		\draw[keps] (right) to node[labl,pos=0.4] {\tiny 3,0} (int);
		\draw[keps] (cent) to node[labl,pos=0.4] {\tiny 3,0} (int);
	\end{tikzpicture}
	+ 10 
	\begin{tikzpicture}[scale=0.35, baseline=0cm]
		\node at (0,-2.2)  [root] (root) {};
		\node at (0,-2.2) [rootlab] {$z$};
		\node at (0,0)  [dot] (int) {};
		\node at (0,3.2)  [var_blue] (cent) {};
		\node at (-2,2)  [var_very_pink] (cent_left) {};
		\node at (2.1,2.8)  [var_blue] (cent_right) {};
		\node at (3.5,1.5)  [var_blue] (right) {};
		\draw[testfcn] (int) to (root);
		\draw[keps] (cent_left) to[bend left=45] node[labl,pos=0.45] {\tiny 3,0} (int);
		\draw[keps] (cent_left) to[bend left=-45] node[labl,pos=0.45] {\tiny 3,0} (int);
		\draw[keps] (cent) to[bend left=45] node[labl,pos=0.4] {\tiny 3,0} (int);
		\draw[keps] (right) to[bend left=45] node[labl,pos=0.4] {\tiny 3,0} (int);
		\draw[keps] (cent_right) to[bend left=45] node[labl,pos=0.4] {\tiny 3,0} (int);
	\end{tikzpicture} 
	+ 10
	\begin{tikzpicture}[scale=0.35, baseline=0cm]
		\node at (0,-2.2)  [root] (root) {};
		\node at (0,-2.2) [rootlab] {$z$};
		\node at (0,0)  [dot] (int) {};
		\node at (-2,2)  [var_very_blue] (left) {};
		\node at (1.5,2.5)  [var_blue] (cent_right) {};
		\node at (3,1.5)  [var_blue] (right) {};
		\draw[testfcn] (int) to (root);
		\draw[keps] (left) to[bend left=90] node[labl,pos=0.45] {\tiny 3,0} (int);
		\draw[keps] (left) to[bend left=-90] node[labl,pos=0.45] {\tiny 3,0} (int);
		\draw[keps] (left) to node[labl,pos=0.4] {\tiny 3,0} (int);
		\draw[keps] (right) to[bend left=60] node[labl,pos=0.4] {\tiny 3,0} (int);
		\draw[keps] (cent_right) to[bend left=45] node[labl,pos=0.4] {\tiny 3,0} (int);
	\end{tikzpicture} \\
& \qquad + 10
	\begin{tikzpicture}[scale=0.35, baseline=0cm]
		\node at (0,-2.2)  [root] (root) {};
		\node at (0,-2.2) [rootlab] {$z$};
		\node at (0,0)  [dot] (int) {};
		\node at (-2,2)  [var_very_blue] (left) {};
		\node at (2.5,0)  [var_very_pink] (right) {};
		\draw[testfcn] (int) to (root);
		\draw[keps] (left) to[bend left=90] node[labl,pos=0.45] {\tiny 3,0} (int);
		\draw[keps] (left) to[bend left=-90] node[labl,pos=0.45] {\tiny 3,0} (int);
		\draw[keps] (left) to node[labl,pos=0.4] {\tiny 3,0} (int);
		\draw[keps] (right) to[bend left=30] node[labl,pos=0.45] {\tiny 3,0} (int);
		\draw[keps] (right) to[bend left=-30] node[labl,pos=0.45] {\tiny 3,0} (int);
	\end{tikzpicture}
+ 5
	\begin{tikzpicture}[scale=0.35, baseline=0cm]
		\node at (0,-2.2)  [root] (root) {};
		\node at (0,-2.2) [rootlab] {$z$};
		\node at (0,0)  [dot] (int) {};
		\node at (-2,2)  [var_very_blue] (left) {};
		\node at (3,1.5)  [var_blue] (right) {};
		\draw[testfcn] (int) to (root);
		\draw[keps] (left) to[out=45, in=45, distance=2.5cm] node[labl,pos=0.55] {\tiny 3,0} (int);
		\draw[keps] (left) to[out=225, in=225, distance=2.5cm] node[labl,pos=0.55] {\tiny 3,0} (int);
		\draw[keps] (left) to[out=10, in=80] node[labl,pos=0.45] {\tiny 3,0} (int);
		\draw[keps] (left) to[out=270, in=180] node[labl,pos=0.45] {\tiny 3,0} (int);
		\draw[keps] (right) to[bend left=60] node[labl,pos=0.4] {\tiny 3,0} (int);
	\end{tikzpicture} 
	+
	\begin{tikzpicture}[scale=0.35, baseline=0cm]
		\node at (0,-2.2)  [root] (root) {};
		\node at (0,-2.2) [rootlab] {$z$};
		\node at (0,0)  [dot] (int) {};
		\node at (0,2.9)  [var_very_blue] (left) {};
		\draw[testfcn] (int) to (root);
		\draw[keps] (left) to[out=180, in=180, distance=3.2cm] node[labl,pos=0.55] {\tiny 3,0} (int);
		\draw[keps] (left) to[out=0, in=0, distance=3.2cm] node[labl,pos=0.55] {\tiny 3,0} (int);
		\draw[keps] (left) to[out=0, in=0, distance=1.5cm] node[labl,pos=0.4] {\tiny 3,0} (int);
		\draw[keps] (left) to[out=180, in=180, distance=1.5cm] node[labl,pos=0.4] {\tiny 3,0} (int);
		\draw[keps] (left) to node[labl,pos=0.45] {\tiny 3,0} (int);
	\end{tikzpicture}
	+ 15
	\begin{tikzpicture}[scale=0.35, baseline=0cm]
		\node at (0,-2.2)  [root] (root) {};
		\node at (0,-2.2) [rootlab] {$z$};
		\node at (0,0)  [dot] (int) {};
		\node at (-2,2)  [var_very_pink] (cent_left) {};
		\node at (2,2)  [var_very_pink] (cent_right) {};
		\node at (3.5,1.5)  [var_blue] (right) {};
		\draw[testfcn] (int) to (root);
		\draw[keps] (cent_left) to[bend left=45] node[labl,pos=0.45] {\tiny 3,0} (int);
		\draw[keps] (cent_left) to[bend left=-45] node[labl,pos=0.45] {\tiny 3,0} (int);
		\draw[keps] (cent_right) to[bend left=45] node[labl,pos=0.45] {\tiny 3,0} (int);
		\draw[keps] (cent_right) to[bend left=-45] node[labl,pos=0.45] {\tiny 3,0} (int);
		\draw[keps] (right) to[bend left=45] node[labl,pos=0.4] {\tiny 3,0} (int);
	\end{tikzpicture}.
\end{align*}
The first diagram does not satisfy Assumption~3(2) in \cite{Martingales}, which means that we cannot bound it uniformly in $\ga$. This is the reason why we assumed a weaker bound on this element in Definition~\ref{def:model}. Multiplying the diagram by $\gamma \approx \emezo^{1/3}$, we can decrease the order of singularity of one of the kernels:
\begin{align*} 
\gamma\;
	\begin{tikzpicture}[scale=0.35, baseline=0cm]
		\node at (0,-2.2)  [root] (root) {};
		\node at (0,-2.2) [rootlab] {$z$};
		\node at (0,0)  [dot] (int) {};
		\node at (0,3.2)  [var_blue] (cent) {};
		\node at (-2.8,0.8)  [var_blue] (left) {};
		\node at (-2,2.5)  [var_blue] (cent_left) {};
		\node at (2,2.5)  [var_blue] (cent_right) {};
		\node at (2.8,0.8)  [var_blue] (right) {};
		\draw[testfcn] (int) to (root);
		\draw[keps] (left) to node[labl,pos=0.45] {\tiny 3,0} (int);
		\draw[keps] (cent_left) to node[labl,pos=0.45] {\tiny 3,0} (int);
		\draw[keps] (cent_right) to node[labl,pos=0.45] {\tiny 3,0} (int);
		\draw[keps] (right) to node[labl,pos=0.4] {\tiny 3,0} (int);
		\draw[keps] (cent) to node[labl,pos=0.4] {\tiny 3,0} (int);
	\end{tikzpicture}
	\;=\;
	\begin{tikzpicture}[scale=0.35, baseline=0cm]
		\node at (0,-2.2)  [root] (root) {};
		\node at (0,-2.2) [rootlab] {$z$};
		\node at (0,0)  [dot] (int) {};
		\node at (0,3.2)  [var_blue] (cent) {};
		\node at (-2.8,0.8)  [var_blue] (left) {};
		\node at (-2,2.5)  [var_blue] (cent_left) {};
		\node at (2,2.5)  [var_blue] (cent_right) {};
		\node at (2.8,0.8)  [var_blue] (right) {};
		\draw[testfcn] (int) to (root);
		\draw[keps] (left) to node[labl,pos=0.45] {\tiny 8/3,0} (int);
		\draw[keps] (cent_left) to node[labl,pos=0.45] {\tiny 3,0} (int);
		\draw[keps] (cent_right) to node[labl,pos=0.45] {\tiny 3,0} (int);
		\draw[keps] (right) to node[labl,pos=0.4] {\tiny 3,0} (int);
		\draw[keps] (cent) to node[labl,pos=0.4] {\tiny 3,0} (int);
	\end{tikzpicture}
\end{align*}
Then Assumption~3 in \cite{Martingales} is satisfied and by \cite[Cor.~5.6]{Martingales}, the moments of the first diagram are bounded by a constant multiple of $(\lambda \vee \emezo)^{-13/6}$, and, reducing the singularity of kernels as before, the moments of the other diagrams, multiplied by $\ga$, are bounded by $c_\ga (\lambda \vee \emezo)^{-3/6}$ with $c_\ga$ vanishing when $\ga \to 0$. 

\subsubsection{Proof of the bounds \eqref{eq:model_bound-delta}}

We draw ``\,\tikz[baseline=-0.1cm] \draw[kepsdot] (0,0) to node[labl,pos=0.45] {\tiny 3,0} (1,0);\,'' for the kernel $\mywidetilde{\SK}^{\ga, \de}$, because it has the same singularity as $\mywidetilde{\SK}^{\ga}$ (see Appendix~\ref{sec:decompositions}), and we draw ``\tikz[baseline=-0.1cm] \draw[keps] (0,0) to node[labl,pos=0.45] {\tiny 3+$\theta$,0} (1.5,0);\,'' for the difference $\delta^{-\theta}(\mywidetilde{\SK}^\ga - \mywidetilde{\SK}^{\ga, \de})$, because it satisfies \cite[Assum.~4]{Martingales} with $a_e=3+\theta$, for any $\theta > 0$ small enough (see Appendix~\ref{sec:decompositions}).

We start with proving the bound \eqref{eq:model_bound-delta} for the element $\tau = \protect\<1b>$. As we described in the beginning of Section~\ref{sec:Models_bounds}, the difference $\mywidetilde{\SK}^\ga - \mywidetilde{\SK}^{\ga, \de}$ satisfies \cite[Assum.~4]{Martingales} with $a_e=3+\theta$, for any $\theta > 0$ small enough, and we represent this difference by the edge ``\,\tikz[baseline=-0.1cm] \draw[keps] (0,0) to node[labl,pos=0.45] {\tiny 3+$\theta$,0} (1.5,0);\,'' with the multiplier $\de^\theta$. Then we write the function $\Pi^{\ga, \fa}_{z} \tau - \Pi_{z}^{\ga, \de, \fa} \tau$ as
\begin{equation*}
\iota_\eps \bigl(\Pi^{\ga, \fa}_{z} \tau - \Pi_{z}^{\ga, \de, \fa} \tau \bigr)(\phi_z^\lambda)
\;=\; \delta^{\theta}
\begin{tikzpicture}[scale=0.35, baseline=0cm]
	\node at (0.9,0.2)  [root] (root) {};
	\node at (0.9,0.2) [rootlab] {$z$};
	\node at (-2.3, 0.2)  [dot] (int) {};
	\node at (-6,0.2)  [var_blue] (left) {};	
	\draw[testfcn] (int) to (root);	
	\draw[keps] (left) to node[labl,pos=0.45] {\tiny 3+$\theta$,0} (int);	
\end{tikzpicture}\;,
\end{equation*} 
with the kernel given by
\begin{equation*}
	\CK^{\lambda, \emezo, \de}_{\CCG, z}(z^{\var}) = \int_{D_\eps} \!\!\phi_z^\lambda(\bar z)\, \big( \mywidetilde{\SK}^\ga - \mywidetilde{\SK}^{\ga, \de} \big)(\bar z - z^{\var}) \, \d \bar z.
\end{equation*}
Applying \cite[Cor.~5.6]{Martingales}, we get for any $\bar\kappa > 0$ and $p \geq 2$ large enough
\begin{equation*}
	\Bigl( \E \bigl| \iota_\eps \bigl(\Pi^{\ga, \fa}_{z} \tau - \Pi_{z}^{\ga, \de, \fa} \tau \bigr)(\phi_z^\lambda) \bigr|^p\Bigr)^{\frac{1}{p}} \lesssim \de^{\theta} (\lambda \vee \emezo)^{- \frac{1}{2}-\theta} \Bigl( 1 + \eps^{\frac{9}{4} - \bar\kappa} \emezo^{-\frac52} \Bigr) \lesssim \de^{\theta} (\lambda \vee \emezo)^{- \frac{1}{2}-\theta},
\end{equation*}
which is the required bound \eqref{eq:model_bound-delta} for the element $\tau$.

Now, we will prove the bound \eqref{eq:model_bound-delta} for the element $\tau = \protect\<2b>$. Similarly to \eqref{eq:Pi2-new} we can write
\begin{align}\label{eq:Pi2-conv}
	\iota_\eps \bigl(\Pi^{\ga, \fa}_{z} \tau - \Pi_{z}^{\ga, \de, \fa} \tau\bigr)(\phi_z^\lambda)
	& \;=\; \de^{\theta}
	\begin{tikzpicture}[scale=0.35, baseline=0cm]
		\node at (0,-2.2)  [root] (root) {};
		\node at (0,-2.2) [rootlab] {$z$};
		\node at (0,-2.5) {$$};
		\node at (0,0)  [dot] (int) {};
		\node at (-2,2.5)  [var_blue] (left) {};
		\node at (2,2.5)  [var_blue] (right) {};
		\draw[testfcn] (int) to (root);
		\draw[keps] (left) to node[labl,pos=0.45] {\tiny $3$,0} (int);
		\draw[keps] (right) to node[labl,pos=0.45] {\tiny 3+$\theta$,0} (int);
	\end{tikzpicture}
	\; +\; \de^{\theta}
	\begin{tikzpicture}[scale=0.35, baseline=0cm]
		\node at (0,-2.2)  [root] (root) {};
		\node at (0,-2.2) [rootlab] {$z$};
		\node at (0,-2.5) {$$};
		\node at (0,0)  [dot] (int) {};
		\node at (-2,2.5)  [var_blue] (left) {};
		\node at (2,2.5)  [var_blue] (right) {};
		\draw[testfcn] (int) to (root);
		\draw[keps] (left) to node[labl,pos=0.45] {\tiny 3+$\theta$,0} (int);
		\draw[kepsdot] (right) to node[labl,pos=0.45] {\tiny $3$,0} (int);
	\end{tikzpicture}
	\; +\; \de^\theta 
	\begin{tikzpicture}[scale=0.35, baseline=0cm]
		\node at (0,-2.2)  [root] (root) {};
		\node at (0,-2.2) [rootlab] {$z$};
		\node at (0,0)  [dot] (int) {};
		\node at (0,2.5)  [var_red_square] (left) {};
		\draw[testfcn] (int) to (root);
		\draw[keps] (left) to[bend left=80] node[labl,pos=0.45] {\tiny 3+$\theta$,0} (int);
		\draw[keps] (left) to[bend left=-80] node[labl,pos=0.45] {\tiny 3,0} (int);
	\end{tikzpicture} 
	\; + \; \de^\theta
	\begin{tikzpicture}[scale=0.35, baseline=0cm]
		\node at (0,-2.2)  [root] (root) {};
		\node at (0,-2.2) [rootlab] {$z$};
		\node at (0,0)  [dot] (int) {};
		\node at (0,2.5)  [var_red_square] (left) {};
		\draw[testfcn] (int) to (root);
		\draw[kepsdot] (left) to[bend left=80] node[labl,pos=0.45] {\tiny $3$,0} (int);
		\draw[keps] (left) to[bend left=-80] node[labl,pos=0.45] {\tiny 3+$\theta$,0} (int);
	\end{tikzpicture} \; \\[0.2cm]
	&  +\; \left(
	\begin{tikzpicture}[scale=0.35, baseline=0cm]
		\node at (0,-2.2)  [root] (root) {};
		\node at (0,-2.2) [rootlab] {$z$};
		\node at (0,0)  [dot] (int) {};
		\node at (0,2.5)  [var_red_triangle] (left) {};
		\draw[testfcn] (int) to (root);
		\draw[keps] (left) to[bend left=80] node[labl,pos=0.45] {\tiny $3$,0} (int);
		\draw[keps] (left) to[bend left=-80] node[labl,pos=0.45] {\tiny $3$,0} (int);
	\end{tikzpicture} \; - \; (\fc_\ga + \fc_\ga' ) \, 
	\begin{tikzpicture}[scale=0.35, baseline=-0.5cm]
		\node at (0,-2.2)  [root] (root) {};
		\node at (0,-2.2) [rootlab] {$z$};
		\node at (0,0)  [dot] (int) {};
		\draw[testfcn] (int) to (root);
	\end{tikzpicture} \right)
	\; - \; \left(
	\begin{tikzpicture}[scale=0.35, baseline=0cm]
		\node at (0,-2.2)  [root] (root) {};
		\node at (0,-2.2) [rootlab] {$z$};
		\node at (0,0)  [dot] (int) {};
		\node at (0,2.5)  [var_red_triangle] (left) {};
		\draw[testfcn] (int) to (root);
		\draw[kepsdot] (left) to[bend left=80] node[labl,pos=0.45] {\tiny $3$,0} (int);
		\draw[kepsdot] (left) to[bend left=-80] node[labl,pos=0.45] {\tiny $3$,0} (int);
	\end{tikzpicture} \; - \; (\fc_{\ga, \de} + \fc_{\ga, \de}' ) \, 
	\begin{tikzpicture}[scale=0.35, baseline=-0.5cm]
		\node at (0,-2.2)  [root] (root) {};
		\node at (0,-2.2) [rootlab] {$z$};
		\node at (0,0)  [dot] (int) {};
		\draw[testfcn] (int) to (root);
	\end{tikzpicture} \right). \nonumber
\end{align} 
The moments of the first four terms in \eqref{eq:Pi2-conv} are bounded using \cite[Cor.~5.6]{Martingales} by a constant multiple of $\de^\theta (\lambda \vee \emezo)^{-1-\theta}$ in the same way as we bounded the respective terms in \eqref{eq:Pi2-new}. We prefer to provide more details for the two expressions in parentheses in \eqref{eq:Pi2-conv}. In the same way as in \eqref{eq:Pi-3}, we can write the whole expression in the last line in \eqref{eq:Pi2-conv} as
\begin{align*}
&\frac{1}{2} \int_{D_\eps} \phi^\lambda_z (\bar z) \biggl( \eps^3 \sum_{\tilde{y} \in \Lattice} \int_0^\infty \Bigl(\mywidetilde{\SK}^\ga_{\bar t - \tilde{s}}(\bar x - \tilde{y})^2 - \mywidetilde{\SK}^{\ga, \de}_{\bar t - \tilde{s}}(\bar x - \tilde{y})^2\Bigr) \bigl(\bC_{\ga, \fa}(\tilde{s}, \tilde{y}) - 2 + 2 \be \varkappa_{\ga, 3} \ga^6 \un{\fC}_\ga\bigr) \d \tilde{s} \biggr) \d \bar z \\
&+ \frac{1}{2} \int_{D_\eps} \phi^\lambda_z (\bar z) \biggl( \eps^3 \sum_{\tilde{y} \in \Lattice} \int_{-\infty}^0 \Bigl(\mywidetilde{\SK}^\ga_{\bar t - \tilde{s}}(\bar x - \tilde{y})^2 - \mywidetilde{\SK}^{\ga, \de}_{\bar t - \tilde{s}}(\bar x - \tilde{y})^2\Bigr) \bigl(\widetilde{\bC}_{\ga, \fa}(-\tilde{s}, \tilde{y}) - 2 + 2 \be \varkappa_{\ga, 3} \ga^6 \un{\fC}_\ga\bigr) \d \tilde{s} \biggr) \d \bar z.
\end{align*}
We bound this expression in the same way as we bounded \eqref{eq:Pi-3}, with the only difference that now we bound the difference of the two kernels as 
\begin{equation*}
\biggl| \int_{D_\eps} \Bigl(\mywidetilde{\SK}^\ga(z)^2 - \mywidetilde{\SK}^{\ga, \de}(z)^2\Bigr)\, \d z \biggr| \lesssim \delta^\theta \emezo^{-1 - \theta}
\end{equation*}
(see explanations at the beginning of this section). Hence, the expression in the last line in \eqref{eq:Pi2-conv} is bounded by constant times $\de^\theta (\lambda \vee \emezo)^{-1-\theta}$, as required. 
\medskip

The bound \eqref{eq:model_bound-delta} for the other elements in $\bar\CW$ can be proved by analogy, and we prefer to provide only the idea of the proof. For any element $\tau \in \bar\CW$ we can write 
\[
	\iota_\eps \bigl(\Pi^{\ga, \fa}_{z}\tau - \Pi_{z}^{\ga, \de, \fa}\tau\bigr)(\phi_z^\lambda) = \sum_{i \in A} \iota_\eps \bigl(\Pi_{z}^{\ga, i}\tau - \Pi_{z}^{\ga, (\de), i}\tau\bigr)(\phi_z^\lambda),
\]
for a finite set $A$, and where the new maps $\Pi_{z}^{\ga, i} \tau$ and $\Pi_{z}^{\ga, (\de), i} \tau$ are coming from expanding products of martingales \cite[Eq.~5.1]{Martingales}. These two maps can be represented by diagrams, as we did above, with the only difference that the edges in the diagram of $\Pi_{z}^{\ga, (\de), i} \tau$ incident to the noise nodes are given by the kernels $\mywidetilde{\SK}^{\ga, \de}$. We can further write
\begin{equation}\label{eq:Pi-expansion}
	\iota_\eps \bigl(\Pi_{z}^{\ga, i}\tau - \Pi_{z}^{\ga, (\de), i}\tau\bigr)(\phi_z^\lambda) = \sum_{j \in B_i} \iota_\eps \bigl(\Pi_{z}^{\ga, (\de), i, j}\tau\bigr)(\phi_z^\lambda),
\end{equation}
where the diagram for $\iota_\eps \bigl(\Pi_{z}^{\ga, (\de), i, j}\tau\bigr)(\phi_z^\lambda)$ is obtained from $\iota_\eps \bigl(\Pi_{z}^{\ga, (\de), i}\tau\bigr)(\phi_z^\lambda)$ by replacing one of the kernels incident to noise nodes by $\mywidetilde{\SK}^\ga - \mywidetilde{\SK}^{\ga, \de}$, and some other nodes by $\mywidetilde{\SK}^{\ga, \de}$.

Applying \cite[Cor.~5.6]{Martingales} to each element in \eqref{eq:Pi-expansion}, in the same way as we did in the previous sections, we get the required bound \eqref{eq:model_bound-delta}.

\subsection{Proof of Proposition~\ref{prop:models-converge}} 
\label{sec:convergence-of-models-proof}

We start with proving the required bounds on the maps $\Pi^{\ga, \fa}$ and $\Pi^{\ga, \de, \fa}$. From the preceding sections we conclude that in the setting of this proposition, for $\bar \kappa > 0$ sufficiently small and for every $\tau \in \CW^\ex \setminus \{\<5b>, \<40eb>, \<50eb>, \<20b>, \<30b>\}$ with $|\tau| < 0$, we have
\begin{subequations}\label{eqs:moment-bounds-Pi-Gamma}
\begin{equation}\label{eq:moment-bounds-Pi}
\E \Bigl[ \bigl|\bigl(\iota_\eps \Pi^{\ga, \fa}_{z} \tau\bigr) (\varphi^\lambda_{z})\bigr|^p\Bigr] \lesssim \lambda^{(|\tau| + \bar \kappa) p}, \qquad \E \Bigl[ \bigl| \bigl(\Pi^{\ga, \fa}_{z} \tau\bigr) (\bar z) \bigr|^p \Bigr] \lesssim \emezo^{(|\tau| + \bar \kappa) p},
\end{equation}
and 
\begin{align}
\E \Bigl[ \bigl|\bigl(\iota_\eps \Pi^{\ga, \fa}_{z} \tau - \iota_\eps \Pi^{\ga, \de, \fa}_{z} \tau\bigr) (\varphi^\lambda_{z})\bigr|^p\Bigr] &\lesssim \lambda^{(|\tau| + \bar \kappa) p} \de^{\theta p}, \\
 \E \Bigl[ \bigl|\bigl(\Pi^{\ga, \fa}_{z} \tau - \Pi^{\ga, \de, \fa}_{z} \tau\bigr) (\bar z) \bigr|^p \Bigr] &\lesssim \emezo^{(|\tau| + \bar \kappa) p} \de^{\theta p},
\end{align}
\end{subequations}
uniformly over $z \in [-T, T] \times [-1, 1]^3$, $\| \bar z - z \|_\s \leq \emezo$ and other quantities as in \eqref{eq:Pi-bounds}. For the element $\tau = \<5b>$ these bounds hold with $|\tau|$ replaced by $|\tau| + \frac{1}{3}$ and the proportionality constants of order $\ga^{-p}$. In these and the following bounds the proportionality constants depends on $p$ and $T$, but are independent of all the other quantities. These bounds readily yield the respective bounds for the elements $\<40eb>$ and $\<50eb>$, because of the definition \eqref{eq:Pi-E} and the simple bounds $\ga^6 \lesssim \emezo^2 \lesssim (\lambda \vee \emezo)^2$ and $\ga^6 \lesssim \gamma \emezo^{2 - 1/3} \lesssim \gamma (\lambda \vee \emezo)^{2-1/3}$.

It is left to prove these bounds for the symbols $\<20b>$ and $\<30b>$. We will prove stronger bounds 
\begin{equation}\label{eq:Pi-strong-bound}
\begin{aligned}
\E \Bigl[ \bigl|\bigl(\Pi^{\ga, \fa}_{z} \bar \tau\bigr)(\bar z)\bigr|^p\Bigr] &\lesssim \bigl(\| z - \bar z \|_\s \vee \emezo\bigr)^{(|\bar \tau| + \bar \kappa) p}, \\
 \E \Bigl[ \bigl|\bigl(\Pi^{\ga, \fa}_{z} \bar \tau - \Pi^{\ga, \de, \fa}_{z}\bar \tau\bigr)(\bar z)\bigr|^p\Bigr] &\lesssim \bigl(\| z - \bar z \|_\s \vee \emezo\bigr)^{(|\bar \tau| + \bar \kappa) p} \de^{\theta p},
\end{aligned}
\end{equation}
for $\bar \tau \in \{\<20b>, \<30b>\}$, from which the required bounds \eqref{eqs:moment-bounds-Pi-Gamma} follow at once. From the definition \eqref{eq:Pi-I} and the expansion of $\mywidetilde{\SK}^\ga$, provided in Appendix~\ref{sec:decompositions}, we have 
\begin{equation}\label{eq:Pi-strong-bound-proof}
\Bigl(\E \Bigl[ \bigl|\bigl(\Pi^{\ga, \fa}_{z} \bar \tau\bigr)(\bar z)\bigr|^p\Bigr]\Bigr)^{\frac{1}{p}} \lesssim \sum_{n = 0}^{\mywidetilde{M}} \Bigl(\E \Bigl[ \bigl| \iota_\eps \bigl(\Pi^{\ga, \fa}_z \tau\bigr) \bigl(\widetilde{K}^{\ga, n}(\bar z - \bigcdot) - \widetilde{K}^{\ga, n}(z - \bigcdot)\bigr) \bigr|^p\Bigr]\Bigr)^{\frac{1}{p}},
\end{equation}
for $\tau \in \{\<2b>, \<3b>\}$. In order to estimate this sum, we need to consider two cases: $\| z - \bar z \|_\s \geq 2^{-n}$ and $\| z - \bar z \|_\s < 2^{-n}$.

If $\| z - \bar z \|_\s \geq 2^{-n}$, then we apply the Minkowski inequality to bound the $n$-th term in \eqref{eq:Pi-strong-bound-proof} by
\begin{equation}\label{eq:Pi-strong-bound-proof1}
\Bigl(\E \Bigl[ \bigl| \iota_\eps \bigl(\Pi^{\ga, \fa}_z \tau\bigr) \bigl(\widetilde{K}^{\ga, n}(\bar z - \bigcdot)\bigr) \bigr|^p\Bigr]\Bigr)^{\frac{1}{p}} + \Bigl(\E \Bigl[ \bigl| \iota_\eps \bigl(\Pi^{\ga, \fa}_z \tau\bigr) \bigl(\widetilde{K}^{\ga, n}(z - \bigcdot)\bigr) \bigr|^p\Bigr]\Bigr)^{\frac{1}{p}}.
\end{equation}
Moreover, from the identities $\Pi^{\ga, \fa}_{z} = \Pi^{\ga, \fa}_{\bar z} \Gamma^{\ga, \fa}_{\!\bar z z}$ and $\Gamma^{\ga, \fa}_{\!\bar z z} \tau = \tau$ for $\tau \in \{\<2b>, \<3b>\}$ (the first identity follows from the definition of the model, and the second identity follows from Table~\ref{tab:linear_transformations}), we can replace $\Pi^{\ga, \fa}_z$ in the first term in \eqref{eq:Pi-strong-bound-proof1} by $\Pi^{\ga, \fa}_{\bar z}$. Then the bounds \eqref{eq:Kn_bound} and \eqref{eq:moment-bounds-Pi} allow to estimate \eqref{eq:Pi-strong-bound-proof1} by a constant multiple of $2^{- (|\bar \tau| + \bar \kappa) n}$. Then the part of the sum in \eqref{eq:Pi-strong-bound-proof} over $n$ satisfying $\| z - \bar z \|_\s \geq 2^{-n}$ is bounded by a constant times
\begin{equation*}
\sum_{\substack{ 0 \leq n \leq M : \\ \| z - \bar z \|_\s \geq 2^{-n}}} 2^{- (|\bar \tau| + \bar \kappa) n} \lesssim \bigl(\| z - \bar z \|_\s \vee \emezo\bigr)^{|\bar \tau| + \bar \kappa}.
\end{equation*}

If $\| z - \bar z \|_\s < 2^{-n}$, then we need to distinguish the two cases $n < \mywidetilde{M}$ and $n = \mywidetilde{M}$. For $n < \mywidetilde{M}$ we can write 
\begin{equation*}
\widetilde{K}^{\ga, n}(\bar z - \tilde z) - \widetilde{K}^{\ga, n}(z - \tilde z) = \sum_{i = 0}^3 \int_{L_i} \partial_{u_i} \widetilde{K}^{\ga, n}(z + u - \tilde z) \d u,
\end{equation*}
for line segments $L_i$, parallel to the coordinate axes, such that their union is a path connecting the origin and $\bar z - z$. In particular, the length of each $L_i$ is bounded by $\| z - \bar z \|^{\s_i}_\s$, where $\s_0 = 2$ and $\s_i = 1$ for $i = 1,2,3$. Then we have 
\begin{equation*}
\iota_\eps \bigl(\Pi^{\ga, \fa}_z \tau\bigr) \bigl(\widetilde{K}^{\ga, n}(\bar z - \bigcdot) - \widetilde{K}^{\ga, n}(z - \bigcdot)\bigr) = \sum_{i = 0}^3 \int_{L_i} \iota_\eps \bigl(\Pi^{\ga, \fa}_{z + u} \tau\bigr) \bigl(\partial_{u_i} K^{\ga, n}(z + u - \bigcdot)\bigr) \d u,
\end{equation*}
where we replaced $\Pi^{\ga, \fa}_{z}$ by $\Pi^{\ga, \fa}_{z + u}$ in the same was as we did in \eqref{eq:Pi-strong-bound-proof1}. The bounds \eqref{eq:Kn_bound} and \eqref{eq:moment-bounds-Pi} yield 
\begin{align*}
\Bigl(\E \Bigl[ \bigl| \iota_\eps \bigl(\Pi^{\ga, \fa}_z \tau\bigr) \bigl(K^{\ga, n}(\bar z - \bigcdot) - K^{\ga, n}(z - \bigcdot)\bigr) \bigr|^p\Bigr]\Bigr)^{\frac{1}{p}} \lesssim \sum_{i = 0}^3 2^{- (|\bar \tau| - \s_i + \bar \kappa) n} \| z - \bar z \|^{\s_i}_\s.
\end{align*}
Since $|\bar \tau| - \s_i < 0$, we can take $\bar \kappa > 0$ small enough, such that $|\bar \tau| - \s_i + \bar \kappa < 0$. Then the part of the sum in \eqref{eq:Pi-strong-bound-proof} over $n$ satisfying $\| z - \bar z \|_\s < 2^{-n}$ is bounded by a constant times
\begin{equation*}
\sum_{i = 0}^3 \| z - \bar z \|^{\s_i}_\s \sum_{\substack{ 0 \leq n < \mywidetilde{M} : \\ \| z - \bar z \|_\s < 2^{-n}}} 2^{- (|\bar \tau| - \s_i + \bar \kappa) n} \lesssim \bigl(\| z - \bar z \|_\s \vee \emezo\bigr)^{|\bar \tau| + \bar \kappa}.
\end{equation*}

In the case $n = \mywidetilde{M}$ we have $\| z - \bar z \|_\s < \emezo$, and the radius of support of the function $\widetilde{K}^{\ga, \mywidetilde{M}}(\bar z - \tilde z) - \widetilde{K}^{\ga, \mywidetilde{M}}(z - \tilde z)$ in $\tilde z$ is of order $\emezo$. Then \eqref{eq:Kn_bound} and the second bound in \eqref{eq:moment-bounds-Pi} yield
\begin{align*}
\Bigl(\E \Bigl[ \bigl| \iota_\eps &\bigl(\Pi^{\ga, \fa}_z \tau\bigr) \bigl(\widetilde{K}^{\ga, \mywidetilde{M}}(\bar z - \bigcdot) - \widetilde{K}^{\ga, \mywidetilde{M}}(z - \bigcdot)\bigr) \bigr|^p\Bigr]\Bigr)^{\frac{1}{p}} \\
&\qquad \lesssim \emezo^{|\tau| + \bar \kappa} \int_{D_\eps} \bigl| \widetilde{K}^{\ga, \mywidetilde{M}}(\bar z - \tilde z) - \widetilde{K}^{\ga, \mywidetilde{M}}(z - \tilde z)\bigr| \d \tilde z \lesssim \emezo^{|\bar \tau| + \bar \kappa} \lesssim \bigl(\| z - \bar z \|_\s \vee \emezo\bigr)^{|\bar \tau| + \bar \kappa}.
\end{align*}
This finishes the proof of the first bound in \eqref{eq:Pi-strong-bound}.

The second bound in \eqref{eq:Pi-strong-bound} can be proved by analogy, but instead of \eqref{eq:Kn_bound} we need to use 
\begin{equation*}
\bigl|D^k \bigl(\widetilde{K}^{\ga, n} - \widetilde{K}^{\ga, n} \star_\eps \varrho_{\ga, \de}\bigr)(z)\bigr| \leq C \de^\theta 2^{n(3 + |k|_\s - \theta)},
\end{equation*}
for respective $n$ and $k$. This bound follows readily from the properties of $\widetilde{K}^{\ga, n}$ and $\varrho_{\ga, \de}$.

The bounds on $\Pi^{\ga, \fa}$ yield the bounds on $\Gamma^{\ga, \fa}$. Indeed, the definition provided above \eqref{eq:Gamma-lift} yields $\Gamma^{\ga, \fa}_{\!z \bar z} \tau = \tau - (\Pi^{\ga, \fa}_{z} \tau)(\bar z) \blueOne$ for $\tau= \<20b>$, and from \eqref{eq:Pi-strong-bound} we get
\begin{align*}
\E \Bigl[ \bigl| \Gamma^{\ga, \fa}_{\!z \bar z} \tau \bigr|_0^p\Bigr] &= \E \Bigl[ \bigl|\bigl(\Pi^{\ga, \fa}_{z} \tau\bigr)(\bar z)\bigr|^p\Bigr] \lesssim \bigl(\| z - \bar z \|_\s \vee \emezo\bigr)^{(|\tau| + \bar \kappa) p}, \\
\E \Bigl[ \bigl| \bigl(\Gamma^{\ga, \fa}_{\!z \bar z} - \Gamma^{\ga, \de, \fa}_{\!z \bar z}\bigr) \tau \bigr|_0^p\Bigr] &= \E \Bigl[ \bigl|\bigl(\Pi^{\ga, \fa}_{z} \tau - \Pi^{\ga, \de, \fa}_{z} \tau\bigr)(\bar z)\bigr|^p\Bigr] \lesssim \bigl(\| z - \bar z \|_\s \vee \emezo\bigr)^{(|\tau| + \bar \kappa) p} \de^{\theta p},
\end{align*}
which is the required bound. In the same way we get bounds for all other elements $\tau \in \CW^\ex$ such that $\Gamma^{\ga, \fa}_{\!z \bar z} \tau \neq \tau$.

\medskip
Now we will use a Kolmogorov-type result to show that the bounds \eqref{eqs:moment-bounds-Pi-Gamma} and \eqref{eq:Pi-strong-bound} yield \eqref{eq:prop:models-converge}, with a small loss of regularity. For every $\tau \in \CW^\ex \setminus \{\<5b>, \<40eb>, \<50eb>\}$ with $|\tau| < 0$ the bounds 
\begin{align*}
\E \biggl[ \sup_{\lambda \in [\emezo, 1]}\sup_{\varphi \in \CB^{2}_\s} \sup_{z \in K} \lambda^{-|\tau| p} \bigl|\bigl(\iota_\eps \Pi^{\ga, \fa}_{z} \tau\bigr) (\varphi^\lambda_{z})\bigr|^p\biggr] &\lesssim 1,\\
\E \biggl[ \sup_{\lambda \in [\emezo, 1]}\sup_{\varphi \in \CB^{2}_\s} \sup_{z \in K} \lambda^{-|\tau| p} \bigl|\bigl(\iota_\eps \Pi^{\ga, \fa}_{z} \tau - \iota_\eps \Pi^{\ga, \de, \fa}_{z} \tau\bigr) (\varphi^\lambda_{z})\bigr|^p\biggr] &\lesssim \de^{\theta p},
\end{align*}
uniformly in $\ga > 0$, can be proved in exactly the same way as \cite[Lem.~10.2]{Regularity}. For the element $\tau = \<5b>$ these bounds hold with $|\tau| + \frac{1}{3}$ in place of $|\tau|$ and the proportionality constant of order $\ga^p$. These bounds for the elements $\<40eb>$ and $\<50eb>$ readily follow as in \eqref{eqs:moment-bounds-Pi-Gamma}. Furthermore, from \eqref{eq:Pi-strong-bound} and the Kolmogorov continuity criterion \cite{Kallenberg} we conclude that 
\begin{equation*}
\E \biggl[ \sup_{z, \bar z \in K} \frac{|(\Pi^{\ga, \fa}_{z} \bar \tau)(\bar z)|^p}{(\| z - \bar z \|_\s \vee \emezo)^{|\bar \tau| p}}\biggr] \lesssim 1, \qquad \E \biggl[ \sup_{z, \bar z \in K} \frac{|(\Pi^{\ga, \fa}_{z} \bar \tau - \Pi^{\ga, \de, \fa}_{z}\bar \tau)(\bar z)|^p}{(\| z - \bar z \|_\s \vee \emezo)^{|\bar \tau| p}}\biggr] \lesssim \de^{\theta p},
\end{equation*}
for $\bar \tau \in \{\<20b>, \<30b>\}$ and for any compact set $K \subset \R^4$. Finally, we get the required bounds on the maps $\Gamma^{\ga, \fa}$ and $\Gamma^{\ga, \de, \fa}$, because they are defined in \eqref{eq:Gamma-lift} via $\Pi^{\ga, \fa}$ and $\Pi^{\ga, \de, \fa}$.

\section{A discrete solution map}
\label{sec:discrete-solution}

In order to prove the desired convergence in Theorem~\ref{thm:main}, we first need to write equation \eqref{eq:IsingKacEqn-new} in the framework of regularity structures. 

We use the discrete model $Z^{\ga, \fa}_\lift$ construction in Section~\ref{sec:lift}, and define the integration operators on the space of modelled distributions via the kernel $\widetilde{G}^\ga$ as in \cite[Sec.~4]{erhard2017discretisation}. More precisely, we write $\widetilde{G}^\ga = \mywidetilde{\SK}^\ga + \mywidetilde{\SR}^\ga$ as in the beginning of Section~\ref{sec:lift}. Then we use the singular part $\mywidetilde{\SK}^\ga$ to define the map $\CK^\ga_{\kappa}$ as in \cite[Eq.~4.6]{erhard2017discretisation} for the value $\beta = 2$. We use the regularity $\kappa$ by analogy with \eqref{eq:abstract-integral}. We note that we do not need to consider the map $\CA^\ga$ from \cite[Eq.~4.16]{erhard2017discretisation}, since it vanishes in our case (see \cite[Rem.~4.10]{erhard2017discretisation}). We lift the smooth part $\mywidetilde{\SR}^\ga$ to a modelled distribution $R^\ga_{1 + 3 \kappa}$ by a Taylor's expansion as in \cite[Eq.~5.17]{HairerMatetski}. Then we define the map
\begin{equation}\label{eq:P-operator}
\CP^\ga := \CK^\ga_{\kappa} + R^\ga_{1 + 3 \kappa} \CR^{\ga, \fa}
\end{equation}
on a suitable space of modelled distributions, where $\CR^{\ga, \fa}$ is the reconstruction map associated to the model by \eqref{eq:def_rec_op}. In order to use Theorem~4.8 and Lemma~6.2 in \cite{erhard2017discretisation}, we need to show that the respective assumptions in \cite{erhard2017discretisation} are satisfied. Assumptions~4.1 and 4.4 hold trivially, because the action of the model $Z^{\ga, \fa}_\lift$ on polynomials coincides with the canonical continuous polynomial model. Assumption~4.3 follows from our definition of the space $\CX_\eps$ in Section~\ref{sec:DiscreteModels} and properties of the kernel $\mywidetilde{\SK}^\ga$. Assumption~4.7 can be shown by brutal bounds of the terms in \cite[Eq.~4.6]{erhard2017discretisation}, combined with the definitions of discrete models and modelled distributions from Sections~\ref{sec:DiscreteModels} and \ref{sec:DiscreteModelledDistributions}. Finally, Assumption~6.1 follows readily from the Taylor's approximations and smoothness of the function $\mywidetilde{\SK}^\ga$. As we said above, the map $\CA^\ga$ vanishes in our case and Assumption~6.3 trivially holds.

Our goal is to write the solution of \eqref{eq:IsingKacEqn-new} as a reconstruction of an abstract equation of the form \eqref{eq:abstract_equation}. However, the complicated non-linearity in \eqref{eq:IsingKacEqn-new} makes the definition of this equation more difficult. 

As follows from \eqref{eq:hom5}, applying $\CE$ increases homogeneity by $2$. However, applying $\CE$ to a modelled distribution $f \in \CD^\zeta$ does not give in general an element of $\CD^{\zeta+2}$, because $\CE$ vanishes on polynomials. To resolve this problem, we define the domain of this map 
\begin{equation*}
\Dom_{\CE} := \{\<4b>, \<5b>, \blueOne\}
\end{equation*}
and consider a modelled distribution of the form 
\begin{equation}\label{eq:f-in-the-domain}
f(z) = \sum_{\tau \in \Dom_{\CE}} f_{\tau}(z) \tau.
\end{equation}
Then we define the map
\begin{equation}\label{eq:E-hat}
	\bigl(\widehat{\CE}_\ga f\bigr)(z) := \CE\bigl(f(z)\bigr) + \ga^6 f_{\blueOne}(z) \blueOne.
\end{equation}
We need to consider $f$ of the form \eqref{eq:f-in-the-domain}, because $f$ should be in the domain of the map $\CE$. If $\CR^{\ga, \fa}$ is the reconstruction map for the model $Z^{\ga, \fa}_\lift$, then we use Remark~\ref{rem:Eps-reconstruct} to conclude
\begin{equation}\label{eq:E-reconstruction}
\CR^{\ga, \fa} \bigl(\widehat{\CE}_\ga f\bigr)(z) = \ga^6 \sum_{\tau \in \Dom_{\CE}} f_{\tau}(z) \bigl(\CR^{\ga, \fa}\tau\bigr) (z) = \ga^6 \bigl(\CR^{\ga, \fa} f\bigr)(z).
\end{equation}
Moreover, we can show that this map increases regularity. Throughout this section we are going to use the time-dependent norms on modelled distributions introduced in Remark~\ref{rem:T-space}. 

As we showed in Remark~\ref{rem:extension-to-Xi}, the model and the reconstruction map are extended to the symbol $\blueXi$. Then the map \eqref{eq:P-operator} can be applied to this symbol and we define the modelled distribution 
\begin{equation}\label{eq:V-gamma}
W_{\ga, \fa}(z) := \CP^\ga \1_+ (\blueXi)(z).
\end{equation}
Furthermore, for $\zeta = 1 + 3 \kappa$ and $\eta \in \R$ we define the abstract equation
\begin{equation}\label{eq:AbstractDiscreteEquation} 
U_{\ga, \fa} = \CQ_{< \zeta} \Bigl(G^\ga X_\ga^0 + \CP^\ga \1_+ \bigl( F_\ga(U_{\ga, \fa}) + E^{(1)}_\ga(U_{\ga, \fa}) + E^{(2)}_\ga(U_{\ga, \fa}) \bigr) + \sqrt 2\, W_{\ga, \fa}\Bigr),
\end{equation}
for a modelled distribution $U_{\ga, \fa} \in D_\eps^{\zeta, \eta} (Z^{\ga, \fa}_\lift)$, where $G^\ga X_\ga^0$ is the polynomial lift of the operator \eqref{eq:P-gamma} applied to $X_\ga^0$ where the discrete heat kernel $G_t^\ga$ is defined on $\Lattice$ by \eqref{eq:From-P-to-G}. The function $F_\ga$ describes the non-linearity in \eqref{eq:IsingKacEqn-new} and is defined as
\begin{equation}\label{eq:F-ga}
F_\ga(U_{\ga, \fa}) := \CQ_{\leq 0} \Bigl( \Bigl(- \frac{\be^3}{3} + B_\ga\Bigr) U_{\ga, \fa}^3 + (A_\ga + A) U_{\ga, \fa}\Bigr),
\end{equation}
for constants $A_\ga$ and $B_\ga$ whose values will be chosen in Lemma~\ref{lem:solution}. We need to consider these constants because of our definition of the renormalised products in \eqref{eq:model-Hermite}. As we will see in the following lemma, we need to take constants $A_\ga$ and $B_\ga$ in \eqref{eq:F-ga}, vanishing as $\ga \to 0$, in order to obtain exactly \eqref{eq:IsingKacEqn-new} after reconstruction of \eqref{eq:AbstractDiscreteEquation}. The function $E^{(1)}_\ga$ in \eqref{eq:AbstractDiscreteEquation} describes the remainder after the Taylor's approximation of the function $\tanh$ in \eqref{eq:expr_error_term}, and is given by
\begin{equation}\label{eq:E1-ga}
E^{(1)}_\ga(U_{\ga, \fa}) := \frac{1}{\delta \alpha} R_5 \bigl(\be \ga^3 \CR^{\ga, \fa} U_{\ga, \fa}\bigr) \blueOne,  
\end{equation}
where $\CR^{\ga, \fa}$ is the reconstruction map, defined in \eqref{eq:def_rec_op}, and $R_5 : \R \to \R$ is the remainder in the Taylor's approximation of the fifth order of the function $\tanh$, i.e. 
\begin{equation}\label{eq:Taylor}
R_5(x) := \tanh x - x + \frac{x^3}{3} - \frac{x^5}{5}.
\end{equation}
The function $E^{(2)}_\ga$ in \eqref{eq:AbstractDiscreteEquation} is defined as
\begin{equation}\label{eq:E2-ga}
E^{(2)}_\ga(U_{\ga, \fa}) := \frac{\be^5}{5} \widehat{\CE}_\ga \biggl(\sum_{\tau \in \{\scalebox{0.7}{\<4b>}, \scalebox{0.7}{\<5b>}\}} \Bigl( \CQ_{\tau} U_{\ga, \fa}^5 - \bigl(\CR^{\ga, \fa} \CQ_{\tau} U_{\ga, \fa}^5\bigr) \blueOne \Bigr) + H_5 \bigl(\CR^{\ga, \fa} U_{\ga, \fa}, 2 \fc_\ga\bigr) \blueOne \biggr),
\end{equation}
where $\CQ_\tau$ is the projection from the model space to the span of $\tau$, $H_5$ is the $5$-th Hermite polynomial \eqref{eq:def_Hermite} and the renormalisation constant $\fc_\ga$ is defined in \eqref{eq:renorm-constant1}. The expression in the brackets in \eqref{eq:E2-ga} is spanned by the elements of $\Dom_{\CE}$, which allows us to apply the map $\widehat{\CE}_\ga$. 

A natural definition of the non-linearity \eqref{eq:E2-ga} could be $\frac{\be^5}{5} \CQ_{\leq 0} \widehat{\CE}_\ga \CQ_{\leq 0}U_{\ga, \fa}^5$. This definition however uses elements of negative homogeneities which appear in the product $U_{\ga, \fa}^5$. We can make sense of it only if we add extra elements into the model space $\CT^\ex$ and define the map $\widehat{\CE}_\ga$ on these elements. In order to have the dimension of $\CT^\ex$ minimal, we have to make a more complicated definition \eqref{eq:E2-ga}. More precisely, in the brackets in \eqref{eq:E2-ga} we keep only the two elements of $U_{\ga, \fa}^5$ with the smallest homogeneities (these are $\CQ_{\tau} U_{\ga, \fa}^5$ with $\tau \in \{\<4b>, \<5b>\}$). The other parts of $U_{\ga, \fa}^5$ we reconstruct and write in \eqref{eq:E2-ga} as a multiplier of $\blueOne$. Then if we apply the reconstruction map $\CR^{\ga, \fa}$ to the expression in the brackets in \eqref{eq:E2-ga}, we get $H_5 \bigl(\CR^{\ga, \fa} U_{\ga, \fa}, 2 \fc_\ga\bigr)$, which is a renormalised fifth power of the solution $\CR^{\ga, \fa} U_{\ga, \fa}$. We use the renormalisation constant $2 \fc_\ga$, because of the multiplier $\sqrt 2$ of the force term $Y_{\ga, \fa}$ in \eqref{eq:IsingKacEqn-new} and a scaling property of Hermite polynomials. More precisely, in order to renormalise the fifth power of $\sqrt 2\, Y_{\ga, \fa}$, we need to use $H_5 \bigl(\sqrt 2\, Y_{\ga, \fa}, 2 \fc_\ga\bigr)$.

We can show that reconstruction of \eqref{eq:AbstractDiscreteEquation} recovers the discrete equation \eqref{eq:IsingKacEqn-new}.

\begin{lemma}\label{lem:solution}
	Let $Z^{\ga, \fa}_{\lift}$ be the model constructed in Section~\ref{sec:lift}, and let the reconstruction map $ \CR^{\ga, \fa}$ be defined for the model $Z^{\ga, \fa}_{\lift}$ in \eqref{eq:def_rec_op}. Let $U_{\ga, \fa} \in \CD^{1 + 3 \kappa, \eta}_\emezo (Z^{\ga, \fa}_{\lift})$ be a solution of \eqref{eq:AbstractDiscreteEquation}. Then it may be written as
	\begin{equation}\label{eq:expansion}
		U_{\ga, \fa} (z) = \sqrt2\, \<1b> + v_{\ga, \fa}(z) \blueOne + \Bigl(- \frac{\be^3}{3} + B_\ga\Bigr) \Bigl(2 \sqrt{2}\, \<30b> + 6 v_{\ga, \fa}(z) \<20b>\Bigr) + \sum_{i = 1,2,3} v^{(i)}_{\ga, \fa}(z) \X_i,
	\end{equation}
	 for some functions $v_{\ga, \fa}, v^{(i)}_{\ga, \fa} : \R_+ \times \Le \to \R$. More precisely, we have $v_{\ga, \fa} = X_{\ga, \fa} - \sqrt2\, Y_{\ga, \fa}$, where $X_{\ga, \fa} := \CR^{\ga, \fa} U_{\ga, \fa}$, $Y_{\ga, \fa} = \CR^{\ga, \fa}\<1b>$, and $v_{\ga, \fa}$ solves the ``remainder equation''
	 \begin{align}\label{eq:discrete-remainder}
	 v_{\ga, \fa}(t, x) = P^\ga_t X_\ga^0(x) + \int_0^t \widetilde{P}^\ga_{t-s} &\Bigl( \Bigl(- \frac{\be^3}{3} + B_\ga\Bigr) \bigl(v_{\ga, \fa} + \sqrt2\, Y_{\ga, \fa}\bigr)^3 \\
	 &\qquad + (\fC_\ga + A) \bigl(v_{\ga, \fa} + \sqrt2\, Y_{\ga, \fa}\bigr) + E_{\ga, \fa} \Bigr)(s, x)\, \d s, \nonumber
	\end{align}
	 where $E_{\ga, \fa}$ is given by \eqref{eq:expr_error_term} with $X_{\ga, \fa}$ replaced by $v_{\ga, \fa} + \sqrt2\, Y_{\ga, \fa}$.
	
	Furthermore, there exist $A_\ga$ and $B_\ga$, vanishing as $\ga \to 0$, such that the function $X_{\ga, \fa} = \CR^{\ga, \fa} U_{\ga, \fa}$ solves \eqref{eq:IsingKacEqn-new} with the renormalisation constant 
	\begin{equation}\label{eq:C-exact}
		\fC_\ga = 2 \bigl(\fc_\ga + \fc_\ga' - 2 \fc_\ga''\bigr),
	\end{equation}
	where $\fc_\ga$, $\fc_\ga'$ and $\fc_\ga''$ are defined in \eqref{eq:renorm-constant1}, \eqref{eq:renorm-constant3} and \eqref{eq:renorm-constant2} respectively.
\end{lemma}

\begin{proof}
The expansion \eqref{eq:expansion} is obtained in the same way as \eqref{eq:U-expansion}, by iteration of \eqref{eq:AbstractDiscreteEquation}. If we define the functions $X_{\ga, \fa} = \CR^{\ga, \fa} U_{\ga, \fa}$ and $Y_{\ga, \fa} = \CR^{\ga, \fa}\<1b>$, then we obtain 
\begin{equation}\label{eq:XandY}
X_{\ga, \fa}(z) = \sqrt2\, Y_{\ga, \fa}(z) + v_{\ga, \fa}(z),
\end{equation}
where the reconstructions of the elements with strictly positive homogeneities vanish (see Remark~\ref{rem:positive-vanish}). Using \eqref{eq:expansion}, we can write 
\begin{align*}
&\CQ_{\leq 0} U_{\ga, \fa}^3(z) = 2 \sqrt 2\, \<3b> + 6 v_{\ga, \fa}(z)\, \<2b> + 12 \sqrt 2 \Bigl(- \frac{\be^3}{3} + B_\ga\Bigr) \<32b> + 3 \sqrt 2 v_{\ga, \fa}(z)^2\, \<1b> \\
& + 24 \Bigl(- \frac{\be^3}{3} + B_\ga\Bigr) v_{\ga, \fa}(z)\, \<31b> + 36 \Bigl(- \frac{\be^3}{3} + B_\ga\Bigr) v_{\ga, \fa}(z)\, \<22b> + 4 \sum_{i = 1, 2,3} v^{(i)}_{\ga, \fa}(z) \X_i \<2b> + v_{\ga, \fa}(z)^3 \blueOne.
\end{align*}
From our definition of the model in Section~\ref{sec:model-lift} and the reconstruction map in \eqref{eq:def_rec_op} we have $( \CR^{\ga, \fa} \blueOne)(z) = 1$, $( \CR^{\ga, \fa} \<2b>)(z) = H_2(Y_{\ga, \fa}(z), \fc_\ga + \fc_\ga')$, $(\CR^{\ga, \fa} \<1b>^n)(z) = H_n(Y_{\ga, \fa}(z), \fc_\ga)$ for $n \neq 2$, $( \CR^{\ga, \fa} \X_i \<2b>)(z) = 0$, $(\CR^{\ga, \fa} \<32b>)(z) = - 3 \fc_\ga'' Y_{\ga, \fa}(z)$, $(\CR^{\ga, \fa} \<31b>)(z) = 0$ and $(\CR^{\ga, \fa} \<22b>)(z) = - \fc_\ga''$. Applying the reconstruction map to the preceding expansion, we get
\begin{align*}
&\bigl(\CR^{\ga, \fa} \CQ_{\leq 0} U_{\ga, \fa}^3 \bigr)(z) = 2 \sqrt 2 \bigl(Y_{\ga, \fa}(z)^3 - 3\fc_\ga Y_{\ga, \fa}(z)\bigr) + 6 v_{\ga, \fa}(z) \bigl(Y_{\ga, \fa}(z)^2 - \fc_\ga - \fc_\ga'\bigr) \\ 
& - 36 \sqrt 2 \fc_\ga'' \Bigl(- \frac{\be^3}{3} + B_\ga\Bigr) Y_{\ga, \fa}(z) + 3 \sqrt 2 v_{\ga, \fa}(z)^2 Y_{\ga, \fa}(z) - 36 \fc_\ga'' \Bigl(- \frac{\be^3}{3} + B_\ga\Bigr) v_{\ga, \fa}(z) + v_{\ga, \fa}(z)^3 \\
&\hspace{3cm}= X_{\ga, \fa}(z)^3 - 6 \bigl(\fc_\ga + \fc_\ga' - 2 \fc_\ga'' (\beta^3 - 3 B_\ga)\bigr) X_{\ga, \fa}(z),
\end{align*}
where we used \eqref{eq:XandY}. Hence, the reconstruction $\bigl(\CR^{\ga, \fa} F_\ga(U_{\ga, \fa}) \bigr)(z)$ of the function \eqref{eq:F-ga} gives
\begin{equation*}
\Bigl(- \frac{\be^3}{3} + B_\ga\Bigr) \Bigl( X_{\ga, \fa}(z)^3 - 6 \bigl(\fc_\ga + \fc_\ga' - 2 \fc_\ga'' (\beta^3 - 3 B_\ga)\bigr) X_{\ga, \fa}(z) \Bigr) + (A_\ga + A) X_{\ga, \fa}(z).
\end{equation*}
Reconstruction of the function \eqref{eq:E1-ga} is trivial: $\bigl(\CR^{\ga, \fa} E^{(1)}_\ga(U_{\ga, \fa}) \bigr)(z) = \frac{1}{\delta \alpha} R_5\bigl(\be \ga^3 X_{\ga, \fa}(z)\bigr)$.

Now, we turn to reconstruction of the function \eqref{eq:E2-ga}. Expansion \eqref{eq:expansion} yields 
\begin{equation}\label{eq:U5-expansion}
	U_{\ga, \fa}(z)^5 = 4 \sqrt2\, \<5b> + 20 v_{\ga, \fa}(z)\, \<4b> + \widetilde{U}_{\ga, \fa}(z),
\end{equation}
where the remainder $\widetilde{U}_{\ga, \fa}(z)$ takes values in the span of elements with homogeneities greater than $-\frac{3}{2} - 3 \kappa$. Then the expression in the brackets in \eqref{eq:E2-ga} is
\begin{equation*}
 4 \sqrt2\, \<5b> + 20 v_{\ga, \fa}(z)\, \<4b> - \Bigl(4 \sqrt2\, \bigl(\CR^{\ga, \fa} \<5b>\bigr)(z) + 20 v_{\ga, \fa}(z)\, \bigl(\CR^{\ga, \fa} \<4b>\bigr)(z) - H_5 \bigl(X_{\ga, \fa}(z), 2 \fc_\ga\bigr)\Bigr) \blueOne.
\end{equation*}
Using \eqref{eq:E-hat}, the function \eqref{eq:E2-ga} equals 
\begin{align*}
E^{(2)}_\ga(U_{\ga, \fa})(z) &= \frac{\be^5}{5} \biggl(4 \sqrt2\, \<50eb> + 20 v_{\ga, \fa}(z) \<40eb> \\
&\qquad - \ga^6 \Bigl(4 \sqrt2\, \bigl(\CR^{\ga, \fa} \<5b>\bigr)(z) + 20 v_{\ga, \fa}(z)\, \bigl(\CR^{\ga, \fa} \<4b>\bigr)(z) - H_5 \bigl(X_{\ga, \fa}(z), 2 \fc_\ga\bigr)\Bigr) \blueOne \biggr),
\end{align*}
and applying the reconstruction map gives
\begin{align*}
	\bigl(\CR^{\ga, \fa} E^{(2)}_\ga(U_{\ga, \fa})\bigr)(z) &= \ga^6 \frac{\be^5}{5} H_5 \bigl(X_{\ga, \fa}(z), 2 \fc_\ga\bigr) \\
	&= \ga^6 \frac{\be^5}{5} \Bigl( X_{\ga, \fa}(z)^5 - 20 \fc_\ga X_{\ga, \fa}(z)^3 + 60 \fc_\ga^2 X_{\ga, \fa}(z) \Bigr). 
\end{align*}
Here, we used the definition of the reconstruction map \eqref{eq:def_rec_op} and Remark~\ref{rem:Eps-reconstruct}. 

Applying the reconstruction map to both sides of equation \eqref{eq:AbstractDiscreteEquation}, using the property $\CR^{\ga, \fa} \CP^\ga = \widetilde{G}^\ga$ and using all previous identities, we obtain 
\begin{align*}
	X_{\ga, \fa}(t, x) &= G^\ga_t X_\ga^0(x) + \sqrt 2\, Y_{\ga, \fa}(t, x) \\
&\qquad + \int_0^t \widetilde{G}^\ga_{t-s} \Bigl( \Bigl(- \frac{\be^3}{3} + B_\ga - 4 \ga^6 \be^5 \fc_\ga\Bigr) X^3_{\ga, \fa} + \bigl(\fC_\ga + A\bigr) X_{\ga, \fa} + E_{\ga, \fa} \Bigr)(s, x)\, \d s,
\end{align*}
where the error term $E_{\ga, \fa}$ is the same as in \eqref{eq:IsingKacEqn-new} and where 
\begin{equation}\label{eq:C-A}
\fC_\ga = 2 \bigl(\be^3 - 3 B_\ga\bigr) \bigl(\fc_\ga + \fc_\ga' - 2 \fc_\ga'' (\beta^3 - 3 B_\ga)\bigr) + 12 \fc_\ga^2 \ga^6 \be^5 + A_\ga.
\end{equation}
In order to have this equation equal to \eqref{eq:IsingKacEqn-new}, we need to take $B_\ga = 4 \ga^6 \be^5 \fc_\ga$ and $A_\ga$ from the previous identity. Lemma~\ref{lem:renorm-constants} suggests that $|B_\ga| \lesssim \emezo$ which vanishes as $\ga \to 0$.

It is left to show that if we take $\fC_\ga$ of the form \eqref{eq:C-exact}, then the constant $A_\ga$, defined via \eqref{eq:C-A}, vanish as $\ga \to 0$. We recall that $\be$ depends on $\fC_\ga$ via \eqref{eq:beta}. From \eqref{eq:C-exact} and \eqref{eq:C-A} we have
\begin{equation}\label{eq:A-expansion}
A_\ga = - 2 (\be^3 - 1) (\fc_\ga + \fc_\ga') + 4 (\be^6 - 1) \fc_\ga'' + 6 (\fc_\ga + \fc_\ga') B_\ga - 12 \fc_\ga'' B_\ga (2 \beta^3 - 3 B_\ga) - 12 \fc_\ga^2 \ga^6 \be^5.
\end{equation}
Using \eqref{eq:beta} and \eqref{eq:C-exact}, we can write 
\begin{align*}
\be^3 - 1 &= \sum_{k = 1, 2, 3} {3 \choose k} \ga^{6 k} \bigl(2 (\fc_\ga + \fc_\ga' - 2 \fc_\ga'') + c + A \bigr)^k, \\
 \be^6 - 1 &= \sum_{k = 1, \ldots, 6} {6 \choose k} \ga^{6 k} \bigl(2 (\fc_\ga + \fc_\ga' - 2 \fc_\ga'') + c + A \bigr)^k.
\end{align*}
From Lemma~\ref{lem:renorm-constants} we have $\fc_\ga = c_2 \emezo^{-1} + \tilde{\fc}_\ga$ and $\fc_\ga'' = c_1 \log \emezo + \tilde{\fc}''_\ga$, where $|\tilde{\fc}_\ga| \leq C |\log \emezo|$ and $|\tilde{\fc}''_\ga| \leq C$ for some constant $C > 0$ independent of $\ga$. Moreover, the definition \eqref{eq:renorm-constant3} boundedness of $\fc_\ga'$ uniformly in $\ga \in (0,1]$. From \eqref{eq:c-gamma-2} we furthermore have $\emezo = \ga^3 \varkappa_{\ga, 3}$ and hence $\emezo^{-1} = \ga^{-3} +  \ga^{-3} c_{\ga, 3}$, where $|c_{\ga, 3}| \leq \ga^4 / (1 - \ga^4) \to 0$ as $\ga \to 0$. Using these bounds in \eqref{eq:A-expansion}, we can see that $A_\ga$ vanishes as $\ga \to 0$. 
\end{proof}

\begin{remark}
In what follows we will always consider equation \eqref{eq:AbstractDiscreteEquation} with the values $A_\ga$ and $B_\ga$ from Lemma~\ref{lem:solution}, which makes the reconstructed solution of \eqref{eq:AbstractDiscreteEquation} coincide with the solution of \eqref{eq:IsingKacEqn-new}.
\end{remark}

 Let $Z^{\ga, \de, \fa}_{\lift}$ be another random discrete model constructed in Section~\ref{sec:convergence-of-models} and let us consider equation
\begin{equation}\label{eq:AbstractDiscreteEquation-delta}
U_{\ga, \de, \fa} = \CQ_{< \zeta} \Bigl(G^{\ga} X_{\ga, \de}^{0} + \CP^{\ga, \de} \1_+ \CQ_{\leq 0} \Bigl(- \frac{\be^3}{3} \bigl(U_{\ga, \de, \fa}\bigr)^3 + (A_{\ga, \de} + A) U_{\ga, \de, \fa}\Bigr) + \sqrt 2\, W_{\ga, \de, \fa}\Bigr),
\end{equation}
which is defined in the same way as \eqref{eq:AbstractDiscreteEquation}, but with respect to the model $Z^{\ga, \de, \fa}_\lift$. The initial condition at time $0$ is
\begin{equation}\label{eq:initial-gamma-delta}
X_{\ga, \de}^{0}(x) := \eps^3 \sum_{y \in \Lattice} \psi_{\ga, \de}(x - y) X_\ga^0(y),
\end{equation}
where $X_\ga^0$ is defined in the statement of Theorem~\ref{thm:main} and the function $\psi_{\ga, \de}$ is a discrete approximation of the function $\psi_{\de}$ from \eqref{eq:Phi43-delta}:
\begin{equation*}
\psi_{\ga, \de}(x) := \eps^{-3} \int_{|y - x| \leq \eps/2} \psi_\de(y) \d y.
\end{equation*}
As in Lemma~\ref{lem:solution} we can readily conclude that there is a choice of $A_{\ga, \de}$ such that the function $X_{\ga, \de, \fa} = \CR^{\ga, \de, \fa} U_{\ga, \de, \fa}$ solves 
\begin{equation}\label{eq:Discrete-equation-regular-noise}
	X_{\ga, \de, \fa}(t, x) = P^\ga_t X_{\ga, \de}^{0}(x) + \int_0^t \widetilde{P}^\ga_{t-s} \Bigl( -\frac{\be^3}{3} \bigl(X_{\ga, \de, \fa}\bigr)^3 + (\fC_{\ga, \de} + A) X_{\ga, \de, \fa} + \sqrt 2\,\xi_{\ga, \de, \fa} \Bigr)(s, x)\, \d s.
\end{equation}
where the driving noise is defined in \eqref{eq:xi-gamma-delta}. This equation is a modification of the Ising-Kac equation \eqref{eq:IsingKacEqn-new}, driven by a mollified noise and without the error term. The renormalisation $\fC_{\ga, \de}$ we take to be in the form \eqref{eq:C-exact}, but defined via the constants $\fc_{\ga, \de}$, $\fc_{\ga, \de}'$ and $\fc_{\ga, \de}''$ introduced in \eqref{eq:convolved_constants}.

Now we will study the solution map of \eqref{eq:AbstractDiscreteEquation}. In particular, we need to show that it is continuous with respect to the model and the initial state.

\begin{proposition}\label{prop:SolutionMap}
Let $Z^{\ga, \fa}_{\lift}$ be the random discrete model constructed in Section~\ref{sec:lift} and let the initial state $X_\ga^0$ satisfy the assumptions of Theorem~\ref{thm:main}. Then for almost every realisation of $Z^{\ga, \fa}_{\lift}$ there exists (possibly infinite) $T_{\ga, \fa} > 0$ such that \eqref{eq:AbstractDiscreteEquation} has a unique solution $U_{\ga, \fa} \in \CD^{\zeta, \eta}_\emezo(Z^{\ga, \fa}_\lift)$ on the time interval $[0, T_{\ga, \fa})$, where $\zeta = 1 + 3 \kappa$ and the constant $\eta$ is from Theorem~\ref{thm:main}. 

Let moreover $X_{\ga, \fa} = \CR^{\ga, \fa} U_{\ga, \fa}$ where $\CR^{\ga, \fa}$ is the reconstruction map \eqref{eq:def_rec_op} associated to the model. Then for every $L > 0$ there is $T^L_{\ga, \fa} \in (0, T_{\ga, \fa})$, such that $\lim_{L \to \infty} T^L_{\ga, \fa} = T_{\ga, \fa}$ almost surely, and 
\begin{equation}\label{eq:SolutionMap-bound}
\sup_{t \in [0, T \wedge T^L_{\ga, \fa}]} \| X_{\ga, \fa}(t) \|^{(\emezo)}_{\CC^\eta} \leq C,
\end{equation}
for any $T > 0$, provided $\| X^0_{\ga, \fa} \|^{(\emezo)}_{\CC^\eta} \leq L$ and $\$ Z^{\ga, \fa}_{\lift}\$_{T+1}^{(\emezo)} \leq L$, where we use the norm \eqref{eq:eps-norm-1}. The constant $C$ depends on $L$ and is independent of $\ga$.

Let $Z^{\ga, \de, \fa}_\lift$ be the model defined in Section~\ref{sec:convergence-of-models}. Then there is a solution $U_{\ga, \de, \fa} \in \CD^{\zeta, \eta}_\emezo(Z^{\ga, \de, \fa}_\lift)$ of equation \eqref{eq:AbstractDiscreteEquation-delta} on an interval $[0, T_{\ga, \de, \fa})$. Let furthermore $X_{\ga, \de, \fa} = \CR^{\ga, \de, \fa} U_{\ga, \de, \fa}$, where $\CR^{\ga, \de, \fa}$ is the respective reconstruction map. Then there exist $\delta_0 > 0$, $\theta > 0$ and $T^L_{\ga, \de, \fa} \in (0, T_{\ga, \de, \fa})$, such that $\lim_{L \to \infty} T^L_{\ga, \de, \fa} = T_{\ga, \de, \fa}$ almost surely and 
\begin{equation}\label{eq:SolutionMap-bound-delta}
\sup_{t \in [0, T \wedge T^L_{\ga, \fa} \wedge T^L_{\ga, \de, \fa}]} \| (X_{\ga, \fa} - X_{\ga, \de, \fa})(t) \|^{(\emezo)}_{\CC^\eta} \leq C \de^{\theta},
\end{equation}
uniformly over $\de \in (0, \de_0)$, provided $\| X^0_{\ga, \fa} \|^{(\emezo)}_{\CC^\eta} \leq L$, $\$ Z^{\ga, \fa}_{\lift}\$_{T+1}^{(\emezo)} \leq L$, $\| X^0_{\ga, \fa} - X^{0}_{\ga, \delta, \fa} \|^{(\emezo)}_{\CC^\eta} \leq \de^{\theta}$ and $\$ Z^{\ga, \fa}_{\lift}; Z^{\ga, \de, \fa}_{\lift}\$_{T+1}^{(\emezo)} \leq \de^{\theta}$. 
\end{proposition}

\begin{proof}
To prove existence of a local solution, we use a purely deterministic argument. For this, we take $T > 0$ and any realisation of the discrete model $Z^{\ga}_{\lift}$ such that $\$ Z^{\ga, \fa}_{\lift}\$_{T+1}^{(\emezo)}$ is finite. Proposition~\ref{prop:models-converge} suggests that it happens almost surely. The spaces of modelled distributions are considered below with respect to $Z^{\ga}_{\lift}$.

We proved in Lemma~\ref{lem:solution} that if a solution $U_{\ga, \fa}$ exists, then it has the form \eqref{eq:expansion}. Hence, in this proof we will be looking for a solution in this form. 

Let $\CM^{\ga, \fa}_{T}(U_{\ga, \fa})$ be the right-hand side of \eqref{eq:AbstractDiscreteEquation}, restricted to the time interval $[0, T]$. We need to prove that $\CM^{\ga, \fa}_{T}$ is a contraction map on $\CD^{\zeta, \eta}_{\emezo, T}$, uniformly in $\ga$, for $T > 0$ small enough (see Remark~\ref{rem:T-space} for the definition of the time-dependent space). More precisely, let us take $U_{\ga, \fa}, \bar U_{\ga, \fa} \in \CD^{\zeta, \eta}_{\emezo, T}$. Then we will prove that for some $\nu > 0$ we have 
\begin{subequations}\label{eqs:contraction}
\begin{align}
\$\CM^{\ga, \fa}_{T}(U_{\ga, \fa}) \$_{\zeta, \eta; T}^{(\emezo)} &\lesssim \| X_\ga^0 \|^{(\emezo)}_{\CC^\eta} + T^\nu \bigl(1 + \$ U_{\ga, \fa} \$_{\zeta, \eta; T}^{(\emezo)}\bigr)^5, \label{eq:contraction1}\\
\$ \CM^{\ga, \fa}_{T}(U_{\ga, \fa}); \CM^{\ga, \fa}_{T}(\bar U_{\ga, \fa}) \$_{\zeta, \eta; T}^{(\emezo)} &\lesssim T^\nu \bigl(1 + \$ U_{\ga, \fa} \$_{\zeta, \eta; T}^{(\emezo)}\bigr)^4 \$ U_{\ga, \fa}; \bar U_{\ga, \fa} \$_{\zeta, \eta; T}^{(\emezo)}. \label{eq:contraction2}
\end{align}
\end{subequations}
Then for any $T > 0$ small enough, $\CM^{\ga, \fa}_{T}$ is a contraction map on $\CD^{\zeta, \eta}_{\emezo, T}$. The proportionality constants in these bounds are multiples of $\$ Z^{\ga, \fa}_\lift \$^{(\emezo)}_{T+1}$, which implies that $0 < T < T_{\ga, \fa}$, for some $T_{\ga, \fa} > 0$, depending on $\$ Z^{\ga, \fa}_\lift \$^{(\emezo)}_{T+1}$. 

We first prove the bound \eqref{eq:contraction1}. For $\bar\zeta > 0$ and $\bar\eta > - 2$, we apply \cite[Thm.~4.22]{erhard2017discretisation} and get
\begin{align}\label{eq:M-bound}
\$\CM^{\ga, \fa}_{T}(U_{\ga, \fa}) \$_{\zeta, \eta; T}^{(\emezo)} &\lesssim \$ G^\ga X_\ga^0 \$_{\zeta, \eta; T}^{(\emezo)} + \$ W_{\ga, \fa} \$_{\zeta, \eta; T}^{(\emezo)} \\
&\qquad + T^\nu \Bigl( \$ F_\ga(U_{\ga, \fa}) \$_{\bar \zeta, \bar\eta; T}^{(\emezo)} + \$ E^{(1)}_\ga(U_{\ga, \fa}) \$_{\bar \zeta, \bar\eta; T}^{(\emezo)} + \$ E^{(2)}_\ga(U_{\ga, \fa}) \$_{\bar \zeta, \bar\eta; T}^{(\emezo)} \Bigr), \nonumber
\end{align}
for some $\nu > 0$. We are going to bound the terms on the right-hand side one by one, and a precise choice of $\bar \zeta$ and $\bar \eta$ will be clear from these bounds.

 Similarly to \cite[Lem.~7.5]{Regularity}, we get $\$ G^\ga X_\ga^0 \$_{\zeta, \eta; T}^{(\emezo)} \lesssim \| X_\ga^0 \|^{(\emezo)}_{\CC^\eta}$. Furthermore, from \cite[Lem.~2.3]{Martingales} and \cite[Thm.~4.22]{erhard2017discretisation} we have the bound $\$ W_{\ga, \fa} \$_{\zeta, \eta; T}^{(\emezo)} \lesssim T^\nu$ on the term \eqref{eq:V-gamma}.

Now, we will bound the function \eqref{eq:F-ga}. From \cite[Sec.~4 and 6.2]{Regularity} we get
\begin{equation*}
\$ F_\ga(U_{\ga, \fa}) \$_{\zeta_1, \eta_1; T}^{(\emezo)} \lesssim \$ U_{\ga, \fa}^3 \$_{\zeta_1, \eta_1; T}^{(\emezo)} + \$ U_{\ga, \fa} \$_{\zeta_1, \eta_1; T}^{(\emezo)} \lesssim \bigl(\$ U_{\ga, \fa} \$_{\zeta, \eta; T}^{(\emezo)}\bigr)^3 + \$ U_{\ga, \fa} \$_{\zeta, \eta; T}^{(\emezo)},
\end{equation*}
for $\zeta_1 \leq \zeta - 1 - 2\kappa$ and $\eta_1 \leq \eta - 1 - 2 \kappa$. Here, we used the fact that $U_{\ga, \fa}$ lives in a sector of regularity $\alpha = -\frac{1}{2} - \kappa$. Recalling that $\zeta = 1 + 3 \kappa$ and $\kappa < \frac{1}{14}$, the ranges of $\zeta_1$ and $\eta_1$ allow to chose $\bar \zeta$ and $\bar \eta$ as in \eqref{eq:M-bound}.

Now we will bound the function \eqref{eq:E1-ga}. From Proposition~\ref{prop:ReconThm_v2} we get $| (\CR^{\ga, \fa} U_{\ga, \fa}) (z) | \lesssim \emezo^{\alpha \wedge \eta} \$ U_{\ga, \fa} \$_{\zeta, \eta; T}^{(\emezo)}$. Then for $r \in (6, 7)$ the definition \eqref{eq:Taylor} yields
\begin{align*}
\Bigl|\frac{1}{\delta \alpha} R_5 \bigl(\be \ga^3 \CR^{\ga, \fa} U_{\ga, \fa}(z)\bigr)\Bigr| &\lesssim \frac{1}{\delta \alpha} |\be \ga^3 \CR^{\ga, \fa} U_{\ga, \fa}(z)|^{r} \\
&\lesssim \ga^{3 r - 9} \emezo^{r (\alpha \wedge \eta)} \$ U_{\ga, \fa} \$_{\zeta, \eta; T}^{(\emezo)} \lesssim \ga^{\frac{3 r}{2} - 9 - 3 r\kappa} \$ U_{\ga, \fa} \$_{\zeta, \eta; T}^{(\emezo)}.
\end{align*}
From this we obtain the following bound on the function \eqref{eq:E1-ga}:
\begin{equation*}
\$ E^{(1)}_\ga(U_{\ga, \fa}) \$_{\bar\zeta, \bar\eta; T}^{(\emezo)} \lesssim \ga^{\frac{3 r}{2} - 9 - 3 r\kappa} \$ U_{\ga, \fa} \$_{\zeta, \eta; T}^{(\emezo)}.
\end{equation*}
If $\kappa < \frac{1}{14}$, then there is a value of $r$ such that the last term vanishes as $\ga \to 0$. 

In order to bound the function \eqref{eq:E2-ga}, we need to bound the modelled distribution inside the brackets in \eqref{eq:E2-ga}, which we denote by $\widetilde{V}_{\ga, \fa}$. Using the expansion \eqref{eq:U5-expansion}, we can write  
\begin{equation}\label{eq:U5-expansion-new}
U_{\ga, \fa}(z)^5 = 4 \sqrt2\, \<5b> + 20 v_{\ga, \fa}(z)\, \<4b> + \widetilde{U}_{\ga, \fa}(z),
\end{equation}
where the elements spanning $\widetilde{U}_{\ga, \fa}$ have homogeneities greater than $-\frac{3}{2} - 7 \kappa$. We note that $\widetilde{U}_{\ga, \fa}(z)$ does not belong to $\CT^\ex$, but is rather an element of $\fT^\ex$ (see Section~\ref{sec:model-space} for the definition of this space). In particular, we cannot apply the model to $\widetilde{U}_{\ga, \fa}(z)$ and hence we cannot measure the regularity of $\widetilde{U}_{\ga, \fa}(z)$ as a modelled distribution. Instead, we write 
\begin{equation}\label{eq:H}
\widetilde{V}_{\ga, \fa}(z) = 4 \sqrt2\, \<5b> + 20 v_{\ga, \fa}(z)\, \<4b>  + \biggl(H_5 \bigl(\CR^{\ga, \fa} U_{\ga, \fa}(z), 2 \fc_\ga\bigr) - \sum_{\tau \in \{\scalebox{0.7}{\<4b>}, \scalebox{0.7}{\<5b>}\}} \bigl(\CR^{\ga, \fa} \CQ_{\tau} U_{\ga, \fa}^5\bigr)(z) \biggr) \blueOne,
\end{equation}
and we are going to show that this is a modelled distribution in a suitable space. Table~\ref{tab:linear_transformations} suggests that $\Gamma^{\ga, \fa}_{\bar z z} \widetilde{V}_{\ga, \fa}(z) = \widetilde{V}_{\ga, \fa}(z)$, and hence the second term in the definition \eqref{eq:def_dgamma_norm} of modelled distributions contains the difference $\widetilde{V}_{\ga, \fa}(z) - \widetilde{V}_{\ga, \fa}(\bar z)$. Now, we will derive bounds on $\widetilde{V}_{\ga, \fa}(z)$ and $\widetilde{V}_{\ga, \fa}(z) - \widetilde{V}_{\ga, \fa}(\bar z)$.

For the first term in \eqref{eq:H} we have 
\begin{equation}\label{eq:H-bound1}
| \widetilde{V}_{\ga, \fa}(z) |_{|\scalebox{0.7}{\<5b>}|} = 4 \sqrt2, \qquad\qquad | \widetilde{V}_{\ga, \fa}(z) - \widetilde{V}_{\ga, \fa}(\bar z)|_{|\scalebox{0.7}{\<5b>}|} = 0.
\end{equation}
Since $U_{\ga, \fa} \in \CD^{\zeta, \eta}_{\emezo, T}$ and the expansion \eqref{eq:expansion} holds, we conclude that
\begin{align}
|v_{\ga, \fa}(z)| &= \frac{1}{2 \beta^3} |U_{\ga, \fa} (z)|_{0} \lesssim \bigl(\| z \|_\s \vee \emezo\bigr)^{\eta} \$ U_{\ga, \fa} \$_{\zeta, \eta; T}^{(\emezo)}, \label{eq:v-bound}\\
|v_{\ga, \fa}(z) - v_{\ga, \fa}(\bar z)| &= \frac{1}{2 \beta^3} |U_{\ga, \fa} (z) - \Gamma^{\ga, \fa}_{\!z \bar z} U_{\ga, \fa}(\bar z)|_{|\scalebox{0.7}{\<20b>}|} \\
&\lesssim \bigl(\| z - \bar z \|_\s \vee \emezo\bigr)^{\zeta - |\scalebox{0.7}{\<20b>}|} \Vert z, \bar z\Vert_{\emezo}^{\eta - \zeta} \$ U_{\ga, \fa} \$_{\zeta, \eta; T}^{(\emezo)},\nonumber
\end{align}
where we used the definition of the modelled distribution \eqref{eq:def_dgamma_norm}. Hence, for the second term in \eqref{eq:H} we have 
\begin{equation}\label{eq:H-bound2}
\begin{aligned}
| \widetilde{V}_{\ga, \fa}(z) |_{|\scalebox{0.7}{\<4b>}|} &\lesssim \bigl(\| z \|_\s \vee \emezo\bigr)^{\eta} \$ U_{\ga, \fa} \$_{\zeta, \eta; T}^{(\emezo)}, \\
 | \widetilde{V}_{\ga, \fa}(z) - \widetilde{V}_{\ga, \fa}(\bar z)|_{|\scalebox{0.7}{\<4b>}|} &\lesssim \bigl(\| z - \bar z \|_\s \vee \emezo\bigr)^{\zeta - |\scalebox{0.7}{\<20b>}|} \Vert z, \bar z\Vert_{\emezo}^{\eta - \zeta} \$ U_{\ga, \fa} \$_{\zeta, \eta; T}^{(\emezo)}.
\end{aligned}
\end{equation}

Now, we will bound the last term in \eqref{eq:H}. From the expansion \eqref{eq:expansion} and Remark~\ref{rem:positive-vanish}, we get 
\begin{equation*}
\CR^{\ga, \fa} U_{\ga, \fa}(z) = \sqrt2\, Y_{\ga, \fa}(z) + v_{\ga, \fa}(z),
\end{equation*}
where $Y_{\ga, \fa} = \CR^{\ga, \fa}\<1b>$. Using then the expansion \eqref{eq:U5-expansion-new}, the definition of the reconstruction map \eqref{eq:def_rec_op} and the definition of the model \eqref{eq:model-Hermite}, the last term in \eqref{eq:H} may be written as
\begin{align}
&H_5 \bigl(\CR^{\ga, \fa} U_{\ga, \fa}(z), 2 \fc_\ga\bigr) - \sum_{\tau \in \{\scalebox{0.7}{\<4b>}, \scalebox{0.7}{\<5b>}\}} \bigl(\CR^{\ga, \fa} \CQ_{\tau} U_{\ga, \fa}^5\bigr) (z) \label{eq:last-term} \\
&\qquad = H_5 \bigl(\CR^{\ga, \fa} U_{\ga, \fa}(z), 2 \fc_\ga\bigr) - 4 \sqrt2\, \bigl(\CR^{\ga, \fa} \<5b>\bigr)(z) - 20 v_{\ga, \fa}(z)\, \bigl(\CR^{\ga, \fa} \<4b>\bigr)(z) \nonumber \\
&\qquad = H_5 \bigl(\sqrt2\, Y_{\ga, \fa}(z) + v_{\ga, \fa}(z), 2 \fc_\ga\bigr) - 4 \sqrt2\, H_5 \bigl(Y_{\ga, \fa}(z), \fc_\ga\bigr) - 20 v_{\ga, \fa}(z)\, H_4 \bigl(Y_{\ga, \fa}(z), \fc_\ga\bigr). \nonumber
\end{align}
The following expansion holds for the Hermite polynomials
\begin{equation*}
H_n(u + v, c) = \sum_{m = 0}^n {n \choose m} H_m(u, c) v^{n - m},
\end{equation*}
which can be found in \cite{Abramowitz}. Moreover, from the definition \eqref{eq:def_Hermite} we get the scaling identity $H_n(a u, a^2 c) = a^n H_n(u, c)$ for any $a > 0$. Applying these identities, expression \eqref{eq:last-term} turns to
\begin{align}
&H_5 \bigl(\CR^{\ga, \fa} U_{\ga, \fa}(z), 2 \fc_\ga\bigr) - \sum_{\tau \in \{\scalebox{0.7}{\<4b>}, \scalebox{0.7}{\<5b>}\}} \bigl(\CR^{\ga, \fa} \CQ_{\tau} U_{\ga, \fa}^5\bigr) (z) \label{eq:last-term-new} \\
&\quad = \sum_{m = 0}^3 {5 \choose m} 2^{\frac{m}{2}} H_m \bigl(Y_{\ga, \fa}(z), \fc_\ga\bigr) v_{\ga, \fa}(z)^{5 - m} = \sum_{m = 0}^3 {5 \choose m} 2^{\frac{m}{2}} \bigl(\CR^{\ga, \fa} \<1b>^{m}\bigr)(z) v_{\ga, \fa}(z)^{5 - m}. \nonumber
\end{align}
where we postulate $\bigl(\CR^{\ga, \fa} \<1b>^{0}\bigr)(z) = 1$. For $m \in \{1, 2, 3\}$, the definitions \eqref{eq:def_rec_op} and the bound \eqref{eq:Pi-bounds} yield $| (\CR^{\ga, \fa} \<1b>^{m} )(z) | \lesssim \emezo^{|\scalebox{0.7}{\<1b>}| m} \$ Z^{\ga, \fa}_{\lift}\$_{T+1}^{(\emezo)}$. Combining this with the bound on the function $v_\ga$ in \eqref{eq:v-bound}, we estimate expression \eqref{eq:last-term-new} as
\begin{equation}\label{eq:H-bound3}
 | \widetilde{V}_{\ga, \fa}(z) |_{0} \lesssim \sum_{m = 0}^3 \emezo^{|\scalebox{0.7}{\<1b>}| m} \bigl(\| z \|_\s \vee \emezo\bigr)^{(5 - m) \eta},
\end{equation}
where the proportionality constant is a multiple of $\bigl(1 + \$ Z^{\ga, \fa}_{\lift}\$_{T+1}^{(\emezo)}\bigr) \bigl(1 + \$ U_{\ga, \fa} \$_{\zeta, \eta; T}^{(\emezo)} \bigr)^5$.

Using the derived bounds, we can now estimate the function \eqref{eq:E2-ga}. From Table~\ref{tab:linear_transformations} we conclude that $\Gamma^{\ga, \fa}_{\bar z z} (\widehat{\CE}_\ga\widetilde{V}_{\ga, \fa})(z) = (\widehat{\CE}_\ga\widetilde{V}_{\ga, \fa})(z)$, and the second term in the definition of the norm \eqref{eq:def_dgamma_norm} contains only the difference $(\widehat{\CE}_\ga\widetilde{V}_{\ga, \fa})(z) - (\widehat{\CE}_\ga\widetilde{V}_{\ga, \fa})(\bar z)$. Hence, we need to bound $(\widehat{\CE}_\ga\widetilde{V}_{\ga, \fa})(z)$ and $(\widehat{\CE}_\ga\widetilde{V}_{\ga, \fa})(z) - (\widehat{\CE}_\ga\widetilde{V}_{\ga, \fa})(\bar z)$.

From \eqref{eq:H-bound1} we get 
\begin{equation*}
| (\widehat{\CE}_\ga\widetilde{V}_{\ga, \fa})(z) |_{|\scalebox{0.7}{\<5b>}| + 2} = 4 \sqrt2, \qquad \qquad | (\widehat{\CE}_\ga\widetilde{V}_{\ga, \fa})(z) - (\widehat{\CE}_\ga\widetilde{V}_{\ga, \fa})(\bar z) |_{|\scalebox{0.7}{\<5b>}| + 2} = 0.
\end{equation*}
Similarly, from \eqref{eq:H-bound2} we have 
\begin{align*}
| (\widehat{\CE}_\ga\widetilde{V}_{\ga, \fa})(z) |_{|\scalebox{0.7}{\<4b>}| + 2} &\lesssim (\| z \|_\s \vee \emezo)^{\eta} \$ U_{\ga, \fa} \$_{\zeta, \eta; T}^{(\emezo)}, \\
 | (\widehat{\CE}_\ga\widetilde{V}_{\ga, \fa})(z) - (\widehat{\CE}_\ga\widetilde{V}_{\ga, \fa})(\bar z)|_{|\scalebox{0.7}{\<4b>}| + 2} &\lesssim (\| z - \bar z \|_\s \vee \emezo)^{\zeta - |\scalebox{0.7}{\<20b>}|} \Vert z, \bar z\Vert_{\emezo}^{\eta - \zeta} \$ U_{\ga, \fa} \$_{\zeta, \eta; T}^{(\emezo)}.
\end{align*}
Finally, \eqref{eq:H-bound3} yields
\begin{align*}
| (\widehat{\CE}_\ga \widetilde{V}_{\ga, \fa})(z) |_{0} &\lesssim \ga^6 \sum_{m = 0}^3 \emezo^{|\scalebox{0.7}{\<1b>}| m} \bigl(\| z \|_\s \vee \emezo\bigr)^{(5 - m) \eta} \lesssim \emezo^{\vartheta} (\| z \|_{\s} \vee \emezo)^{\eta_2-\vartheta} ,
\end{align*}
for any $0 < \vartheta \leq \frac{1}{2} - 3 \kappa$ and $\eta_2 = 5 \eta + 2$ (recall that $|\<1b>| = -\frac{1}{2} - \kappa$ and $|\<20b>| = 1 - 2\kappa$), and where the proportionality constant is a multiple of $\bigl(1 + \$ Z^{\ga, \fa}_{\lift}\$_{T+1}^{(\emezo)}\bigr) \bigl(1 + \$ U_{\ga, \fa} \$_{\zeta, \eta; T}^{(\emezo)} \bigr)^5$. Using this bound, we get furthermore 
\begin{align*}
&| (\widehat{\CE}_\ga \widetilde{V}_{\ga, \fa})(z) - (\widehat{\CE}_\ga \widetilde{V}_{\ga, \fa})(\bar z) |_{0} \leq | (\widehat{\CE}_\ga \widetilde{V}_{\ga, \fa})(z) |_{0} + |(\widehat{\CE}_\ga \widetilde{V}_{\ga, \fa})(\bar z) |_{0}\\
&\qquad \lesssim \emezo^{\vartheta} \Bigl((\| z \|_{\s} \vee \emezo)^{\eta_2-\vartheta} + (\| \bar z \|_{\s} \vee \emezo)^{\eta_2-\vartheta}\Bigr) \lesssim \emezo^{\vartheta - \bar \vartheta} (\| z - \bar z\|_\s \vee \emezo)^{\bar \vartheta} \Vert z, \bar z\Vert_{\emezo}^{\eta_2-\vartheta} ,
\end{align*}
for any $0 < \bar \vartheta < \vartheta$. Combining the preceding bounds on $\widehat{\CE}_\ga \widetilde{V}_{\ga, \fa}$, we conclude that the following bound holds for the function \eqref{eq:E2-ga}:
\begin{equation}\label{eq:E2-bound}
\$ E^{(2)}_\ga(U_{\ga, \fa}) \$_{\zeta_3, \eta_3; T}^{(\emezo)} \lesssim \emezo^{\vartheta - \bar \vartheta} \bigl(1 + \$ Z^{\ga, \fa}_{\lift}\$_{T+1}^{(\emezo)}\bigr) \bigl(1 + \$ U_{\ga, \fa} \$_{\zeta, \eta; T}^{(\emezo)} \bigr)^5,
\end{equation}
for any $\zeta_3$ and $\eta_3$ satisfying $\zeta_3 \leq \bar \vartheta$, $\zeta_3 \leq \zeta - |\<20b>| + |\<4b>| + 2$, $\eta_3 \leq \eta_2 - \vartheta$, $\eta_3 \leq \eta - |\<20b>| + |\<4b>| + 2$, $\eta_3 - \zeta_3 \leq \eta - \zeta$ and $\eta_3 - \zeta_3 \leq \eta_2 - \vartheta$. Taking $\vartheta = 2 \kappa$, $\bar \vartheta = \kappa$, $\zeta_3 = \kappa$ and $\eta_3 = \eta - 1 - 2 \kappa$, all these conditions are satisfied and moreover we have $\zeta_3 > 0$ and $\eta_3 > -2$, which allows to take $\bar \zeta$ and $\bar \eta$ as in \eqref{eq:M-bound}. We note that \eqref{eq:E2-bound} vanishes as $\ga \to 0$, because the power of $\emezo$ is strictly positive. 

We have just finished the proof of the bound \eqref{eq:M-bound}, from which \eqref{eq:contraction1} follows. The bound \eqref{eq:contraction2} can be proved similarly and we prefer to omit the details. Then the Banach fixed point theorem yields existence of a fixed point of the map $\CM^{\ga, \fa}_{T}$, and hence we get a local solution of equation \eqref{eq:AbstractDiscreteEquation}. By patching the local solution in the standard way, we get the maximal time $T_{\ga, \fa}$ such that the solution exists on the time interval $[0, T_{\ga, \fa})$. One can see that the time $T_{\ga, \fa}$ is the one at which $\| X_{\ga, \fa}(t) \|^{(\emezo)}_{\CC^\eta}$ diverges. Applying Proposition~\ref{prop:ReconThm_v2} to the function $X_{\ga, \fa} = \CR^{\ga, \fa} U_{\ga, \fa}$, we then get the required bound \eqref{eq:SolutionMap-bound}.

A bound on the solutions $U_{\ga, \de, \fa}$ can be proved respectively. Furthermore, in the same way as we proved \eqref{eq:SolutionMap-bound}, we get the bound \eqref{eq:SolutionMap-bound-delta}.
\end{proof}

Proposition~\ref{prop:SolutionMap} gives a local solution $X_{\ga, \fa}$, and by analogy with \eqref{eq:remainder-equation} we can also study the respective solution $v_{\ga, \fa}$ of the remainder equation \eqref{eq:discrete-remainder}. More precisely, we define it as $v_{\ga, \fa} = X_{\ga, \fa} - \sqrt2\, Y_{\ga, \fa}$, where $Y_{\ga, \fa} = \CR^{\ga, \fa}\<1b>$. Then from Proposition~\ref{prop:SolutionMap} we can conclude that in the setting of \eqref{eq:SolutionMap-bound} we have 
\begin{equation}\label{eq:discrete-remainder-bound}
\sup_{t \in [0, T \wedge T^L_{\ga, \fa}]} \| v_{\ga, \fa}(t) \|^{(\emezo)}_{\CC^{3 / 2 + 3 \eta}} \leq C.
\end{equation}
In the same way, for a local solution $X_{\ga, \de, \fa}$ we set $Y_{\ga, \de, \fa} = \CR^{\ga, \de, \fa}\<1b>$ and $v_{\ga, \de, \fa} = X_{\ga, \de, \fa} - \sqrt2\, Y_{\ga, \de, \fa}$. Then in the setting of \eqref{eq:SolutionMap-bound-delta} we have
\begin{equation}\label{eq:discrete-remainder-bound-delta}
\sup_{t \in [0, T \wedge T^L_{\ga, \fa} \wedge T^L_{\ga, \de, \fa}]} \| (v_{\ga, \fa} - v_{\ga, \de, \fa})(t) \|^{(\emezo)}_{\CC^{3 / 2 + 3 \eta}} \leq C \de^\theta.
\end{equation}

\subsection{Controlling the process $\un{X}_{\ga, \fa}$}

Similarly to $X_{\ga, \fa}$, we can also control the process $\un{X}_{\ga, \fa}$ defined in \eqref{eq:X-under}. For this, we define the discrete kernel $\un{P}^\ga_t(x) := \bigl(P^\ga_t \ae \un{K}_\ga\bigr)(x)$ on $x \in \Le$ and by analogy with \eqref{eq:IsingKacEqn} we then get 
\begin{align*}
	\un{X}_\ga(t, x) &= P^\ga_t \un{X}^0_\ga(x) + \sqrt 2\, \un{Y}_\ga(t, x) \\
	&\qquad + \int_0^t \un{P}^{\ga}_{t-s} \Bigl( -\frac{\be^3}{3} X^3_\ga + \bigl(\fC_\ga + A\bigr) X_\ga + E_\ga \Bigr)(s, x)\, \d s,
\end{align*}
where 
\begin{equation*}
	\un{Y}_\ga (t, x) := \frac{1}{\sqrt 2} \eps^3 \sum_{y \in \Le} \int_0^t \un{P}^{\ga}_{t-s}(x-y) \,\d \M_\ga(s, y).
\end{equation*}
We defined the respective kernel $\un{G}^\ga_t(x)$ on $x \in \Lattice$ by \eqref{eq:From-P-to-G}. This kernel is different from $\widetilde{G}^\ga$ only by the scale, which is $\emezo$ for the latter and $\un{\emezo} := \emezo \ga^{\un\kappa}$ for the former. Hence, in the same way as we did in Appendix~\ref{sec:decompositions}, we may write $\un{G}^\ga = \un{\SK}^\ga + \un{\SR}^\ga$ and we may defined the respective abstract map $\un{\CP}^\ga$ as in \eqref{eq:P-operator}. We also define the respective lift of the martingales $\un{Z}^{\ga, \fa}_{\lift}$, which is defined in the same way as $Z^{\ga, \fa}_{\lift}$ in Section~\ref{sec:model-lift}, but where in the definitions \eqref{eq:lift-hermite} and \eqref{eq:Pi-I} we use the kernel $\un{\SK}^\ga$. We note that we need to use the norms on scale $\un{\emezo}$ to work with these objects, i.e. we have $\$ \un{Z}^{\ga, \fa}_\lift \$^{(\un{\emezo})}_{T}$ bounded and $\un{\CP}^\ga$ acts on suitable spaces $\CD^{\zeta, \eta}_{\un{\emezo}, T}$. If $U_{\ga, \fa}$ is a solution of \eqref{eq:AbstractDiscreteEquation}, then we define 
\begin{equation}\label{eq:U-bar}
\un{U}_{\ga, \fa} = \CQ_{< \zeta} \Bigl(G^\ga \un{X}_\ga^0 + \un{\CP}^\ga \1_+ \bigl( F_\ga(U_{\ga, \fa}) + E^{(1)}_\ga(U_{\ga, \fa}) + E^{(2)}_\ga(U_{\ga, \fa}) \bigr) + \sqrt 2\, \un{W}_{\ga, \fa}\Bigr),
\end{equation}
where $\un{W}_{\ga, \fa}(z) := \un{\CP}^\ga \1_+ (\blueXi)(z)$. We have from Lemma~\ref{lem:solution} that the solution of \eqref{eq:IsingKacEqn-new} is obtained as $X_{\ga, \fa} = \CR^{\ga, \fa} U_{\ga, \fa}$. Recalling that $X_{\ga, \fa}$ equals $X_{\ga}$, the solution of \eqref{eq:IsingKacEqn}, on the time interval $[0, \tau_{\ga, \fa}]$, we conclude that $\un{X}_{\ga} = \un{\CR}^\ga \un{U}_{\ga, \fa}$ on $[0, \tau_{\ga, \fa}]$. Furthermore, we may get a bound on $\un{X}_{\ga, \fa} = \un{\CR}^\ga \un{U}_{\ga, \fa}$.

\begin{proposition}\label{prop:under-X}
Let $X_{\ga, \fa}$ be the local solution defined in Proposition~\ref{prop:SolutionMap}, and let $\un{X}_{\ga, \fa}$ be as above. Then in the setting of \eqref{eq:SolutionMap-bound} one has
\begin{equation*}
\sup_{t \in [0, T \wedge T^L_{\ga, \fa}]} \| \un{X}_{\ga, \fa}(t) \|^{(\un{\emezo})}_{\CC^\eta} \leq C,
\end{equation*}
where we use the norm \eqref{eq:eps-norm-1} with the scale $\un{\emezo} := \emezo \ga^{\un\kappa}$.
\end{proposition}

\begin{proof}
For any $0 < \tilde{\emezo} \leq \emezo$ we have $\| X^{0}_\ga\|^{(\tilde{\emezo})}_{\CC^{\bar \eta}} \leq \| X^{0}_\ga\|^{(\emezo)}_{\CC^{\bar \eta}}$. Taking $\un{\emezo} < \tilde{\emezo} < \emezo$, Lemma~\ref{lem:Z-bound} yields $\bigl| \bigl( \iota_\eps S_{\ga} (t)\bigr) (\phi_x^\lambda) \bigr| \lesssim \la^{\bar\eta} \| X^{0}_\ga\|^{(\emezo)}_{\CC^{\bar \eta}}$ for any smooth compactly supported $\varphi$ and $\lambda \in [\un{\emezo}, 1]$. From the assumption \eqref{eq:initial-convergence} on the initial condition, the last quantity is bounded uniformly in $\ga \in (0, \ga_\star)$. Since the function $\un{K}_\ga$ is smooth and rescaled by $\un{\emezo}$, we get $\sup_{\ga \in (0, \ga_\star)} \| \un{X}^{0}_\ga\|^{(\un{\emezo})}_{\CC^{\bar \eta}} < \infty$. Estimating then the right-hand side of \eqref{eq:U-bar} in exactly the same way as we bounded \eqref{eq:contraction1}, we get $\$\un{U}_{\ga, \fa} \$_{\zeta, \eta; T}^{(\un\emezo)} \lesssim 1$, for any $T \in (0, T_{\ga, \fa})$, where the proportionality constant is independent of $\ga$ and $T$. Recalling that $\un{X}_{\ga, \fa} = \un{\CR}^\ga \un{U}_{\ga, \fa}$, the bound \eqref{eq:SolutionMap-bound} follows from Proposition~\ref{prop:ReconThm_v2} and moment bounds for the model.
\end{proof}

Let us define $\un{v}_{\ga, \fa} := \un{X}_{\ga, \fa} - \sqrt2\, \un{Y}_{\ga, \fa}$ with $\un{Y}_{\ga, \fa} = \un{\CR}^{\ga, \fa}\<1b>$. Then by bounding the right-hand side of \eqref{eq:U-bar} without the term $\sqrt 2\, \un{W}_{\ga, \fa}$, in the same way as we did in the proof of Proposition~\ref{prop:under-X}, in the setting of \eqref{eq:SolutionMap-bound} we get 
\begin{equation}\label{eq:remainder-bound}
\sup_{t \in [0, T \wedge T^L_{\ga, \fa}]} \| \un{v}_{\ga, \fa}(t) \|^{(\un{\emezo})}_{\CC^{3 / 2 + 3 \eta}} \leq C.
\end{equation}

We also need to control the process $\un{X}_\ga X_\ga$ which appears in the definition of the stopping time \eqref{eq:tau-2}. In what follows, when using the norm $\| \bigcdot \|_{L^\infty}$ of these processes, we compute the norm on $\Le$. Writing as before $X_{\ga, \fa} = \sqrt2\, Y_{\ga, \fa} + v_{\ga, \fa}$ and $\un{X}_{\ga, \fa} = \sqrt2\, \un{Y}_{\ga, \fa} + \un{v}_{\ga, \fa}$ with $Y_{\ga, \fa} = \CR^{\ga, \fa}\<1b>$ and $\un{Y}_{\ga, \fa} = \un{\CR}^{\ga, \fa}\<1b>$, we get 
\begin{align}
\bigl\| \bigl(\un{X}_{\ga, \fa} X_{\ga, \fa} &- 2\, \un{Y}_{\ga, \fa} Y_{\ga, \fa}\bigr)(t) \bigr\|^{(\un{\emezo})}_{\CC^{3 / 2 + 3 \eta}} \label{eq:X-un-X} \\
&\qquad \lesssim \| \un{Y}_{\ga, \fa}(t)\|_{L^\infty} \| v_{\ga, \fa}(t) \|_{L^\infty} + \| \un{v}_{\ga, \fa}(t) \|_{L^\infty} \bigl(\| Y_{\ga, \fa}(t) \|_{L^\infty} + \| v_{\ga, \fa}(t) \|_{L^\infty}\bigr). \nonumber
\end{align}
Propositions~\ref{prop:models-converge} and \ref{prop:ReconThm_v2} yield $\E \bigl[\sup_{t \in [0, T]} \| Y_{\ga, \fa}(t) \|_{L^\infty}^p \bigr] \lesssim \emezo^{|\<1b>| p}$ for all $p \geq 1$ large enough, and respectively $\E \bigl[\sup_{t \in [0, T]} \| \un{Y}_{\ga, \fa}(t) \|_{L^\infty}^p \bigr] \lesssim \un{\emezo}^{|\<1b>| p}$. Moreover, from \eqref{eq:discrete-remainder-bound} and \eqref{eq:remainder-bound} we get 
\begin{equation*}
\sup_{t \in [0, T \wedge T^L_{\ga, \fa}]} \| v_{\ga, \fa}(t) \|_{L^\infty} \lesssim \emezo^{\frac{3}{2} + 3 \eta}, \qquad \sup_{t \in [0, T \wedge T^L_{\ga, \fa}]} \| \un{v}_{\ga, \fa}(t) \|_{L^\infty} \lesssim \un{\emezo}^{\frac{3}{2} + 3 \eta}.
\end{equation*}
Using these bounds and Minkowski inequality, we get from \eqref{eq:X-un-X}
\begin{equation}\label{eq:unX-X-minus-Y}
\sup_{t \in [0, T \wedge T_{\ga, \fa}]} \bigl\| \bigl(\un{X}_{\ga, \fa} X_{\ga, \fa} - 2\, \un{Y}_{\ga, \fa} Y_{\ga, \fa}\bigr)(t) \bigr\|_{L^\infty} \lesssim \emezo^{|\<1b>| + \frac{3}{2} + 3 \eta} \ga^{|\<1b>| \un \kappa},
\end{equation}
where we used the definition $\un{\emezo} = \emezo \ga^{\un\kappa}$ and the bounds on $\eta$ in the statement of Theorem~\ref{thm:main}. If we take $\un \kappa \leq \kappa < \frac{1}{10}$, where $\kappa$ is the value used in the definition of the regularity structure \eqref{eqs:hom}, the preceding expression is bounded by $C \emezo^{\un\kappa - 1}$. Furthermore, for any $\un{\eta} < -1$  we get the estimate 
\begin{equation*}
\E \biggl[\sup_{t \in [0, T]} \Bigl(\| \un{Y}_{\ga, \fa}(t) Y_{\ga, \fa}(t) - \tfrac{1}{2} \un{\fC}_\ga(t) \|^{(\un{\emezo})}_{\C^{\un{\eta}}}\Bigr)^p\biggr] \lesssim 1,
\end{equation*}
for any $p \geq 1$ large enough and any $T > 0$. This estimate is obtained in the same way as Lemma~\ref{lem:X2-under-prime}, because the difference in the processes involved in these estimates is only in the initial states. Moreover, we have $|\un{\fC}_\ga - 2 \un{\fC}_\ga(t)| \lesssim 1$ where the constant $\un{\fC}_\ga$ is defined in \eqref{eq:c-under}. Combining this bound with \eqref{eq:unX-X-minus-Y}, we get the following result. 

\begin{lemma}\label{lem:unX-X-bound}
Let $\un \kappa \leq \kappa$, where $\kappa$ is the value used in \eqref{eqs:hom} and $\un{\kappa}$ is from \eqref{eq:under-K-def}. For any $\un{\eta} < -1$, $T > 0$ and any $p \geq 1$ large enough, in the setting of \eqref{eq:SolutionMap-bound} one has
\begin{equation*}
\E \biggl[\sup_{t \in [0, T \wedge T^L_{\ga, \fa}]} \Bigl(\| \un{X}_{\ga, \fa}(t) X_{\ga, \fa}(t) - \un{\fC}_\ga \|^{(\un{\emezo})}_{\C^{\un{\eta}}}\Bigr)^p\biggr] \lesssim \emezo^{(\un\kappa - 1) p},
\end{equation*}
where $\eta$ is from the statement of Theorem~\ref{thm:main} and $\un{\emezo} = \emezo \ga^{\un\kappa}$.
\end{lemma}

\section{Proof of Theorem~\ref{thm:main}} 
\label{sec:ConvFinal}

Let $X_\ga$ be the rescaled spin field of the Ising-Kac model \eqref{eq:X-gamma}, and let $X$ be the solution of the $\Phi^4_3$ equation \eqref{eq:Phi43}. Our goal is to prove that 
\begin{equation}\label{eq:weak-limit}
\lim_{\ga \to 0} \E \bigl[ F (\iota_\eps X_\ga) \bigr] = \E \bigl[ F ( X )\bigr], 
\end{equation}
for any bounded, uniformly continuous function $F : \CD \bigl([0, T], \SD'(\T^3)\bigr) \to \R$. We note that the processes $X_\ga$ and $X$ are not required to be coupled, and the expectations in \eqref{eq:weak-limit} may be on different probability spaces. We fix the value $T > 0$ throughout this section. 

It will be convenient to introduce some intermediate processes. More precisely, for $\delta > 0$ we define $X_{\de}$ to be the solution of the SPDE \eqref{eq:Phi43-delta} and we define $X_{\ga, \de, \fa}$ to be the solution of equation \eqref{eq:Discrete-equation-regular-noise}. Then the limit \eqref{eq:weak-limit} follows if for some $\ga_0 > 0$ we have
\begin{subequations}
\begin{align}
&\lim_{\ga \to 0} \E \bigl[ F (\iota_\eps X_{\ga, \fa}) \bigr] = \E \bigl[ F ( X )\bigr], \label{eq:bound-1}\\
\lim_{\fa \to \infty} &\lim_{\ga \in (0, \ga_0)} \E \bigl| F (\iota_\eps X_\ga) - F (\iota_\eps X_{\ga, \fa}) \bigr| = 0, \label{eq:bound-2}
\end{align}
\end{subequations}
where \eqref{eq:bound-1} holds for each fixed $\fa \geq 1$. Note that the two processes in \eqref{eq:bound-2} are defined on the same probability space. Furthermore, \eqref{eq:bound-1} follows if for some $\de_0 > 0$ we have
\begin{subequations}
\begin{align}
&\lim_{\de \to 0} \E \bigl| F (X_\de) - F ( X ) \bigr| = 0, \label{eq:bound-3} \\
&\lim_{\ga \to 0} \E \bigl[ F (\iota_\eps X_{\ga, \de, \fa}) \bigr] = \E \bigl[ F ( X_\de )\bigr], \label{eq:bound-4} \\
\lim_{\de \to 0} &\sup_{\ga \in (0, \ga_0)} \E \bigl| F (\iota_\eps X_{\ga, \fa}) - F (\iota_\eps X_{\ga, \de, \fa}) \bigr| = 0, \label{eq:bound-5}
\end{align}
\end{subequations}
where \eqref{eq:bound-4} holds for every fixed $\de \in (0, \de_0)$. Again we used that the pairs of processes in \eqref{eq:bound-3} and \eqref{eq:bound-5} are defined on the same probability spaces. The limit \eqref{eq:bound-3} follows from a much stronger convergence stated in Theorem~\ref{thm:Phi-solution}. The limit \eqref{eq:bound-4} is proved in Lemma~\ref{lem:convergence2}.
\medskip

In order to compare the discrete and continuous heat kernels, we introduce the metric
\begin{equation}\label{eq:heat-metric}
\| G^\ga_t; G_t \|^{(\emezo)}_{L^1} := \sum_{x \in \Lattice} \int_{|\bar x - x| \leq \eps} \bigl|  G^\ga_t(x) - G_t(\bar x)\bigr| \d \bar x.
\end{equation}
Here, we use the heat kernel $G_t(x) = (2 \pi t)^{-3/2} e^{- |x|^3 / t}$ and the discrete kernel $G^\ga_t : \Lattice \to \R$ defined in \eqref{eq:From-P-to-G}. We will also use the discrete kernel $\widetilde{G}^\ga$ defined in \eqref{eq:From-P-to-G-tilde}. 

\begin{lemma}\label{lem:heat-bounds}
For any $0 < t \leq 1$ one has
\begin{equation}\label{eq:heat-limits}
\lim_{\ga \to 0} \| G^\ga_t; G_t \|^{(\emezo)}_{L^1} = 0, \qquad\qquad \lim_{\ga \to 0} \| \widetilde{G}^\ga_t; G_t \|^{(\emezo)}_{L^1} = 0.
\end{equation}
\end{lemma}

\begin{proof}
From the explicit formula for the heat kernel we can get (see \cite[Lem.~7.4]{Regularity})
\begin{equation*}
|G_t(x) - G_t(\bar x)| \leq C \bigl(t^{1/2} + (|x| \wedge |\bar x|)\bigr)^{-3 - \theta} |x - \bar x|^\theta,
\end{equation*}
for any $\theta \in [0, 1]$. Similarly, from the bounds on the discrete kernels provided at the end of Appendix~\ref{sec:decompositions} we get
\begin{equation}\label{eq:heat-bounds}
 |G^\ga_t(x) - G_t(x)| \leq C \eps^\theta \bigl(t^{1/2} + |x| + \eps\bigr)^{-3 - \theta},
\end{equation}
Then the integral in \eqref{eq:heat-metric} is estimated by a constant multiple of $\eps^\theta \bigl(t^{1/2} + |x| + \eps\bigr)^{-3 - \theta}$, and the whole expression \eqref{eq:heat-metric} can be estimated by a constant times $\eps^\theta$. This gives the first limit in \eqref{eq:heat-limits}, and the second follows in the same way, where the bounds for $\widetilde{G}^\ga_t$ are of the form \eqref{eq:heat-bounds} with $\eps$ being replaced by $\emezo$.
\end{proof}

\begin{lemma}\label{lem:convergence2}
For any $\fa \geq 1$, $\de \in (0, 1)$ and $T > 0$, the process $X_{\ga, \de, \fa}(t)$ is almost surely uniformly bounded on $[0, T]$. Moreover, the limit \eqref{eq:bound-4} holds.
\end{lemma}

\begin{proof}
We note that the formula \eqref{eq:xi-gamma-delta} makes sense on $\R \times \T^3$ (and not just $\R \times \Le$). Let then $\bar{\xi}_{\ga, \de, \fa}$ be defined by \eqref{eq:xi-gamma-delta} on $\R \times \T^3$. In will be convenient to introduce an additional process $\bar{X}_{\ga, \de, \fa}$ on $\R \times \T^3$, which is the solution of the SPDE
	\begin{equation}\label{eq:auxiliary}
		\bigl( \partial_t - \Delta \bigr) \bar{X}_{\ga, \de, \fa} = - \frac{\be^3}{3} \bar{X}_{\ga, \de, \fa}^3 + \bigl( \fC_{\de} + A \bigr) \bar{X}_{\ga, \de, \fa} + \sqrt 2 \,\bar \xi_{\ga, \de, \fa},
	\end{equation}
	with the initial condition $X_\de^{0}$, the same as for equation \eqref{eq:Phi43-delta}. Then the limit \eqref{eq:bound-4} follows from 
	\begin{subequations}
	\begin{align}
	&\lim_{\ga \to 0} \E \bigl[ F (\bar X_{\ga, \de, \fa}) \bigr] = \E \bigl[ F ( X_\de )\bigr], \label{eq:bound-4-1} \\
	&\lim_{\ga \to 0} \E \bigl| F (\iota_\eps X_{\ga, \de, \fa}) - F (\bar X_{\ga, \de, \fa}) \bigr| = 0, \label{eq:bound-4-2}
	\end{align}
	\end{subequations}
	and we are going to prove these two limits.
	
For $T > 0$ we will use the shorthand notation $L^\infty_T := L^{\infty}([0, T] \times \T^3)$, and we will consider all the spaces and norms on $\T^3$ in the spatial variable, which we prefer not to write every time.
	
	We start with analysing the second term in \eqref{eq:bound-4-1}. For this we will show the continuous dependence of the solution of equation \eqref{eq:auxiliary} on the driving noise and the initial state. More precisely, for $f_0 \in L^\infty$, for $T > 0$ and for a function $\zeta \in L^{\infty}_T$ we consider the PDE
	\begin{equation}\label{eq:auxiliary-general}
		( \partial_t - \Delta ) f = - \frac{\be^3}{3} f^3 + ( \fC_{\de} + A ) f + \sqrt 2\, \zeta
	\end{equation}
	on $[0, T] \times \T^3$ with an initial condition $f_0 \in L^{\infty}(\T^3)$ at time $0$. Of course, the solution $f$ depends on $\delta$ and $\gamma$ through the constants $\fC_{\de}$ and $\beta$ (see \eqref{eq:beta}), but we prefer not to indicate this dependence to have a lighter notation. By our assumptions, there exists $L > 0$ such that $\| f_0 \|_{L^{\infty}} \leq L$ and $\| \zeta \|_{L^\infty_T} \leq L$. We are going to prove that there is a unique solution $f \in L^\infty_T$, and the solution map $f = \CS_T(\zeta, f_0)$ is locally continuous from $L^{\infty}_T \times L^{\infty}$ to $L^{\infty}_T$. 
	
	Let $P : \R_+ \times \T^3$ be the heat kernel, i.e., the Green's function of the parabolic operator $\partial_t - \Delta$. Then, with a little ambiguity, we write $P_t$ for the semigroup, whose action on functions is given by the convolution with the heat kernel $P_t$ on $\T^3$. Then the mild form of \eqref{eq:auxiliary-general} is 
	\begin{equation}\label{eq:f-mild}
	 f_t(x) = P_t f_0 (x) + \int_0^t P_{t-s} \Bigl(- \frac{\be^3}{3} f_s^3 + ( \fC_{\de} + A ) f_s + \sqrt 2\, \zeta_s\Bigr)(x) \d s.
	\end{equation}
	We denote by $\CM_t(f)(x)$ the right-hand side, and we are going to prove that $\CM_t(f)$ is a contraction map on $\CB_{L, t} := \{f : \| f \|_{L^\infty_t} \leq L+1\}$ for a sufficiently small $0 < t < T$.
	
	Taking $f \in \CB_{L, t}$, using the Young inequality and using the identity $\| P_t \|_{L^1} = 1$, we get 
	\begin{align*}
		\| \CM_t(f) \|_{L^\infty} &\leq \| f_0 \|_{L^\infty} + t \| f\|^3_{L^\infty_t} + t | \fC_{\de} + A | \| f \|_{L^\infty_t} + t \sqrt 2\, \| \zeta \|_{L^\infty_t} \\
		&\leq L + t \bigl( (L +1)^3 + | \fC_{\de} + A | (L+1) + \sqrt 2\, L \bigr),
	\end{align*}
	where we estimated $\be^3 \leq 3$, which follows from \eqref{eq:beta} for all $\gamma > 0$ sufficiently small. Taking $t > 0$ small enough, we get 
	\begin{equation*}
	\| \CM_t(f) \|_{L^\infty} \leq L + 1,
	\end{equation*}
	which means that $\CM_t$ maps $\CB_{L, t}$ to itself. 
	
	Let us now take $f, \bar f \in \CB_{L, t}$ with $f_0 = \bar f_0$. Then 
	\begin{equation*}
		\big( \CM_t(f) - \CM_t(\bar f) \big)(x) = - \frac{\be^3}{3} \int_{0}^t P_{t-s} \big( f_s^3 - \bar f_s^3 \big)(x) \d s + ( \fC_{\de} + A ) \int_{0}^t P_{t-s} \big( f_s - \bar f_s \big)(x) \d s,
	\end{equation*}
	which yields similarly to how we did above 
	\begin{align*}
		\| \CM_t(f) - \CM_t(\bar f)\|_{L^\infty} &\leq t \| f^3 - \bar f^3\|_{L^\infty_t} + t | \fC_{\de} + A | \| f - \bar f \|_{L^\infty_t} \\
		&\leq t \bigl( 3 L^2 + | \fC_{\de} + A | \bigr) \| f - \bar f \|_{L^\infty_t}.
	\end{align*}
	Taking $t > 0$ small enough, we get $t \bigl( 3 L^2 + | \fC_{\de} + A | \bigr) < 1$, which means that $\CM_t$ is a contraction on $\CB_{L, t}$. By the Banach fixed point theorem, there exists a unique solution $f \in L^\infty_t$ of equation \eqref{eq:f-mild}. 
	
	Let us now denote by $f = \CS_t(\zeta, f_0)$ the solution map of \eqref{eq:f-mild} on $\CB_{L, t}$. We are going to show that it is continuous with respect to $\zeta$ and $f_0$, satisfying $\| f_0 \|_{L^{\infty}} \leq L$ and $\| \zeta \|_{L^\infty_T} \leq L$. For this we take $\| \bar f_0 \|_{L^{\infty}} \leq L$ and $\| \bar \zeta \|_{L^\infty_T} \leq L$, and for $\bar f = \CS_t(\bar \zeta, \bar f_0)$ we have 
	\begin{align*}
		\big( f_t - \bar f_t \big)(x) &= P_t (f_0 - \bar f_0) (x) - \frac{\be^3}{3} \int_{0}^t P_{t-s} \big( f_s^3 - \bar f_s^3 \big)(x) \d s \\ 
		&\qquad + ( \fC_{\de} + A ) \int_{0}^t P_{t-s} \big( f_s - \bar f_s \big)(x) \d s + \sqrt 2 \int_{0}^t P_{t-s} \big( \zeta_s - \bar \zeta_s \big)(x) \d s.
	\end{align*}
	Computing the norms as above, we get 
	\begin{equation*}
		\| f - \bar f \|_{L^\infty_t} \leq \| f_0 - \bar f_0 \|_{L^\infty} + t \bigl( 3 L +  | \fC_{\de} + A | \bigr) \| f - \bar f\|_{L^\infty_t} + t \sqrt 2\, \| \zeta - \bar \zeta \|_{L^\infty_t}.
	\end{equation*}
	Since $t$ is such that $t \bigl( 3 L +  | \fC_{\de} + A | \bigr) < 1$, we can move the term proportional to $\| f - \bar f\|_{L^\infty_t}$ to the left-hand side and get 
	\begin{equation*}
		\| f - \bar f \|_{L^\infty_t} \leq C \| f_0 - \bar f_0 \|_{L^\infty} + C \| \zeta - \bar \zeta \|_{L^\infty_t},
	\end{equation*}
	where the proportionality constant $C$ depends on $\delta$ and $L$. Thus, we have a locally Lipschitz continuity of the solution map. 
	
	The extension of the solution to longer time intervals $[0, T]$ is the standard procedure, and is done by patching local solutions. Since the function $V : \R \to \R$ given by $V(u) = u^2$ is a Lyapunov function for equation \eqref{eq:auxiliary-general}, the solution is global in time and $T$ can be taken arbitrary (this standard result can be found for example in \cite[Prop.~6.23]{Hai09}). 
	
	Let us now look back at \eqref{eq:bound-4-1}. Using the constructed solution map we can write $X_{\de} = \CS(\xi_{\de}, X_{\de}^{0})$ and $\bar{X}_{\ga, \de, \fa} = \CS(\bar \xi_{\ga, \de, \fa}, X_{\de}^{0})$. By Lemma $2.3$ in \cite{Martingales} we have the convergence in law in the topology of the Skorokhod space $\CD(\R_+, \SD'(\T^3))$ of the family of martingales $(\M_{\ga, \fa}(\bigcdot, x))_{x \in \Le}$ to a cylindrical Wiener process on $L^2(\T^3)$. For any $T > 0$, we therefore get convergence in law of $\bar \xi_{\ga, \de, \fa}$ to $\xi_{\de}$, as $\ga \to 0$, in the topology of $L^\infty([0, T] \times \T^3)$. Then from continuity of the solution map $\CS$ we conclude that $\bar{X}_{\ga, \de, \fa}$ converges in law to $X_{\de}$, as $\ga \to 0$, in the topology of $L^\infty([0, T] \times \T^3)$. This yields the required limit \eqref{eq:bound-4-1}.
	\medskip
	
	Now, we will prove the limit \eqref{eq:bound-4-2}. We observe that these two processes are driven by the same noise and the live on the same probability space. It will be convenient to define an analogue of the $L^\infty$ norm to compare a discrete and continuous functions. Namely, for $f_\ga : \Lattice \to \R$ and $f : \R^3 \to \R$ we set 
	\begin{equation*}
	\| f_\ga; f \|^{(\emezo)}_{L^\infty} := \sup_{\substack{x \in \Lattice, \bar x \in \R^3 \\ |\bar x - x|_\sinfty \leq \eps / 2}} \bigl| f_{\ga}(x) - f(\bar x)\bigr|.
	\end{equation*}
	If moreover functions depend on the time variable, then set $\| f_\ga; f \|^{(\emezo)}_{L^\infty_T} := \sup_{t \in [0, T]} \| f_\ga(t); f(t) \|^{(\emezo)}_{L^\infty}$. Then the limit \eqref{eq:bound-4-2} holds if we show
	\begin{equation}\label{eq:second-term-converges}
	\lim_{\ga \to 0} \E \| X_{\ga, \de, \fa}; \bar{X}_{\ga, \de, \fa} \|^{(\emezo)}_{L^\infty_T} = 0.
	\end{equation}
	
	Now, we will prove the limit \eqref{eq:second-term-converges}. The mild form of \eqref{eq:auxiliary} is 
	\begin{equation}\label{eq:mild1}
		 \bar{X}_{\ga, \de, \fa}(t, x) = P_t X_{\de}^{0} (x) + \int_0^t P_{t-s} \Bigl(- \frac{\be^3}{3} \bar{X}_{\ga, \de, \fa}^3 + ( \fC_{\de} + A ) \bar{X}_{\ga, \de, \fa} + \sqrt 2\, \bar \xi_{\ga, \de, \fa}\Bigr)(s, x) \d s.
	\end{equation}
	As a consequence of our analysis of equation \eqref{eq:auxiliary-general}, if we take $\| X_{\de}^{0} \|_{L^{\infty}} \leq L$ and $\|\bar \xi_{\ga, \de, \fa} \|_{L^\infty_T} \leq L$, then for $0 < t \leq T$ small enough we have $\|\bar{X}_{\ga, \de, \fa} \|_{L^\infty_t} \leq L + 1$. We will use this value $t$ in what follows. We can perform the same analysis as above and conclude that if $\| X_{\de}^{0} \|_{L^{\infty}} \leq L$ then $\|X_{\ga, \de, \fa} \|_{L^\infty_t} \leq L + 1$. We prefer not to repeat the same argument twice. 
	
	We extend the processes periodically in the spatial variables. This means that we need to replace $P^\ga$ and $\widetilde{P}^\ga$ by $G^\ga$ and $\widetilde{G}^\ga$ respectively; and we need to replace $P$ by $G$ in \eqref{eq:mild1}. In what follows we are going to work with these periodic extensions. 
	
	Using the metric \eqref{eq:heat-metric}, one can readily get the bound
	\begin{equation}\label{eq:discret-continuous-bound}
	\| G^\ga_t X_{\ga, \de, \fa}^{0}; G_t X_{\de}^{0} \|^{(\emezo)}_{L^\infty} \leq \| G_t \|_{L^1} \| X_{\ga, \de, \fa}^{0}; X_{\de}^{0} \|^{(\emezo)}_{L^\infty} + \| G^\ga_t; G_t \|^{(\emezo)}_{L^1} \| X_{\de}^{0}\|_{L^\infty}.
	\end{equation}
	We have $\| G_t \|_{L^1} = 1$, and from Lemma~\ref{lem:heat-bounds} we have that $Q^\ga_t := \| G^\ga_t; G_t \|^{(\emezo)}_{L^1}$ and $\widetilde{Q}^\ga_t := \| \widetilde{G}^\ga_t; G_t \|^{(\emezo)}_{L^1}$ vanish as $\ga \to 0$. Using then the bound $\|\bar{X}_{\ga, \de, \fa} \|_{L^\infty_t} \leq L + 1$, subtracting equations, and using the bound similarly to \eqref{eq:discret-continuous-bound}, we get 
	\begin{align*}
	&\| X_{\ga, \de, \fa}; \bar{X}_{\ga, \de, \fa} \|^{(\emezo)}_{L^\infty_t} \leq \| X_{\ga, \de, \fa}^{0}; X_{\de}^{0} \|^{(\emezo)}_{L^\infty} + Q^\ga_t L + t \sqrt 2\, \| \xi_{\ga, \de, \fa}; \bar \xi_{\ga, \de, \fa}\|^{(\emezo)}_{L^\infty_t}\\
	&\qquad + t \Bigl( \| X_{\ga, \de, \fa}^3; \bar{X}_{\ga, \de, \fa}^3 \|^{(\emezo)}_{L^\infty_t} + | \fC_{\de} + A | \| X_{\ga, \de, \fa}; \bar{X}_{\ga, \de, \fa} \|^{(\emezo)}_{L^\infty_t} + |\fC_{\ga, \de} - \fC_{\de}| L\Bigr) \\
	&\qquad + t \widetilde{Q}^\ga_t \Bigl( (L + 1)^3 + | \fC_{\de} + A | (L + 1) + \sqrt 2\, L\Bigr).
	\end{align*}
	We can readily show that $\| X_{\ga, \de, \fa}^3; \bar{X}_{\ga, \de, \fa}^3 \|^{(\emezo)}_{L^\infty_t} \leq 3 L^2 \| X_{\ga, \de, \fa}; \bar{X}_{\ga, \de, \fa} \|^{(\emezo)}_{L^\infty_t}$, and the choice of $t$ allows to absorb the term proportional to $\| X_{\ga, \de, \fa}; \bar{X}_{\ga, \de, \fa} \|^{(\emezo)}_{L^\infty_t}$ to the left-hand side and get the bound 
	\begin{align*}
	\| X_{\ga, \de, \fa}; \bar{X}_{\ga, \de, \fa} \|^{(\emezo)}_{L^\infty_t} &\lesssim \| X_{\ga, \de, \fa}^{0}; X_{\de}^{0} \|^{(\emezo)}_{L^\infty} + Q^\ga_t L + t \| \xi_{\ga, \de, \fa}; \xi_{\ga, \de, \fa} \|^{(\emezo)}_{L^\infty_t} + t |\fC_{\ga, \de} - \fC_{\de}| L\\
	&\qquad + t \widetilde{Q}^\ga_t \Bigl( (L + 1)^3 + | \fC_{\de} + A | (L + 1) + \sqrt 2\, L\Bigr),
	\end{align*}
	where the proportionality constant depends on $t$ and $L$. From our assumptions in Theorem~\ref{thm:main} on the initial states we conclude that $\lim_{\ga \to 0} \| X_{\ga, \de, \fa}^{0}; X_{\de}^{0} \|^{(\emezo)}_{L^\infty} = 0$. Furthermore, we have $\lim_{\ga \to 0} \E \| \xi_{\ga, \de, \fa}; \bar \xi_{\ga, \de, \fa} \|^{(\emezo)}_{L^\infty_t} = 0$. Finally, from the definitions of the renormalisation constants we get $\lim_{\ga \to 0} \fC_{\ga, \de} = \fC_{\de}$, because the constants are defined in terms of the heat kernels and these converge uniformly as $\ga \to 0$ (see Lemma~\ref{lem:Pgt}). Then from the preceding inequality we obtain 
	\begin{equation*}
	\E \| X_{\ga, \de, \fa}; \bar{X}_{\ga, \de, \fa} \|^{(\emezo)}_{L^\infty_t} \leq C_{\ga} (L, t),
	\end{equation*}
	where $\lim_{\ga \to 0} C_{\ga} (L, t) = 0$. Since $\bar \xi_{\ga, \de, \fa}$ is almost surely bounded, the process $\bar{X}_{\ga, \de, \fa}$ almost surely does not blow up in a finite time (see the argument above), and we conclude that the same is true for $X_{\ga, \de, \fa}$ and \eqref{eq:second-term-converges} holds for any $T > 0$.
	\end{proof}

Our next aim is to prove the limit \eqref{eq:bound-5}. It will be convenient to prove the required convergence in probability. For this we need to restrict the time interval to $[0, T^L_{\ga, \fa} \wedge T^L_{\ga, \de, \fa}]$, where the stopping times $T^L_{\ga, \fa}$ and $T^L_{\ga, \de, \fa}$ are defined in Proposition~\ref{prop:SolutionMap}. Moreover, we need to introduce auxiliary stopping times providing a bound on the models. More precisely, for $L > 0$ we define
\begin{align*}
\tau^{L}_{\ga, \fa} &:= \inf \Bigl\{ t \geq 0 : \$ Z^{\ga, \fa}_{\lift}\$_{t+1}^{(\emezo)} \geq L \Bigr\} \wedge T^L_{\ga, \fa}, \quad \tau^{L}_{\ga, \de, \fa} := \inf \Bigl\{ t \geq 0 : \$ Z^{\ga, \de, \fa}_{\lift}\$_{t+1}^{(\emezo)} \geq L \Bigr\} \wedge T^L_{\ga, \de, \fa}. 
\end{align*}
Then for any $A > 0$, $L > 0$ and $T > 0$ we have 
\begin{align}
&\P\biggl( \sup_{t \in [0, T]} \| (X_{\ga, \de, \fa} - X_{\ga, \fa})(t) \|^{(\emezo)}_{\CC^\eta} \geq A \biggr) \nonumber\\
&\hspace{1cm}\leq \P\biggl( \sup_{t \in [0, T \wedge \tau^{L}_{\ga, \fa} \wedge \tau^{L}_{\ga, \de, \fa}]} \| (X_{\ga, \de, \fa} - X_{\ga, \fa})(t) \|^{(\emezo)}_{\CC^\eta} \geq A \biggr) + \P \bigl(\tau^{L}_{\ga, \fa} \wedge \tau^{L}_{\ga, \de, \fa} < T \bigr). \label{eq:weak-limit-3}
\end{align}

From the assumptions of Theorem~\ref{thm:main} we conclude that there exists $L_\star > 0$ such that $\| X^0_{\ga} \|^{(\emezo)}_{\CC^{\bar \eta}} \leq L_\star$ uniformly in $\ga \in (0, \ga_\star)$. Moreover, the definition \eqref{eq:initial-gamma-delta} yields $\sup_{\ga \in (0, \ga_\star)} \| X^0_{\ga} - X^{0}_{\ga, \delta} \|^{(\emezo)}_{\CC^\eta} \lesssim \de^{\theta}$ for any $\eta < \bar \eta$ and any $\theta > 0$ small enough. We fix $0 < \ga_0 \leq \ga_\star$ such that the result of Proposition~\ref{prop:models-converge} holds. Then from Proposition~\ref{prop:SolutionMap} we conclude that 
\begin{equation}\label{eq:X-delta-stopped}
\lim_{L \to \infty} \lim_{\de \to 0} \sup_{\ga \in (0, \ga_0)} \P\biggl( \sup_{t \in [0, T \wedge \tau^{L}_{\ga, \fa} \wedge \tau^{L}_{\ga, \de, \fa}]} \| (X_{\ga, \de, \fa} - X_{\ga, \fa})(t) \|^{(\emezo)}_{\CC^\eta} \geq A \biggr) = 0.
\end{equation}
Furthermore, we have 
\begin{equation}\label{eq:tau-bounds}
\P \bigl(\tau^{L}_{\ga, \fa} \wedge \tau^{L}_{\ga, \de, \fa} < T \bigr) \leq \P \bigl(T^{L}_{\ga, \fa} \wedge T^{L}_{\ga, \de, \fa} < T \bigr) + \P\bigl( \$ Z^{\ga, \fa}_{\lift}\$_{T+1}^{(\emezo)} \geq L \bigr) + \P\bigl( \$ Z^{\ga, \de, \fa}_{\lift}\$_{T+1}^{(\emezo)} \geq L \bigr).
\end{equation}
Markov's inequality yields $\P \bigl( \$ Z^{\ga, \fa}_{\lift}\$_{T+1}^{(\emezo)} \geq L\bigr) \leq L^{-p}\, \E \bigl( \$ Z^{\ga, \fa}_{\lift}\$_{T+1}^{(\emezo)}\bigr)^p$, for any $p \geq 1$. From Proposition~\ref{prop:models-converge} we conclude that for any $p$ the preceding expectation is bounded uniformly in $\ga \in (0,\ga_0)$. In the same way from Proposition~\ref{prop:models-converge} we conclude that $\P \bigl( \$ Z^{\ga, \de, \fa}_{\lift}\$_{T+1}^{(\emezo)} \geq L\bigr)$ is bounded uniformly in $\ga \in (0,\ga_0)$ and $\de \in (0,1)$, and hence from \eqref{eq:tau-bounds} we get 
\begin{equation}\label{eq:X-delta-stopped-better}
\lim_{L \to \infty} \sup_{\de \in (0,1)} \sup_{\ga \in (0, \ga_0)} \P \bigl(\tau^{L}_{\ga, \fa} \wedge \tau^{L}_{\ga, \de, \fa} < T \bigr) \leq \lim_{L \to \infty} \sup_{\de \in (0,1)} \sup_{\ga \in (0, \ga_0)} \P \bigl(T^{L}_{\ga, \fa} \wedge T^{L}_{\ga, \de, \fa} < T \bigr).
\end{equation}

Lemma~\ref{lem:convergence2} implies that the living time $T_{\ga, \de, \fa}$ of the process $X_{\ga, \de, \fa}$ is almost surely infinite, and hence Proposition~\ref{prop:SolutionMap} yields $\lim_{L \to \infty} T^L_{\ga, \de, \fa} = +\infty$ almost surely. Then the right-hand side of \eqref{eq:X-delta-stopped-better} equals 
\begin{equation}\label{eq:X-delta-stopped-better-2}
\lim_{L \to \infty} \sup_{\de \in (0,1)} \sup_{\ga \in (0, \ga_0)} \P \bigl(T^{L}_{\ga, \fa} < T \bigr).
\end{equation}

Furthermore, as we stated after \eqref{eq:IsingKacEqn-periodic}, we have $X_{\ga, \fa}(t) = X_{\ga}(t)$ for $t \leq \tau_{\ga, \fa}$ and $X_{\ga, \fa}(t) = X'_{\ga, \fa}(t)$ for $t > \tau_{\ga, \fa}$. Then $X_{\ga, \fa}$ is almost surely bounded on each bounded time interval, because for $t \leq \tau_{\ga, \fa}$ the process is bounded due to the definition of the stopping times \eqref{eq:tau-1}-\eqref{eq:tau}, and for $t > \tau_{\ga, \fa}$ the process is bounded due to Lemma~\ref{lem:X-prime-bound}. Hence, we conclude that the living time of the process $X_{\ga, \fa}$ is almost surely infinite, and $\lim_{L \to \infty} T^L_{\ga, \fa} = +\infty$ almost surely. This implies that \eqref{eq:X-delta-stopped-better-2} vanishes. 

From the preceding argument we conclude that the expression in \eqref{eq:weak-limit-3} vanishes, which yields convergence of the process $X_{\ga, \de, \fa}$ to $X_{\ga, \fa}$ as $\de \to 0$ in probability in the topology as in \eqref{eq:weak-limit-3}. Then the processes converge in distribution, i.e. \eqref{eq:bound-5} holds.
\medskip

We have proved the limit \eqref{eq:bound-1} and it is left to prove \eqref{eq:bound-2}. We are going to prove this limit in probability. Recalling the definition of $X_{\ga, \fa}$, for any $A > 0$ we get
\begin{equation}
\P\biggl( \sup_{t \in [0, T]} \| (X_{\ga, \fa} - X_\ga)(t) \|^{(\emezo)}_{\CC^\eta} \geq A \biggr) \leq \P\bigl( \tau_{\ga, \fa} < T \bigr), \label{eq:weak-limit-4}
\end{equation}
where the supremum vanishes if $\tau_{\ga, \fa} \geq T$. From the definition \eqref{eq:tau} we have 
\begin{equation}
\P \bigl(\tau_{\ga, \fa} < T \bigr) \leq \P \bigl(\tau^{(1)}_{\ga, \fa} < T \bigr) + \P \bigl(\tau^{(2)}_{\ga, \fa} < T \bigr). \label{eq:weak-limit-5}
\end{equation}
The stopping time \eqref{eq:tau-1} we write as $\tau^{(1)}_{\ga, \fa} = \inf \bigl\{t \geq 0 :  \| X_{\ga, \fa}(t) \|^{(\emezo)}_{\CC^{\eta}} \geq \fa \bigr\}$, and hence it coincides with the stopping time $T^{L_\fa}_{\ga, \fa}$ defined in Proposition~\ref{prop:SolutionMap} with a suitable values $L_\fa$, depending on $\fa$ and such that $\lim_{\fa \to \infty} L_\fa = \infty$. Then we have $\lim_{\fa \to \infty}\sup_{\ga \in (0, 1)} \P \bigl(\tau^{(1)}_{\ga, \fa} < T \bigr) = 0$. Convergence of the last term in \eqref{eq:weak-limit-5} to zero uniformly in $\ga \in (0, 1)$ as $\fa \to \infty$ follows from Lemma~\ref{lem:unX-X-bound}.

\subsection{The renormalisation constant}
\label{sec:renormalisation}

We readily conclude from Lemma~\ref{lem:renorm-constants} that the renormalisation constant \eqref{eq:C-exact} may be written in the form \eqref{eq:C-expansion}. 

\appendices

\section{Properties of the discrete kernels}
\label{sec:kernels}

The main result of this appendix is provided in Lemma~\ref{lem:Pgt}, which provides bounds on continuous extensions of the functions $G^\ga$ and $\widetilde{G}^\ga$ defined in \eqref{eq:From-P-to-G} and \eqref{eq:From-P-to-G-tilde}. 

Before proving these main results, we need to prove several bounds on the function $K_\ga$. By the definitions \eqref{eq:K-gamma} and \eqref{eq:Laplacian-gamma} we conclude that there exists $\ga_0 > 0$ (depending on the radius of support of the function $\fK$) such that for $\ga \in (0, \ga_0)$ and $\om \in \{-N, \ldots, N\}^3$
\begin{equation}\label{eq:B1}
	\widehat{K}_\ga (\om) = \eps ^3 \sum_{x\in \Le} \Kg(x) e^{- i \pi \om \cdot x} = \varkappa_{\ga, 1} \ga^{3} \sum_{x \in \ga \Z^3} \fK(x) e^{- i \pi \ga^3 \om \cdot x}, 
\end{equation} 
where we used the fact that $\fK$ is compactly supported to extend the sum to all $x \in \ga \Z^3$. In what follows we will always consider $\ga \in (0, \ga_0)$. Furthermore, it will be convenient to view $\widehat{K}_\ga$ as a function of a continuous argument by evaluating \eqref{eq:B1} for all $\om \in \R^3$. In this way, the function $\widehat{K}_\ga(\om)$ is smooth and we will use the notation $\om = (\om_1, \om_2, \om_3)$ and $\partial_j$ for the partial derivative with respect to $\om_j$. For a multiindex $k \in \N_0^3$ we will write $D^k = \prod_{j= 1}^3 \partial_j^{k_j}$ for a mixed derivative.

\begin{lemma}\label{lem:Kg0}
For any $c > 0$ there exists a constant $C >0$ such that
\begin{subequations}
	\begin{align}
		\bigl| \ga^{-6} \bigl(1 -\widehat{K}_\ga(\om) \bigr) -  \pi^2 |\om|^2 \bigr| &\leq C_1 \ga^3 | \om|^3, \label{eq:K2.4} \\
		\bigl| \ga^{-6}  \partial_j \widehat{K}_\ga(\om) + 2  \pi^2 \om_j \bigr| &\leq C_1 \ga^3 | \om|^2, \label{eq:K2.3}
	\end{align}
\end{subequations}
uniformly over $\ga \in (0, \ga_0)$, $|\om| \leq c\ga^{-3}$ and $j \in \{1,2,3\}$.
\end{lemma}

\begin{proof}
For $|\om| \leq c \ga^{-3}$ a Taylor expansion and \eqref{eq:B1} yield
\begin{equation}\label{eq:one-minus-K}
\begin{aligned}
1 -  \widehat{K}_\ga (\om) 
&= \varkappa_{\ga, 1} \ga^{3} \sum_{x\in \ga \Z^3} \fK( x) \bigl(1 - e^{- i \pi \ga^3 \om \cdot x} \bigr) \\
&= \varkappa_{\ga, 1} \ga^{3} \sum_{x\in \ga \Z^3} \fK( x) \Big( i \pi \ga^3 \om \cdot x + \tfrac{1}{2} \big( \pi \ga^3 \om \cdot x\big)^2\Big) + \Err_\ga(\om),
\end{aligned}
\end{equation}
where the error term satisfies $| \Err_\ga(\om)| \leq \varkappa_{\ga, 1} \frac{\pi^3}{6} \ga^{12} |\om|^3 \sum_{x\in \ga \Z^3}  |x|^3  \fK(x) \lesssim \ga^9 |\om|^3$. In the first identity in \eqref{eq:one-minus-K} we used the definition of the constant $\varkappa_{\ga, 1}$ in \eqref{eq:K-gamma}. By the symmetry of the kernel $\fK(x)$, we have $\sum_{x\in \ga \Z^3} \fK( x) ( \om \cdot x ) =0$ and $\sum_{x\in \ga \Z^3} x_i x_j \fK(x)  =0$ for $i \neq j$. Furthermore, the sums $\ga^{3} \sum_{x\in \ga \Z^3} \fK( x) x_j^2$ converge to $\int_{\R^3} \fK(x) x_j^2 \d x = 2$ as $\ga \to 0$ with an error $\CO(\ga^3)$. The last identity follows from \eqref{eq:K-moments} and symmetry of the function $\fK$. Then from \eqref{eq:one-minus-K} we obtain \eqref{eq:K2.4}.

The remaining bound \eqref{eq:K2.3} follows in a similar manner. More precisely, using a Taylor expansion we write 
\begin{align*}
-\partial_j \widehat{K}_\ga(\om) &= \varkappa_{\ga, 1} i \pi \ga^6 \sum_{x\in \ga \Z^3} x_j  \fK(x) \big( e^{- i \pi \ga^3 \om \cdot x} -1 \big) \\
&= \varkappa_{\ga, 1} \pi^2 \ga^9 \om_j \sum_{x \in \ga \Z^3} x_j^2 \fK(x) + \Err'_\ga(\om),
\end{align*}
for an error term satisfying $|\Err'_\ga(\om)| \lesssim \ga^9 |\om|^2$. Here, we have used the symmetry of the kernel $\fK$ to add the term $-1$ in the first equality and to remove the sums containing the products $x_i x_j$ for $i \neq j$ in the Taylor expansion in the second line. The bound \eqref{eq:K2.3} then follows similarly to \eqref{eq:K2.4}.
\end{proof}

\begin{lemma}\label{lem:Kg}
For any $k \in \N_0^3$ and $m \in \N_0$ there are constants $C_1, C_2, C_3 > 0$ (where only $C_2$ and $C_3$ depend on $k$, and only $C_3$ depends on $m$) such that the following estimates hold uniformly over $\ga \in (0, \ga_0)$, $\om \in \bigl[-N-\frac12, N +\frac12\bigr]^3$ and $j \in \{1,2,3\}$:
\begin{enumerate}
\item (Most useful for $|\om| \lesssim \ga^{-3}$)
\begin{subequations}\label{eqs:K-bounds}
\begin{align}
	|\widehat{K}_\ga(\om) | &\leq 1, \label{eq:K1} \\
	|\partial_j \widehat{K}_\ga(\om) | &\leq C_1 \ga^3 \big( |\ga^3 \om| \wedge 1  \big), \label{eq:K1A} \\
	|D^{k} \widehat{K}_\ga(\om) | &\leq C_2 \ga^{3 | k |_{\sone}}, \label{eq:K1B}
\end{align}
\end{subequations}
\item (Most useful for $|\om| \gtrsim \ga^{-3}$)
\begin{equation}
    | \ga^3 \om|^{2m} \big| D^{k} \widehat{K}_\ga (\om) \big|\leq C_3 \ga^{3 |k|_{\sone}}. \label{eq:K3} 
\end{equation}
\end{enumerate}
Furthermore, the value of $\ga_0 > 0$ can be chosen small enough so that
\begin{equation}\label{eq:K2}
	 1 -\widehat{K}_\ga(\om) \geq C_4 \big( |\ga^3 \om|^2  \wedge 1 \big),
\end{equation} 
uniformly over the same values of $\ga$ and $\om$, for some $C_4 > 0$.
\end{lemma}

\begin{proof}
We can get \eqref{eq:K1} from \eqref{eq:B1} as $|\widehat{K}_\ga(\om) | \leq \varkappa_{\ga, 1} \ga^{3} \sum_{x \in \ga \Z^3} \fK(x) = 1$, where we used the definition of the constant $\varkappa_{\ga, 1}$ in \eqref{eq:K-gamma}. Similarly, from \eqref{eq:B1} we get
\begin{equation}\label{eq:K-hat-deriv}
D^k \widehat{K}_\ga(\om) = \varkappa_{\ga, 1} \ga^3 \sum_{x\in \ga \Z^3} (-i \pi \ga^3 x)^k \fK(x) e^{- i \pi \ga^3 \om \cdot x },
\end{equation}
with the notation $x^k = \prod_{j= 1}^3 x_j^{k_j}$. Then we can prove \eqref{eq:K1B} as follows
\begin{equation*}
|D^{k} \widehat{K}_\ga(\om) | \lesssim \varkappa_{\ga, 1} \ga^{3 (|k|_\sone + 1)} \sum_{x \in \ga \Z^3} \fK(x) |x|^{|k|_\sone} \lesssim \ga^{3 |k|_\sone},
\end{equation*}
where we estimated the sum by an integral, which is bounded because $\fK$ is bounded and compactly supported. For $|\om| \geq \ga^{-3}$, the estimate \eqref{eq:K1A} is a particular case of \eqref{eq:K1B}, and for $|\om| \leq \ga^{-3}$ it follows from \eqref{eq:K2.3}.

The proof of \eqref{eq:K3} is more involved. If $|\om| \leq \ga^{-3}$, then the bound \eqref{eq:K3} follows from \eqref{eq:K1B}, and we need to prove it for $|\om| \geq \ga^{-3}$. For any function $f \colon \ga \Z^3 \to \R$, we define the discrete Laplacian 
\begin{equation*}
	\uDelta_\ga f (x) := \ga^{-2} \sum_{y \sim x} \bigl(f(y) - f(x)\bigr),
\end{equation*}
where the sum runs over $y \in \ga \Z^3$, which are nearest neighbours of $x$, i.e. $|y - x| = 1$. For a fixed $\om \in \R^3$ we define the function $\ee_\om : x \mapsto e^{-i \pi \ga^3 \om \cdot x}$, for which we have
\begin{equation*}
	\uDelta_\ga \ee_\om(x) = \ff_\ga(\om) \ee_\om(x) \qquad \text{with}\quad \ff_\ga(\om) := -2 \ga^{-2} \sum_{j = 1}^3 \bigl(1 - \cos(\pi \ga^3 \om_j)\bigr).
\end{equation*}
We note that for $\om \neq 0$ we have $\ff_\ga(\om) \neq 0$, and this identity allows to write \eqref{eq:B1} as 
\begin{equation*}
	\widehat{K}_\ga (\om) = \frac{\varkappa_{\ga, 1} \ga^{3}}{\ff_\ga(\om)^m} \sum_{x \in \ga \Z^3} \fK(x) \bigl(\uDelta^m_\ga \ee_\om\bigr)(x),
\end{equation*} 
for any integer $m \geq 0$. After a summation by parts we get
\begin{equation}\label{eq:K-hat-bound}
	\widehat{K}_\ga (\om) = \frac{\varkappa_{\ga, 1} \ga^{3}}{\ff_\ga(\om)^m} \sum_{x \in \ga \Z^3} \bigl(\uDelta_\ga^m \fK\bigr)(x) \,\ee_\om(x).
\end{equation}
The function $\uDelta_\ga^m \fK(x)$ converges uniformly to $\Delta^m \fK$ as $\ga \to 0$, where $\Delta$ is the three-dimensional Laplace operator (recall that $\fK$ is smooth). Hence, $\ga^{3} \sum_{x \in \ga \Z^3} \bigl(\uDelta_\ga^m \fK\bigr)(x) \ee_\om(x)$ can be absolutely estimated by an integral of $|\Delta^m \fK(x)|$. Recalling the scaling \eqref{eq:scalings}, one can see that there is a constant $C > 0$ such that $|\ff_\ga(\om)^{-m}| \leq C ( \ga^{3} |\om| )^{-2m}$ uniformly over $\ga > 0$ and $|\om_j| \leq N+\frac12$ for $j \in \{1, 2, 3\}$. Then from \eqref{eq:K-hat-bound} we get the required bound \eqref{eq:K3} for $k = 0$.

For $k \neq 0$, we use \eqref{eq:K-hat-deriv} and similarly to \eqref{eq:K-hat-bound} we get
\begin{equation*}
D^{k} \widehat{K}_\ga (\om) = \frac{\varkappa_{\ga, 1} \big(-i \pi \ga^3\big)^{|k|_{\sone}}}{\ff_\ga(\om)^m} \ga^{3} \sum_{x \in \ga \Z^3} \bigl(\uDelta_\ga^m \widetilde{\fK}_k\bigr)(x) \,\ee_\om(x),
\end{equation*}
where $\widetilde{\fK}_k(x) := x^k \fK(x)$. Estimating the sum and the function $\ff_\ga$ as before, we get \eqref{eq:K3} for any $k$.

Let us proceed to the proof of \eqref{eq:K2}. From \eqref{eq:K3}, we conclude that there exists $\overline{c}>0$ such that for $|\om| \geq \overline{c} \ga^{-3}$ we have $|\widehat{K}_\ga(\om)| \leq \frac12$. Hence, \eqref{eq:K2} holds for such $\om$. Next, we consider $\om$ such that $|\om| < \underline{c} \ga^{-3}$ for a constant $\underline{c}>0$ to be fixed below. For such $\om$, \eqref{eq:K2.4} implies the existence of $\underline{C}$ such that
\begin{equation*}
1 -  \widehat{K}_\ga (\om)  \geq \pi^2 |\om|^2 \ga^6 - \underline{C} |\om|^3 \ga^9 \geq \bigl( \pi^2 -\underline{C}\,  \underline{c}\bigr) |\om|^2 \ga^6, 
\end{equation*}
which can be bounded from below by $\pi^2 |\om|^2 \ga^6 / 2$ if we choose $\underline{c}$ small enough.

Finally, in order to treat the case $\underline{c} \ga^{-3} \leq |\om | \leq \overline{c} \ga^{-3}$, we observe that the Riemann sums 
\begin{equation*}
\Kg( \ga^{-3} \om ) = \varkappa_{\ga, 1} \ga^{3} \sum_{x \in \ga \Z^3} \fK(x) e^{- i \pi \om \cdot x}
\end{equation*}
approximate $(\SF \fK)(\om)$ uniformly for $ |\om| \in [\underline{c}, \overline{c}]$, where $\SF \fK$ is the continuous Fourier transform on $\R^3$. On the other hand, $\SF \fK$ is the Fourier transform of a probability measure with a density on $\R^3$, and as such, it is continuous and $|(\SF \fK)(\om)|<1$ if $\om \neq 0$. In particular, $|(\SF \fK)(\om)|$ is bounded away from $1$ uniformly for  $ |\om| \in [\underline{c}, \overline{c}]$. Combining these facts, we see that for $\ga$ small enough, $\Kg(\om)$ is bounded away from $1$ uniformly in $\underline{c} \ga^{-3} \leq |\om | \leq \overline{c} \ga^{-3}$.
\end{proof}

The next lemma provides estimates on the kernels $G^\ga$ and $\widetilde{G}^\ga$, defined in \eqref{eq:From-P-to-G} and \eqref{eq:From-P-to-G-tilde} respectively. One way to extend the function $G^\ga_t$ off the grid is by its Fourier transform
\begin{equation*}
	(\SF G^\ga_t)(\om) = \exp \Bigl( \varkappa_{\ga, 3}^2 \bigl( \widehat{K}_\ga(\om) - 1 \bigr) \frac{t}{\alpha} \Bigr) \1_{|\om|_{\sinfty} \leq N}
\end{equation*}
for all $\om \in \R^3$, where $\SF$ is the Fourier transform on $\R^3$ (this formula follows from \eqref{eq:tildeP}, \eqref{eq:From-P-to-G} and the Poisson summation formula). However, such extension is not convenient to work with because its Fourier transform is not smooth, which in particular does not allow to get the bounds in Lemma~\ref{lem:Pgt} below.

In order to define an extension of $G^\ga_t$ with a smooth Fourier transform we use the idea of \cite[Sec.~5.1]{HairerMatetski}. Namely, from \eqref{eq:From-P-to-G} and \eqref{eq:kernels-P} we conclude that the function $G^\ga_t$ solves the equation 
\begin{equation*}
\frac{\d}{\d t} G^\ga_t(x) = \Delta_\ga G^\ga_t(x), \qquad x \in \Lattice,
\end{equation*} 
with the initial condition $G^\ga_0(x) = \delta^{(\eps)}_{x,0}$ (the latter is defined below \eqref{eq:M-bracket} and $\Delta_\ga$ is defined in \eqref{eq:Laplacian-discrete}). Then we can write 
\begin{equation*}
G^\ga_t(x) = \bigl(e^{t \Delta_\ga} \delta^{(\eps)}_{\cdot,0}\bigr)(x), \qquad x \in \Lattice,
\end{equation*}
where $e^{t \Delta_\ga}$ is the semigroup generated by the bounded operator $\Delta_\ga$, acting on the space of bounded functions on $\Lattice$. We applied the semigroup to the function $x \mapsto \delta^{(\eps)}_{x,0}$. We take a Schwartz function $\phi : \R^3 \to \R$, such that $\phi(0) = 1$ and $\phi(x) = 0$ for all $x \in \Z^3 \setminus \{0\}$, and such that $(\SF \phi)(\om) = 0$ for $|\om|_\sinfty \geq \frac{3}{4}$.\footnote{We can define $\phi$ by its Fourier transform $\SF \phi = (\SF \fD) * \psi$, where $\fD(x) = \prod_{j=1}^3 \frac{\sin(\pi x_j)}{\pi x_j}$ is the Dirichlet kernel and $\psi \in \CC^\infty(\R^3)$ is supported in the ball of radius $\frac{1}{4}$ with center at the origin and satisfies $\int_{\R^3} \psi(x) \d x = 1$. Then $(\SF \phi)(\om)$ is smooth and vanishes for $|\om|_\sinfty \geq \frac{3}{4}$, because the Fourier transform of $\fD$ vanishes for $|\om|_\sinfty > \frac{1}{2}$. Moreover, $\phi$ is Schwartz, because its Fourier transform is Schwartz. Finally, $\phi$ takes the required values at the integer points, because $\fD$ takes the same values.} We note that the formula \eqref{eq:Laplacian-gamma} makes sense for all functions $f$ from $\CC_b(\R^3)$, which is the space of bounded continuous functions on $\R^3$, equipped with the supremum norm. Then \eqref{eq:Laplacian-discrete} allows to view $\Delta_\ga$ as a bounded operator acting on $\CC_b(\R^3)$. Setting $\phi^\eps(x) := \eps^{-3} \phi(\eps^{-1} x)$ we then define the extension of $G^\ga_t$ off the grid by 
\begin{equation}\label{eq:G-extension}
G^\ga_t(x) := \bigl(e^{t \Delta_\ga} \phi^{\eps}\bigr)(x), \qquad x \in \R^3.
\end{equation}
The respective extension of the function $\widetilde{G}^{\ga}_{t}$ is given by \eqref{eq:From-P-to-G-tilde} for all $x \in \R^3$. The advantage of such definition of the extension is that its Fourier transform is smooth. 

It will be convenient to treat these functions on the space-time domain $\R_+ \times \R^3$. For this, we write $G^\ga(z)$ where $z = (t,x)$ with $t \in \R_+$ and $x \in \R^3$, and we write $D^k G^\ga(z)$ for the mixed derivative of order $k = (k_0, \ldots, k_3) \in \N_0^4$, where the index $k_0$ corresponds to the time variable $t$ and the other indices $k_i$ correspond to the respective spatial variables. We recall the parabolically rescaled quantities $|k|_\s$ and $\| z\|_\s$ defined in Section~\ref{sec:notation}.

\begin{lemma}\label{lem:Pgt}
Let the constant $\ga_0>0$ be as in the statement of Lemma~\ref{lem:Kg}, and let $|t|_a := |t|^{1/2} + a$ for any $a > 0$. Then for every $r \in \N$, $k \in \N_0^4$ with $k_0 \leq r$ and $n \in \N_0$ there is $C > 0$ such that 
\begin{subequations}\label{eqs:G-bounds}
\begin{align}\label{eq:G-bound}
	\big| D^k G^\ga(t,x) \big| &\leq C |t|_\eps^{-3 -|k|_\s + n} \bigl(\| (t,x) \|_{\s} + \eps\bigr)^{-n},\\
	\label{eq:tildeG-bound}
	\big| D^k \widetilde{G}^\ga(t,x) \big| &\leq C |t|_\emezo^{-3 -|k|_\s + n} \bigl(\| (t,x) \|_{\s} + \emezo\bigr)^{-n},
\end{align}
\end{subequations}
uniformly over $(t,x) \in \R^4$ with $t > 0$ and $\gamma \in (0, \ga_0)$.
\end{lemma}

From the bounds \eqref{eqs:G-bounds} we can apply \cite[Lem.~5.4]{HairerMatetski} and get the expansion as described in the beginning of Section~\ref{sec:lift}. For this, we note that the bounds \eqref{eq:tildeG-bound} imply that $\widetilde{G}^\ga$ is a Schwartz function in $x$, which satisfies 
\begin{equation*}
\big| D^k \widetilde{G}^\ga(t,x) \big| \leq C \bigl(\| (t,x) \|_{\s} + \emezo\bigr)^{-3 -|k|_\s}.
\end{equation*}
Moreover, we can smoothly extend $\widetilde{G}^\ga$ to $\R^4$ in the same way as in \cite[Sec.~5.1]{HairerMatetski}, so that $\widetilde{G}^\ga(t) \equiv 0$ for $t < 0$.

\begin{proof}[of Lemma~\ref{lem:Pgt}]
We start with proving \eqref{eq:G-bound}. Using \eqref{eq:tildeP}, the Fourier transform of \eqref{eq:G-extension} equals 
\begin{equation*}
(\SF G^\ga_t)(\om) = (\SF \phi^{\eps}) (\om) \exp \Bigl( \varkappa_{\ga, 3}^2 \bigl( \widehat{K}_\ga(\om) - 1 \bigr) \frac{t}{\alpha} \Bigr),
\end{equation*}
where $(\SF \phi^{\eps}) (\om) = (\SF \phi) (\eps \om)$. Then the inverse Fourier transform yields 
\begin{equation}\label{eq:G-deriv}
D^k G^\ga (t, x) = \int_{\R^3} F^{\ga}_t(\om) e^{2 \pi i \om \cdot x}\, \d \om
\end{equation}
with 
\begin{equation}\label{eq:F-def}
F^{\ga}_t(\om) := (\SF \phi) \bigl(\eps \om\bigr) \Bigl( \bigl( \widehat{K}_\ga(\om) - 1 \bigr) \frac{\varkappa_{\ga, 3}^2}{\alpha} \Bigr)^{k_0} \bigl(2 \pi i \om\bigr)^{\bar k} \exp \Bigl( \varkappa_{\ga, 3}^2 \bigl( \widehat{K}_\ga(\om) - 1 \bigr) \frac{t}{\alpha} \Bigr).
\end{equation}
To bound the integral, we consider two cases: $|\om| \leq \ga^{-3}$ and $|\om| > \ga^{-3}$.

In the case $|\om| \leq \ga^{-3}$, according to \eqref{eq:K2.4} and \eqref{eq:K2}, there exists $c > 0$ such that for all $\ga \in (0, \ga_0)$ we have
\begin{equation*}
|F^{\ga}_t(\om)| \lesssim \bigl( |\om|^2 + \ga^3 | \om|^3 \bigr)^{k_0} |\om^{\bar k}| \exp \bigl(- c |\om|^2 t \bigr) \lesssim |\om|^{|k|_\s} \exp \bigl(- c |\om|^2 t \bigr),
\end{equation*}
where we used the scaling variables \eqref{eq:scalings} and the bound \eqref{eq:c-gamma-2}. Here, we bounded the Fourier transform of $\phi$ by a constant. Restricting the domain of the integration in \eqref{eq:G-deriv} to $|\om| \leq \ga^{-3}$, we estimate the integral by a constant times 
\begin{equation*}
\int_{|\om| \leq \ga^{-3}} |\om|^{|k|_\s} \exp \bigl(- c |\om|^2 t \bigr) \d \om.
\end{equation*}
If $t \geq \ga^6$, then we change the variable of integration to $u = \sqrt t \om$ and the integral can be estimated by $C t^{- (3 + |k|_\s) / 2}$. On the other hand, if $t < \ga^6$, then we change the variable to $u = \ga^{3} \om$ and the integral gets bounded by $C \emezo^{- 3 - |k|_\s}$ (recall that $\emezo \approx \ga^3$).

Now we will consider the case $|\om| > \ga^{-3}$. Since $\varphi$ is Schwartz, the same is true for its Fourier transform, and for any $m \in \N_0$ we have $\bigl| (\SF \phi) \bigl(\eps \om\bigr) \bigr| \lesssim (1 + \eps |\om|)^{-m}$. Using then \eqref{eq:K3} and \eqref{eq:K2}, we get
\begin{align}\label{eq:F-bound}
|F^{\ga}_t(\om)| &\lesssim (1 + \eps |\om|)^{-m} \Bigl( \bigl(1 + | \ga^3 \om|^{-8} \bigr)/ \alpha \Bigr)^{k_0} |\om^{\bar k}| \exp \bigl(- c t / \alpha \bigr) \\
&\lesssim \ga^{-6 k_0} (1 + \eps |\om|)^{-m} |\om|^{|\bar k|_\sone} \exp \bigl(- c \ga^{-6} t \bigr). \nonumber
\end{align}
Then the part of the integral \eqref{eq:G-deriv}, with the domain of integration restricted to $|\om| > \ga^{-3}$, is estimated by a constant times
\begin{equation}\label{eq:integral-bound}
\ga^{-6 k_0} \exp \bigl(- c \ga^{-6} t \bigr) \int_{|\om | > \ga^{-3}} (1 + \eps |\om|)^{-m} |\om|^{|\bar k|_\sone} \d \om \lesssim \ga^{ -6 k_0} \eps^{-3-|\bar k|_\sone} \exp \bigl(- c \ga^{-6} t \bigr).
\end{equation}
The integral is finite as soon as we take $m > |\bar k|_\sone + 3$. If $t \leq \ga^4$, then this expression is bounded by $C \ga^{ -6 k_0} \eps^{-3-|\bar k|_\sone} \lesssim C \eps^{-3 - |k|_\s}$ (recall that $\eps \approx \ga^{-4}$). If $t \geq \ga^4$, then we bound $\exp (- c \ga^{-6} t ) \leq \exp (- c \ga^{-2}/2 ) \exp (- c \ga^{-6} t/2 )$ and the exponentials can be estimated by rational functions as $\exp (- c \ga^{-2}/2 ) \lesssim \ga^{(3 + |\bar k|_\sone)/2}$ and $\exp (- c \ga^{-6} t/2 ) \lesssim (\ga^{-6} t)^{-(3 + |k|_\s)/2}$. Then \eqref{eq:integral-bound} is bounded by $C t^{- (3 + |k|_\s) / 2}$.

From the preceding analysis we conclude that 
\begin{equation}\label{eq:G-bound-1}
\big| D^k G^\ga(t,x) \big| \leq C \big( |t|^{1/2} + \eps \big)^{-3-|k|_\s},
\end{equation}
and to complete the proof of \eqref{eq:G-bound} we need to bound this function with respect to $x$. For this, we need to consider $|x| \geq t^{1/2} \vee \eps$ (the required bound \eqref{eq:tildeG-bound} for $|x| \leq t^{1/2} \vee \eps$ follows from \eqref{eq:G-bound-1}). 

For the function $\ee_x : \om \mapsto e^{2 \pi i \om \cdot x}$ we have $\Delta_\om \ee_x(\om) = |2 \pi i x|^2 \ee_x(\om)$, where $\Delta_\om$ is the Laplace operator with respect to $\om$. Then the function $e^{2 \pi i \om \cdot x}$ in \eqref{eq:G-deriv} can be replaced by $|2 \pi i x|^{-2 \ell} \Delta_\om^\ell\ee_x(\om)$ for any $\ell \geq 0$. Applying a repeated integration by part we get
\begin{equation}\label{eq:tildeG-new}
D^k G^\ga(t,x) = |2 \pi i x|^{-2\ell} \int_{\R^3} \Delta_\om^\ell F^{\ga}_t(\om)\, \ee_x(\om)\, \d \om.
\end{equation}
There are no boundary terms in the integration by parts, because $F^{\ga}_t(\om)$ and its derivatives decay at infinity. The Fa\`{a} di Bruno formula allows to absolutely bound the function inside the integral by a constant multiple of
\begin{align}
&\max_{|n_1|_{\sone} + \cdots + |n_4|_{\sone} = 2\ell} \biggl| \eps^{|n_1|_\sone} D^{n_1} (\SF \phi) \bigl(\eps \om\bigr) D^{n_2}\Bigl( \bigl( \widehat{K}_\ga(\om) - 1 \bigr) \frac{\varkappa_{\ga, 3}^2}{\alpha} \Bigr)^{k_0} \label{eq:G-spatial-bound}\\
&\hspace{5cm}\times \bigl(D^{n_3}\om^{\bar k}\bigr) D^{n_4}\exp \Bigl( \varkappa_{\ga, 3}^2 \bigl( \widehat{K}_\ga(\om) - 1 \bigr) \frac{t}{\alpha} \Bigr)\biggr|, \nonumber 
\end{align}
where the maximum is over $n_1, \ldots, n_4 \in \N_0^3$ with $n_3 \leq \bar k$. As before, we need to consider two cases: $|\om| \leq \ga^{-3}$ and $|\om| > \ga^{-3}$.

In the case $|\om| \leq \ga^{-3}$ from \eqref{eqs:K-bounds} we conclude that 
\begin{equation*}
\Bigl|D^{n}\Bigl( \bigl( \widehat{K}_\ga(\om) - 1 \bigr) \frac{\varkappa_{\ga, 3}^2}{\alpha} \Bigr)\Bigr| \lesssim p_{\ga, n}(\om) \qquad \text{with} \quad 
p_{\ga, n}(\om) = 
\begin{cases}
|\om|^{2 - |n|_{\sone}} &\text{for}~ |n|_{\sone} \leq 2, \\
\ga^{3 (|n|_{\sone}-2)} &\text{for}~ |n|_{\sone} \geq 3,
\end{cases}
\end{equation*}
and Fa\`{a} di Bruno formula yields 
\begin{align*}
&\Bigl| D^{n}\Bigl( \bigl( \widehat{K}_\ga(\om) - 1 \bigr) \frac{\varkappa_{\ga, 3}^2}{\alpha} \Bigr)^{k_0} \Bigr| \\
&\qquad \lesssim \Bigl| \bigl( \widehat{K}_\ga(\om) - 1 \bigr) \frac{\varkappa_{\ga, 3}^2}{\alpha} \Bigr|^{(k_0 - |n|_{\sone}) \vee 0} \max_{r_1 + \cdots + r_{|n|_\sone} = n} \prod_{i = 1}^{|n|_\sone} \Bigl| D^{r_i} \Bigl( \bigl( \widehat{K}_\ga(\om) - 1 \bigr) \frac{\varkappa_{\ga, 3}^2}{\alpha} \Bigr) \Bigr| \\
&\qquad \lesssim |\om|^{2 (k_0 - |n|_{\sone}) \vee 0} \max_{r_1 + \cdots + r_{|n|_\sone} = n} \prod_{i = 1}^{|n|_\sone} p_{\ga, r_i}(\om).
\end{align*}
Combining this bound with Fa\`{a} di Bruno formula and \eqref{eq:K2}, we get
\begin{align*}
&\Bigl|D^{n}\exp \Bigl( \varkappa_{\ga, 3}^2 \bigl( \widehat{K}_\ga(\om) - 1 \bigr) \frac{t}{\alpha} \Bigr)\Bigr| \\
&\qquad \lesssim \exp \Bigl( \varkappa_{\ga, 3}^2 \bigl( \widehat{K}_\ga(\om) - 1 \bigr) \frac{t}{\alpha} \Bigr) \max_{\ell_1 + \cdots + \ell_{|n|_\sone} = n} \prod_{i = 1}^{|n|_\sone} \Bigl| D^{\ell_i} \Bigl( \varkappa_{\ga, 3}^2 \bigl( \widehat{K}_\ga(\om) - 1 \bigr) \frac{t}{\alpha} \Bigr) \Bigr| \\
&\qquad \lesssim t^{|n|_{\sone}} \exp \bigl(- c |\om|^2 t \bigr) \max_{\ell_1 + \cdots + \ell_{|n|_\sone} = n} \prod_{i = 1}^{|n|_\sone} p_{\ga, \ell_i}(\om).
\end{align*}
Using these bounds and $|D^{n_1}(\SF \phi) \bigl(\om\bigr) \lesssim 1$, the expression inside the maximum in \eqref{eq:G-spatial-bound} is estimated by a constant times
\begin{equation*}
\eps^{|n_1|_\sone} |\om|^{2 (k_0 - |n_2|_{\sone}) \vee 0 + |\bar k|_{\sone} - |n_3|_{\sone}} t^{|n_4|_{\sone}} F_{n_2, n_4}(\om) \exp \bigl(- c |\om|^2 t \bigr),
\end{equation*}
with 
\begin{equation*}
F_{n_2, n_4}(\om) := \left(\max_{r_1 + \cdots + r_{|n_2|_\sone} = n_2} \prod_{i = 1}^{|n_2|_\sone} p_{\ga, r_i}(\om)\right) \left(\max_{\ell_1 + \cdots + \ell_{|n_4|_\sone} = n_4} \prod_{i = 1}^{|n_4|_\sone} p_{\ga, \ell_i}(\om)\right).
\end{equation*}
Hence, the part of the integral \eqref{eq:tildeG-new}, in which the integration variables is restricted to $|\om| \leq \ga^{-3}$, is bounded by a multiple of
\begin{equation*}
|x|^{-2\ell} \eps^{|n_1|_\sone} t^{|n_4|_{\sone}} \int_{|\om| \leq \ga^{-3}} |\om|^{2 (k_0 - |n_2|_{\sone}) \vee 0 + |\bar k|_{\sone} - |n_3|_{\sone}} F_{n_2, n_4}(\om) \exp \bigl(- c |\om|^2 t \bigr) \d \om.
\end{equation*}
If $t \geq \ga^6$, then we change the variable to $u = t^{1/2} \om$ and estimate this expression by $C |x|^{-2\ell} t^{\ell - \frac{1}{2}(|k|_\s + 3)}$. If $t < \ga^6$, then we change the variable to $u = \ga^3 \om$ and estimate the preceding expression by a multiple of $|x|^{-2\ell} \eps^{2\ell - (|k|_\s + 3)}$. 

In the case $|\om| > \ga^{-3}$ we use \eqref{eq:K2} and $|D^{n_1}(\SF \phi) \bigl(\om\bigr) \lesssim (1 + |\om|)^{-m}$ to bound the function inside the maximum in \eqref{eq:G-spatial-bound} by a constant multiple of
\begin{align*}
&\eps^{|n_1|_{\sone}} \ga^{3 (|n_2|_{\sone} + |n_4|_{\sone} )} (1 + |\eps \om|)^{-m} |\om|^{|\bar k|_{\sone} - |n_3|_{\sone}} \exp \bigl(- c t / \alpha \bigr) \frac{t^{|n_4|_{\sone}}}{\alpha^{k_0 + |n_4|_{\sone}}} \\
&\qquad = \eps^{|n_1|_{\sone}} \ga^{3 ( |n_2|_{\sone} - |n_4|_{\sone} - 2 k_0)} (1 + |\eps \om|)^{-m} |\om|^{|\bar k|_{\sone} - |n_3|_{\sone}} t^{|n_4|_{\sone}} \exp \bigl(- c \ga^{-6} t \bigr).
\end{align*}
Then the part of the integral \eqref{eq:tildeG-new} for $|\om| > \ga^{-3}$ is bounded by a multiple of
\begin{equation*}
|x|^{-2\ell} \eps^{|n_1|_{\sone}} \ga^{3 ( |n_2|_{\sone} - |n_4|_{\sone} - 2 k_0 )} t^{|n_4|_{\sone}} \exp \bigl(- c \ga^{-6} t \bigr) \int_{|\om| > \ga^{-3}} (1 + |\eps \om|)^{-m} |\om|^{|\bar k|_{\sone} - |n_3|_{\sone}} \d \om.
\end{equation*}
This integral is finite if we take $m$ sufficiently large. We proceed in the same way as in \eqref{eq:integral-bound}. For $t \leq \ga^4$ this expression is bounded by $|x|^{-2\ell} \ga^{3 ( 2\ell - | k|_\s - 3)}$. For $t \geq \ga^4$ we estimate the exponential by a rational function and bound the preceding expression by $C |x|^{-2\ell} t^{\ell - \frac{1}{2}(|k|_\s + 3)}$. 

Taking $n = 2 \ell$, we have just proved that for $|x| \geq |t|^{1/2} \vee \eps$ we have 
\begin{equation*}
\big| D^k G^\ga(t,x) \big| \leq C |t|_\eps^{-3 -|k|_\s + n} |x|^{-n},
\end{equation*}
which together with \eqref{eq:G-bound-1} gives the required bound \eqref{eq:G-bound}. 

The bound \eqref{eq:tildeG-bound} can be proved in a similar way. More precisely, from \eqref{eq:G-deriv} we get
\begin{equation*}
D^k \widetilde{G}^\ga(t,x) = \int_{\R^3} \widehat{K}_\ga(\om) F^{\ga}_t(\om) e^{2 \pi i \om \cdot x} \d \om.
\end{equation*}
The rest of the proof goes in the same way as before, with the only difference that now we use the fact that $\widehat{K}_\ga(\om)$ is Schwartz and for every $m \geq 0$ it satisfies $|D^n \widehat{K}_\ga(\om)| \lesssim \emezo^{|n|_\sone} (1 + \emezo |\om|)^{-m}$. Hence, all the scalings $\eps$ should be replaced by $\emezo$.
\end{proof}

As a corollary of the previous lemma, we can obtain a bound on the periodic heat kernel $\widetilde{P}^\ga$.

\begin{lemma}\label{lem:tilde-P-bound}
In the setting of Lemma~\ref{lem:Pgt} one has the following bound uniformly in $t \geq 0$:
\begin{equation}\label{eq:tilde-P-bound}
\|\widetilde{P}^\ga_t\|_{L^\infty} \leq C |t|_\emezo^{-3}.
\end{equation}
\end{lemma}

\begin{proof}
From \eqref{eq:From-P-to-G} we have $\widetilde{P}^\ga_t(x) = \sum_{m \in 2\Z^3} \widetilde{G}^\ga_t(x + m)$. Using \eqref{eq:tildeG-bound} with any $n > 3$ and estimating the sum by an integral, we get the required bound \eqref{eq:tilde-P-bound}.
\end{proof}

\subsection{Decompositions of discrete kernels}
\label{sec:decompositions}

Lemma~B.3 in \cite{Martingales} allows to apply \cite[Lem.~5.4]{HairerMatetski} for any integer $r \geq 2$ and to write the discrete kernel as $G^\ga = \SK^\ga + \SR^\ga$, where
\begin{enumerate}
\item $\SR^\ga$ is compactly supported and anticipative, i.e. $\SR^\ga(t,x) = 0$ for $t < 0$, and $\| \SR^\ga \|_{\CC^r}$ is bounded uniformly in $\gamma \in (0, 1]$.
\item $\SK^\ga$ is anticipative and may be written as $\SK^\ga = \sum_{n = 0}^{M} K^{\ga, n}$ with $M = - \lfloor \log_2 \eps \rfloor$, where the functions $\{K^{\ga, n}\}_{0 \leq n \leq M}$ are defined on $\R^4$ and have the following properties:
	\begin{enumerate}
		\item
		the function $K^{\ga, n}(z)$ is supported on the set $\{ z :  \|z\|_\s \leq c 2^{-n}\}$ for a constant $c \geq 1$ used in the supports of all functions $\{K^{\ga, n}\}_{0 \leq n \leq M}$;
		\item
		for some $C > 0$, independent of $\ga$, one has
		      \begin{equation}\label{eq:Kn_bound}
			      |D^k K^{\ga, n}(z)| \leq C 2^{n(3 + |k|_\s)},
		      \end{equation}
		      uniformly in $z$, $k \in \N_0^4$ such that $|k|_\s \leq r$, and $0 \leq n < M$; for $n = M$ the bound \eqref{eq:Kn_bound} holds only for $k = 0$ (in particular, the function $K^{\ga, M}$ does not have to be differentiable); 
		\item
		for all $0 \leq n < M$ and $k \in \N_0^4$, such that $|k|_\s \leq r$, one has 
		\begin{equation*}
			\int_{D_\eps} z^k K^{\ga, n}(z)\, \d z = 0;
		\end{equation*}
		for $n = M$ this identity holds only for $k = 0$.
	\end{enumerate}	
\end{enumerate}
Throughout the article we use interchangeably the notations $\SK^\ga(z)$ and $\SK^\ga_t(x)$ (and respectively for other kernels) for a point $z = (t,x)$ with $t \in \R$ and $x \in \R^3$.

In the same way we can write $\widetilde{G}^\ga = \mywidetilde{\SK}^\ga + \mywidetilde{\SR}^\ga$, where the last two functions have the same properties as above, with the only difference that $\mywidetilde{\SK}^\ga$ is decomposed into a sum of  $\mywidetilde{M} = - \lfloor \log_2 \emezo \rfloor$ functions as $\mywidetilde{\SK}^\ga = \sum_{n = 0}^{\mywidetilde{M}} \widetilde{K}^{\ga, n}$.

This decomposition in particular allows to bound function convolved with a discrete heat kernel.

\begin{lemma}\label{lem:P-convolved-with-f}
For a function $f : \Le \to \R$ and for $\eta < 0$ the following bound holds uniformly in $\ga \in (0,1)$ and locally uniformly in $t \geq 0$:
\begin{equation*}
\|P^\ga_t *_\eps f\|_{L^\infty} \leq C |t|_\eps^{\eta} \| f \|^{(\emezo)}_{\CC^\eta}.
\end{equation*}
\end{lemma}

\begin{proof}
Using \eqref{eq:From-P-to-G} we write $P^\ga_t *_\eps f = G^\ga_t *_\eps f$, where $f$ is extended periodically on the right-hand side. Using the decomposition of $G^\ga$ as in the beginning of this section, we get $G^\ga_t *_\eps f = \sum_{n = 0}^{M} K^{\ga, n}_t *_\eps f + \SR^\ga_t *_\eps f$. Since $K^{\ga, n}$ is bounded in a ball of radius $c 2^{-n}$, for a fixed $t \geq 0$ we have $K^{\ga, n}_t \equiv 0$ if $|t|^{1/2} > c 2^{-n}$. Then the preceding sum can be restricted to $0 \leq n \leq M$ satisfying $|t|^{1/2} \leq c 2^{-n}$. Furthermore, the definition \eqref{eq:eps-norm-1} yields 
\begin{equation*}
\| K^{\ga, n}_t *_\eps f \|_{L^\infty} \lesssim 2^{-\eta n} \| f \|^{(\emezo)}_{\CC^\eta}, \qquad\qquad \| \SR^\ga_t *_\eps f \|_{L^\infty} \lesssim \| f \|^{(\emezo)}_{\CC^\eta}.
\end{equation*}
Then for $\eta < 0$ we have 
\begin{equation*}
\|G^\ga_t *_\eps f\|_{L^\infty} \lesssim \sum_{\substack{0 \leq n \leq M: \\ |t|^{1/2} \leq c 2^{-n}}} 2^{-\eta n} \| f \|^{(\emezo)}_{\CC^\eta} \lesssim |t|_\emezo^{\eta} \| f \|^{(\emezo)}_{\CC^\eta},
\end{equation*}
as required. 
\end{proof}

Using the function $\varrho_{\ga, \de}$ defined in \eqref{eq:rho-gamma}, we introduce new kernels $G^{\ga, \de} := G^\ga \star_\eps \varrho_{\ga, \de}$ and $\widetilde{G}^{\ga, \de} := \widetilde{G}^\ga \star_\eps \varrho_{\ga, \de}$. Then the decompositions of the kernels yield $G^{\ga, \de} = \SK^{\ga, \de} + \SR^{\ga, \de}$ and $\widetilde{G}^{\ga, \de} = \mywidetilde{\SK}^{\ga, \de} + \mywidetilde{\SR}^{\ga, \de}$, where $\SK^{\ga, \de} = \sum_{n = 0}^{M} K^{\ga, \de, n}$ and $\mywidetilde{\SK}^{\ga, \de} = \sum_{n = 0}^{\mywidetilde{M}} \widetilde{K}^{\ga, \de, n}$, and all the functions have the same properties as described above. Moreover, from the definition \eqref{eq:rho-gamma} we have the bounds 
\begin{equation*}
\bigl|D^k \bigl(K^{\ga, \de, n} - K^{\ga, n}\bigr)(z)\bigr| \leq C \de^\theta 2^{n(3 + \theta + |k|_\s)}, \qquad \bigl|D^k \bigl(\widetilde{K}^{\ga, \de, n} - \widetilde{K}^{\ga, n}\bigr)(z)\bigr| \leq C \de^\theta 2^{n(3 + \theta + |k|_\s)},
\end{equation*}
for any $\theta \in (0,1]$, as well as $\| \SR^{\ga, \de} - \SR^{\ga} \|_{\CC^{r-1}} \leq C \de^\theta$ and $\|\, \mywidetilde{\SR}^{\ga, \de} - \mywidetilde{\SR}^{\ga} \|_{\CC^{r-1}} \leq C \de^\theta$. 

\bibliographystyle{Martin}
\bibliography{bibliography}

\end{document}